\documentclass[a4paper,11pt]{article}
\usepackage[top=1in, left=1in, right=1in, bottom=1in]{geometry}
\usepackage[colorlinks=true,    
            linkcolor=blue,     
            citecolor=cyan,    
            urlcolor=red]       
           {hyperref}
\usepackage{hyperref, amsthm, amsmath, amssymb, amsfonts, mathrsfs, mathtools, graphicx, xcolor, multicol, geometry, comment, tikz, calc, accents}
\usetikzlibrary{shadings}
\allowdisplaybreaks 
\newtheorem{theorem}{Theorem}[section]
\newtheorem{definition}[theorem]{Definition}   
\newtheorem{remark}[theorem]{Remark}
\newtheorem{proposition}[theorem]{Proposition}
\newtheorem{lemma}[theorem]{Lemma}

\newtheorem{problem}[theorem]{Problem}

\newtheorem{assumption}[theorem]{Assumption} 

\usepackage{bm} 
\newcommand{\vertiii}[1]{{\left\vert\kern-0.25ex\left\vert\kern-0.25ex\left\vert #1 
    \right\vert\kern-0.25ex\right\vert\kern-0.25ex\right\vert}}
    
\newcommand{\dbtilde}[1]{\accentset{\approx}{#1}}        
\newcommand{\vect}[1]{\boldsymbol{\mathbf{#1}}} 
\newcommand{\sfTheta}{\mathsf{\Theta}}
\newcommand{\barbigOmega}{\overline{\Omega}}
\newcommand{\baromega}{\overline{\omega}}
\newcommand{\staromega}{\omega^{\star}}
\newcommand{\barstaromega}{\baromega^{\star}}
\newcommand{\staru}{u^{\star}}

\newcommand{\un}{{u}^{(n)}}
\newcommand{\uk}{{u}^{(m)}}
\newcommand{\uprime}{u^{\prime}} 
\newcommand{\udot}{\delta{u}} 
\newcommand{\udote}{\delta{u}_{\rho}} 
\newcommand{\wdot}{\delta{w}} 
\newcommand{\uddot}{\delta^{2}{u}} 
\newcommand{\wddot}{\delta^{2}{w}} 
\newcommand{\nn}{\vect{n}} 
\newcommand{\VV}{\vect{\theta}} 
 
\newcommand{\Vn}{{\theta}_{n}}

\newcommand{\bigdO}{\partial \Omega} 
\newcommand{\domega}{\partial \omega} 
\newcommand{\delmu}{\delta\mu} 
\newcommand{\dmu}{\nu} 
\newcommand{\ddmu}{[\nu,\nu]} 
\newcommand{\muout}{\mu_{0}} 
\newcommand{\muin}{\mu_{1}} 

\newcommand{\mue}{\mu_{\rho}} 
\newcommand{\etaa}{\eta_{a}}
\newcommand{\cb}{c_{b}}
\newcommand{\cf}{c_{f}}
\newcommand{\adset}{B}
\newcommand{\epszed}{\rho_{0}}
\newcommand{\regu}{R}
\newcommand{\Je}{J_{\rho}}
\newcommand{\Jeprime}{J_{\rho}^{\prime}}
\newcommand{\Jedprime}{J_{\rho}^{\prime\prime}}
\newcommand{\op}[1]{\operatorname{#1}}  
\newcommand{\dive}{\operatorname{div}}  
\newcommand{\dn}[1]{\partial_{\nn}{#1}}    
\newcommand{\intO}[1]{\int_{\Omega}{#1}{\, {d} x}}

\newcommand{\intG}[1]{\int_{\bigdO}{#1}{\, {d}{s}}}  
\newcommand{\intdomega}[1]{\int_{\partial\omega}{#1}{\, {d}{s}}}  

\newcommand{\intOt}[1]{\int_{\Omega_{t}}{#1}{\, {d} x_{t}}}  
\newcommand{\intGt}[1]{\int_{\partial\Omega_{t}}{#1}{\, {d}{s}_{t}}}

\newcommand{\abs}[1]{\left\vert{#1}\right\vert}
\newcommand{\norm}[1]{\left\|{#1}\right\|} 
\newcommand{\indO}[2]{\left\langle{#1},{#2}\right\rangle_{\bigdO}}
\newcommand{\inO}[2]{\left({#1},{#2}\right)_{\Omega}} 
\newcommand{\jump}[1]{\left[{#1}\right]_{\pm}}

\newcommand{\dett}{I_{t}}
\newcommand{\At}{A_{t}}
\newcommand{\bt}{b_{t}}
\newcommand{\wt}{w^{t}}
\newcommand{\ut}{u^{t}}
\newcommand{\zt}{z^{t}}
\newcommand{\alphat}{\alpha^{t}}
\newcommand{\mut}{\mu^{t}}
\newcommand{\dzero}{\delta_{\circ}}

\newcommand{\alert}[1]{{\color{black}{#1}}} 

\newcommand{\mua}{{\mu}} 

\usepackage[normalem]{ulem}
\usepackage{soul}

\title{Shape optimization for piecewise parameter identification in inverse diffusion problems with a single boundary measurement}

\author{%
Manabu Machida\thanks{Department of Informatics, Faculty of Engineering, Kindai University, Higashi-Hiroshima 739-2116, Japan. \texttt{machida@hiro.kindai.ac.jp}} \and
Hirofumi Notsu\thanks{Faculty of Mathematics and Physics, Institute of Science and Engineering, Kanazawa University, Kakumamachi, Kanazawa 920-1192, Japan. \texttt{notsu@se.kanazawa-u.ac.jp}} \and
Julius Fergy Tiongson Rabago\footnotemark[2]\thanks{Faculty of Mathematics and Physics, Institute of Science and Engineering, Kanazawa University, Kakumamachi, Kanazawa 920-1192, Japan. \texttt{jfrabago@gmail.com}}
}

\date{} 

\begin{document}




\maketitle

\let\thefootnote\relax

\begin{abstract}
This paper proposes a unified shape and coefficient optimization approach for inverse problems governed by diffusion equations.
The associated forward problem is considered with a Robin boundary condition, physically motivated in diffuse optical tomography to model partial reflection of light at tissue boundaries.
The main objective is the recovery of the piecewise-defined absorption coefficient together with its underlying interface from a single boundary measurement.
To this end, a shape-based reconstruction approach is formulated in which the interface is introduced as a geometric unknown governing the piecewise structure of the absorption coefficient.
While classical approaches rely on the Fr\'{e}chet derivative with respect to spatially varying parameters, the Eulerian derivative with respect to the interface is additionally exploited.
This leads to a unified framework for the simultaneous recovery of the coefficient and the geometry under the single-measurement setting. Numerical experiments demonstrate the effectiveness of the proposed method, even for complex and non-convex interfaces.
\end{abstract}

\bigskip

\tableofcontents

\section{\alert{Introduction}}\label{sec:Introduction}

In this study, we investigate an inverse problem for the steady-state diffusion equation arising in diffuse optical tomography (DOT).
Our main objective is to reconstruct an unknown piecewise-defined absorption coefficient together with its associated interface from a single boundary measurement.
To this end, we employ a unified shape and coefficient optimization framework combining Eulerian shape derivatives for geometric variations of the interface with Fr\'echet derivatives for parametric variations of the absorption coefficient.

\subsection{{Problem setting and background}}
The governing model is given by the steady-state diffusion equation in a bounded Lipschitz domain $ \Omega \subset \mathbb{R}^d $, $ d \in \{2,3\} $:
\begin{equation}\label{eq:main}
\left\{
\begin{aligned}
  -\dive{\left( {\alpha(x)} \nabla u(x) \right)} + \mua(x) u(x) &= f(x), \quad x \in \Omega,\\
  {\alpha(x)} \dn{u(x)} + \frac{1}{\zeta} u(x) &= 0, \quad x \in \bigdO.
\end{aligned}
\right.
\end{equation}
Here, $\zeta > 0$ and $\dn{}$ denotes the directional derivative with respect to the outward unit vector $\nn$ normal to $\bigdO$.
Furthermore, $\alpha$ is the diffusion coefficient while $\mua$ is the absorption coefficient, and $f > 0$ is the source term.

Equation \eqref{eq:main} governs light propagation in biological tissue and arises as the steady-state model in diffuse optical tomography (DOT) \cite{GroenhuisFerwerdaBosch1983, Schweigeretal1995, Arridge1999, Nickelletal2000}. In DOT, the coefficients of the diffusion equation are determined from boundary measurements \cite{Jiang2011}, where $u$ represents the diffuse fluence rate (energy density up to a constant factor). We refer to \cite{Arridge1999, GibsonHebdenArridge2005} for comprehensive overviews of optical tomography, to the review by Durduran et al.~\cite{Durduranetal2010} for applications in tissue monitoring, and to \cite{Asprietal2024} for recent mathematical and numerical developments in inverse problems for DOT. Beyond optical imaging, \eqref{eq:main} also appears in geophysical applications such as reflection seismology, particularly in formulations based on time-harmonic scalar wave models (see, e.g., \cite{Potthast2006, YamanYakhnoPotthast2013}).

In the case of near-infrared light,
the parameter $\zeta$ appearing in the Robin boundary condition
originates from Fresnel reflection \cite{EganHilgeman1979}.
When light is perfectly absorbed at the boundary,
the appropriate boundary condition reduces to the Dirichlet condition.

In this work, we consider the reconstruction of the absorption coefficient under the assumption that the diffusion coefficient is known. The outgoing light intensity, denoted by $u = u(x)$ for $x \in \partial\Omega$, is measured on the exterior boundary $\partial\Omega$ of the domain; {although measurements on a subboundary are also possible, we restrict ourselves to the full boundary for simplicity. The goal is to recover the absorption coefficient from a single boundary measurement, a setting for which identifiability and stability results have been established for closely related inverse problems, including non-homogeneous Neumann boundary conditions \cite[Thms.~2.2, 3.2]{Harrach2012, Meftahi2021} as well as interface-only identification problems \cite[Thm.~3.1]{Bal2005}. For further identifiability results for coefficient reconstruction in elliptic problems, we refer to \cite{Meftahi2021,Bal2005} and the references therein.}

We assume that the absorption coefficient is piecewise defined, taking different values in distinct subregions of the domain.
This structure naturally places the inverse problem within the framework of inverse geometry problems.
Motivated by this observation, we propose a reconstruction approach based on shape optimization techniques and tools from shape calculus.
To the best of our knowledge, this study is the first to develop a unified shape and coefficient optimization framework for this reconstruction problem, extending existing shape-based inverse problem methodologies such as \cite{Arridgeetal2008}. While related recent studies have employed shape optimization for source identification \cite{WuGongGongZhangZhu2026}, the present work addresses a fundamentally distinct inverse coefficient reconstruction problem.

\subsection{{Reconstruction methods in DOT}}

Inverse problems for optical tomography are severely ill-posed and nonlinear. To solve such inverse problems in practice, iterative schemes and linearized approaches have been employed (e.g.,~\cite{Boasetal2001}). However, the former can easily become trapped in local minima, whereas the latter is valid only for small perturbations of the coefficients in the equation of interest. As an alternative, inverse series do not require linearization \cite{MoskowSchotland2019}. Given the severe ill-posedness of the problem, regularization is commonly incorporated into reconstruction methods. As discussed in \cite{Arridgeetal2008}, many iterative reconstruction methods achieve this by penalizing spatial variations or gradients of the parameters to be reconstructed. Although such regularization improves stability, it may also lead to overly smooth reconstructions in which sharp interfaces and discontinuities are blurred (see, e.g.,~\cite{Meftahi2021,ZhengChengGong2020}).

Since discontinuities arise at interfaces separating different tissue types
such as tumors, hematomas, organs, and regions containing contrast agents or
tracers, it is natural to assume a piecewise structure for optical parameters
such as absorption and scattering coefficients.

Interface- and shape-based reconstruction methods can potentially remove
the limitation of iterative schemes and linearized approaches, which
suppress discontinuities, because geometric information is incorporated
directly into the inverse problem formulation.

In this work, we develop a shape-based reconstruction framework. We
simultaneously recover the interface of inhomogeneity of a coefficient and
the value of the inhomogeneity.

\subsection{{Related works on shape-based reconstruction}}

From the viewpoint of shape optimization,
Bal and Ren~\cite{BalRen2006}
studied the reconstruction of singular surfaces using shape sensitivity analysis and level set techniques.
They compared this framework with the classical reconstruction of inclusions characterized by discontinuous diffusion coefficients,
showing that the velocity-field construction underlying the shape optimization procedure
remains applicable even when the coefficients are generalized beyond piecewise constants,
provided that discontinuities across interfaces are preserved.
Their work highlights the flexibility of shape optimization methods
for handling interfaces and geometrically structured coefficients.
Their framework primarily focused on interface reconstruction and shape evolution.

Shape-based reconstruction methods have also been investigated in DOT.
In particular, Arridge et al.~\cite{Arridgeetal2008}
compared voxel-based and shape-based reconstruction approaches
for recovering optical inclusions in scattering media.
Their work considered both explicit and implicit interface representations,
including explicit boundary parameterizations and level-set formulations; see also \cite{BelhachmiDhifMeftahi2023}.
The study demonstrated that shape-based reconstruction methods
can improve interface localization relative to conventional reconstruction techniques.

Although Arridge et al.~\cite{Arridgeetal2008}
discussed the possibility of simultaneously reconstructing shape and optical parameters,
they also emphasized that, in many practical applications,
reliable estimates of the optical coefficients are often available in most subregions,
so that only the parameters associated with unknown anomalies need to be reconstructed.
Consequently, their computational framework and numerical experiments
primarily focused on shape/interface recovery
under prescribed or partially known optical coefficients.
Moreover, their reconstruction procedures relied mainly on iterative linearization strategies,
such as Gauss--Newton methods combined with Levenberg--Marquardt regularization,
together with explicit geometric parametrizations and boundary element methods (BEM).

\subsection{{Contributions and comparison with previous works}}

Motivated by the above discussions, we consider a shape-based optimization framework for DOT in a single-measurement setting, incorporating geometric structure directly into a unified PDE-constrained inverse problem. In contrast to other reconstruction approaches (see e.g. \cite{Meftahi2021,ZhengChengGong2020}, and the references therein), the interface and the absorption coefficient are treated as coupled unknowns to be recovered simultaneously. Compared with related shape-based inverse problems in the literature (e.g., \cite{Arridgeetal2008,BalRen2006}), where numerical implementations typically focus on shape-only reconstruction, the present work considers a joint reconstruction of the interface and the absorption coefficient.

The reconstruction procedure is implemented within a finite element method (FEM) framework as a first-order iterative Lagrangian-based optimization scheme combining Fr\'{e}chet and Eulerian shape derivatives. As a result, the proposed approach yields a relatively sharper delineation of the boundary interface, rather than the blurred or diffuse interfaces commonly observed in conventional approaches (see, e.g., \cite{Meftahi2021,ZhengChengGong2020} as well as \cite[Sec.~2]{Arridgeetal2008}). Furthermore, the framework allows for the reconstruction of not only convex interfaces but also non-convex ones.

In addition to the methodological distinctions discussed above,
the proposed framework possesses several characteristic modeling features:
(i) the use of a homogeneous Robin boundary condition
in contrast to the commonly employed controllable non-homogeneous Neumann conditions;
(ii) a formulation outside the standard input/measurement setting of inverse problems,
since no sources are applied on the accessible boundary; and
(iii) the simultaneous recovery of both coefficient values and interfaces
within a single optimization framework.

These modeling choices naturally lead to a coupled variational structure in which both the geometric and parameter dependencies of the cost functional must be taken into account. Accordingly, we rigorously compute the total derivative of the associated cost functional $J$ with respect to the sub-domain $\omega$ and the absorption parameter $\mu$, given by
\[
J'({\omega}, \mu)[\VV, \nu]
=
\frac{\partial J}{\partial {\omega}}({\omega}, \mu)\VV
+
\frac{\partial J}{\partial \mu}({\omega}, \mu)\nu,
\]
under mild regularity assumptions on the unknown interface boundary
(refer to the succeeding sections for the precise meaning of the terms).
For the Eulerian derivative,
our main result is provided by Theorem~\ref{thm:state_shape_derivative_weak_assumptions}
(see also Theorem~\ref{thm:state_shape_derivative})
in subsection~\ref{subsec:rigorous_computation}.
For the optimality condition with respect to $\mu$,
the main result is given by Theorem~\ref{thm:optimality_result}
in subsection~\ref{subsec:first_order_optimality_condition}.

In addition to the theoretical developments,
we complement the proposed framework with extensive numerical experiments (Sections~\ref{subsec:constant_source}--\ref{subsec:point_sources_close_to_the_boundary})
covering configurations with constant and point sources,
non-circular and non-smooth interfaces,
and source locations near the boundary.
These experiments illustrate the robustness and practical performance
of the proposed reconstruction method
across a range of challenging geometric and limited-data settings.


\section{Recovery of the absorption coefficient}\label{sec:recovery_of_the_absorption_coefficient}

We begin by introducing some of the notation used throughout the paper.
The standard $L^{2}(\Omega)$-, $L^{\infty}(\Omega)$-, $H^{1}(\Omega)$-, $L^{2}(\bigdO)$-norms will be used frequently in this paper.
Throughout the paper, $c$ will denote a generic positive constant that may have a different value at different places.
We also occasionally use the symbol `$\lesssim$', which means that if $x \lesssim y$, then we can find some constant $c > 0$ such that $x \leqslant c y$.
Furthermore, $y \gtrsim x$ is defined as $x \lesssim y$.
{Lastly, for economy of space, we sometimes use the shorthand notation $\int_{\Omega_{\pm}} = \int_{\Omega_+} + \int_{\Omega_-}$, and write $\int_{\Omega}$ instead of explicitly splitting the integral, assuming the context makes the division clear. For instance, the integral $\int_{\Omega} \Psi(x) \, dx$ should be understood as $\int_{\Omega} \Psi(x) \, dx = \int_{\Omega_+} \Psi(x) \, dx + \int_{\Omega_-} \Psi(x) \, dx$, where $\Omega_+ = \Omega \setminus \overline{\omega}$ and $\Omega_- = \omega$.}

\subsection{Geometric and coefficient setting}
Let $\Omega \subset \mathbb{R}^d$, $d \in \{2,3\}$, be a bounded open Lipschitz domain such that
$\mathbb{R}^d \setminus \overline{\Omega}$ is connected.
Let $\omega \Subset \Omega$ be an open subdomain with piecewise smooth boundary $\partial \omega$, satisfying the same connected-complement property, i.e.,
$\mathbb{R}^d \setminus \overline{\omega}$ is connected.

Let $f : \Omega \to \mathbb{R}$ be non-constant, unless stated otherwise.
We assume $\alpha = \alpha_0 \chi_{\Omega \setminus \overline{\omega}} + \alpha_1 \chi_{\omega}$, $\alpha_0, \alpha_1 > 0$ (also denoted $\alpha_{+}, \alpha_{-}$, respectively), or, when unspecified, $\alpha \equiv \text{constant} > 0$.
Let $\mu_{\min}, \mu_{\max} \in \mathbb{R}_{+}$ and define
\[
	L^{\infty}_{+}(\Omega) \coloneqq  \left\{ \eta \in L^{\infty}(\Omega) \;\middle|\; \exists \eta_{0} > 0  \text{ such that } \eta \ge \eta_{0} \text{ a.e. in } \Omega \right\}.
\]
The admissible set is
\begin{equation}\label{eq:admissible_set_A}
\mathcal{A}
\coloneqq  \left\{ \mu \in L^{\infty}_{+}(\Omega) \;\middle|\;
\mu_{\min} \le \mu \le \mu_{\max} \right\},
\end{equation}
and we denote its interior by $\mathcal{A}^{\circ} \coloneqq  \left\{ \mu \in L^{\infty}_{+}(\Omega) \;\middle|\;  \mu_{\min} < \mu < \mu_{\max} \right\}$.

\begin{assumption}\label{assumption:weak_assumptions}
Unless stated otherwise, the following hold:
\begin{itemize}
    \item $\omega \in \mathcal{O}_{\circ}^{1}$ is a $C^{1,1}$ subdomain of $\Omega$,
	where $\mathcal{O}_{\circ}^{1}$ is defined precisely in \eqref{eq:admissible_set_of_subdomains};
    \item $\mu = \mu_{+} \chi_{\Omega_{+}} + \mu_{-} \chi_{\Omega_{-}}$, with
    $\mu_{\pm} \in L^{\infty}_{+}(\Omega) \cap W^{1,\infty}(\Omega_{\pm})$;
    \item $\alpha = \alpha_{+} \chi_{\Omega_{+}} + \alpha_{-} \chi_{\Omega_{-}}$, with
    $\alpha_{\pm} \in L^{\infty}_{+}(\Omega) \cap W^{1,\infty}(\Omega_{\pm})$;
    \item $f \in L^{2}(\Omega) \cap H^{1}(\Omega_{\pm})$.
\end{itemize}
\end{assumption}
With the above assumptions in place, we seek to recover $\mua$ from boundary measurements $h$ on $\partial\Omega$.
The inverse problem is then formulated as follows.
\begin{problem}[Absorption coefficient recovery]\label{prob:absorption_coefficient_recovery}
Given $\alpha, \zeta \in \mathbb{R}_{+}$ and sufficiently regular function $f : \Omega \to \mathbb{R}$, $f\neq 0$, find $\mu \in \mathcal{A}$ {from boundary measurements
\[
h = u(\mu)|_{\partial\Omega},
\]
where $u = u(\mu) \in H^1(\Omega)$ is the solution of \eqref{eq:main}.}
\end{problem}
%
%
We note that boundary measurements alone cannot uniquely determine a general coefficient, making the problem ill-posed.
This implies that existence, uniqueness, and stability of solutions are not guaranteed \cite{Hadamard1923}.
Regularization methods, such as Truncated Singular Value Decomposition (TSVD) \cite{Isakov2006, KaipioSomersalo2005}, iterative regularization \cite{AlberRyazantseva2006,BakushinskyKokurin2004,KaipioSomersalo2005}, and Tikhonov regularization \cite{BakushinskyKokurin2004,EnglKunischNeubauer1989,Isakov2006,TikhonovArsenin1977}, are commonly used to address this issue.

In this paper, we use Tikhonov regularization. {Stability and convergence with respect to noisy data can be established within the classical framework \cite{ChaventKunisch1994, EnglKunischNeubauer1989, EnglHankeNeubauer1996}. The arguments are standard and can be carried out in a similar manner as in \cite{ZhengChengGong2020}; hence they are omitted for brevity.} The Tikhonov-regularized formulation leads to the following minimization problem:
\begin{equation}\label{eq:minimization_mu}
	\mu_{\rho}^{\star}
	= \op{argmin}_{\mu \in {\mathcal{A}}} \left\{ J(\mu) + \regu(\mu,\rho) \right\},
\end{equation}
where
\[
	J(\mu) \coloneqq  \frac{1}{2} \norm{u(\mu) - h}_{L^{2}(\bigdO)}^{2}
	\qquad \text{and} \qquad
	\regu(\mu,\rho) \coloneqq  \frac{\rho}{2} \norm{\mu}^{2}_{L^{2}(\Omega)}.
\]
Here, $\rho > 0$ is a Tikhonov regularization parameter that controls the relative influence of the data misfit and the regularization term.
This regularization is employed to mitigate the ill-posedness of the inverse coefficient problem and to promote stable reconstructions.

In this study, we will depart from the use of multiple measurements, which require solving several partial differential equations, and instead employ shape optimization techniques to recover the absorption coefficient from a single measurement.
We focus on recovering the absorption coefficient when it is given by $\mua = \mu_{0} \chi_{\Omega\setminus\baromega} + \mu_{1} \chi_{\omega}$,
where $\mu_{0}, \mu_{1} \in \mathbb{R}_{+}$ and $\chi_{\omega}$ denote the characteristic function of $ \omega \Subset \Omega $.

In the following subsections, we present several preliminary results, with proofs in Appendix~\ref{appx:proofs}.
The main result is given in Theorem~\ref{thm:optimality_result}.
\subsection{\small Well-posedness of the state and continuity of the coefficient-to-solution map}
Let us define the following sesquilinear and linear forms for
$u,v \in V\coloneqq H^{1}(\Omega)$:
\begin{equation}\label{eq:bilinear_and_linear_forms}
\left\{
\begin{aligned}
	a(u,v) &= \intO{( {\alpha} \nabla{u} \cdot \nabla{v} + {\mu} {u}{v})}
				+ \frac{1}{\zeta} \intG{{u}{v}} ,\\
	l(v) &= \intO{fv}.
\end{aligned}
\right.
\end{equation}
For later reference, we write the bilinear form as $a(\mu;\, \cdot, \cdot)$ to explicitly indicate its dependence on the reaction coefficient $\mu$.

The weak formulation of \eqref{eq:main} can be stated as follows:
\begin{problem}\label{prob:weak_form_optimal_tomography}
	Find $u \in {V}$ such that $a(u,v) = l(v)$, for all $v \in V$.
\end{problem}
We have the following well-posedness result with respect to Problem~\ref{prob:weak_form_optimal_tomography}.
\begin{lemma}\label{lem:wellposedness_weak_formulation}
	Let ${\alpha}, \zeta \in \mathbb{R}_{+}$, ${\mu}\in {\mathcal{A}}$, and $f\in H^{-1}(\Omega)$.
	Then, there exists a unique weak solution $u \in V$ to Problem~\ref{prob:weak_form_optimal_tomography}.
\end{lemma}
The existence of a weak solution to Problem~\ref{prob:weak_form_optimal_tomography} remains guaranteed when $\alpha \in L^{\infty}_{+}(\Omega)$.
Based on Lemma~\ref{lem:wellposedness_weak_formulation}, for each ${\mu} \in {\mathcal{A}}$, Problem~\ref{prob:weak_form_optimal_tomography} is well-posed.
We define the functional $F \coloneqq  F(\mu)$ as follows:
\begin{equation}\label{eq:mapping_F}
	F: \mathcal{A} \longrightarrow V, \quad \mathcal{A} \ni {\mu} \longmapsto F({\mu}) = u(\mu) \in V,
\end{equation}
where $u = u(\mu)$ solves Problem~\ref{prob:weak_form_optimal_tomography}.
We drop $\mu$ when there is no confusion.
\begin{proposition}\label{prop:boundedness_of_F}
	The map $F$ is continuous in {$\mathcal{A}$} and
	\[
		\norm{F(\tilde{\mu})-F(\dbtilde{\mu})}_{V} \lesssim \norm{\tilde{\mu} - \dbtilde{\mu}}_{L^{\infty}(\Omega)}.
	\]
\end{proposition}
{The mapping \eqref{eq:mapping_F} is differentiable with respect to $\mu$ for all $\mu$ belonging to the interior ${\mathcal{A}^{\circ}}$.}

We make the following key assumption.
\begin{assumption}\label{key_assumption}
    The admissible set $\adset$ is a finite-dimensional, closed convex subset of $\mathcal{A}$.
\end{assumption}
For example, $\adset$ is the set of piecewise constant functions defined on a fixed partition of $\Omega$.

{Let $\mu \in {\mathcal{A}^{\circ}}$.
Then, for any $\nu \in {\adset}$ such that $\mu + \nu \in {\mathcal{A}}$, $F(\mu + \nu) \in V$ is well-defined.} Let ${\wdot} = F(\mu + \nu) - F(\mu)$.
By the definition of $F$, ${\wdot} = {\wdot}(x)$ satisfies the equation
\begin{equation}\label{eq:strong_form_of_delta_u}
    \left\{
    \begin{aligned}
      -\dive{\left( {\alpha} \nabla{\wdot} \right)} + (\mu + \nu) {\wdot} &= - \nu u, \quad \text{in }  \Omega,\\
      {\alpha} \dn{{\wdot}} + \frac{1}{\zeta} {\wdot} &= 0, \quad \text{on }  \bigdO.
    \end{aligned}
    \right.
\end{equation}
In variational form, ${\wdot}={\wdot}(x) \in V$ satisfies
\begin{equation}\label{eq:weak_form_del_u}
	{a(\mu + \nu; {\wdot},v)} = -(\nu u, v)_{\Omega}, \quad \forall v \in V.
\end{equation}
We now state the following proposition.

\begin{proposition}\label{prop:state_derivative}
    For any {$\mu \in \mathcal{A}^{\circ}$ and $\nu \in {\adset}$ with $\mu + \nu \in \mathcal{A}$} , $F(\mu)$ is differentiable with respect to $\mu$, and the sensitivity $\udot \coloneqq  \udot(\mu)[\nu] = DF(\mu)\nu$ uniquely satisfies the equation
        	\begin{equation}\label{eq:sensitivity_equation}
        \left\{
        \begin{aligned}
        -\dive{\left( {\alpha} \nabla{\udot}\right)} + \mu \udot &= - \nu u  , \quad x \in \Omega,\\
        {\alpha} \dn{\udot} + \frac{1}{\zeta} \udot &= 0, \quad x \in \bigdO,
        \end{aligned}
        \right.
        	\end{equation}	
     with the variational form
        	\begin{equation}\label{eq:sensitivity_equation_variational_form}
        		{a(\mu;\udot,v)} = -(\nu u, v)_{\Omega}, \quad \forall v \in V,
        \end{equation}
    where $u = F(\mu)$.
    Furthermore, there exists a constant $c>0$ such that\footnote{{Let $X$ and $Y$ be Banach spaces. We denote by $\mathscr{B}(X;Y)$ the space of all bounded linear operators from $X$ to $Y$, equipped with the operator norm $\norm{T}_{\mathscr{B}(X;Y)} \coloneqq  \displaystyle\sup_{\norm{\psi}_X \leqslant 1} \norm{T \psi}_Y$, $T \in \mathscr{B}(X;Y)$.}}
    \begin{equation*}
        \norm{DF(\mu)}_{\mathscr{B}(\mathcal{A};V)} \leqslant c.
    \end{equation*}
\end{proposition}
For future needs, we compute the second-order derivative of $F$.
For this purpose, let us denote $\wddot = DF(\mu + \nu_{1})\nu_{2} - DF(\mu)\nu_{2}$,
where $DF(\mu + \nu_{1})\nu_{2}$ and $DF(\mu)\nu_{2}$ are the derivatives of the map $F$ at points $\mu + \nu_{1}$ and $\mu$, respectively.
In view of Proposition~\ref{prop:state_derivative}, we can easily deduce that $\udot = \udot(\mu) = DF(\mu)\nu_{2}$ satisfies the following
\[
        \left\{
        \begin{aligned}
          - \dive{ \left({\alpha} \nabla{\udot} \right)} + \mu \udot &= - \nu_{2} u, \quad \text{in } \Omega,\\
          {\alpha} \dn{\udot} + \frac{1}{\zeta} \udot &= 0, \quad \text{on } \bigdO,
        \end{aligned}
        \right.
\]
where $u = u(\mu) = F(\mu)$.
A straightforward computation shows that
\[
        \left\{
        \begin{aligned}
          -\dive{\left( {\alpha} \nabla{\wddot} \right)} + (\mu + \nu_{1}) \wddot &= - \nu_{2} {u}(\mu+\nu_{1}) + \nu_{2} {u}(\mu) - \nu_{1} \udot(\mu), \quad \text{in } \Omega,\\
          {\alpha} \dn{\wddot} + \frac{1}{\zeta} \wddot &= 0, \quad \text{on } \bigdO.
        \end{aligned}
        \right.
\]
In variational form, we have ${\wddot}={\wddot}(x) \in V$ satisfies
\begin{equation}\label{eq:weak_form_del_del_u}
\begin{aligned}
	{a(\mu + \nu_{1}; {\wddot},v)}
		&=  -(\nu_{2} [ F(\mu+\nu_{1}) - F(\mu) ]  , v)_{\Omega} - ( \nu_{1} DF(\mu) \nu_{2} , v)_{\Omega},
\end{aligned}
\end{equation}
for all $v \in V$.

The above series of computations lead us to the next proposition.
\begin{proposition}\label{eq:state_second_derivative}
	For any $\mu \in {\mathcal{A}^{\circ}}$, $F(\mu)$ is twice-differentiable with respect to $\mu$, and $\uddot  \coloneqq  \uddot(\mu)[\nu_{1}, \nu_{2}]= D^{2}F(\mu)[\nu_{1},\nu_{2}] \in V$ uniquely satisfies the variational equation
	\begin{equation}\label{eq:second_order_sensitivity_equation_variational_form}
        		{a(\mu;\uddot,v)} = -(\nu_{2} DF(\mu)\nu_{1}  , v)_{\Omega} - ( \nu_{1} DF(\mu) \nu_{2} , v)_{\Omega} , \quad \forall v \in V.
        \end{equation}
        with $u = F(\mu)$.
        Additionally, there exists a constant $c>0$ such that $DF(\mu)$ satisfies the estimate
        \begin{equation}\label{eq:boundedness_of_Hessian}
            \norm{D^{2}F(\mu)}_{\mathscr{B}({\mathcal{A} \times \mathcal{A}};V)} \leqslant c.
        \end{equation}
\end{proposition}
%
%
%
%
\subsection{Regularized problem: well-posedness and {first-order optimality condition}}\label{subsec:first_order_optimality_condition}
We now examine the regularized version of Problem~\ref{prob:absorption_coefficient_recovery} within the optimization framework \eqref{eq:minimization_mu}.
For all $\mu \in L^{\infty}_{+}(\Omega)$, let $u(\mu) \in V$ denote the weak solution of \eqref{eq:main}, solving Problem~\ref{prob:weak_form_optimal_tomography} for the given $\mu$.
From \eqref{eq:minimization_mu}, we define
{
\begin{equation}\label{eq:regularized_cost_function}
	\Je(\mu) \coloneqq  J(\mu) + \regu(\mu,\rho) = \frac{1}{2} \norm{u(\mu) - h}_{L^{2}(\bigdO)}^{2} + \frac{\rho}{2} \norm{\mu}^{2}_{L^{2}(\Omega)}.
\end{equation}
}

We now define the following problem.
%
\begin{problem}\label{prob:minimization_of_mu}
	Find $\mue \in \adset \subset \mathcal{A}$ such that $\mue = \inf_{\mu \in \adset} \Je(\mu)$.
\end{problem}
In this section, we want to prove that the objective functional is convex under specific conditions, and its minimizer in $\adset$ is unique.
We start by noting that for any $\mu \in \adset$ and $\dmu \in \adset$ with $\mu + \dmu \in \adset$, formally, we have
\begin{align}
	\Jeprime(\mu)\dmu &= \indO{u(\mu) - h}{\udot} + \rho\inO{\mu}{\dmu},\label{eq:derivative_of_J}\\
	\Jedprime(\mu)\ddmu &= \norm{\udot}_{L^{2}(\bigdO)}^{2} + \indO{{u(\mu) - h}}{\uddot} + \rho\norm{\dmu}_{L^{2}(\Omega)}^{2},\label{eq:second_derivative_of_J}
\end{align}
where $\udot = DF(\mu)\dmu$ is the unique solution to the variational equation \eqref{eq:sensitivity_equation_variational_form} while $\uddot = D^{2}F(\mu)[\dmu,\dmu]$ uniquely solves the variational equation \eqref{eq:second_order_sensitivity_equation_variational_form}.

On a side note, we underline here that $\nu \in L^{\infty}(\Omega) \subset L^{2}(\Omega)$ and $\norm{\nu}_{L^{2}(\Omega)} \leqslant \abs{\Omega}^{1/2}\norm{\nu}_{L^{\infty}(\Omega)}$ (see \cite[Prop.~6.12, p.~186]{Folland1999} or \cite[Thm.~1.5.5(c), p.~46]{AtkinsonHan2009}).

\begin{proposition}[Strict convexity of $\Je$]\label{prop:strict_convexity_of_J}
	There exists a constant $\epszed > 0$ independent of $\mu \in \adset \subset \mathcal{A}$, such that for all $\rho > \epszed$, $\Je(\mu)$ is strictly convex.
\end{proposition}

{We emphasize that, since $\rho$ is a regularization parameter, it cannot be chosen arbitrarily large in practice. While strict convexity can be guaranteed under a sufficiently strong regularization regime (see Appendix~\ref{subsubsec:strict_convexity_proof} for the proof of Proposition~\ref{prop:strict_convexity_of_J}), this condition is generally not satisfied in applications, particularly in the presence of noisy measurements.} Hence, the objective functional is not convex in general, and uniqueness of Problem~\ref{prob:minimization_of_mu} cannot be ensured.
Nevertheless, the numerical experiments in Section~\ref{sec:numerical_studies} suggest that the proposed method exhibits stable behaviour and produces reasonable reconstructions.

\begin{remark}
We emphasize that Assumption \eqref{key_assumption} is essential for the proof of Proposition~\ref{prop:strict_convexity_of_J}.
The inequality $\norm{\cdot}_{L^{\infty}(\Omega)} \lesssim \norm{\cdot}_{L^{2}(\Omega)}$ generally does not hold (cf. \cite[Prop.~6.12, p.~186]{Folland1999} or \cite[Thm.~1.5.5(c), p.~46]{AtkinsonHan2009}).
However, for some functions $\varphi$ on $\Omega$, the inequality $\norm{\varphi}_{L^{\infty}(\Omega)} \lesssim \norm{\varphi}_{L^{2}(\Omega)}$ does hold. Specifically, for a fixed $k$ and a bounded set $\Omega$, every $\varphi \in {P}_{k}(\Omega)$ (the set of polynomials of degree at most $k$ on $\Omega$) satisfies $\norm{\varphi}_{L^{\infty}(\Omega)} \leqslant c \norm{\varphi}_{L^{2}(\Omega)}$, where the constant $c>0$ depends on $k$.
\end{remark}
Next, we characterize the minimizer of the G\^{a}teaux differentiable convex functional $\Je$ (see \cite[Thm.~5.3.19, p.~233]{AtkinsonHan2009}), which is the main result of this section.
We provide a well-posedness result and the first optimality condition for the solution of Problem~\ref{prob:minimization_of_mu} as follows:
\begin{theorem}\label{thm:optimality_result}
    Let $B$ satisfy Assumption \ref{key_assumption}, and let $\rho > 0$ such that $\Je$ is strictly convex.
    Then, Problem~\ref{prob:minimization_of_mu} has a unique solution $\mue \in {\adset}$, which depends continuously on all data with $h \in H^{1/2}(\bigdO)$.
    Moreover, $\mue$ satisfies the inequality
    \begin{equation}\label{eq:first_order_optimality_condition}
        ( - {u}{p} + \rho \mue, \delmu - \mue)_{\Omega} \geqslant 0, \quad \forall \delmu \in {\adset},
    \end{equation}
    where $p = {p}(\mue) \in {V}$ is the solution to the following PDE system
    \begin{equation}\label{eq:first_main_adjoint}
    \left\{
    \begin{aligned}
      -\dive{\left( {\alpha} \nabla {p} \right)} + \mue {p} &= 0, \quad \text{in } \Omega,\\
      {\alpha} \dn{{p}} + \frac{1}{\zeta} {p} &= u_{\rho} - h, \quad \text{on } \bigdO,
    \end{aligned}
    \right.
    \end{equation}
    with $u_{\rho} = F(\mue)$, where $ u_{\rho} \in {V}$ uniquely solves \eqref{eq:main} with $\mua$ replaced by $\mue$.
\end{theorem}
\begin{proof}
    Let $ u_{\rho} \in {V}$ be the unique weak solution of \eqref{eq:main} with absorption coefficient $\mue \in {\adset}$.
    For convenience, we drop $\cdot_{\rho}$ in $\mue$ during the proof.
    By assumption, $B$ is a closed, convex set in the Hilbert space {$\mathcal{A}$}, and $\Je$ is strictly convex (Proposition~\ref{prop:strict_convexity_of_J}).
    Using the standard result for convex minimization problems \cite[Thm.~5.3.19, p.~233]{AtkinsonHan2009}, there exists a unique stable solution $\mu \coloneqq  \mue \in {\adset}$ to Problem~\ref{prob:minimization_of_mu}, characterized by the optimality condition
    \begin{equation}\label{eq:optimality_inequality}
    	\Jeprime(\mu)(\delmu - \mu) \geqslant 0, \quad \forall \delmu \in {\adset}.
    \end{equation}

    To verify this, we first calculate the derivative of $\Je$ with respect to $\mu$.
    Let $\dmu = {\delmu - \mu}$ and $\udote = DF(\mu)\nu$.
    Then, $\udote \in {V}$ solves \eqref{eq:sensitivity_equation_variational_form} with $\mu = \mue$. From \eqref{eq:derivative_of_J}, the inequality \eqref{eq:optimality_inequality} becomes $\indO{u(\mu) - h}{\udote} + \rho\inO{\mu}{\dmu} \geqslant 0$, for all $\delmu \in {\adset}$.

    Next, we eliminate $\udote$ by introducing the adjoint system {\eqref{eq:first_main_adjoint}} and {multiply} both sides of the first equation by $\udote$.
    {Then, integrating} over $\Omega$ and applying integration by parts gives $\indO{u(\mu) - h}{\udote} = {a(\mu; p, \udote)}$.

    Finally, taking $v = p \in {V}$ in \eqref{eq:sensitivity_equation_variational_form} yields ${a(\mu; \udote, p)} = -(\nu u, p)_{\Omega}$. Since ${a(\mu; \cdot, \cdot)}$ is symmetric, combining the equations results in the inequality $(-{u}{p} + \rho \mu, \dmu)_{\Omega} \geqslant 0$, which holds for all $\delmu \in {\adset}$, where $\dmu = {\delmu - \mu}$ and $\mu = \mue \in {\adset}$. This proves the proposition.
\end{proof}
%
%
%
%
%
%
\section{Shape recovery of the boundary interface}\label{sec:approximation}
The numerical solution of Problem~\ref{prob:minimization_of_mu} via a gradient-based descent method using \eqref{eq:derivative_of_J} fails to provide a satisfactory reconstruction of the absorption coefficient from a single measurement.
A single measurement is insufficient for reasonable recovery, requiring multiple measurements \cite{Meftahi2021}.
Additionally, domain integral-type {cost functionals}  are more effective than boundary integral-type ones for this recovery process \cite{ZhengChengGong2020, Meftahi2021}.
This study proposes the use of shape optimization techniques to improve the recovery of the absorption coefficient, while retaining the {cost functional} form in \eqref{eq:minimization_mu} and relying on a single measurement.

{We assume that $\muout$ is known.
Then, as} a result of the proposed strategy, which includes the regularization term $\regu$ for $\mu$, we minimize the regularized {objective functional}:
\begin{equation}\label{eq:parameter_and_shape_minimization}
	{{\min_{(\muin,\omega)} \Je(\muin,\omega)
		\equiv \min_{(\muin,\omega)} \left\{ J(\muin,\omega)+ \regu(\mu,\rho) \right\} }}
\end{equation}
where $J$ and $\regu$ are defined in \eqref{eq:regularized_cost_function}.
Hereinafter, we write {$J({\omega}) = J(\mu_{1}, \omega)$}, focusing on the variation of $J$ with respect to the sub-domain $\omega$.

Looking at \eqref{eq:parameter_and_shape_minimization}, we highlight that the {objective functional} $J$ depends not only on $\mu$ but also on $\partial\omega$ through the solution $u = u(\mu, \partial\omega)$ to \eqref{eq:main}.
The optimal solution $\omega^{\ast} = \omega^{\ast}(\mu)$, if it exists, depends on $\mu$ via the state equation \eqref{eq:main}.
\begin{remark}
The addition of $\regu$ in \eqref{eq:parameter_and_shape_minimization} addresses the ill-posedness of the minimization.
Regularization for both $\mu$ and $\partial\omega$ enhances stability and improves the approximation of the minimizer as a solution to the inverse problem of recovering $\mu$ and $\partial\omega$.
However, regularizing only $\mu$ is sufficient for stability in numerical approximations, as shown in Section~\ref{sec:numerical_studies}.
\end{remark}
We seek a minimizer of $J$ associated with Problem~\ref{prob:minimization_of_mu}, with $\mu$ prescribed in $\Omega \setminus \baromega$ and in $\omega$.
{The existence of an optimal solution is established in Appendix~\ref{subsec:existence_of_optimal_shape_solution}.}
In the numerical scheme, the state system \eqref{eq:main} is solved iteratively for fixed $\mu$ and $\omega$, followed by updates of both variables based on the derivatives of $J$ with respect to $\mu$ and $\omega$.
This approach, based on variational calculus, calculates the gradients of $J$ with respect to both $\mu$ and $\omega$ in \eqref{eq:main}, accounting for $\omega$ through $\mu$.
Introducing an adjoint variable associated with the measurement $h$, we derive a kernel representation of the total derivative, essential for gradient-based algorithms to minimize $J$.
This also shows the equivalence of the unique minimizer and the critical point of $J$.

We underline here that the proposed approach eliminates the need to compute solutions for the PDEs associated with various input data and multiple measurements.
However, in exchange, along with solving two PDE systems (corresponding to the state and adjoint state problems), we must also compute the solution of a vector-type Laplace equation.
This corresponds to the computation of the extended-regularized deformations field characterized by the so-called shape gradient of $J$; i.e., the variation of $J$ with respect to the region of interest $\omega$.

\subsection{Shape sensitivity analysis}\label{sec:shape_sensitivity_analysis}
The objective of this section is to calculate the {Eulerian derivative} of $J$ with respect to $\omega$ using the chain rule, assuming the {Eulerian derivative} of the state $u$ exists.

\subsubsection{Notations and some definitions}
Let us fix some notations.
We denote by $\nn$ the outward unit normal to $\partial\omega$ pointing into $\Omega \setminus \baromega$.
Thus, $\dn{u}_{-}$ (respectively, $\dn{u}_{+}$) is the normal derivative from the inside of $\omega$ (respectively, $\Omega \setminus \baromega$) at the interface $\partial\omega$, and $\jump{\cdot}$ denotes the jump across the same interface.

We fix a small number $\dzero > 0$ and define the subdomain $\Omega_{{\circ}} \Subset \Omega$ with ${{C}}^{\infty}$ boundary as follows:
\[
	\Omega_{{\circ}}  \coloneqq  \{ x\in \Omega \mid \operatorname{dist}(x,\bigdO) \geqslant \dzero\}.
\]
Let $k \in \mathbb{N}$.
We define $\mathcal{O}_{{\circ}}^{k}$ as the set of all simply connected subdomain $\omega$ with ${{C}}^{k,1}$ boundary $\partial\omega$ such that $\operatorname{dist}(x,\bigdO) > \dzero$ for all $x \in \omega$; i.e.,
%
%
\begin{equation}\label{eq:admissible_set_of_subdomains}
	\mathcal{O}_{{\circ}}^{k}\coloneqq  \left\{ \omega \Subset \Omega_{{\circ}}  \;\left|\;
\begin{aligned}
&  \operatorname{dist}(x,\bigdO) > \dzero, \forall x\in\omega, \Omega\setminus\baromega \ \text{is connected, }\\
&\text{and $\omega$ is of class ${{C}}^{k,1}$}
\end{aligned} \right. \right\}.
\end{equation}

We call $\mathcal{O}_{{\circ}}^{k}$ the set of admissible geometries or interface boundaries.
Notably, the inclusions are assumed to be bounded away from the accessible boundary $\bigdO$, and $\Omega \setminus \baromega$ is connected.
Hereinafter, we call $\Omega$ an \textit{admissible domain} if $\Omega$ contains a subdomain $\omega \in \mathcal{O}_{{\circ}}^{k}$.

The admissible set of interface boundaries is described by a particular class of perturbations of the domain $\Omega$.
We denote by $\VV$ a ${{C}}^{k,1}$ regular vector field with compact support in $\Omega_{{\circ}}$, and let ${{\sfTheta}^{k}}$ stands for the collection of all admissible deformation fields; i.e., we define
\begin{equation}\label{eq:admissible_deformation_fields}
	{{\sfTheta}^{k}} \coloneqq  \{\VV \in {{C}}^{k,1}(\mathbb{R}^{d}) \mid \operatorname{supp}(\VV) \subset \overline{\Omega}_{{\circ}}  \}.
\end{equation}
{For exactness, we assume that there exists $\delta > 0$ such that
$\{x \in \Omega \mid \operatorname{dist}(x,\partial\omega) > \delta/2\} \subset \Omega_{\circ} \subset \{x \in \Omega \mid \operatorname{dist}(x,\partial\omega) > \delta/3\}$, with $\delta$ chosen sufficiently small so that $\delta < 2\,\operatorname{dist}(\partial\omega,\partial\Omega)$.}\footnote{{This assumption ensures a uniform positive separation between the interface $\partial\omega$ and the interior region $\Omega_{\circ}$,
thereby preventing topological changes and guaranteeing the validity of the shape sensitivity analysis.
}}
For $\VV \in {{\sfTheta}^{k}}$, we let $\Vn \coloneqq  \langle \VV, \nn \rangle$ be its normal component.

Let us define $T_{t} : \overline{\Omega} \longmapsto \overline{\Omega}$ as the \textit{perturbation of the identity} $id$ {(the $d$-dimensional identity map)} given by $T_{t} = T_{t}({\VV}) = id + t \VV$, where $\VV \coloneqq  (\theta_{1}, \ldots, \theta_{d})^{\top} \in {{\sfTheta}^{k}}$ is a $t$-independent deformation field.
We define $\Omega_{t}\coloneqq T_{t}(\Omega)$, $\bigdO_{t}\coloneqq T_{t}(\bigdO) = \bigdO$, and $\omega_{t}\coloneqq T_{t}(\omega)$, i.e., $\partial\omega_{t}\coloneqq T_{t}(\partial\omega)$.
In addition, $\Omega_{0}$ is such that the interface is given by $\partial\omega_{0}$.

It can be shown that there exists a sufficiently small number $t_{0}>0$ such that for all $t \in \textsf{I}\coloneqq [0,t_{0})$, the transformation $T_{t}$ is a diffeomorphism from $\Omega \in {{C}}^{k,1}$ onto its image (see, e.g., \cite[Thm.~7]{BacaniPeichl2013} for $k=1$).
Hereinafter, we let $t_{0} \in \mathbb{R}_{+}$ be small enough so that ${[t \mapsto T_{t}]} \in {{C}}^{1}({\textsf{I}},{{C}}^{k,1}(\overline{\Omega})^{d})$ and $ {[t \mapsto T_{t}^{-1}]} \in {{{C}}^{1}}({\textsf{I}},{{C}}^{1}(\overline{\Omega})^{d})$ (cf. \cite{IKP2006,IKP2008}).
In fact, here we assume that, for all $t \in \textsf{I}$, $\dett \coloneqq  \det \,{D}T_t > 0$.
Accordingly, we define the set of all admissible perturbations of $\Omega$ denoted by $\mathcal{O}_{ad}^{k}$ as follows:
\begin{equation}\label{eq:admissible_domains}
	\mathcal{O}_{ad}^{k} =
	\left\{T_{t}({\VV})(\overline{\Omega}) \subset {\barbigOmega} \mid \Omega \in {{C}}^{k,1},
	\baromega \in \mathcal{O}_{{\circ}}^{k}, t \in \textsf{I}, \VV \in {{\sfTheta}^{k}} \right\}.
\end{equation}

It is important to note that the fixed boundary $\bigdO$ only needs Lipschitz continuity, not ${C}^{k,1}$ regularity.
However, for simplicity, we assume higher regularity for some $k \in \mathbb{N}$.
The numerical scheme developed in this study applies to domains $\Omega$ with Lipschitz regularity.
%
%

The following regularities hold (see, e.g., \cite{IKP2006,IKP2008} or \cite[Lem.~3.2, p.~111]{SokolowskiZolesio1992}):
\begin{equation}\label{eq:regular_maps}
\left\{
\begin{aligned}
	{[t \mapsto DT_{t}]} &\in {{C}}^{1}({\textsf{I}},{{C}}^{0,1}({\barbigOmega})^{d\times d}),
		&& {[t \mapsto (DT_{t})^{-\top}]} \in {{C}}^{1}({\textsf{I}},{{C}}({\barbigOmega})^{d\times d}),\\	
	{[t \mapsto \dett]} &\in {{C}}^{1}({\textsf{I}},{{C}}({\barbigOmega})),
		&&\ {[t \mapsto \dett]} \in {{C}}^{1}({\textsf{I}},{{C}}^{0,1}({\barbigOmega})),\\
	{[t \mapsto \At]} &\in {{C}}^{1}({\textsf{I}},{{C}}({\barbigOmega})^{d \times d}),
		&& {[t \mapsto \bt]} \in {{C}}^{1}({\textsf{I}},{{C}}(\partial{\omega})),
\end{aligned}
\right.
\end{equation}
where $\At \coloneqq  \dett({D}T_t^{-1})({D}T_t)^{-\top}$.
The derivatives of the maps $[t \mapsto \dett]$, $[t \mapsto \At]$, and $[t \mapsto b_{t}]$ are respectively given by
\begin{equation}\label{eq:derivatives_of_maps}
\left\{
\begin{aligned}
	\frac{d}{dt}\dett \big|_{t=0}
		&= \lim_{{t \to 0^{+}}} \frac{\dett - 1}{t} = {\operatorname{div}}\, \VV,\\
	\quad \frac{d}{dt}\At \big|_{t=0}
		& = \lim_{{t \to 0^{+}}} \frac{\At - id}{t}
		= ({\operatorname{div}}\VV)id -  {D}\VV - ({D}\VV)^\top =: A,
		\\
	\quad\frac{d}{dt}\bt \big|_{t=0}
		& =  \lim_{{t \to 0^{+}}} \frac{\bt - 1}{t}
		= {\operatorname{div}}_{\tau} \VV
		= {\operatorname{div}} \VV \big|_{{\partial\omega}} - ({D}\VV\nn)\cdot\nn,
\end{aligned}
\right.
\end{equation}
where ${\operatorname{div}}_{\tau} \VV$ denotes the tangential divergence of the vector $\VV$ on ${\partial\omega}$.
Furthermore, we assume that, for any $\alpha \in L^{\infty}_{+}(\Omega)$,
\begin{equation}\label{eq:bounds_At_and_Bt}
	\abs{\xi}^2 \lesssim \alpha\At\xi \cdot \xi \lesssim \abs{\xi}^2,
	\qquad \text{for all}\ \xi \in \mathbb{R}^{d}.
\end{equation}

The functional $J : \mathcal{O}_{ad}^{k} \to \mathbb{R}$ has a directional \textit{first-order} \textit{Eulerian derivative} at $\Omega$ in the direction of the field $\VV \in {{\sfTheta}^{k}}$ if the limit
\begin{equation*}
	\lim_{t \searrow0} \frac{J({\omega}_{t}) - J({\omega})}{t} =: {d}J({\omega})[\VV]
\end{equation*}
exists (see, e.g., \cite[Sec.~4.3.2, Eq.~(3.6), p.~172]{DelfourZolesio2011}).
The functional $J$ is said to be \textit{shape differentiable} at $\omega$ {if the limit exists for all $\VV \in {{\sfTheta}^{k}}$ and the} mapping $\VV \mapsto {d}J({\omega})[\VV]$ is both linear and continuous on ${{\sfTheta}^{k}}$.
In such a case, we call the resulting map as the \textit{shape gradient} of $J$.

In the following subsections, we denote the function $\varphi_{t} : \Omega_{t} \to \mathbb{R}$ on the reference domain $\Omega$ using $T_{t}$ as $\varphi^{t} \coloneqq  \varphi_{t} \circ T_{t}: \Omega \to \mathbb{R}$.
%
%
\subsubsection{{Lagrangian} and {Eulerian derivative} of the states}\label{subsec:material_and_shape_derivative_of_the_state}
The following proposition presents the first result of this section, describing the structure of the \textit{Lagrangian} and \textit{Lagrangian} derivatives of the state.
We stress that $H^{1}$ regularity of the state solution is insufficient to justify the existence of its {Eulerian derivative}.
In fact, higher regularity is required.
Therefore, we consider $C^{k,1}$ bounded domains, for some $k\in\mathbb{N}$ and use an elliptic regularity result to obtain $H^{k+1}$ (local) regularity for the state, which is sufficient to prove Theorem~\ref{thm:state_shape_derivative}.
If we only assume that $\Omega$ is a Lipschitz domain, then $u \in H^1(\Omega)$.
However, if $\Omega$ and $\omega$ are of class $C^{k,1}$, and if $\alpha$, $\mu$, and $f$ are sufficiently smooth,
the regularity of $u$ improves locally, with
$u_{+} \in H^{k+1}(\Omega_{\circ} \setminus \overline{\omega})$ and $u_{-} \in H^{k+1}(\omega)$,
where $u_{\pm} = u|_{\Omega_{\pm}}$, $\Omega_{+} = \Omega \setminus \overline{\Omega}_{-}$, and $\Omega_{-} = \omega \Subset \Omega_{\circ}$.

\begin{theorem}\label{thm:state_shape_derivative}
Let the assumptions of Proposition~\ref{thm:optimality_result}
be satisfied, and assume that $\Omega \in \mathcal{O}_{ad}^{k}$ and $\VV \in {{\sfTheta}^{k}}$, for some $k \in \mathbb{N}$.
Then, the state ${u} \in V$, has the material derivative $\dot{u} \in V$ satisfying the following variational equation
\begin{equation}\label{eq:material_derivative}
	\begin{aligned}
	&\int_{\Omega}{( {{\alpha}\nabla{\dot{u}} \cdot \nabla{v} + {{\mu}} {\dot{u}} {v})} } \, dx\\
			& =  -\int_{\Omega}{ \left( {\alpha} A \nabla{u} \cdot \nabla{v} + \operatorname{div}{\VV} {\mu} {u} {v} - \operatorname{div}{\VV} f {v} \right)} \, dx \\
			& \quad - \int_{\Omega}{ \left[ \nabla{{\alpha}} \cdot \VV (\nabla{u} \cdot \nabla{v}) + \nabla{{\mu}} \cdot \VV {u} {v} - \nabla{f}\cdot\VV {v} \right] } \, dx, \quad \forall v \in V.
	\end{aligned}
\end{equation}
{Furthermore, assuming the additional regularity conditions $\alpha \in W^{k, \infty}(\Omega_{\pm})$, $\mu \in W^{k-1, \infty}(\Omega_{\pm})$, and $f \in H^{k-1}(\Omega_{\pm})$, it follows that $u_{+} \in H^{k+1}(\Omega_{\circ} \setminus \overline{\Omega_{-}})$ and $u_{-} \in H^{k+1}(\Omega_{-})$.}
If $u$ satisfies the continuity conditions
	\begin{equation}\label{eq:continuity_equations_on_the_interface}
		\jump{u} = 0
		\quad \text{and} \quad
		\jump{ \alpha \dfrac{\partial{u}}{\partial{\nn}} } = 0\qquad \text{on $\domega$},
	\end{equation}
then $u$ is shape differentiable and its {Eulerian derivative} $\uprime(\Omega)[\VV] \in H^{1}(\Omega\setminus{\baromega}) \cup H^{1}(\omega)$ solves
        \begin{equation}\label{eq:state_shape_derivative}
        \left\{
        \begin{aligned}
          -\dive{\left( {\alpha} \nabla {\uprime} \right)} + \mue {\uprime} &= 0, &&\quad \text{in $\Omega \setminus \baromega$ and in $\omega$},\\
           \jump{ {\uprime} } &= -\Vn \jump{\dfrac{\partial{u}}{\partial\nn}}, &&\quad \text{on } \domega,\\
           \jump{\alpha\dfrac{\partial\uprime}{\partial\nn}} &= K(u)[\VV], &&\quad \text{on } \domega,\\
          {\alpha} \dn{{\uprime}} + \frac{1}{\zeta} {\uprime} &= 0, &&\quad \text{on } {\partial\Omega},
        \end{aligned}
        \right.
        \end{equation}
        where
        \[
        	K(u)[\VV] \coloneqq \nabla_{{\tau}} \cdot \left( \Vn \jump{\alpha} \nabla_{\tau} {u} \right) - \jump{\mu} {u} \Vn + \jump{f} \Vn.
        \]
\end{theorem}
{
\begin{remark}
The jump term $\jump{f}$ in \eqref{eq:state_shape_derivative} reflects the general
form of the Eulerian derivative when the forcing term may be discontinuous across the
interface $\partial\omega$.
In the present setting, the diffusion coefficient $\alpha$ and the forcing term $f$
are prescribed data. While $\alpha$ is allowed to be piecewise-defined in order to
model material heterogeneity across $\partial\omega$, the forcing term is assumed
to be independent of the geometry.
\end{remark}
}
A rigorous proof of Theorem~\ref{thm:state_shape_derivative} is given in Appendix~\ref{appx:material_and_shape_derivative}.
In equation \eqref{eq:state_shape_derivative}, the jump expression with the normal derivative of $u$ is given by:
\begin{align*} \jump{\alpha \dfrac{\partial u^{\prime}}{\partial \nn}}
	= \alpha_{+} \nabla u^{\prime}_{+} \cdot \nn - \alpha_{-} \nabla u^{\prime}_{-} \cdot \nn
	= \alpha_{+} \nabla u^{\prime}_{+} \cdot \nn_{+} + \alpha_{-} \nabla u^{\prime}_{-} \cdot \nn_{-},
\end{align*}
where $\nn = \nn_{+}$ is the outward unit normal vector to $\Omega$, and $\nn_{-} = -\nn$ is the inward unit normal vector.
For any function defined on $\Omega$, we denote its restrictions to $\Omega_{+}$ and $\Omega_{-}$ as $u_{+}$ and $u_{-}$, respectively, and drop the subscripts when no confusion arises.
The smoothness assumptions for the domain and deformation fields in {Theorem~\ref{thm:state_shape_derivative}} can be replaced by $C^{3,\lambda}$ for some $\lambda \in (0,1)$.

{The following proposition provides the boundary expression of the Eulerian derivative of the cost functional, which is used to construct the shape gradient in the numerical optimization algorithm. The derivation relies on the Hadamard--Zol\'{e}sio structure theorem and therefore requires sufficient regularity of the domain and interfaces. For this reason, the theoretical analysis in the present work is established for (at least) $C^{1,1}$-regular geometries. Nevertheless, the proposed numerical reconstruction algorithm, which is based on this theoretical result, can still be implemented on Lipschitz domains and yields reasonable numerical results in practice.
}
\begin{proposition}\label{prop:shape_derivative_of_the_cost_functional_J}
Let the assumptions of Theorem~\ref{thm:state_shape_derivative} hold.
Then, the least-squares misfit functional $J(\omega)  = \intG{\abs{u-h}^{2}}$ is differentiable with respect to the shape $\Omega \in \mathcal{O}_{ad}^{k}$ (i.e., with respect to $\omega$) in the direction of $\VV \in {{\sfTheta}^{k}}$, for some $k \in \mathbb{N}$, and its {Eulerian derivative} is given by
\begin{equation}\label{eq:shape_gradient}
    dJ(\omega)[\VV] = \intdomega{\left( {-} \jump{\alpha} \nabla_{\tau}{u} \cdot \nabla_{\tau}{p}
    	- \jump{\mu}{u}{p} + \jump{f}{p} \right)  \Vn},
\end{equation}
where $p \in V$ satisfies {\eqref{eq:first_main_adjoint}} subject to the continuity conditions
\begin{equation}\label{eq:continuity_equations_on_the_interface_for_the_adjoint}
	\jump{p} = 0
	\quad \text{and} \quad
	\jump{ \alpha \dfrac{\partial{p}}{\partial{\nn}} } = 0 \quad \text{on } \domega.
\end{equation}
Here, $u = F(\mu)$, and $u \in V$ uniquely solves \eqref{eq:main} and satisfies the continuity equations in \eqref{eq:continuity_equations_on_the_interface}.
\end{proposition}
Before proving the proposition, we note that Assumption~\ref{assumption:weak_assumptions} ensures the well-posedness of \eqref{eq:state_shape_derivative} and \eqref{eq:first_main_adjoint} by the Lax-Milgram lemma.
\begin{proof}[Proof of Proposition~\ref{prop:shape_derivative_of_the_cost_functional_J}]
	Let the assumptions of the proposition be satisfied.
	Then, we have sufficient regularity of the domain and the state to apply Hadamard's boundary differentiation formula (see, e.g., \cite[Thm.~4.3, Ch.~9, p.~486]{DelfourZolesio2011}, \cite[Prop.~5.4.18, Ch.~5.4, p.~225]{HenrotPierre2018}, or \cite[Lem.~3.3, Eq.~(3.44), p.~112]{SokolowskiZolesio1992}).
	That is, we have $dJ({\omega})[\VV] = \intG{(u-h)\uprime}$.
    We remove $\uprime$ from the integral expression above using the adjoint method;
    that is, we introduce $p$ as the solution of the adjoint problem \eqref{eq:first_main_adjoint} and multiply the main equation with $\uprime \in V$, apply integration by parts, and use the boundary conditions to obtain the equation $a(\uprime, p) = \intdomega{ K(u)[\VV] {p} }$.
    Similarly, by applying the same steps to the adjoint system \eqref{eq:first_main_adjoint} with $\uprime$ as the multiplier, we obtain $ a(p, \uprime) = \intG{ (u-h)\uprime }$.
    Observe that -- under the continuity {conditions \eqref{eq:state_shape_derivative} and \eqref{eq:continuity_equations_on_the_interface_for_the_adjoint}} -- we have $a(\uprime, p) = a(p, \uprime)$.
    Thus,
	\begin{align*}
		dJ({\omega})[\VV]
		= \intG{(u-h)\uprime}
		&= \intdomega{ { \left( \operatorname{div}_{\tau}{( \Vn \jump{\alpha} \nabla_{\tau}{u} )} {p}
		 	-   \jump{\mu}{u} {p}  \Vn
			+  \jump{f} {p}  \Vn \right)  } }\\
		&= \intdomega{ \left( {-} \jump{\alpha} \nabla_{\tau}{u} \cdot \nabla_{\tau}{p}
		 	-   \jump{\mu}{u} {p}
			+  \jump{f} {p}  \right)  {\Vn} }.
	\end{align*}
	This proves the characterization of the {Eulerian derivative} of $J$ as claimed in \eqref{eq:shape_gradient}.
\end{proof}
We observe that setting $k = 2$ is sufficient to establish the results in Theorem~\ref{thm:state_shape_derivative} and Proposition~\ref{prop:shape_derivative_of_the_cost_functional_J}.
Furthermore, with $\Omega$ being Lipschitz and $\omega$ belonging to the class $C^{1,1}$, we can derive the shape gradient of $J$ as presented in \eqref{eq:shape_gradient}.
To achieve this, we will employ the so-called \textit{rearrangement method} in the spirit of \cite{IKP2006,IKP2008,HIKKP2009}.
The derivation is provided in the next subsection.
\subsubsection{Computation of the shape gradient without the {Eulerian derivative} of the state}\label{subsec:rigorous_computation}
In this subsection, we provide a direct and rigorous computation of the shape gradient bypassing the need to compute the {Eulerian derivative} of the state.
The method only requires the mild assumptions given in Assumption~\ref{assumption:weak_assumptions}.
Although $f$ is simplified in this section (and later $\alpha$ for convenience), the shape gradient computation remains applicable even if $f$ has jump discontinuities.

We have the following remark on the continuity conditions on the product of the normal derivatives of $u$ and $p$ on the boundary interface $\domega$.
\begin{remark}
The continuity conditions given in \eqref{eq:continuity_equations_on_the_interface} and \eqref{eq:continuity_equations_on_the_interface_for_the_adjoint}, allow us to deduce that the product $ \jump{ {\alpha} \dn{u} \dn{p}} \Vn$
vanishes in the shape gradient’s kernel, which is the only part where these conditions are applied.
Indeed, the following property holds for the jump operator $\jump{\cdot}$:
\[
    \jump{\varphi \psi} = \jump{\varphi} \psi_{+} + \varphi_{-} \jump{\psi}
    = \varphi_{+} \jump{\psi} + \jump{\varphi} \psi_{-}
\]
which implies that
\[
    \jump{\varphi \psi} = 0 \quad \text{if } \jump{\varphi} = \jump{\psi} = 0.
\]
Thus, by this identity, given that the {conditions \eqref{eq:continuity_equations_on_the_interface} and \eqref{eq:continuity_equations_on_the_interface_for_the_adjoint}} hold, we infer that $ \jump{ {\alpha} \dn{u} \dn{p}} = 0$.
On another note, we comment that given the assumptions given in Assumption~\ref{assumption:weak_assumptions}, one can verify that the solution $u \in V$ of Problem~\ref{prob:weak_form_optimal_tomography} satisfies the local higher-regularity result $u_{\pm} \in H^{2}(\Omega_{\pm})$.
\end{remark}
\begin{theorem}\label{thm:state_shape_derivative_weak_assumptions}
Let Assumption~\ref{assumption:weak_assumptions} be satisfied, and assume that $\Omega \in \mathcal{O}_{ad}^{1}$, with Lipschitz smooth boundary $\partial\Omega$, and let $\VV \in {{\sfTheta}^{1}}$.
Then, $J$ is differentiable with respect to $\omega$ in the direction of $\VV$ and its
{Eulerian derivative} $dJ({\omega})[\VV]$ is given by
        \begin{equation}\label{eq:general_shape_gradient}
		dJ({\omega})[\VV]
		= \intdomega{ \left( - \jump{ {\alpha} } (\nabla_{\tau}{u} \cdot \nabla_{\tau}{p})
			+  \jump{ {\alpha}  \dfrac{\partial{u}}{\partial{\nn}}  \dfrac{\partial{p}}{\partial{\nn}} }
		 	-   \jump{\mu{u}} {p}
			+  f {p} \right) \Vn
		},
        \end{equation}
	where $u \in {V}$ solves \eqref{eq:main} and $p \in {V}$ satisfies the corresponding adjoint problem (cf. \eqref{eq:first_main_adjoint}).
\end{theorem}
\begin{proof}
	Assumption~\ref{assumption:weak_assumptions} ensures the boundedness of $dJ(\omega)[\VV]$ for any admissible domain $\Omega$ and deformation field $\VV \in {\sfTheta}^{1}$.
	For simplicity in the proof and to avoid lengthy expressions, we assume $\alpha \in \mathbb{R}_{+}$ and that $\mu$ is piecewise constant.
	This omits the corresponding computation for the spatial derivatives of $\alpha$ and $\mu$ in \eqref{eq:tilde_alpha}.
	
	The proof essentially proceeds in two parts.
	Firstly, we evaluate the limit $\lim_{t\to0} \frac{1}{t}\left( J({\omega}_{t}) - J({\omega}) \right)$.
	Then, using the regularity of the domain as well as the state and adjoint variables, we characterized the boundary integral expression for the computed limit.
	We begin by applying the boundary transformation formula
	\[
		\int_{\domega_{t}}{\varphi_{t}}{\, ds_{t}}
		= \int_{\domega}{\varphi_{t} \circ T_{t} \dett \abs{DT_{t}^{-\top}\nn}}{\, ds}
		=: \int_{\domega}{\varphi^{t} b_{t}}{\, ds}, \qquad (b_{t} = \dett \abs{DT_{t}^{-\top}\nn}),
	\]
	for a function $\varphi_{t} \in L^{1}(\partial\Omega_{t})$ \cite[Eq.~(4.9), p.~484]{DelfourZolesio2011} and the identity $\eta^{2} - \zeta^{2} = (\eta - \zeta)^{2} + 2\zeta (\eta - \zeta)$ to obtain the following calculations:
\begin{align*}
	J({\omega}_{t}) - J({\omega})
	&= \frac{1}{2} \intGt{\abs{{{u}_{t}} - h}^{2}} - \frac{1}{2} \intG{\abs{u - h}^{2}}	\\
	&= \frac{1}{2} \intG{ \left\{ ({b_{t}} - 1) ( \abs{\ut - h}^{2} - \abs{u - h}^{2} ) + ({b_{t}} - 1) \abs{u-h}^{2} \right\} }\\
	&\qquad \quad + \frac{1}{2} \intG{ \left\{  \abs{\ut - h}^{2} - \abs{u - h}^{2} \right\} }\\
	&= \frac{1}{2} \intG{ \left\{ ({b_{t}} - 1) ( \abs{\ut - h}^{2} - \abs{u - h}^{2} ) + ({b_{t}} - 1) \abs{u-h}^{2} \right\} }\\
	&\qquad \quad + \frac{1}{2} \intG{ \left\{ 2(\ut - h - (u - h))({u} - h) + \abs{\ut - h - (u - h)}^{2} \right\} }\\
	&= \frac{1}{2} \intG{ \left\{ ({b_{t}} - 1) ( \abs{\ut - h}^{2} - \abs{u - h}^{2} ) + ({b_{t}} - 1) \abs{u-h}^{2} \right\} }\\
	&\qquad + \frac{1}{2} \intG{ \left\{ 2(\ut - u)(u - h) + \abs{\ut - u}^{2} \right\} }\\
	&= j_{1}(t) + j_{2}(t) + j_{3}(t) + j_{4}(t).
\end{align*}
At this point, we note that, following a similar line of argument as in the first step of the proof of Theorem~\ref{thm:state_shape_derivative}, it can be shown that
\[
    \lim_{t \searrow 0} \frac{1}{\sqrt{t}} \norm{\ut - u}_{V} = 0.
\]
Using this result, together with equation \eqref{eq:derivatives_of_maps} and the fact that $\VV = \vect{0}$ on $\partial\Omega$, it can be verified that $\dot{j}_{1}(0) = \dot{j}_{2}(0) = \dot{j}_{4}(0) = 0$, where $\dot{j}_{i}(0) = \frac{d}{dt} j_{i}(t)\big|_{t=0}$, for $i = 1,2,3,4$.
For clarity, we comment that $ \dot{j}_{2}(0) = 0$ comes from the fact that $\dot{j}_{2}(0) = \frac{1}{2}\intG{ \operatorname{div}_{\tau}\VV \abs{u-h}^{2} }$, which is a consequence of \eqref{eq:derivatives_of_maps} and the observation that $\operatorname{div}_{\tau}\VV \in C(\domega)$ because $x \mapsto \nn(x) \in {{C}}^{0,1}(\domega)$.
This leads us to
\[
	\lim_{t\searrow 0} \frac{1}{t} \left( J({\omega}_{t}) - J({\omega}) \right)
		= \dot{j}_{3}(0)
		= \lim_{t\searrow 0} \intG{ \zt(u - h)  }.
\]
Let us consider the weak formulation of the adjoint problem \eqref{eq:first_main_adjoint}: find $p \in V$ such
\[
	a(p, v) = \intG{ (u-h) v}, \quad \forall v \in V.
\]
Now, let us choose $v = \wt =\ut - u \in V$ and again define $\zt = \wt/t$.
Observe that the result in the first step of the proof of Theorem~\ref{thm:state_shape_derivative} remains valid under Assumption~\ref{assumption:weak_assumptions}; refer to \eqref{eq:tilde_alpha}.
This leads us to the following series of equations
\begin{align*}
		\dot{j}_{3}(0)
		&= \lim_{t\searrow 0} \intG{ \zt(u - h)  }
		\stackrel{\eqref{eq:bilinear_and_linear_forms}}{=} \lim_{t\searrow 0} a(p, \zt)\\
		&=  \lim_{t\searrow 0} \left\{ \intO{( {\alpha} \nabla{p} \cdot \nabla{\zt} + {\mu} {p}{\zt})}
				+ \frac{1}{\zeta} \intG{{p}{\zt}} \right\}\\
		&= 
				-\intO{ \alpha A \nabla{u} \cdot \nabla{p}}
				-\intO{ \operatorname{div}{\VV} \mu {u} {p}}
				+\intO{ \left( \nabla{f}\cdot\VV {p} +  \operatorname{div}{\VV} f {p} \right)}\\
		&=: j_{31} + j_{32} + j_{33}.
\end{align*}
Using identity \eqref{eq:divergence_expansion}, the integral expressions $j_{32}$ and $j_{33}$ can be expressed as follows, respectively:
\begin{align}
	j_{32} &= -\intO{ \operatorname{div}{\VV} \mu {u} {p}}
		= -\intO{ \left[ \mu \operatorname{div}{({u}{p}\VV)} - \mu{u} (\nabla{p} \cdot \VV) - \mu{p} (\nabla{u} \cdot \VV) \right]}, \label{eq:expression_for_j32}\\
	j_{33} &
	= \intO{ \left( \nabla{f}\cdot\VV {p} +  {f}{p}  \operatorname{div}{\VV} \right)}
			= \intO{ \left[ \operatorname{div}{({f}{p}\VV)} - f \nabla{p} \cdot \VV \right]}.\label{eq:expression_for_j33}
\end{align}
Next, let us expand $j_{31}$.
From \eqref{eq:derivatives_of_maps}, we have
\begin{align*}
	j_{31} 
		&= -\intO{ \alpha {\operatorname{div}}{\VV} \nabla{u} \cdot \nabla{p} }
			+ \intO{ \alpha {D}\VV \nabla{u} \cdot \nabla{p} }
			+ \intO{ \alpha ({D}\VV)^\top \nabla{u} \cdot \nabla{p} }\\
		&=: k_{1} + k_{2} + k_{3}.
\end{align*}
%
We manipulate each term above.
First, since ${u}_{\pm}, {p}_{\pm} \in H^{2}(\Omega_{\pm})$, we have $\nabla {u}_{\pm} \cdot \nabla{p}_{\pm} \in W^{1,1}(\Omega_{\pm})$.
Thus, we can utilize the following identity:
\[
	-\intO{\operatorname{div}{\VV} \varphi}
	= \intO{\VV \cdot \nabla{\varphi}}
	- \intG{\varphi \VV \cdot \nn},
\]
which holds for $\VV \in {{C}}^{1}(\overline{\Omega})^{d}$, $\varphi \in W^{1,1}(\Omega)$, and a Lipschitz domain $\Omega$, by assigning $\varphi = \nabla{u}_{\pm} \cdot \nabla{p}_{\pm}$ and replacing $\Omega$ with $\Omega_{\pm}$.
Then, because $\VV = \vect{0}$ on $\partial\Omega$, we get
\begin{align*}
 k_{1}
	%
&= \int_{\Omega_{\pm}}{ {\alpha}_{\pm} \VV \cdot \nabla{(\nabla{u}_{\pm} \cdot \nabla{p}_{\pm})}}{\, dx}\\
& \quad  + \int_{\domega}{ {\alpha}_{+}(\nabla{u}_{+} \cdot \nabla{p}_{+}) \Vn}{\, ds}
		- \int_{\domega}{ {\alpha}_{-}(\nabla{u}_{-} \cdot \nabla{p}_{-}) \Vn}{\, ds}.
\end{align*}
The product $\nabla{(\nabla{u}_{\pm}\cdot\nabla{p}_{\pm})} \cdot \VV$ can be expanded as follows:
\[
	\nabla{(\nabla{u}_{\pm}\cdot\nabla{p}_{\pm})} \cdot \VV
		= (\nabla{p}_{\pm})^{\top} \nabla^{2}{u}_{\pm} \VV + \VV^{\top} \nabla^{2}{p}_{\pm} \nabla{u}_{\pm}
		= (\nabla^{2}{u}_{\pm} \nabla{p}_{\pm} + \nabla^{2}{p}_{\pm} \nabla{u}_{\pm}) \cdot \VV,
\]
where the latter equation follows from the fact that the Hessian $\nabla^{2}$ is symmetric.
These equations yield the following
\begin{align*}
	k_{1} &= \int_{\Omega_{\pm}}{ {\alpha}_{\pm}  (\nabla^{2}{u}_{\pm} \nabla{p}_{\pm} + \nabla^{2}{p}_{\pm} \nabla{u}_{\pm}) \cdot \VV }{\, dx} \\
	& \qquad \qquad + \int_{\domega}{ {\alpha}_{+}(\nabla{u}_{+} \cdot \nabla{p}_{+}) \Vn}{\, ds} \
		- \int_{\domega}{ {\alpha}_{-}(\nabla{u}_{-} \cdot \nabla{p}_{-}) \Vn}{\, ds}.
\end{align*}

Next, we find equivalent forms of $k_{2}$ and $k_{3}$.
For this purpose, considering $k_{2}$, we recall the second identity given in \eqref{eq:nabla_expansions} to get
\[
	{\operatorname{div}}((\VV \cdot \nabla{u}_{\pm})\nabla{p}_{\pm})
	= D\VV \nabla{p}_{\pm} \cdot \nabla{u}_{\pm} + \nabla{p}_{\pm}^{\top} \nabla^{2}{u}_{\pm} \VV + (\VV \cdot \nabla{u}_{\pm}) \Delta{p}_{\pm}.
\]
Integrating both sides of the above equation over $\Omega_{\pm}$, applying Stokes' theorem, and utilizing the boundary condition $\VV = \vect{0}$ on $\partial\Omega$, we arrive at the following results:
%
%
\begin{align*}
	\int_{\Omega_{\pm}}{ {\alpha} D\VV \nabla{p}_{{\pm}} \cdot \nabla{u}_{{\pm}} }{\, dx}
	&= -\int_{\Omega_{\pm}}{ {\alpha} \left\{ \nabla{p}_{{\pm}}^{\top} \nabla^{2}{u}_{{\pm}} \VV + (\VV \cdot \nabla{u}_{{\pm}}) \Delta{p}_{{\pm}} \right\} }{\, dx} \\
	& \quad \ \mp \int_{\domega}{ {\alpha} (\VV \cdot \nabla{u}_{\pm})\nabla{p}_{\pm} \cdot \nn }{\, ds}.
\end{align*}
Interchanging $u_{\pm}$ and $p_{\pm}$ and noting that $(D\VV)^{\top} \nabla{p}_{\pm} \cdot \nabla{u}_{\pm} = (D\VV) \nabla{u}_{\pm} \cdot \nabla{p}_{\pm}$, we get a similar equation for $k_{3}$.
Adding these computed expressions for $k_{1}$, $k_{2}$, and $k_{3}$, we get
\begin{align*}
		j_{31}
		&= k_{1} + k_{2} + k_{3}\\
		&= \int_{\Omega_{\pm}}{ {\alpha}_{\pm}  (\nabla^{2}{u}_{\pm} \nabla{p}_{\pm} + \nabla^{2}{p}_{\pm} \nabla{u}_{\pm}) \cdot \VV }{\, dx} \\
		&\quad -\int_{\Omega_{\pm}}{ {\alpha} \left[ \nabla^{2}{u}_{{\pm}} \nabla{p}_{{\pm}} \cdot \VV + (\VV \cdot \nabla{u}_{{\pm}}) \Delta{p}_{{\pm}} \right] }{\, dx}
				\ \mp \int_{\domega}{ {\alpha} (\VV \cdot \nabla{u}_{\pm}) \dn{p_{\pm}} }{\, ds}\\
		&\quad -\int_{\Omega_{\pm}}{{\alpha} \left[ \nabla^{2}{p}_{{\pm}} \nabla{u}_{{\pm}} \cdot \VV + (\VV \cdot \nabla{p}_{{\pm}}) \Delta{u}_{{\pm}} \right] }{\, dx}
				\ \mp \int_{\domega}{ {\alpha} (\VV \cdot \nabla{p}_{\pm})  \dn{u_{\pm}} }{\, ds}.
\end{align*}
By utilizing the continuity equations for $u$ given in \eqref{eq:continuity_equations_on_the_interface} and combining the resulting expression with \eqref{eq:expression_for_j32} and \eqref{eq:expression_for_j33}, we obtain, after some rearrangements and applying Stokes' theorem, the following result:
\begin{equation}\label{eq:sum_of_js}
\begin{aligned}
	\dot{j}_{3}(0)
		&= \intdomega{ \jump{{\alpha}(\nabla{u} \cdot \nabla{p})} \Vn} \\
		&\quad - \intdomega{ \jump{ {\alpha} (\VV \cdot \nabla{p})  \dn{u} } }
			- \intdomega{ \jump{ {\alpha} (\VV \cdot \nabla{u})  \dn{p} } }\\
		&\quad + \int_{\Omega_{\pm}}{ \left( - {\alpha} \Delta{u}_{\pm} + \mu_{\pm}{u}_{\pm} - f \right) (\VV \cdot \nabla{p}_{\pm}) }{\, dx}\\
		&\quad + \int_{\Omega_{\pm}}{ \left( - {\alpha} \Delta{p}_{\pm} + \mu_{\pm}{p}_{\pm} \right) (\VV \cdot \nabla{u}_{\pm}) }{\, dx}
		\\
		&\quad + \intdomega{ ({f} -  \jump{\mu{u}}) {p}\Vn}.
\end{aligned}
\end{equation}
Since $u_{\pm}, p_{\pm} \in H^{2}(\Omega_{\pm})$, we have  $\nabla u_{\pm} \cdot \mathbf{V}, \nabla p_{\pm} \cdot \mathbf{V} \in H^{1}(\Omega_{\pm})$.
By multiplying equation \eqref{eq:main} by $\nabla p_{\pm} \cdot \mathbf{V}$, where  $\alpha(x) = \alpha \in \mathbb{R}_{+}$ and $\mu$ is piecewise constant, we deduce that the fourth integral
in \eqref{eq:sum_of_js} equals zero.
Likewise, multiplying equation  \eqref{eq:first_main_adjoint} by $\nabla u_{\pm} \cdot \mathbf{V}$, we find that the fifth integral in \eqref{eq:sum_of_js} also equals zero.
Therefore, we have
\begin{align*}
		\dot{j}_{3}(0)
		&= \intdomega{ \jump{{\alpha}(\nabla{u} \cdot \nabla{p})} \Vn}
			-  \intdomega{ \jump{ {\alpha} (\VV \cdot \nabla{p})  \dn{u} } }
			-  \intdomega{ \jump{ {\alpha} (\VV \cdot \nabla{u})  \dn{p} } }\\
		&\quad + \intdomega{ ({f} -  \jump{\mu{u}}) {p}\Vn}.
\end{align*}
As a consequence of \eqref{eq:continuity_equations_on_the_interface}, we see that $\jump{\nabla_{\tau}{u}} = 0$ on $\domega$.
Similarly, $\jump{\dn{p}} = \jump{\nabla_{\tau}{p}} = 0$ on $\domega$.
Using these equations, we deduce that
\[
	 \jump{ {\alpha} (\VV \cdot \nabla{p}) \dfrac{\partial{u}}{\partial{\nn}} }
	 	= \jump{ {\alpha}  \dfrac{\partial{p}}{\partial{\nn}}  \dfrac{\partial{u}}{\partial{\nn}} } \Vn
		\quad\text{and}\quad
	 \jump{ {\alpha} (\VV \cdot \nabla{u}) \dfrac{\partial{p}}{\partial{\nn}}  }
	 	= \jump{ {\alpha}  \dfrac{\partial{u}}{\partial{\nn}}  \dfrac{\partial{p}}{\partial{\nn}} } \Vn.
\]
Thus, by using the identity
$
\nabla{u} \cdot \nabla{p}
		= \dn{u} \dn{p}
			+ \nabla_{\tau}{u} \cdot \nabla_{\tau}{p}
$,
we finally obtain the desired expression:
\begin{align*}
		dJ({\omega})[\VV]
		&= \intdomega{ \left\{ -  \jump{ {\alpha} } (\nabla_{\tau}{u} \cdot \nabla_{\tau}{p})
			+ \jump{ {\alpha}  \dfrac{\partial{u}}{\partial{\nn}}  \dfrac{\partial{p}}{\partial{\nn}} }
			- \jump{\mu{u}}{p} + {f} {p}  \right\} \Vn }.
\end{align*}
This concludes the proof.
\end{proof}
\begin{remark}
As established by the preceding proof, the Eulerian derivative of the cost functional in Theorem~\ref{thm:state_shape_derivative_weak_assumptions} holds under the requirements specified in Assumptions~\ref{assumption:weak_assumptions}. Furthermore, the numerical algorithm presented in the following section relies on these conditions and appears to remain effective in certain broader settings---for instance, in empirical tests where $f \notin L^{2}(\Omega)$.
\end{remark}
%
%
%
%
%
%
%
%
%
\section{Numerical Algorithm and Examples}\label{sec:numerical_studies}
In this section, we present the numerical implementation of the proposed approach.
We begin by discussing the computation of the forward problem and the selection of the regularization parameter $\rho$.
Next, we develop a numerical algorithm using a Sobolev-gradient descent scheme for boundary interface variation.
Finally, we validate the scheme with various numerical examples.
\subsection{Forward problem}
The computational setup is as follows:
In the forward problem, all free parameters in the PDE system are specified, including the input source $f$ and the exact geometry of $\omega$.
We emphasize that, in contrast to the usual approach based on non-destructive testing and evaluation, our method does not rely on input data from the boundary; instead, we use the observed data -- the single measurement $h$.
This data is synthetically generated by solving the direct problem \eqref{eq:main}.

To avoid `inverse crimes' (see \cite[p.~179]{ColtonKress2019}), the forward problem is solved using a fine mesh and ${P}_{2}$ finite element basis functions, while coarser triangulations and ${P}_{1}$ basis functions are used in the inversion.
Gaussian noise with mean zero and standard deviation $\gamma \norm{h}_{\infty}$, where $\gamma$ is a free parameter, is added to $h$ to simulate noise.
\subsection{Choice of regularization functional}
Regularization is commonly incorporated by adding specific terms to the numerical implementation, either during minimization or when addressing ill-posed systems of equations.
These terms often depend on the parametrization of the variable of interest or the discretization of the ill-posed systems \cite{Rundell2008, Fang2022}.

In our numerical experiments, we introduce a regularization functional that combines the perimeter of $\partial\omega$ with Tikhonov regularization for $\mu$ on $\partial\omega$:
\[
	\frac{\rho_{1}}{2} P(\partial\omega) + \frac{\rho_{2}}{2}  {R}(\mu)
	\coloneqq  \frac{\rho_{1}}{2} \int_{\partial\omega}{1}\, ds + \frac{\rho_{2}}{2} \intO{\mu^{2}},
\]
where $\rho_{1}$ and $\rho_{2}$ are small positive constants.

In most cases, we omit the perimeter penalization and rely solely on Tikhonov regularization for $\mu$ on $\partial\omega$, as this approach is sufficient for accurately reconstructing $\mu$ and $\partial\omega$, even with noise-contaminated data.
\subsection{Choice of regularization parameter}
In the reconstruction of noisy data, selecting the regularization parameter $\rho = \rho_{2}$ in \eqref{eq:minimization_mu} is critical.
This parameter is often determined using the discrepancy principle, which relies heavily on accurate knowledge of the noise level.
However, precise noise-level information is often unavailable or unreliable in many applications.
Errors in noise estimation can significantly reduce reconstruction accuracy when using the discrepancy principle.
To overcome this difficulty, we propose a heuristic rule for choosing $\rho$ that avoids dependence on noise-level knowledge.
This rule is based on the balancing principle \cite{ClasonJinKunisch2010b}: fix $\beta > 1$ and compute $\rho > 0$ such that
\begin{equation}\label{eq:balancing_principle}
    (\beta - 1) J({\muin},\omega) - \frac{\rho}{2} {R}(\mu)
    \coloneqq  (\beta - 1) \frac{1}{2} \norm{u(\mu) - h}_{L^{2}(\bigdO)}^{2} - \frac{\rho}{2} \norm{\mu}^{2}_{L^{2}(\Omega)} = 0.
\end{equation}
This approach balances the data-fitting term $J({\muin},\omega)$ with the penalty term ${R}(\mu)$, where $\beta > 1$ controls their trade-off.
It eliminates the need for noise-level knowledge and has been successfully applied to both linear and nonlinear inverse problems \cite{ClasonJinKunisch2010, ClasonJinKunisch2010b, Clason2012, ClasonJin2012, ItoJinTakeuchi2011, Meftahi2021}.

In the following numerical experiments, we explore two approaches: (1) directly assigning a fixed value to $\rho$ and (2) applying the heuristic rule \eqref{eq:balancing_principle}.
Specifically, choosing $\rho \in (0,1)$ indicates the first approach, while selecting $\beta > 1$ corresponds to the balancing principle \eqref{eq:balancing_principle}.
%
%
%
\subsection{Numerical algorithm}
\label{subsec:Numerical_Algorithm}
The main steps of our numerical algorithm follows a standard approximation procedure, the important details of which we provide as follows.

\subsubsection{Choice of descent direction for the boundary interface variation}
The domain $\omega$ is approximated using boundary interface variation, following an approach similar to domain variation in shape optimization \cite{Doganetal2007}.
We employ a Riesz representation of the shape gradient $G$ to suppress oscillations on the unknown interface boundary during the approximation process.
Rapid oscillations on the interface boundary may destabilize the approximation and potentially halt the process prematurely.

To compute a Riesz representative of the kernel $-G\nn$, we generate an $H^1$-smooth extension of the vector by seeking a weak solution $\VV \in H_{0}^1(\Omega)^{d} \coloneqq  \{ \vect{\varphi} \in H^1(\Omega)^{d} \mid \vect{\varphi}  = \vect{0} \ \text{on $\bigdO$}\}$ to the variational equation
\begin{equation}\label{eq:extension_regularization}
	(\nabla \VV, \nabla \vect{\varphi})_{\Omega} + (\VV , \vect{\varphi} )_{\Omega}
		= - \langle {G} \nn , \vect{\varphi}\rangle_{\partial\omega}, \ \text{for all $\vect{\varphi} \in H_{0}^1(\Omega)^{d}$}.
\end{equation}
In this way, we obtain a \textit{Sobolev gradient} \cite{Neuberger1997} representation $\VV \in H_{0}^1(\Omega)^{d}$ of $G$, which is only supported on $\partial\omega$.
More importantly, this approach produces a smoothed, preconditioned extension of $-{G}\nn$ over the entire domain $\Omega$.
Such an extension allows us to deform the discretized computational domain by moving the (movable) nodes of the computational mesh -- thus moving not only the nodes on the boundary interface but also all internal nodes within the discretized domain.
Further discussions about discrete gradient flows for shape optimization problems are provided in \cite{Doganetal2007}.
%
%
%
\subsubsection{Moving-mesh scheme}\label{subsubsec:algorithm}
To compute the $k$th boundary interface $\partial\omega^{{[k]}}$, we carry out the following procedures:
\begin{description}
\setlength{\itemsep}{2pt}
	\item[1. \it{Initilization}] Fix $\rho \in (0,1)$ (or $\beta > 1$) and choose initial guesses $\mu^{0}$ and $\partial\omega^{0}$.
	\item[2. \it{Iteration}] For $k = 0, 1, 2, \ldots$, do the following:
		\begin{enumerate}
			\item[2.1] Solve the state's and adjoint's variational equations on the current domain $\Omega^{{[k]}}$.
			\item[2.2] For a sufficiently small $t_{1}^{{[k]}} > 0$, do the update $\mu^{{[k+1]}} = \mu^{{[k]}} - t_{1} \Je'(\mu^{{[k]}})$.
			\item[2.3] Choose $t_{2}^{{[k]}}>0$, and compute the deformation vector $\VV^{{[k]}}$ using \eqref{eq:extension_regularization} in $\Omega^{{[k]}}$.
			\item[2.4] Update the current domain by setting $\Omega^{{[k+1]}} \coloneqq  \{ x + t_{2}^{{[k]}} \VV^{{[k]}}(x) \in \mathbb{R}^{d} \mid  x \in \Omega^{{[k]}}\}$.			
		\end{enumerate}
	\item[3. \it{Stop Test}] Repeat \textit{Iteration} until convergence.
\end{description}
        {In Step~2.4, the domain $\Omega$ is approximated by a conforming triangulation $\Omega_h$ with nodes $\{x_i\}_{i=1}^N$.
        To update the domain during the optimization, we define the mapping
        \[
        T^{[k]}(x) \coloneqq  x + t_2^{[k]} \VV^{[k]}(x),
        \]
        and obtain the next domain iterate as
        \[
        \Omega_h^{[k+1]} \coloneqq  T^{[k]}(\Omega_h^{[k]}),
        \]
        where $t_2^{[k]}$ is chosen sufficiently small to ensure stability.
        At the discrete level, this corresponds to moving each node according to
        \[
        x_i^{[k+1]} = x_i^{[k]} + t_2^{[k]} \VV^{[k]}(x_i^{[k]}), \quad i=1,\dots,N,
        \]
        while preserving the mesh connectivity. This update is applied to all nodes: interior nodes and free boundary nodes are moved according to the mapping, whereas nodes on fixed boundary parts remain unchanged.
        We refer the reader to Remark~\ref{rem:remeshing_in_3D} for a comment on practical aspects of mesh deformation.
        %
        %
        %
        %

In practice, the mesh deformation remains stable throughout the simulations, provided that sufficiently small step sizes $t^{[k]}$, $k=0,1,\ldots,$ are used. Excessive deformations may deteriorate the mesh quality and adversely affect the numerical accuracy. Therefore, the step size is additionally reduced whenever necessary to prevent highly distorted or inverted mesh elements. Further details are provided in the next subsection.
}

\subsubsection{Step-size computation and stopping condition} In Step 2.2, $t_{1}^{{[k]}} = t_{2}^{{[k]}} = t^{{[k]}}$ for all $k = 0, 1, 2, \ldots$, and $t^{{[k]}}$ is computed using a backtracking line search inspired by \cite[p.~281]{RabagoAzegami2020}, with the formula
\[
	t^{{[k]}} = s J({\omega}^{{[k]}}) / |\VV^{{[k]}}|^2_{H^{1}(\Omega^{{[k]}})^{d}}
\]
at each iteration step $k$, 
{where $s > 0$ is a scaling parameter chosen via empirical tuning over a finite set of trial values, so as to satisfy stability requirements and promote sufficient decrease of the objective functional. In particular, overly large values of $s$ may induce instability in the iteration, while overly small values can significantly deteriorate the convergence rate.}

Although more sophisticated step-size strategies may be considered, the present approach already provides stable and effective reconstructions in the numerical experiments.
To prevent inverted triangles in the mesh after the update, the step size $t^{{[k]}}$ is further reduced.

The algorithm terminates when $t^{{[k]}} < t_{0}$, where $t_{0} > 0$ is a small real number, or when the maximum number of iterations is reached.
In all experiments, $t_{0}$ is set to $10^{-12}$ for convex boundary interfaces and adjusted to $10^{-6}$ for noisy data.
For non-convex interfaces, where boundary reconstruction is more challenging under noise, the step size is further reduced when the cost value drops below $10^{-3}$ to avoid overshooting.

In the following subsections, we first test the proposed scheme with exact measurements in a simple setting, then extend it to more complex setups with noisy data.

\subsection{2D radial  problem}\label{sec:radial_case}
We first consider a series of 2D radial problems, similar to the test setup in \cite{MoskowSchotland2009,Machida2023}.
We let $g \in \mathbb{R}_{+}$, $\eta \geqslant -1$, and suppose that
\begin{equation}\label{eq:focused_case}
	{\mu_{1}}(x) = g(1 + \eta(x))
\end{equation}
where $\eta$ is supported in a closed ball $B_a$ of radius $a$, i.e., $\op{supp}\,\eta \subset B_a \subset \Omega$.
%
%

Let us assume the 2D radial geometry and consider \eqref{eq:main}.
In the polar coordinate system we have $x = (r, \theta)$, where $r$ is the radial coordinate and $\theta$ is the angular coordinate.
Let $\Omega$ be the disk of radius $R$ centered at the origin.
Assuming that $\eta$ has the radial symmetry, we write $\eta(x) = \eta(r)$, for $r \in (0,R)$.
Let $0< R_a < R$ and  $\etaa > 0$.
We suppose that $\eta(r) = \etaa$ for $r \in (0, R_{a})$ and $0$ for $r \in (R_{a}, R)$.

Hereinafter, we write $g = k^{2}$, $k > 0$.
We {let} $l = \zeta {\alpha}$, set ${\alpha} = 1$, and consider $\Omega_{-} = \{x \mid |x| \leqslant R_a\}$ and $\Omega_{+} = \{x \mid R_a < |x| < R\}$.
We put $\cdot^{\ast}$ when referring to the exact parameter value; e.g., we denote the exact absorption coefficient by $\mu^{\ast} = \mu_{0}^{\ast}  \chi_{\Omega\setminus\baromega}  + \mu_{1}^{\ast} \chi_{\omega}$.
We consider the form $\mu^{\ast} = \mu_{0}^{\ast}  \chi_{\Omega\setminus\baromega}  + \mu_{1}^{\ast} \chi_{\omega} = g (\chi_{\Omega\setminus\baromega}  + (1+\etaa^{\ast}) \chi_{\omega} )$.

In all experiments, we set $k=1$, $R=3$, $l=0.3$, and $R_{a}^{\ast} = 1.5$ for the axisymmetric case.
%
\subsection{Numerical tests with a constant source function}\label{subsec:constant_source}
We consider a constant source function $f=1$.
We take $s = 0.1$ and choose $\partial\omega^{0} = {B}_{r} \coloneqq  \{x \in \mathbb{R}^{2} \mid |x| = r = 2.8\}$ as the initial guess for the boundary interface.
The numerical results with $\mu_{1}^{\ast} = 1.2$ are shown in Figures \ref{fig:constant_source_example_1} with fixed $\rho$ and $\beta$.
The reconstructions show precision, and the plots indicate that when using a constant source function, the regularization parameter selection is almost the same, regardless of whether the balancing principle \eqref{eq:balancing_principle} is applied.
This suggests that we can choose a suitable value for $\rho$ to obtain a reconstruction of $\mu$ consistent with the case when \eqref{eq:balancing_principle} is used.
%
%
\begin{figure}[htp!]
\centering
\resizebox{0.2\textwidth}{!}{\includegraphics{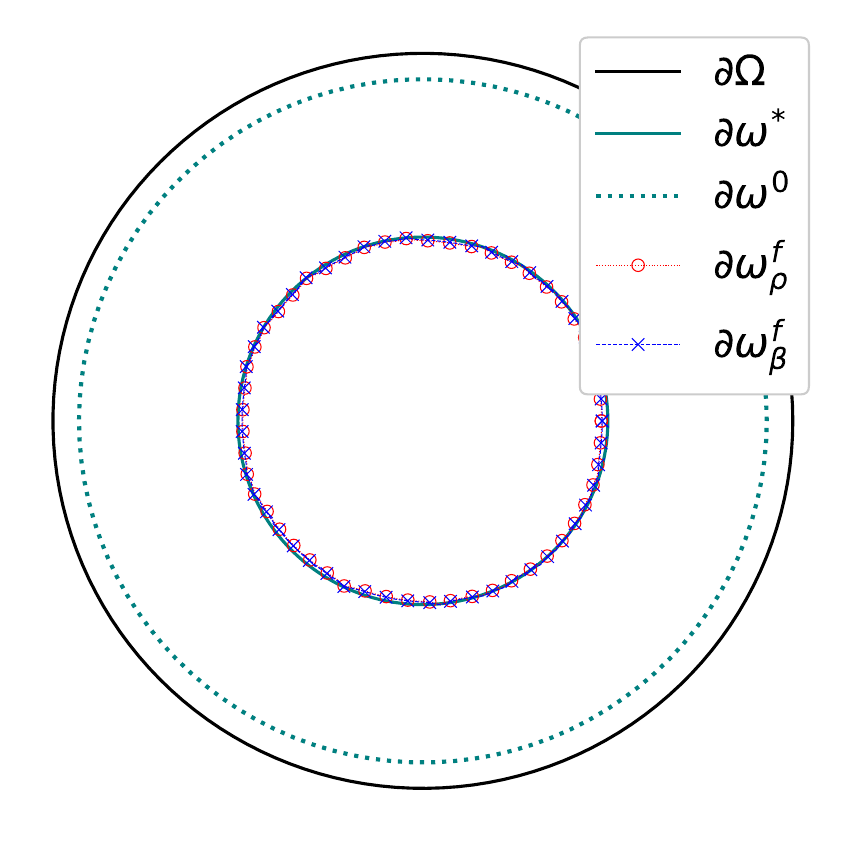}}
\resizebox{0.21\textwidth}{!}{\includegraphics{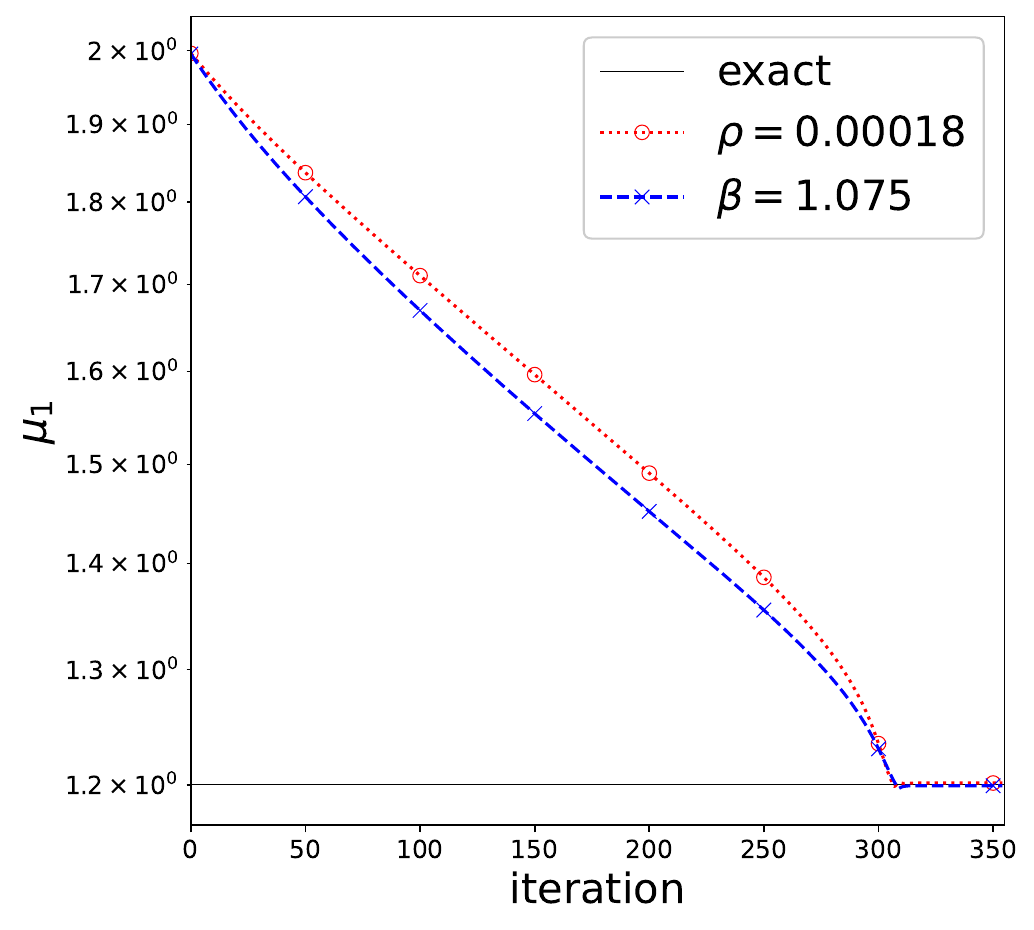}} \quad
\resizebox{0.2\textwidth}{!}{\includegraphics{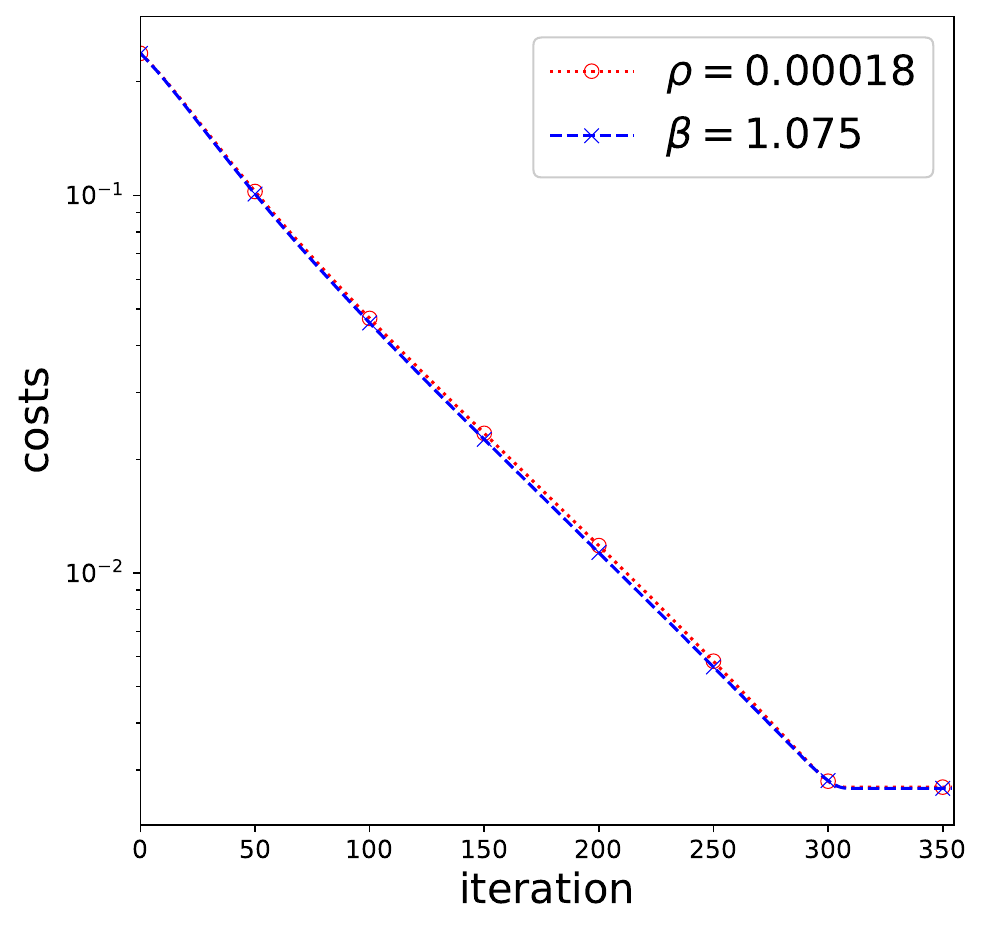}} \quad
\caption{A radial problem in 2D with $\mu_{1}^{\ast} = 1.2$ and source function $f=1$}
\label{fig:constant_source_example_1}
\end{figure}
%
%

We repeat the experiment with higher values of $\eta^{\ast}$. Using the balancing principle \eqref{eq:balancing_principle}, we obtain the results in Figure~\ref{fig:constant_source_example_2}.
Accurate identification of $\mu$ and $\partial\omega$ is possible for higher $\eta^{\ast}$ values when an appropriate $\beta$ is chosen with exact measurements.
\begin{figure}[htp!]
\centering \quad
\resizebox{0.2\textwidth}{!}{\includegraphics{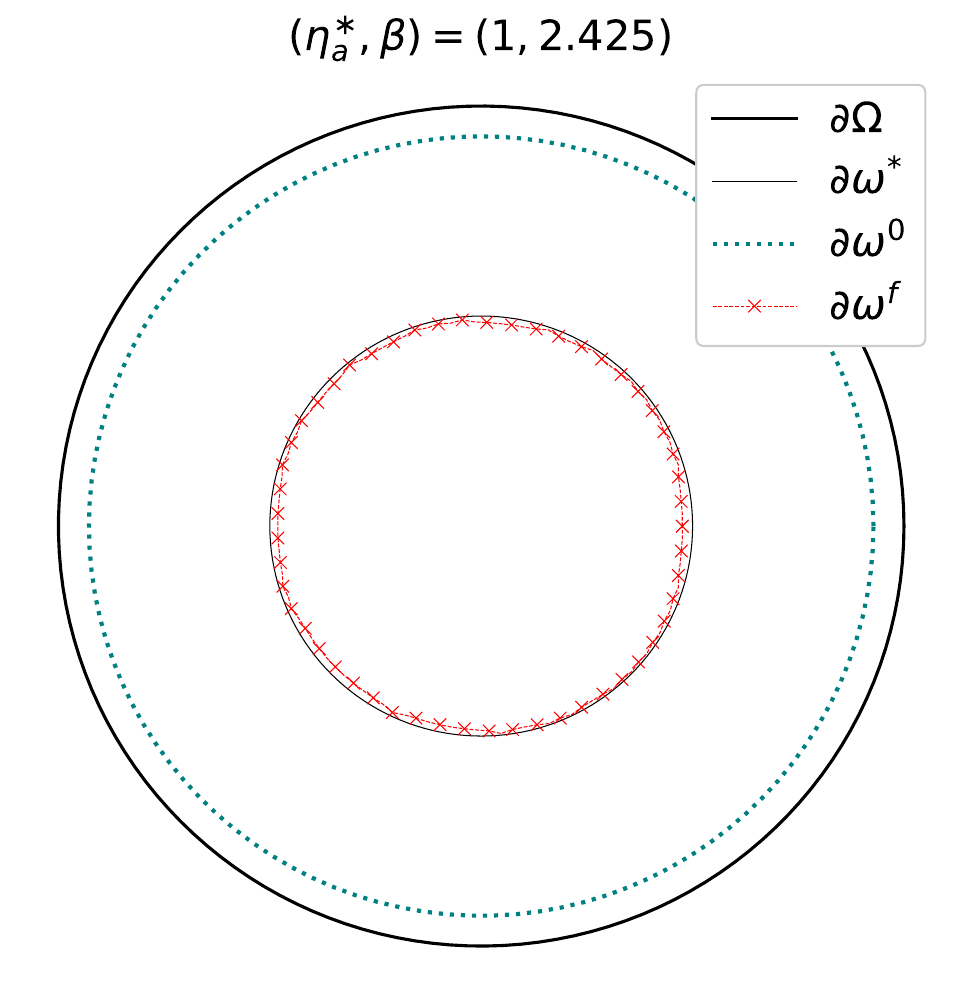}} \quad
\resizebox{0.2\textwidth}{!}{\includegraphics{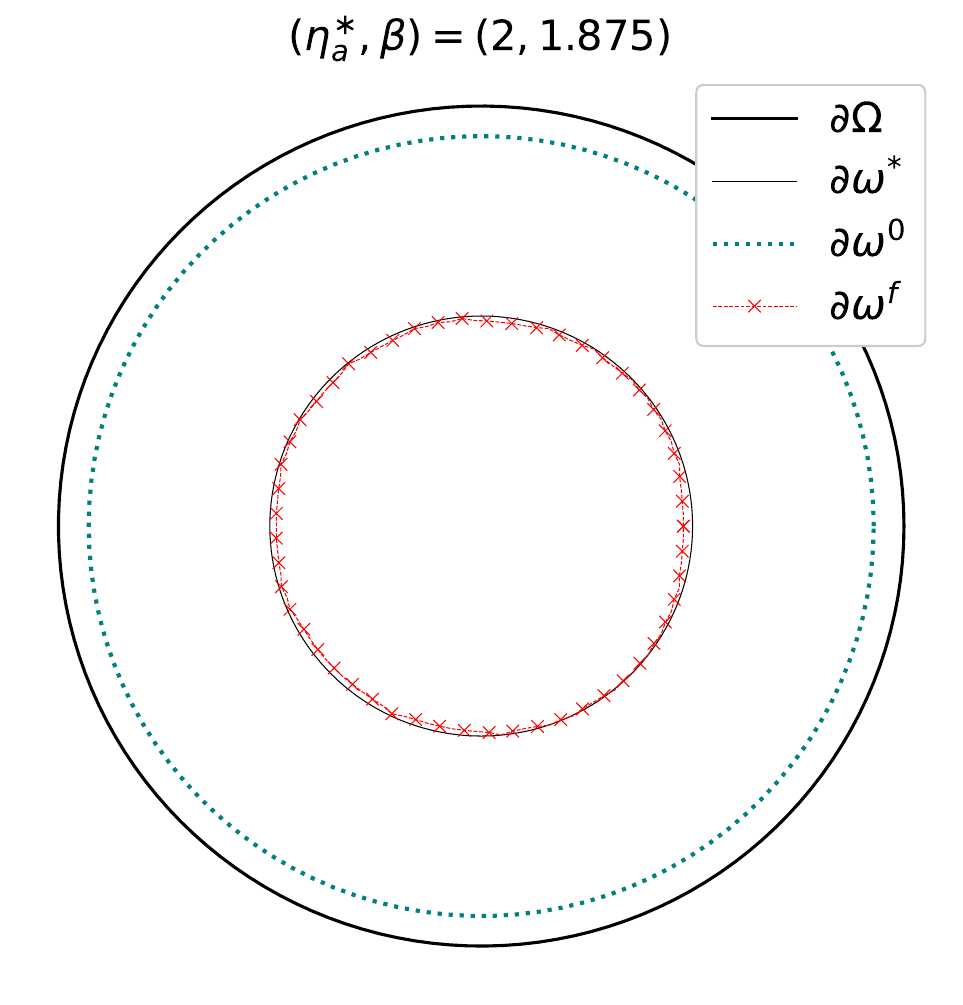}} \quad
\resizebox{0.2\textwidth}{!}{\includegraphics{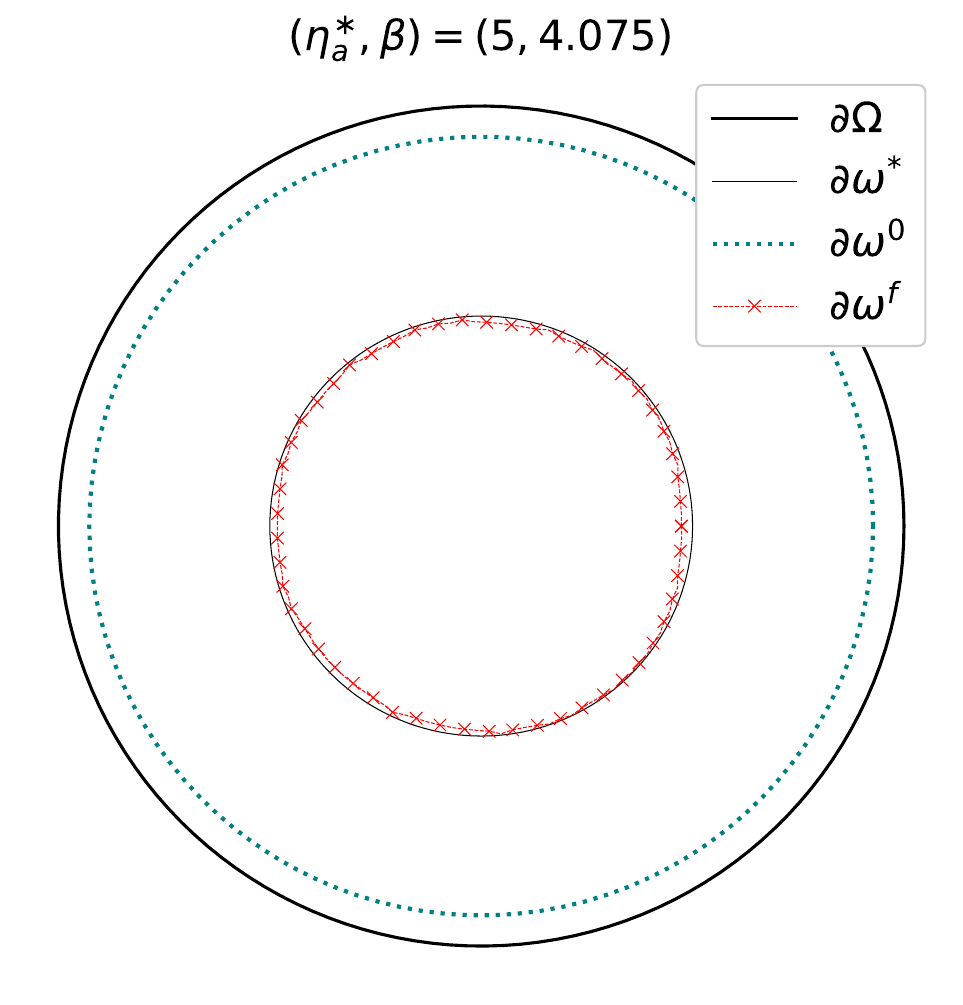}} \\
\resizebox{0.2\textwidth}{!}{\includegraphics{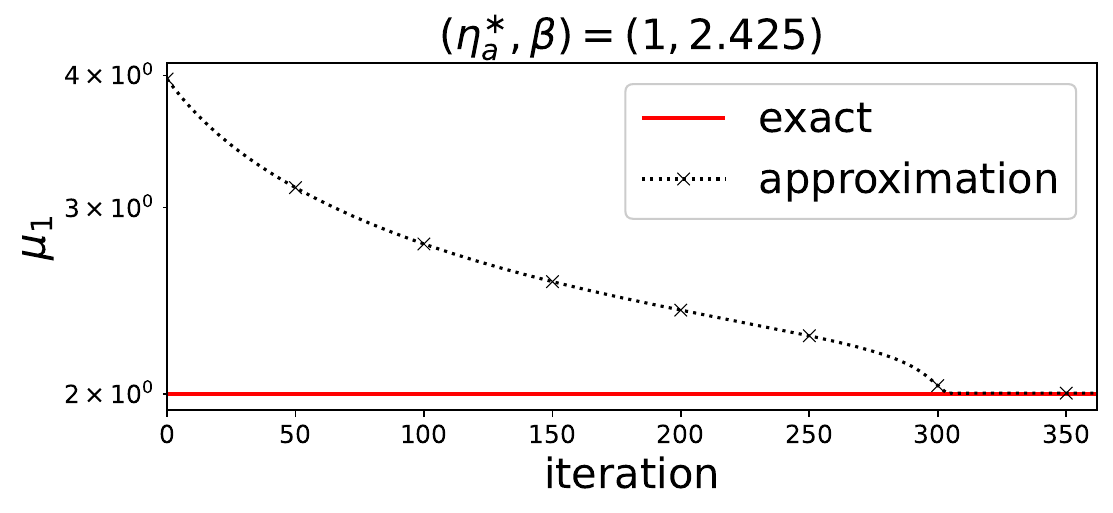}} \quad
\resizebox{0.2\textwidth}{!}{\includegraphics{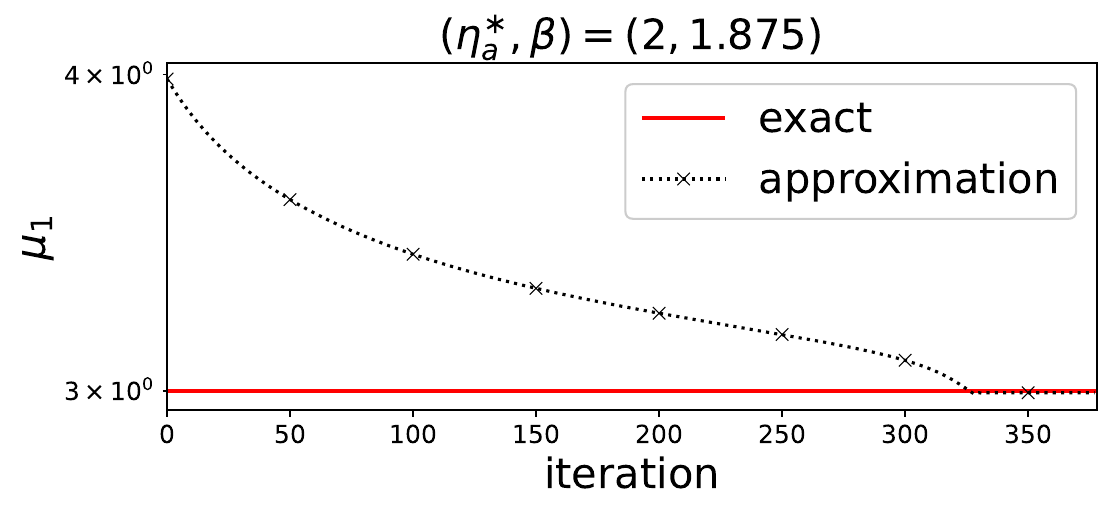}} \quad
\resizebox{0.2\textwidth}{!}{\includegraphics{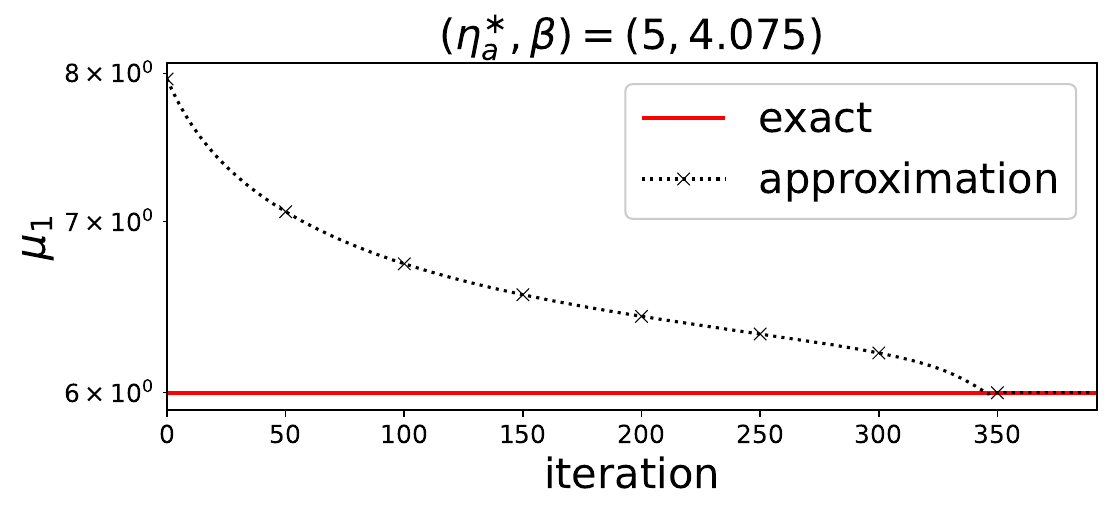}}
\caption{2D radial  problem with source function $f=1$ and varying $\mu_{1}^{\ast}$}
\label{fig:constant_source_example_2}
\end{figure}
\subsection{Numerical tests with a point source}\label{subsec:point_source}
We next consider the case of a point source\footnote{Note that this violates our regularity assumption on the source function.}, specifically when $f = \delta(x - x_{0})$, where $x_{0}$ is the position of the point source within $\Omega$.
{The inverse problem considered in this setup is related to fluorescence DOT.}

\begin{remark}
While the shape sensitivity analysis in subsection~\ref{sec:shape_sensitivity_analysis} relies on the assumption that $f \in H^{1}(\Omega)$,
an assumption violated in the case of point sources and for which the applicability of the derived shape derivative remains unclear,
the numerical experiments reported in the following suggest that the method may still exhibit reasonable practical performance.
\end{remark}

We set $x_{0} = 0$ and $(\mu_{0}^{\ast}, \mu_{1}^{\ast}) = (1, 1.2)$.
Without applying \eqref{eq:balancing_principle}, we obtain the results shown in Figure~\ref{fig:point_source_example_3}.
The leftmost plots correspond to $s = 1$, and the middle plots correspond to $s = 10$.
As expected, increasing the step-size parameter $s$ accelerates convergence toward the exact solution.
However, in both cases, the boundary interface approximation is inaccurate.
We then repeat the experiment with a new initial guess, $\partial\omega^{0} = B_{0.8}$, yielding a highly accurate approximation of the exact solution, as shown in the rightmost plot of Figure~\ref{fig:point_source_example_3}.
The scheme depends on the initialization, as anticipated.
By selecting a good initial guess for $\mu$ and $\partial\omega$, an accurate identification of the unknown coefficient and boundary interface can be achieved.

%
%
%
\begin{figure}[htp!]
\centering
\quad
\resizebox{0.2\textwidth}{!}{\includegraphics{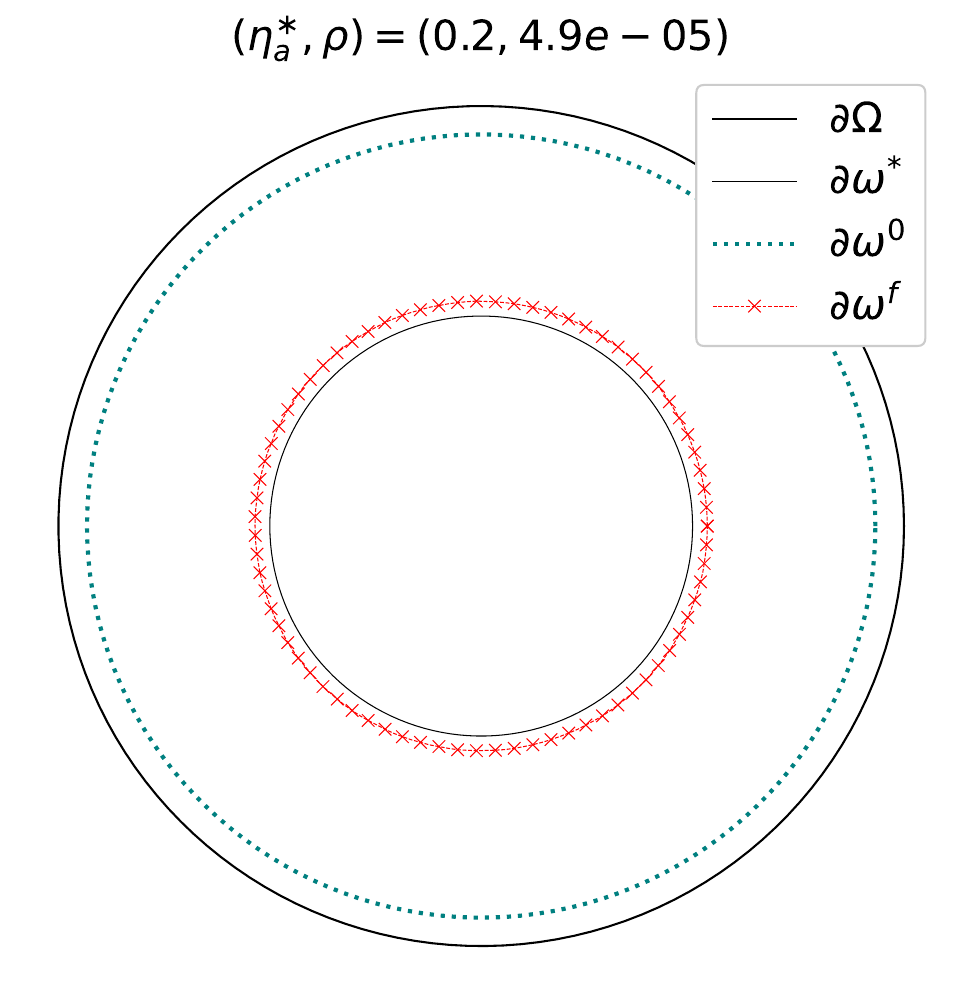}} \quad
\resizebox{0.2\textwidth}{!}{\includegraphics{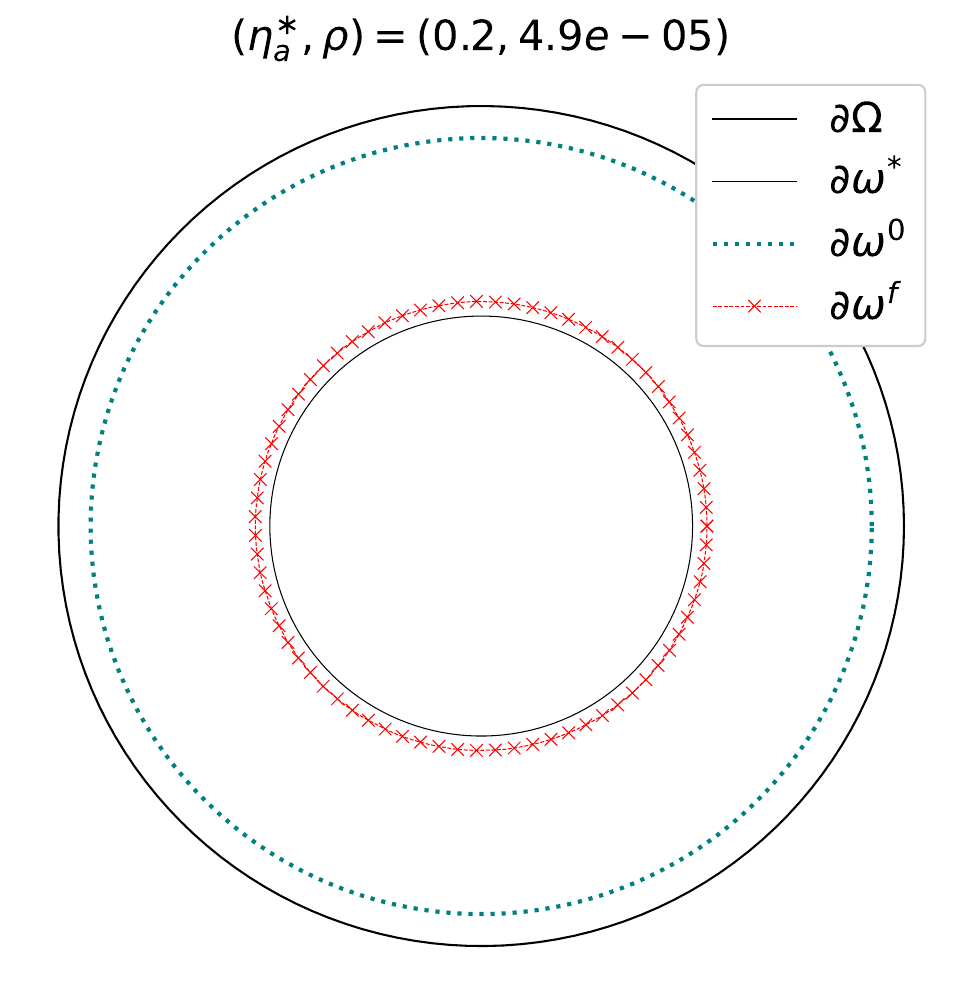}} \quad
\resizebox{0.2\textwidth}{!}{\includegraphics{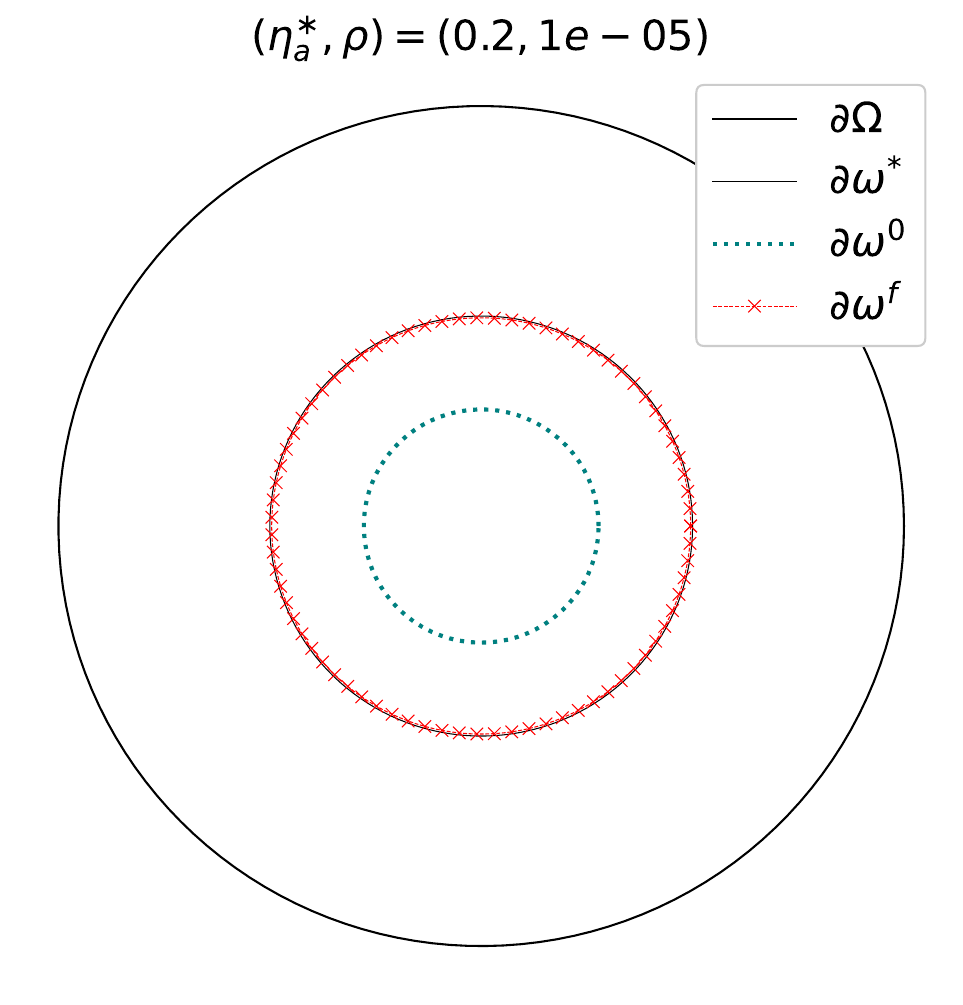}} \\
\resizebox{0.2\textwidth}{!}{\includegraphics{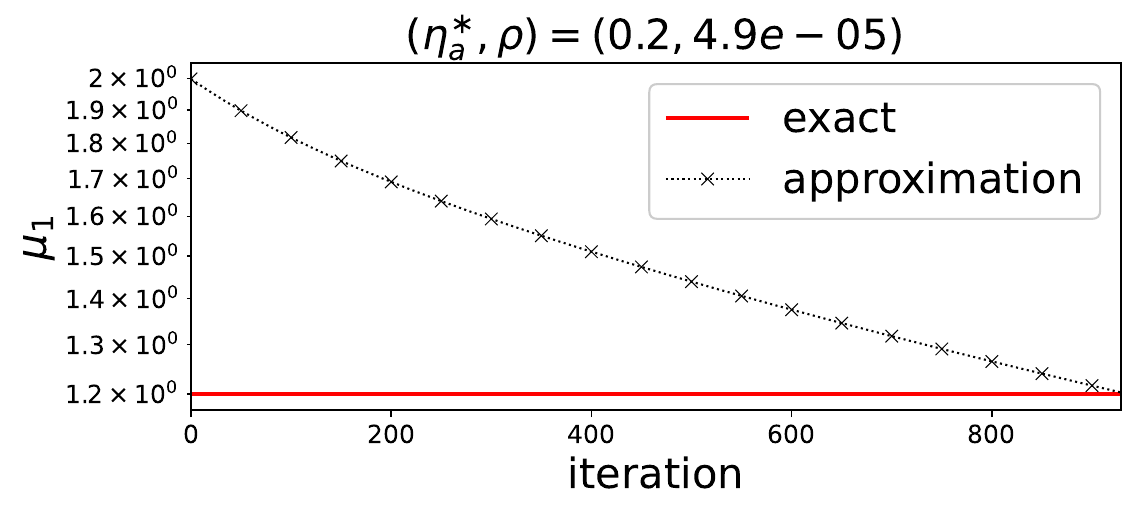}} \quad
\resizebox{0.2\textwidth}{!}{\includegraphics{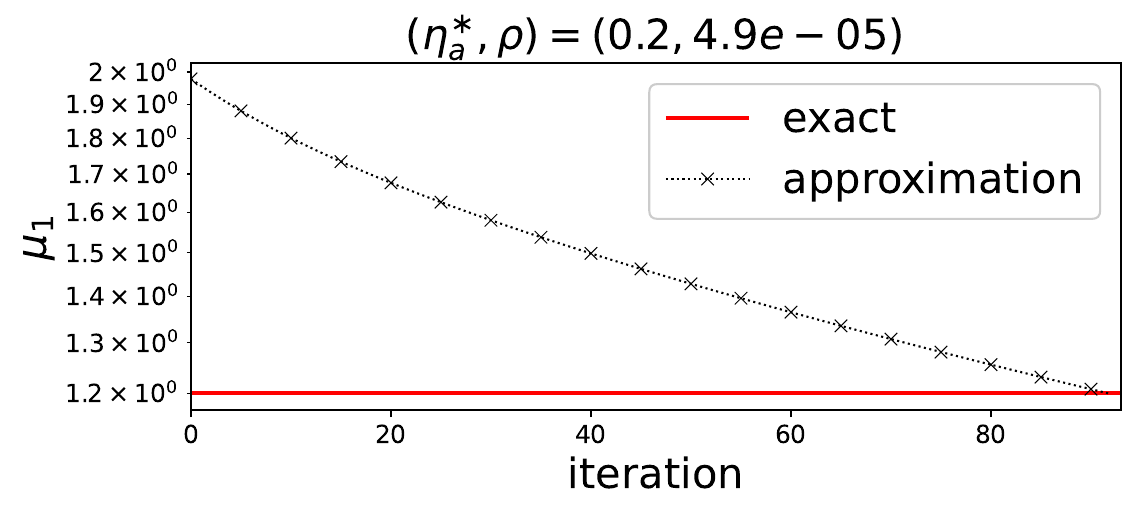}} \quad
\resizebox{0.2\textwidth}{!}{\includegraphics{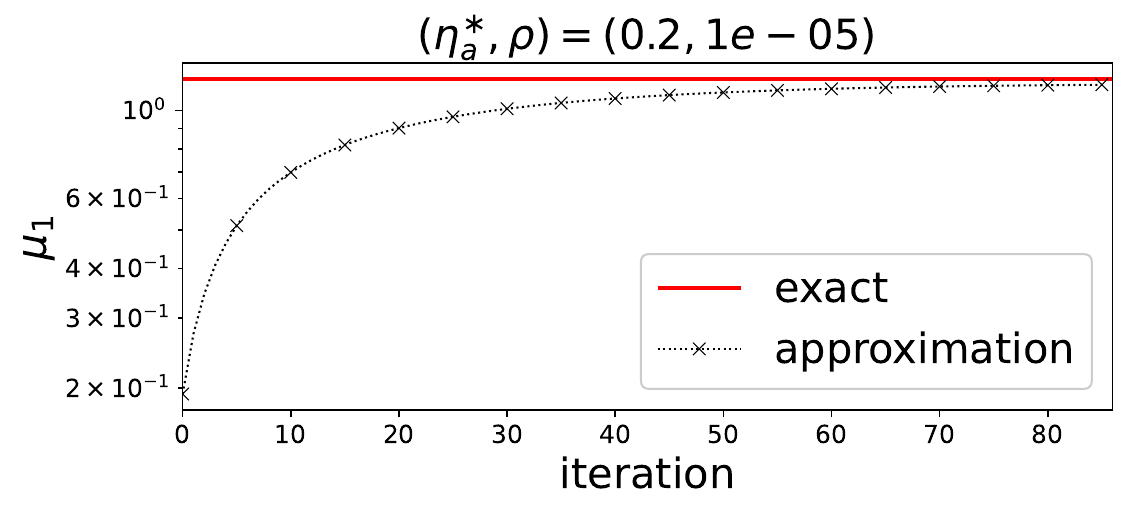}}
\caption{Results for a 2D radial problem with $f(x) = \delta(x)$}
\label{fig:point_source_example_3}
\end{figure}
%

We repeat the experiments with slight modifications to $\mu^{\ast}$.
Specifically, we examine the main case \eqref{eq:focused_case} with varying $\eta^{\ast} = 0.2,1,2,5$, but using a point source instead of a spatially oscillating source term.
With $s=100$, the results are summarized in Figure \eqref{fig:point_source_example_4}.
By carefully selecting initial guesses for the unknown parameter and boundary interface (see also Figure~\ref{fig:point_source_example_3}), we achieve precise identification of $\mu$ and $\partial\omega$.
%
\begin{figure}[htp!]
\centering
\hfill
\resizebox{0.2\textwidth}{!}{\includegraphics{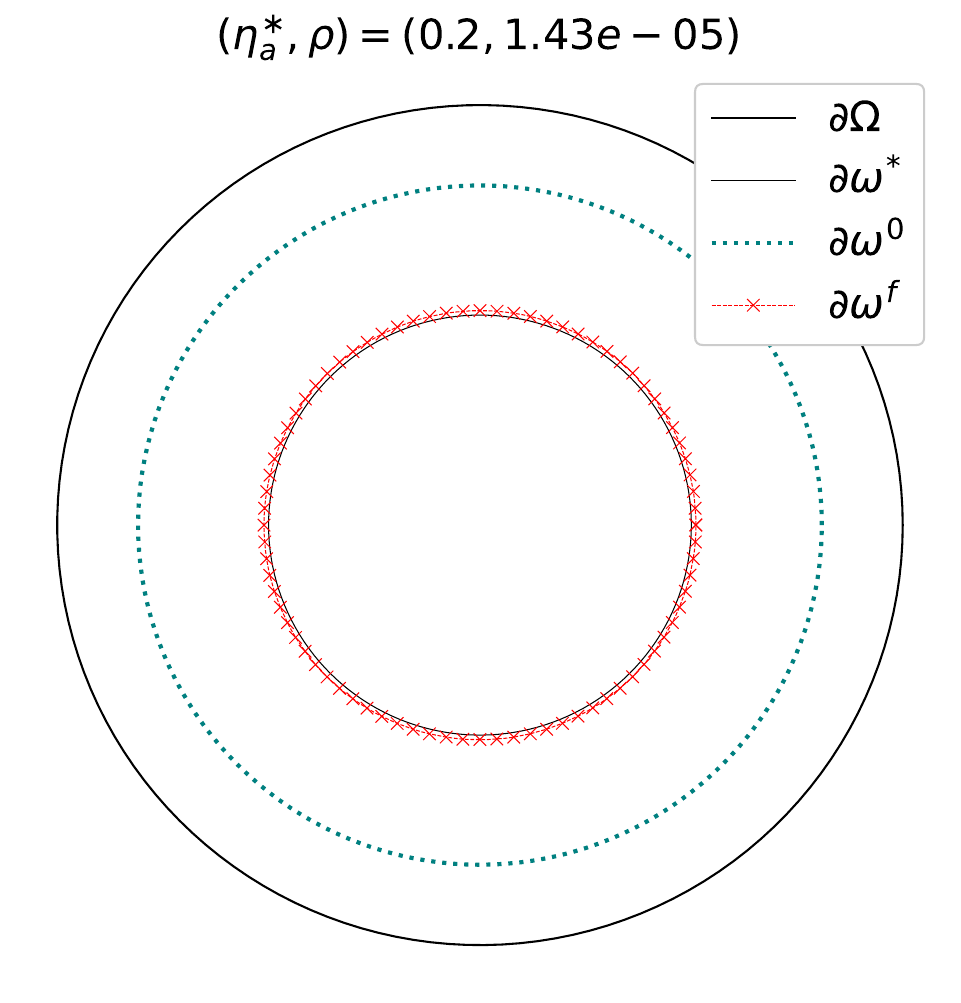}} \hfill
\resizebox{0.2\textwidth}{!}{\includegraphics{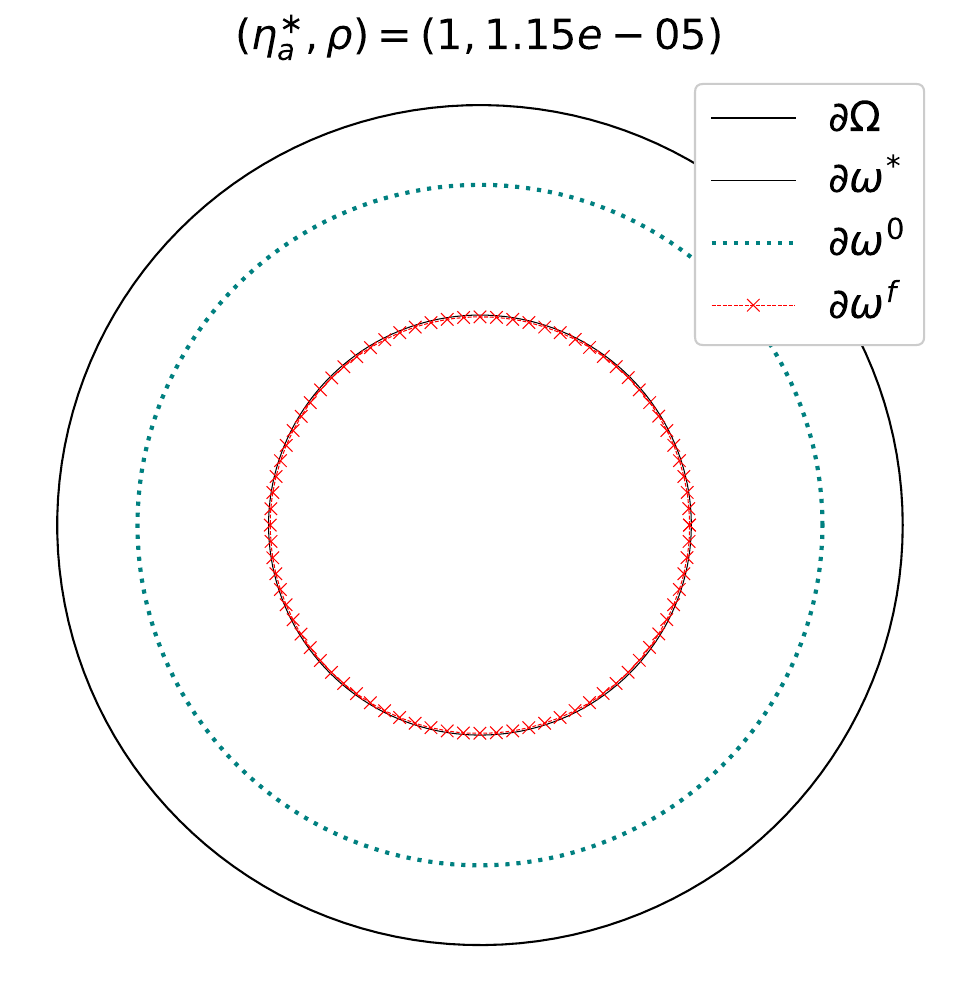}} \hfill
\resizebox{0.2\textwidth}{!}{\includegraphics{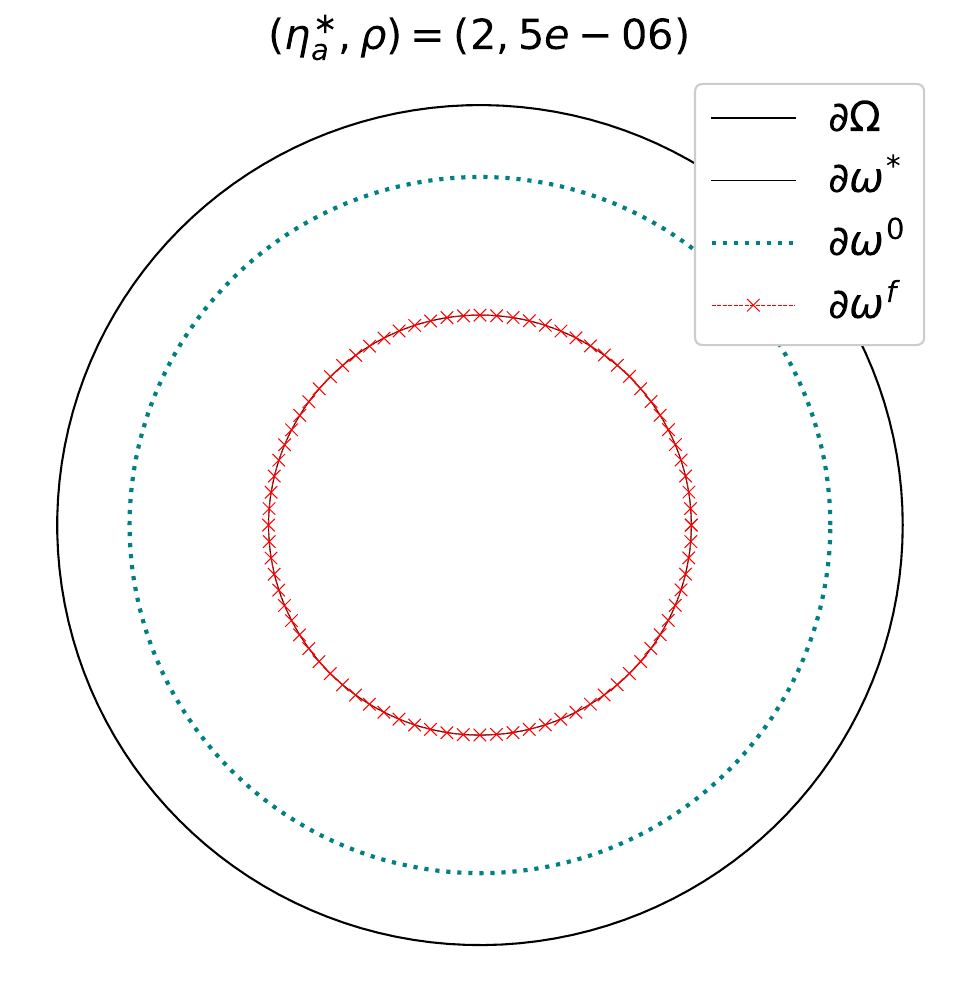}} \hfill
\resizebox{0.2\textwidth}{!}{\includegraphics{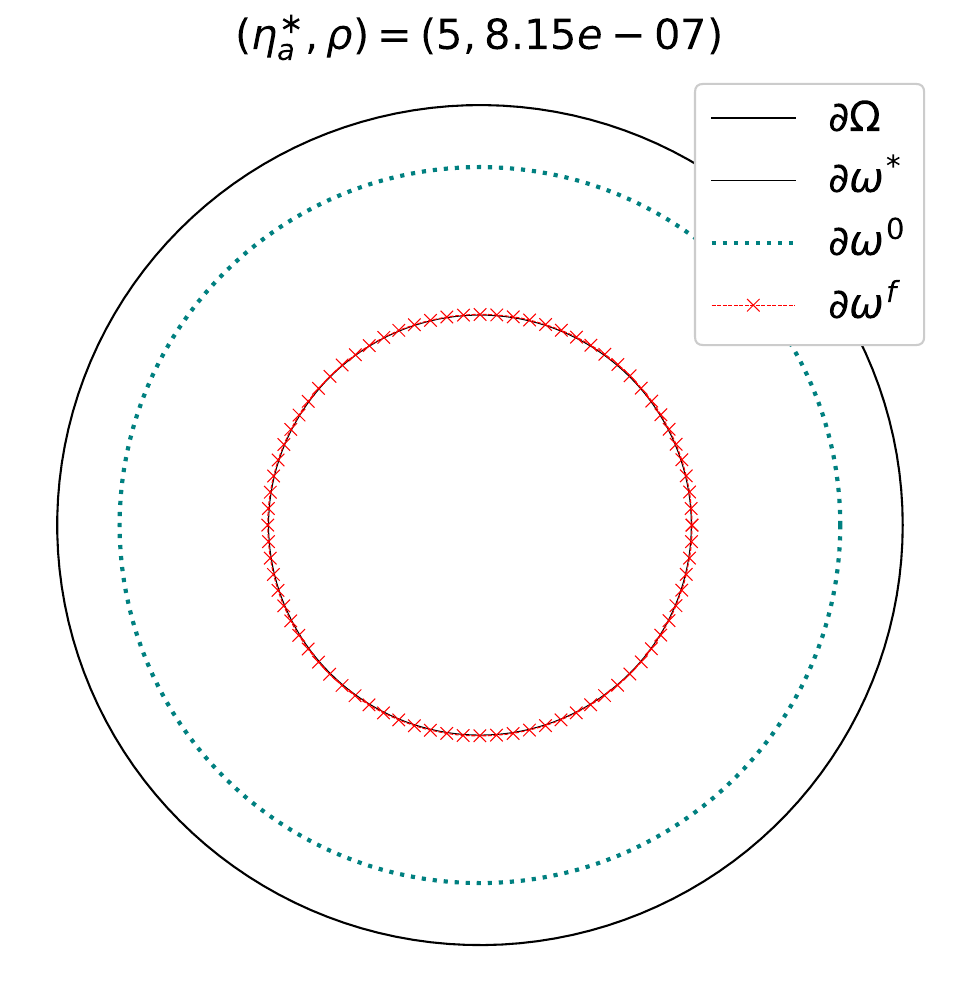}} \hfill\\
\resizebox{0.1275\textwidth}{!}{\includegraphics{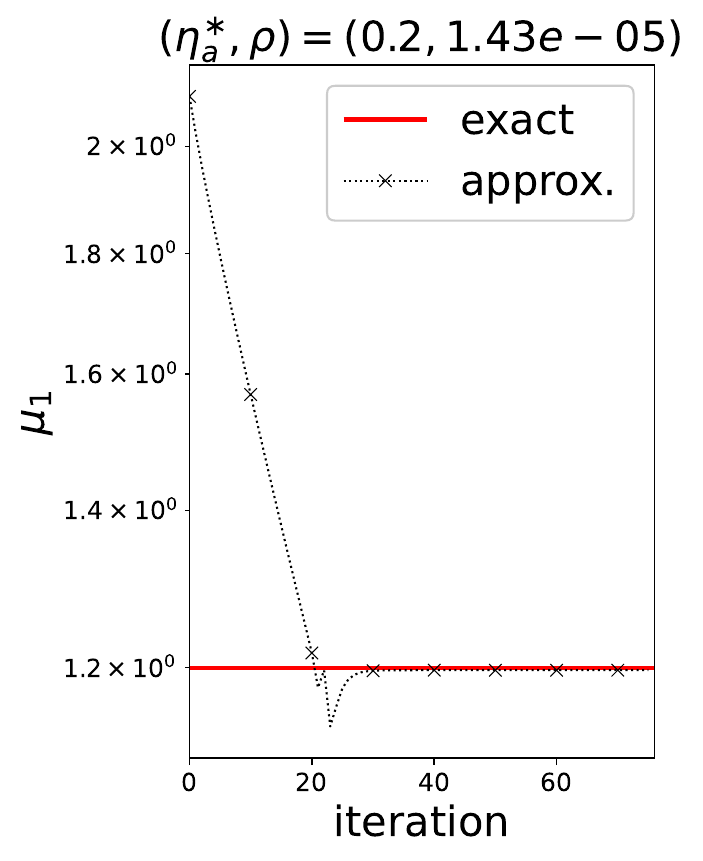}} \hfill
\resizebox{0.112\textwidth}{!}{\includegraphics{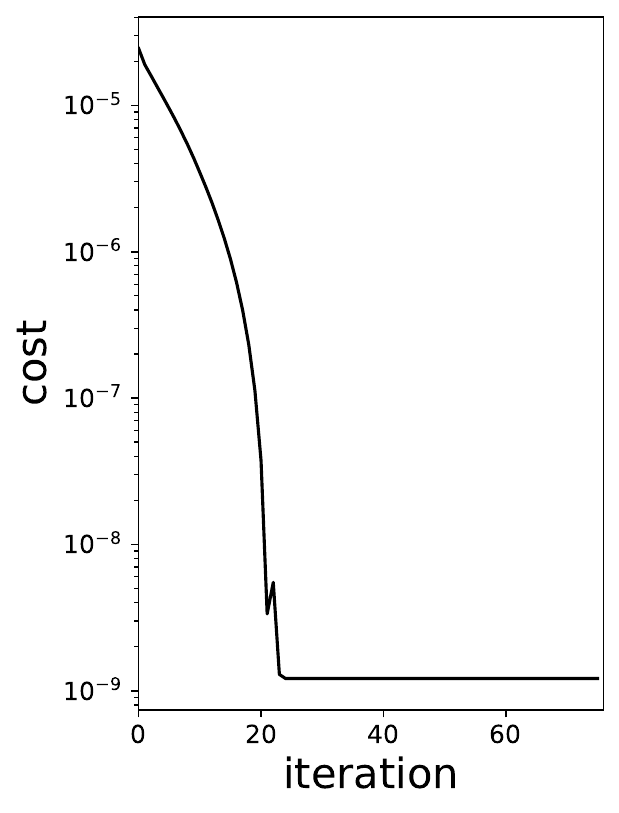}} \hfill
\resizebox{0.12\textwidth}{!}{\includegraphics{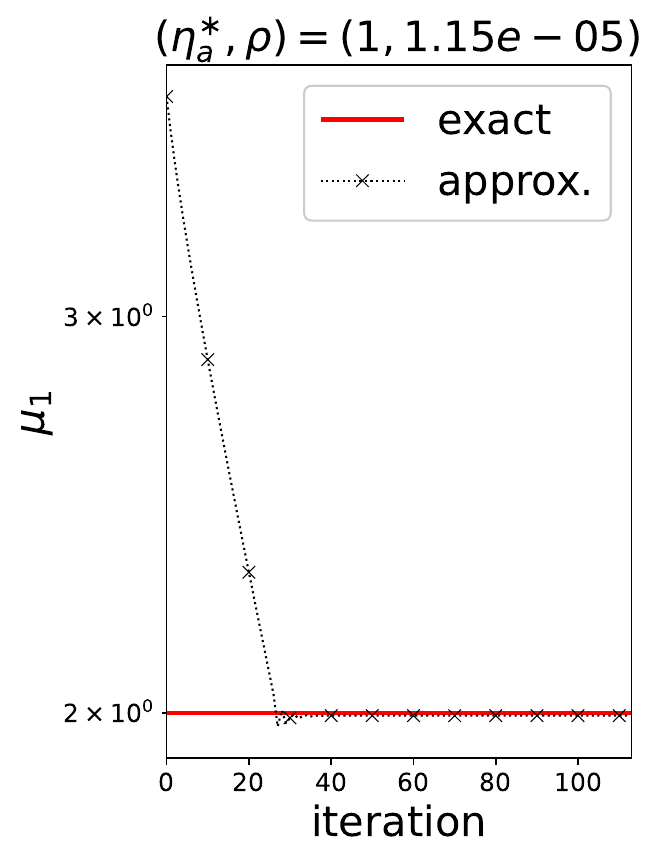}} \hfill
\resizebox{0.112\textwidth}{!}{\includegraphics{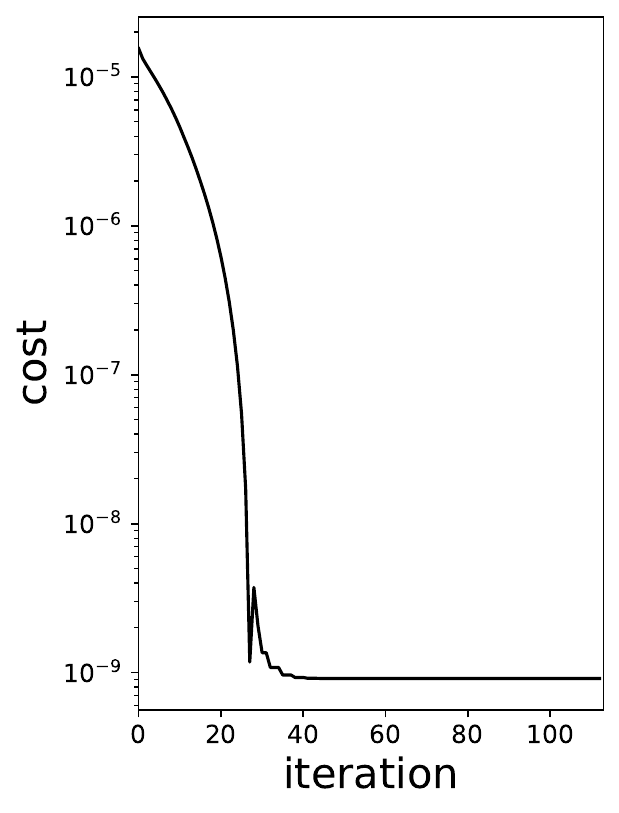}} \hfill
\resizebox{0.12\textwidth}{!}{\includegraphics{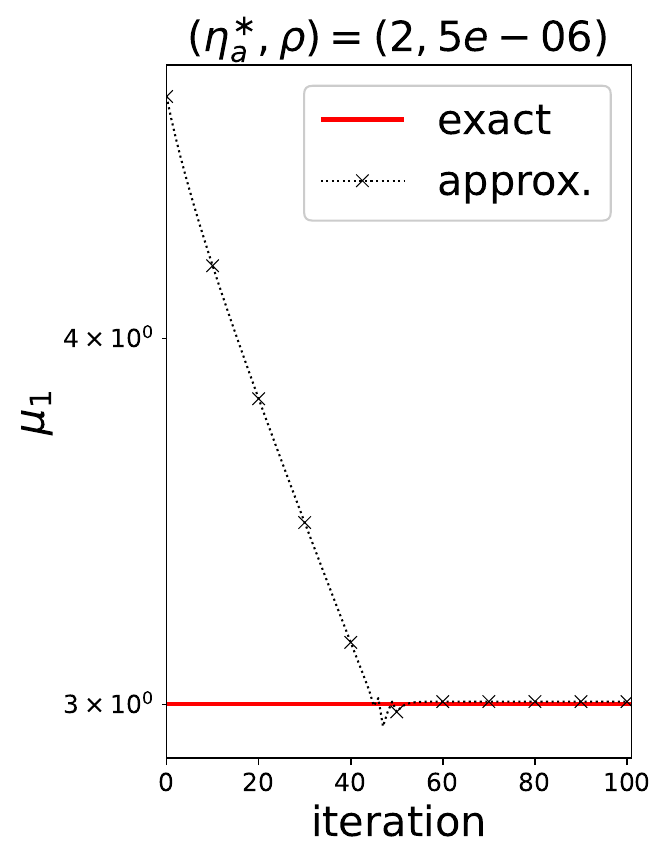}} \hfill
\resizebox{0.112\textwidth}{!}{\includegraphics{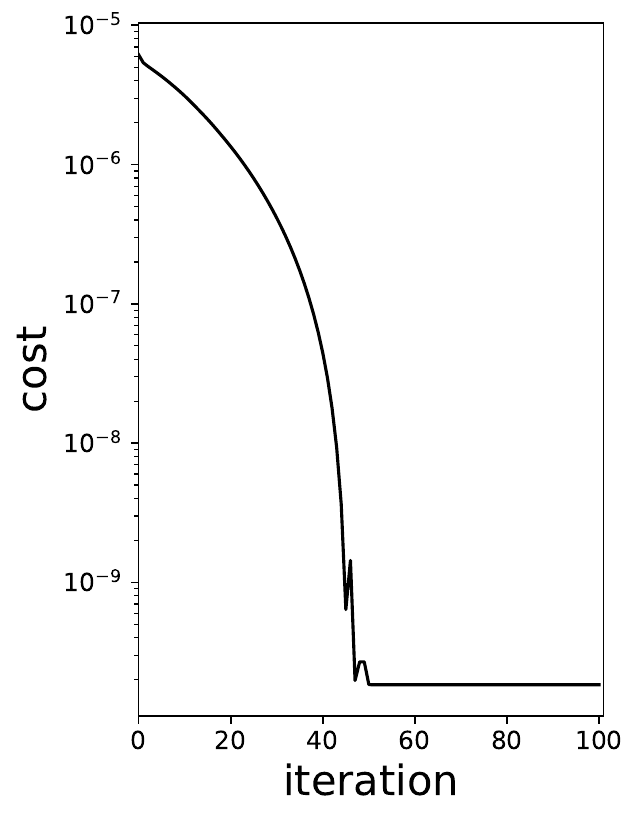}} \hfill
\resizebox{0.12\textwidth}{!}{\includegraphics{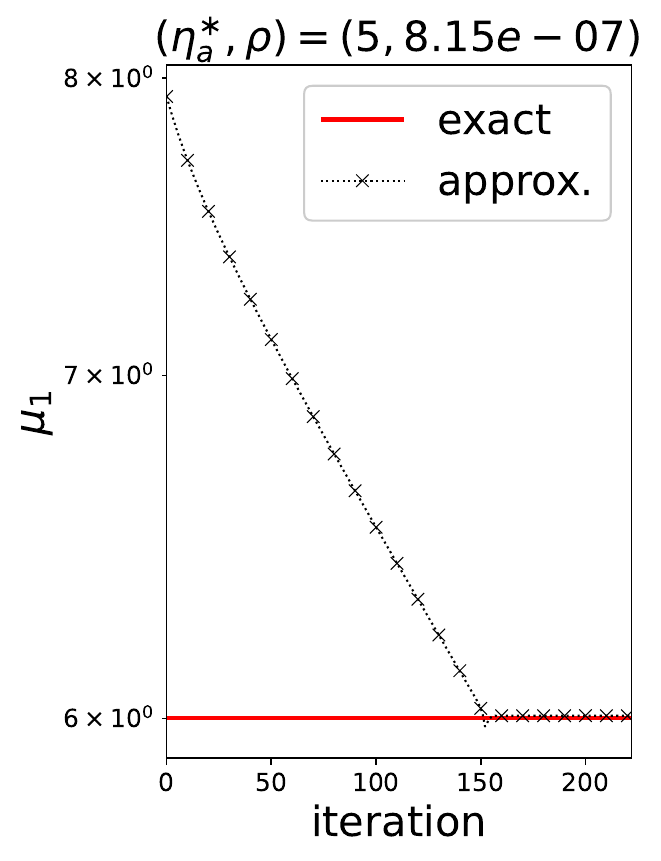}} \hfill
\resizebox{0.112\textwidth}{!}{\includegraphics{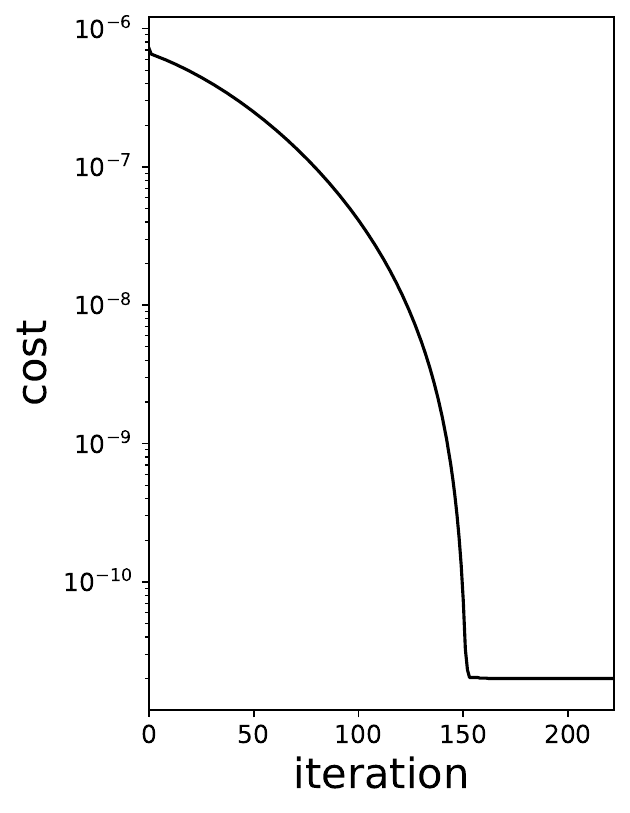}} \hfill
\caption{Results for a 2D radial  problem with $f(x) = \delta(x)$ and varying $\eta_{a}^{\ast}$}
\label{fig:point_source_example_4}
\end{figure}
%
%

The previous tests were based on precise measurements. For noisy data, we summarize the reconstructions of $\partial\omega^{\ast} = B_{1.5}$ and identifications when $\eta^{\ast} = 0.2, 5$ in Figure~\ref{fig:point_source_with_noise_example_5}.
Despite high noise levels, the identifications of both values and shapes were satisfactory.
We used the same $\rho$ as in the non-noisy case to demonstrate the impact of noise under a fixed Tikhonov regularization parameter.
As expected, the reconstructions were less accurate than with exact measurements, as shown by the cost value history in the rightmost plots of Figure~\ref{fig:point_source_with_noise_example_5}.
The reconstructions can be improved by adjusting $\rho$.
%
%
\begin{figure}[htp!]
\centering
\resizebox{0.165\textwidth}{!}{\includegraphics{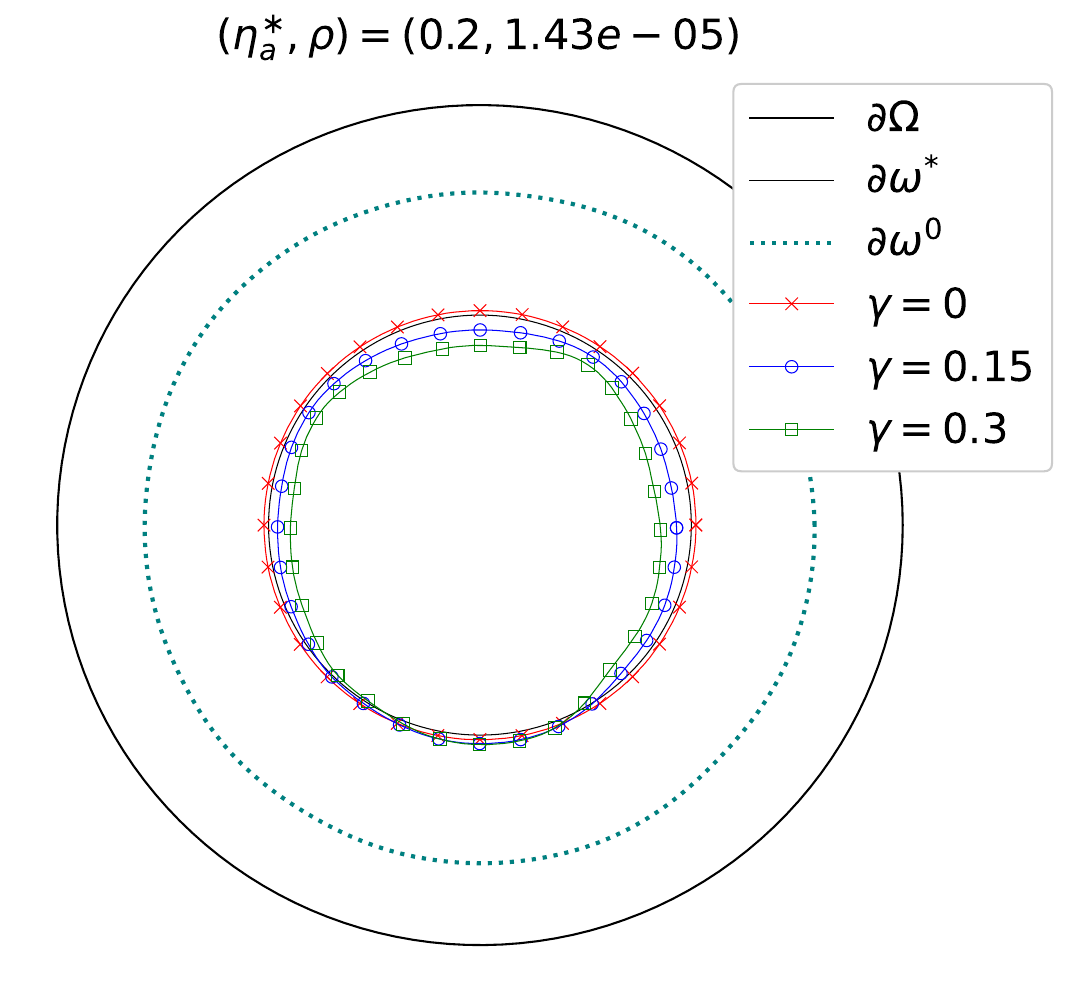}}
\resizebox{0.15\textwidth}{!}{\includegraphics{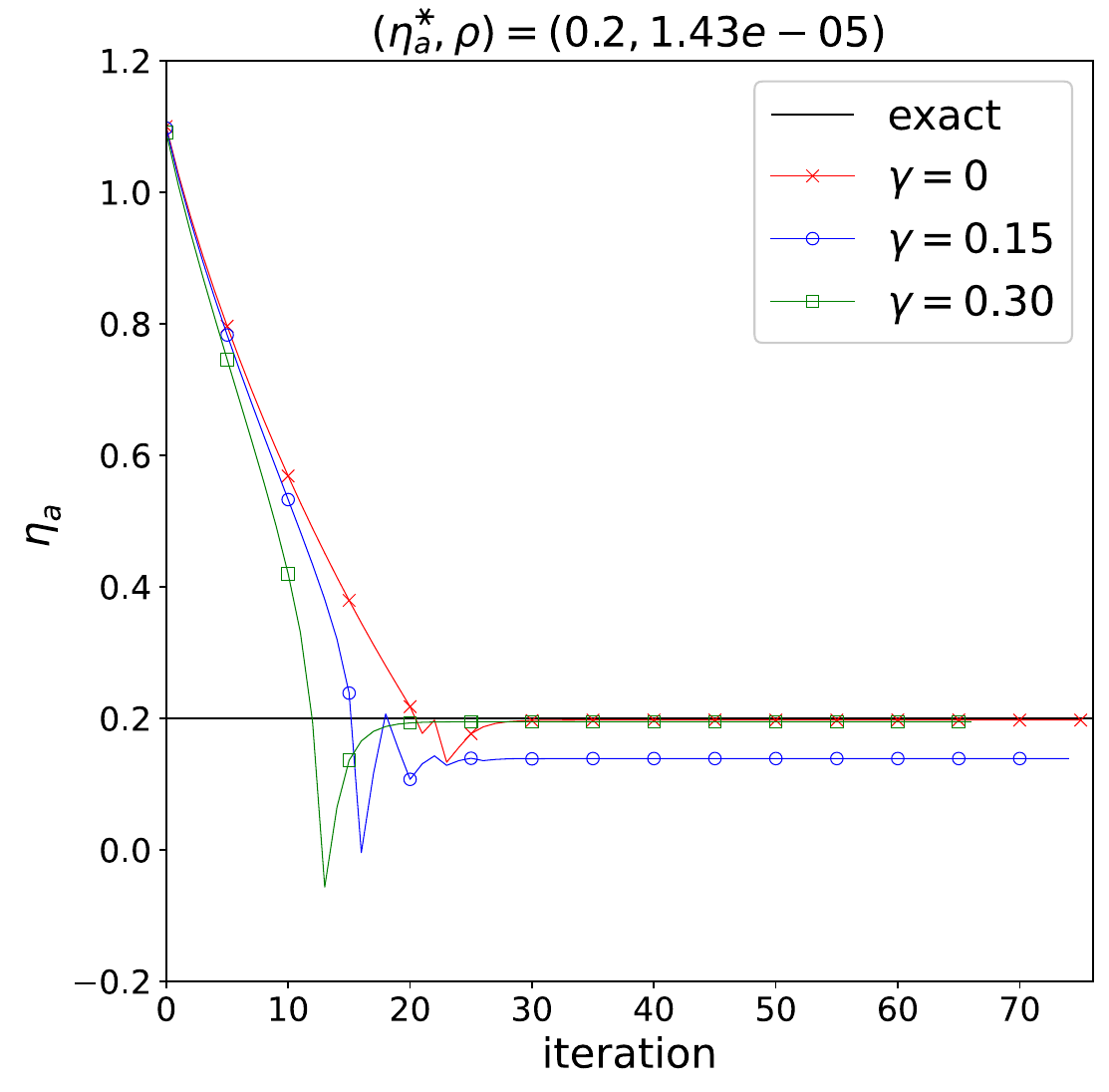}}
\resizebox{0.15\textwidth}{!}{\includegraphics{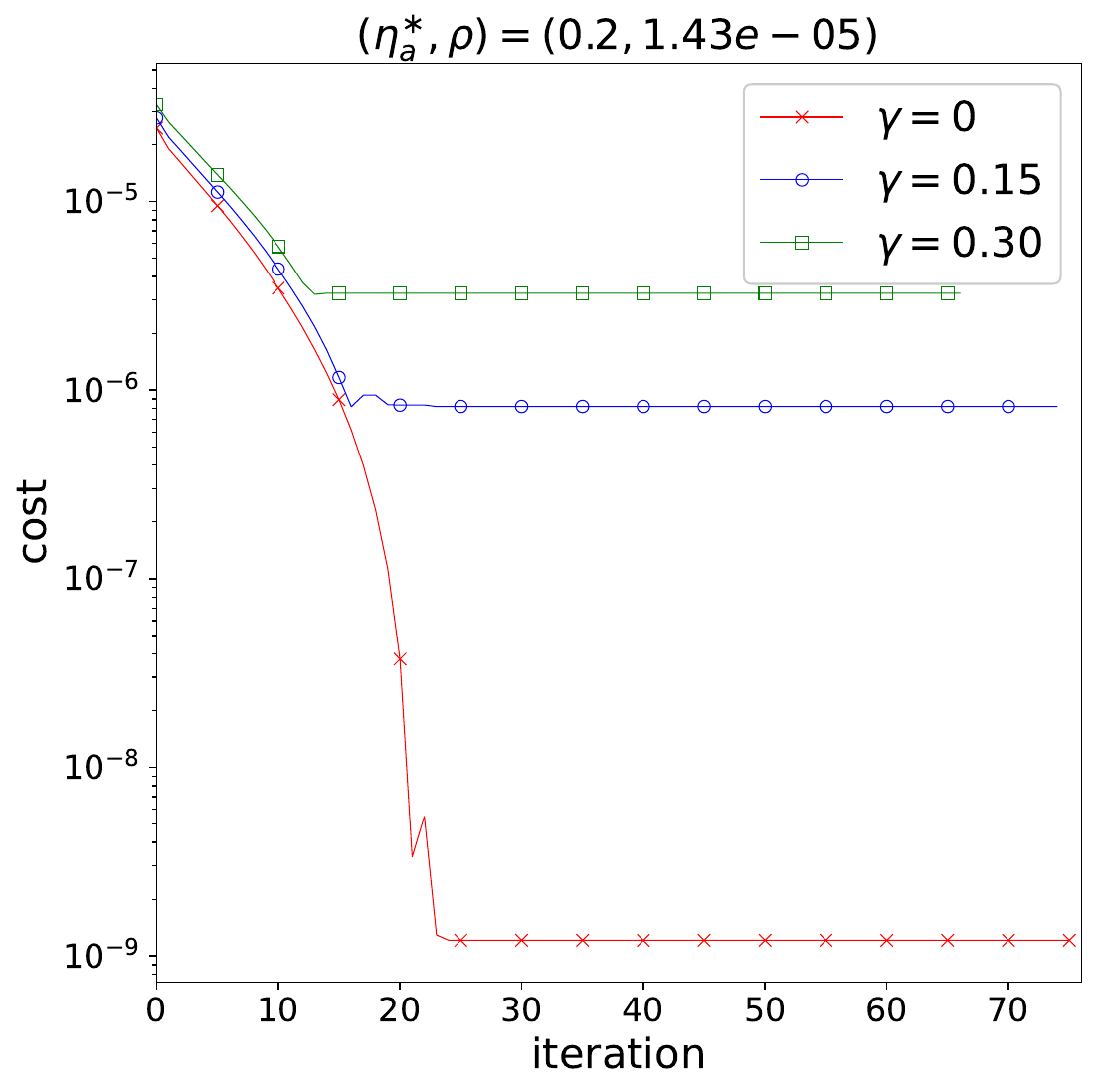}}\quad
\resizebox{0.165\textwidth}{!}{\includegraphics{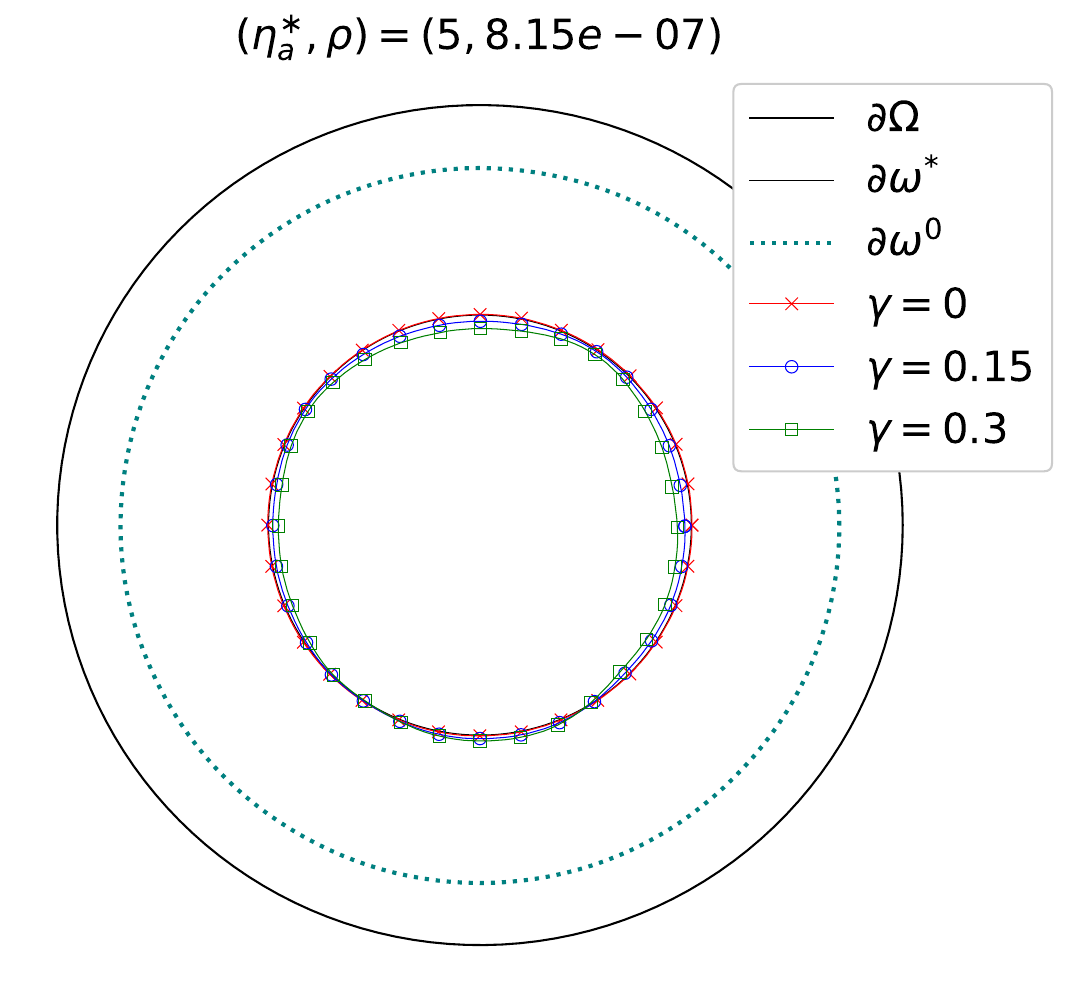}}
\resizebox{0.15\textwidth}{!}{\includegraphics{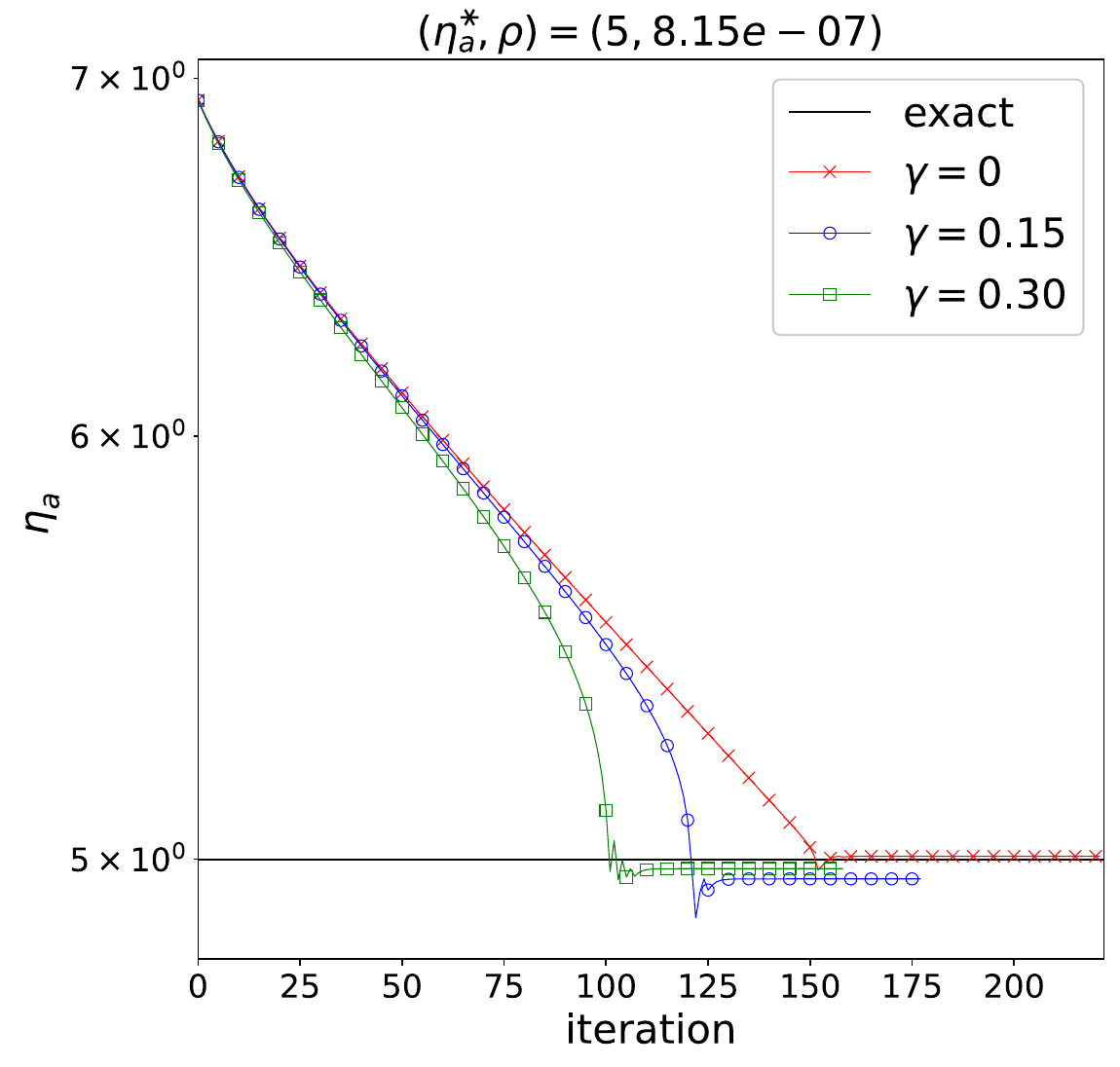}}
\resizebox{0.15\textwidth}{!}{\includegraphics{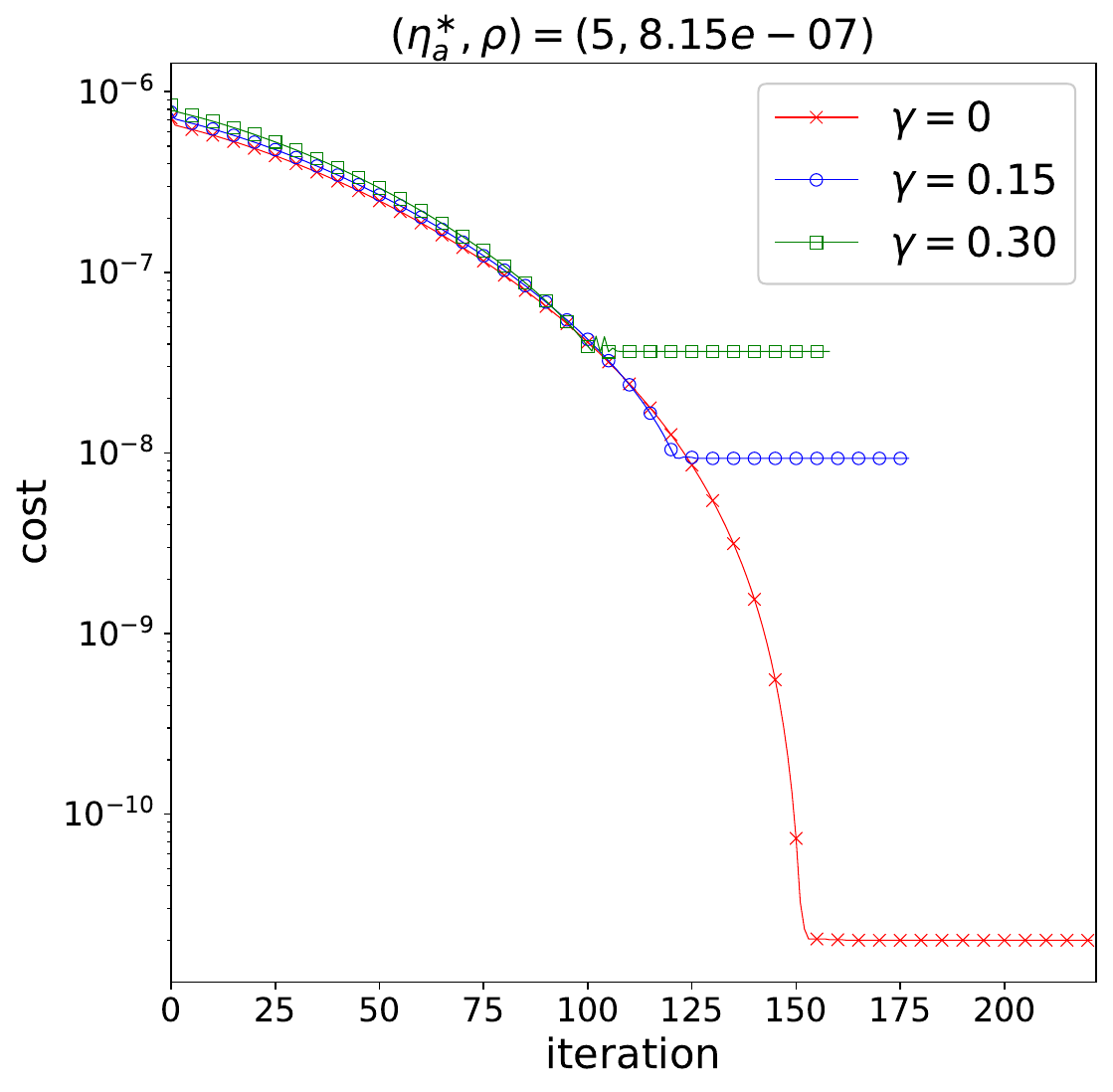}}
\caption{Results for a 2D radial problem with $f(x) = \delta(x)$ under noisy data, with $\eta^{\ast} = 0.2$ (left three plots) and $\eta^{\ast} = 5$ (right three plots)}
\label{fig:point_source_with_noise_example_5}
\end{figure}
%
\subsection{Numerical tests with a constant source function and non-circular boundary interface}\label{subsec:constant_source_non_circular}
We repeat the experiments from subsection~\ref{subsec:constant_source}, but with the exact boundary interface given by:
\[
	\domega^{\ast} = \left\{5(0.4+0.06\cos(3t))\begin{pmatrix}\cos{t}\\ \sin{t}\end{pmatrix}, \forall t \in [0, 2\pi) \right\}.
\]
This boundary is non-convex with minor concavities.
As before, we set $f=1$, $k=1$, $R=3$, and $l=0.3$, and show the results for $\etaa^{\ast} = 0.2$ in Figures \ref{fig:constant_source_flower_1} with $\rho = 0.0002$ and $s = 4$.
Though the reconstruction misses the concavities, it closely matches the exact boundary, and the method accurately identifies $\mu_{1}$.
%
%
\begin{figure}[htp!]
\centering
\resizebox{0.18\textwidth}{!}{\includegraphics{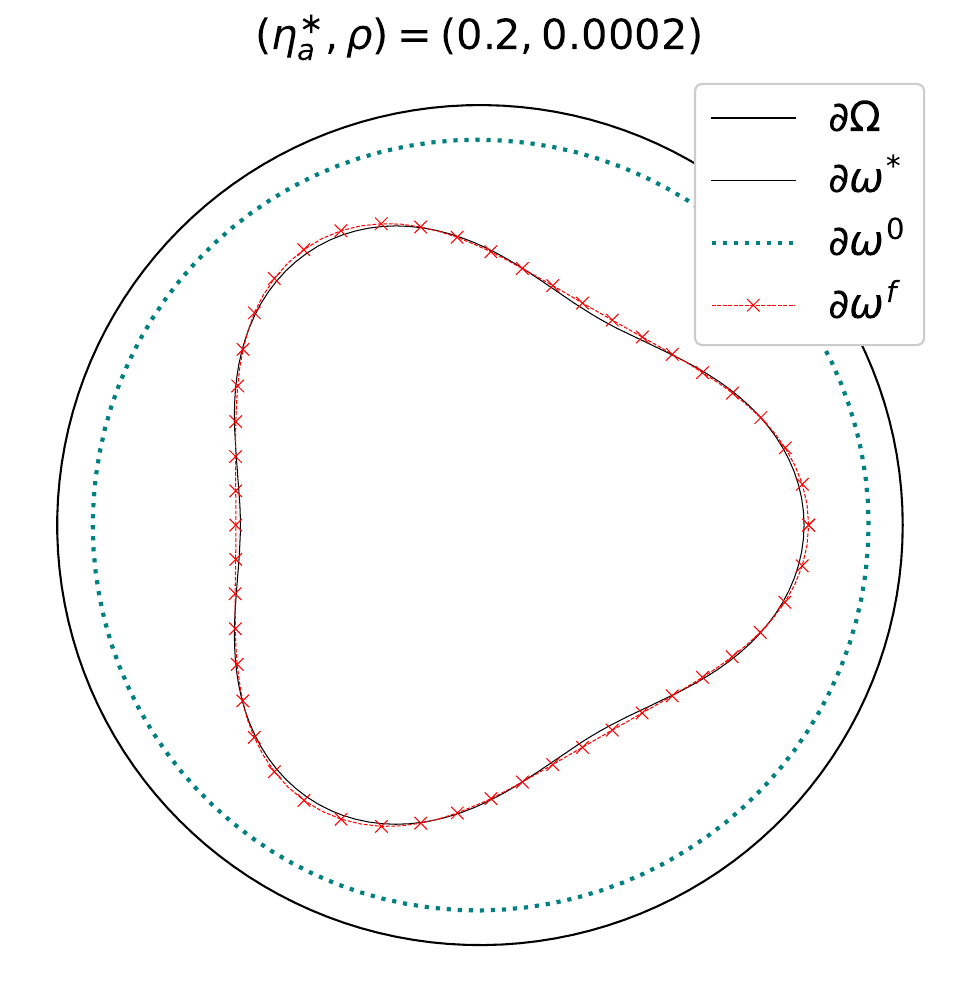}} \quad
\resizebox{0.15\textwidth}{!}{\includegraphics{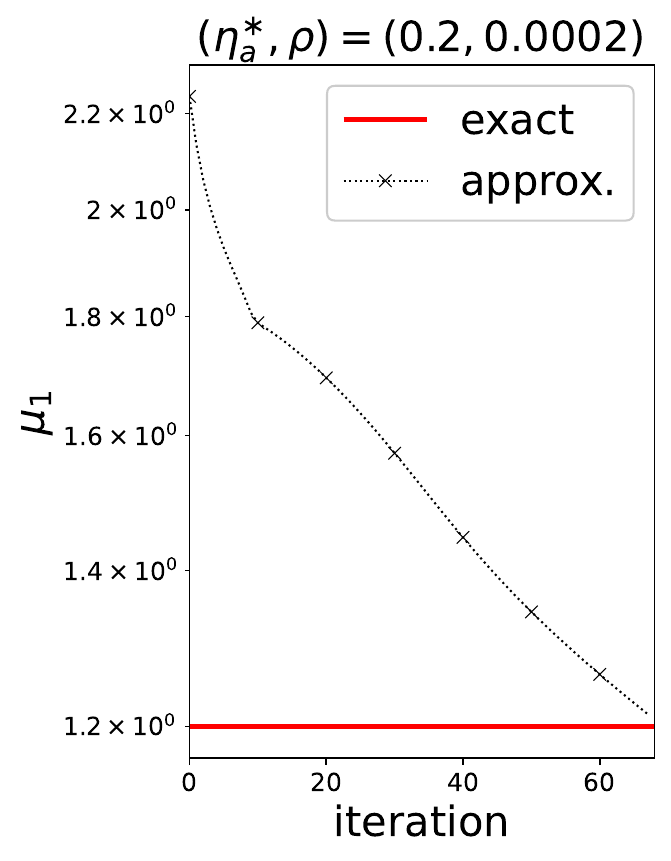}} \quad
\resizebox{0.14\textwidth}{!}{\includegraphics{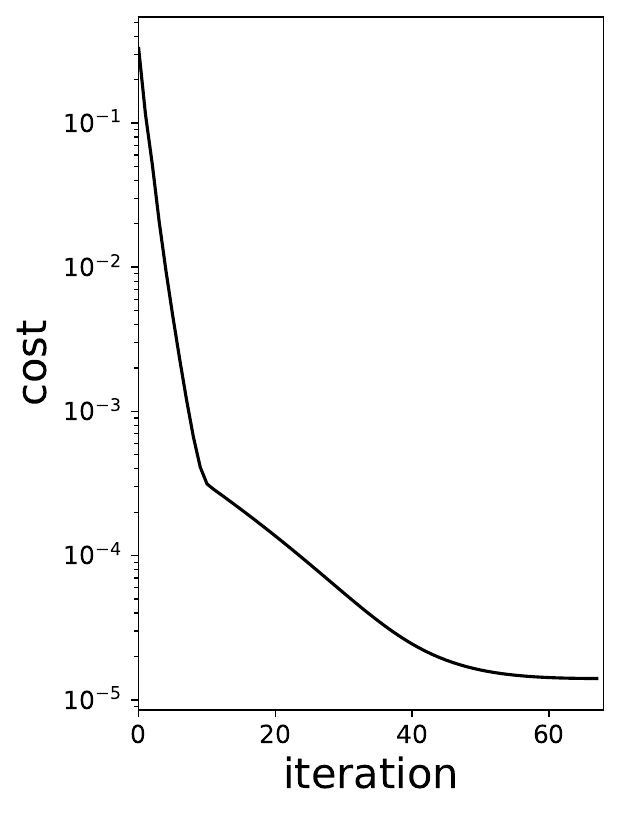}}
\caption{An example featuring a non-circular boundary interface}
\label{fig:constant_source_flower_1}
\end{figure}
%
%

We perform another set of experiments with noisy data at two different levels.
The numerical results, displayed in Figures \ref{fig:constant_source_flower_2}, show the results when $\gamma = 0.5, 0.10$.
{As expected, the reconstructions under noisy measurements are less accurate than those obtained with exact data, but they remain reasonable, as illustrated in the figures.
Furthermore, the cost values are higher in the presence of noise, which is consistent with expectations.}
%
%
\begin{figure}[htp!]
\centering
\resizebox{0.2\textwidth}{!}{\includegraphics{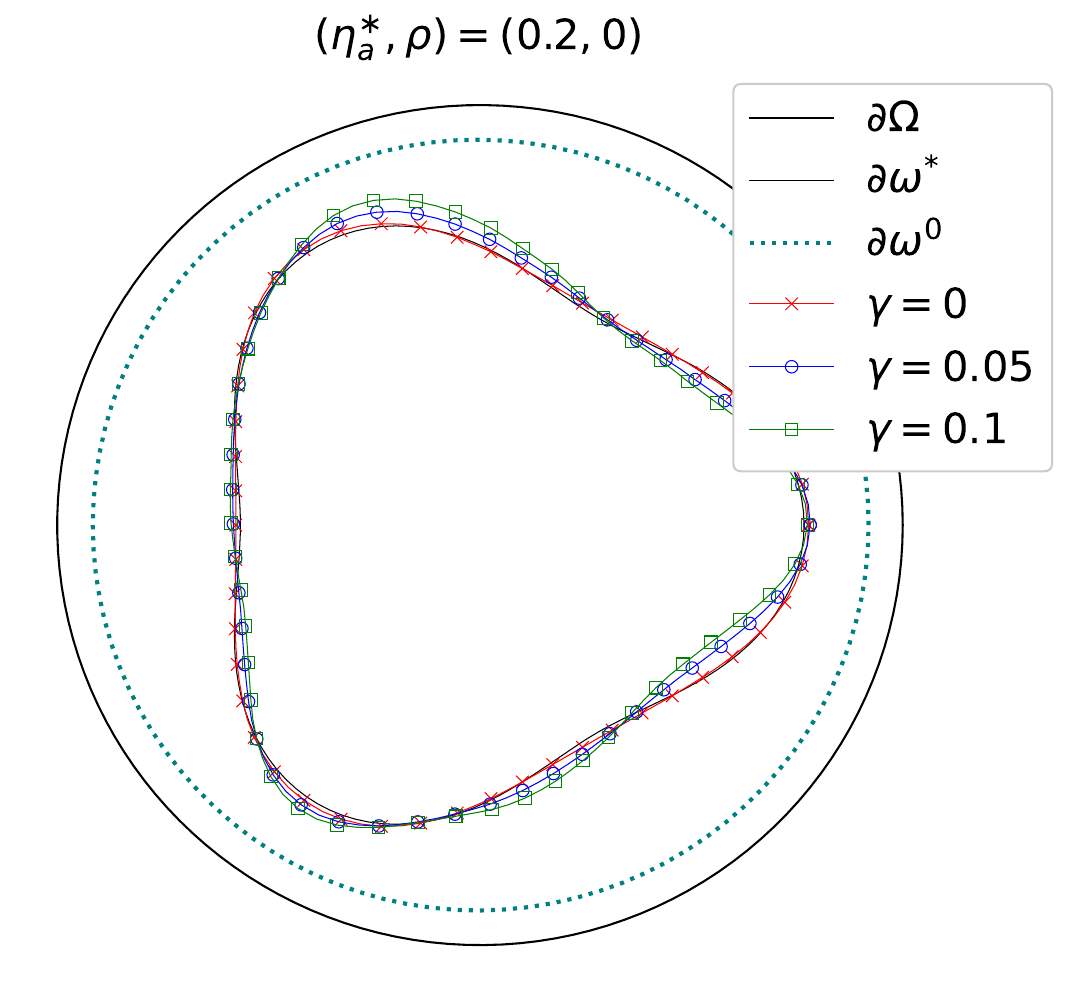}} \quad
\resizebox{0.19\textwidth}{!}{\includegraphics{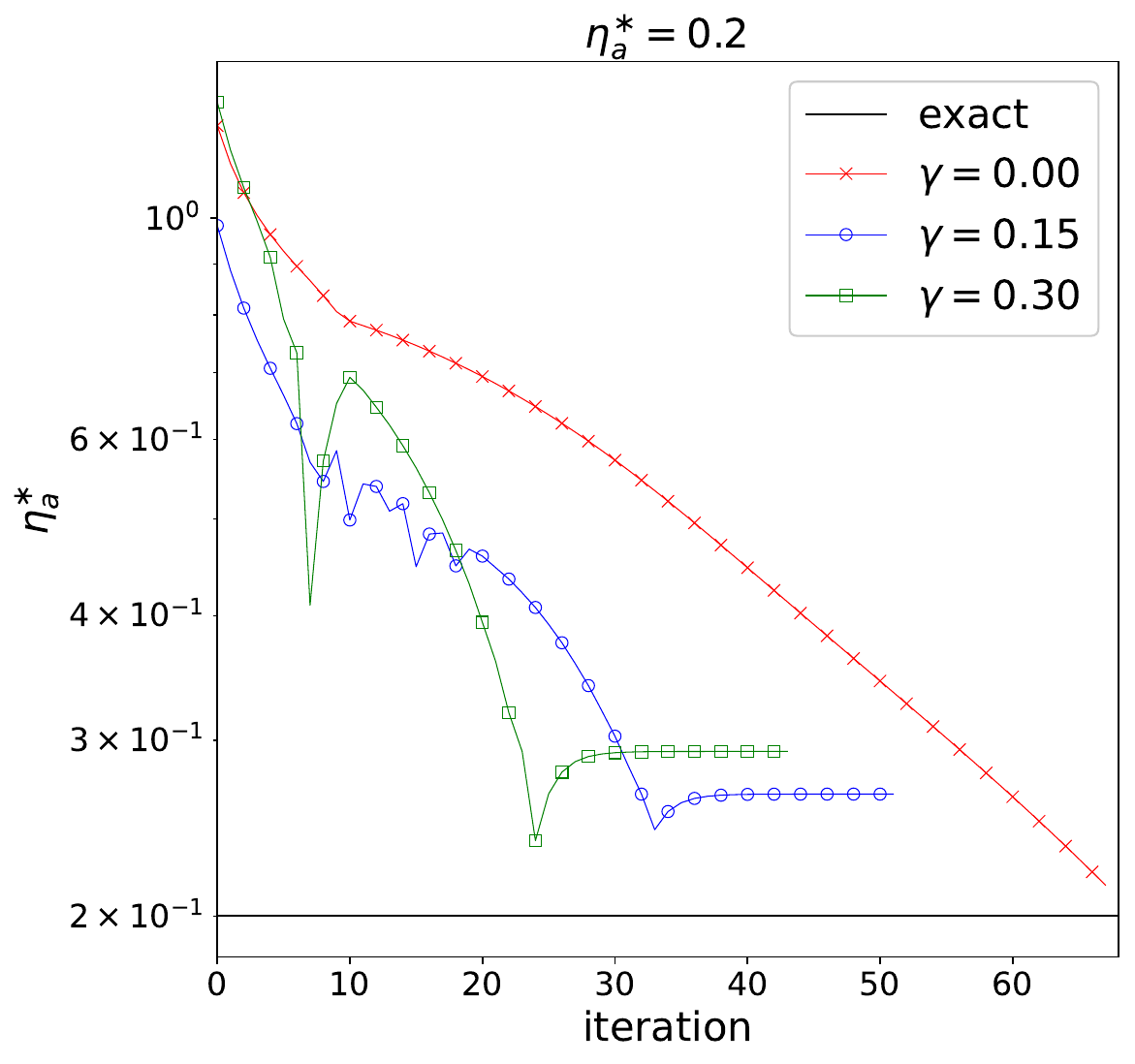}} \quad
\resizebox{0.18\textwidth}{!}{\includegraphics{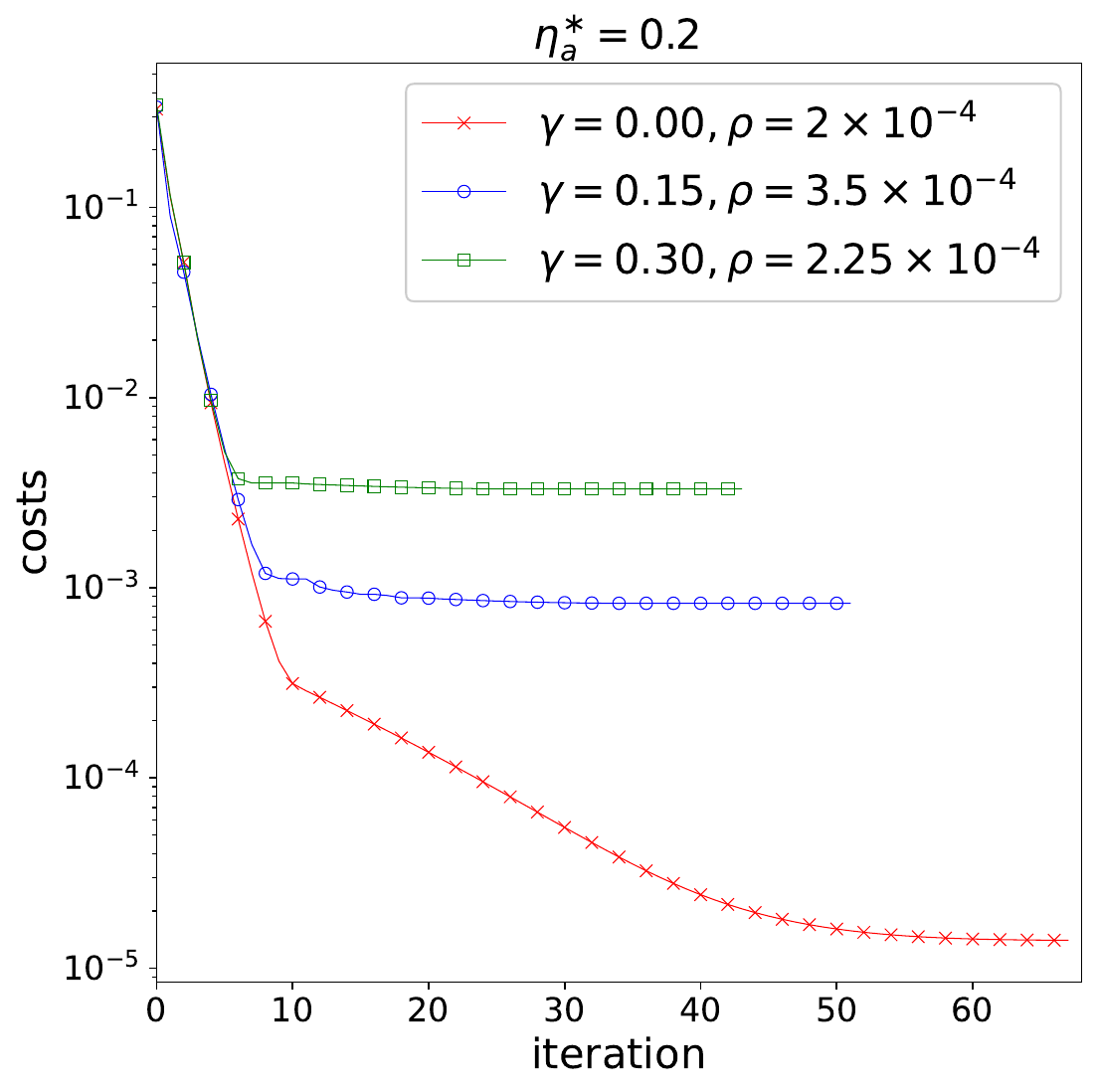}}
\caption{An example featuring a non-circular boundary interface with noisy data}
\label{fig:constant_source_flower_2}
\end{figure}
%
%

Next, we conduct experiments by varying $\etaa^{\ast}$ with exact measurements, where the initial guess is smaller than the exact boundary interface and positioned inside $\domega^{\ast}$.
Specifically, we use $\domega^{0} = B_{0.5}$ as the initial boundary geometry and set $s=10$.
The results, including the histories of $\mu_{1}$ and cost values, are shown in Figure~\ref{fig:constant_source_flower_3}.
The method successfully identified the boundary interface, including its concave parts, and the values for $\mu_{1}$ closely match the exact values, demonstrating the approach's robustness.
%
%
\begin{figure}[htp!]
\centering
\hfill
\resizebox{0.2\textwidth}{!}{\includegraphics{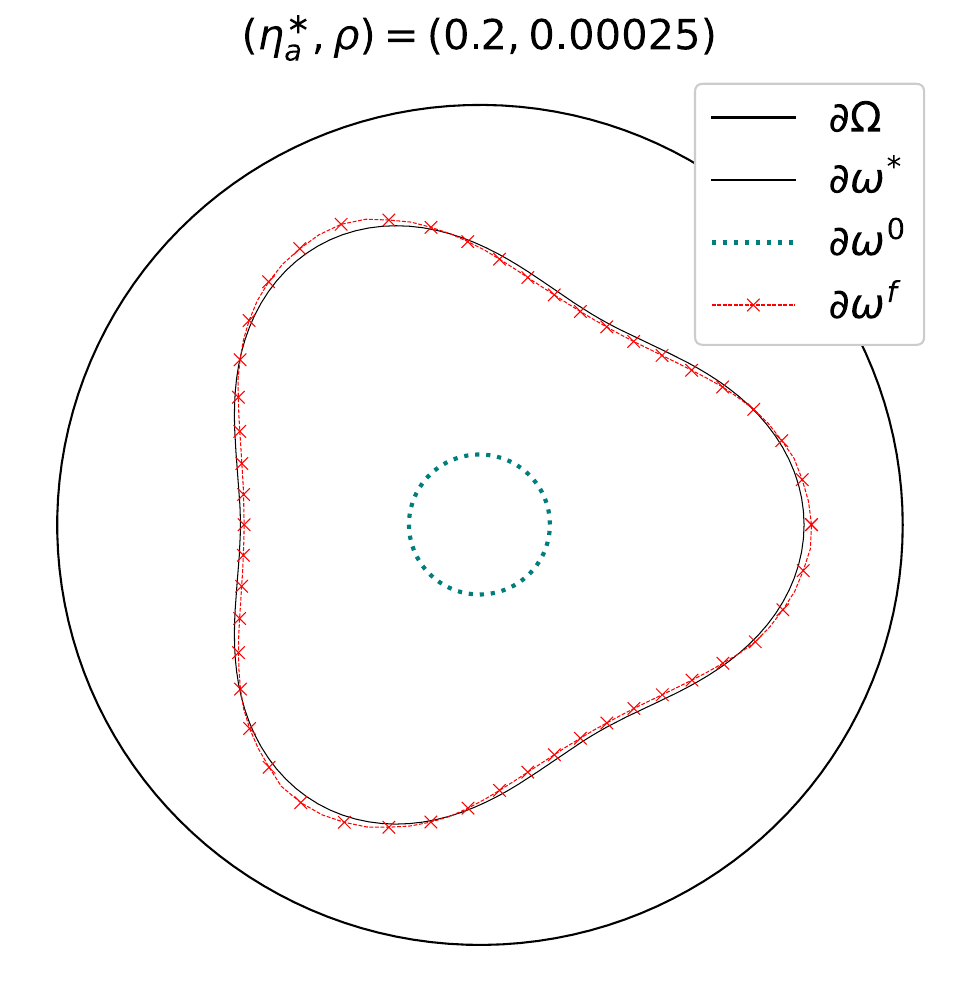}} \hfill
\resizebox{0.2\textwidth}{!}{\includegraphics{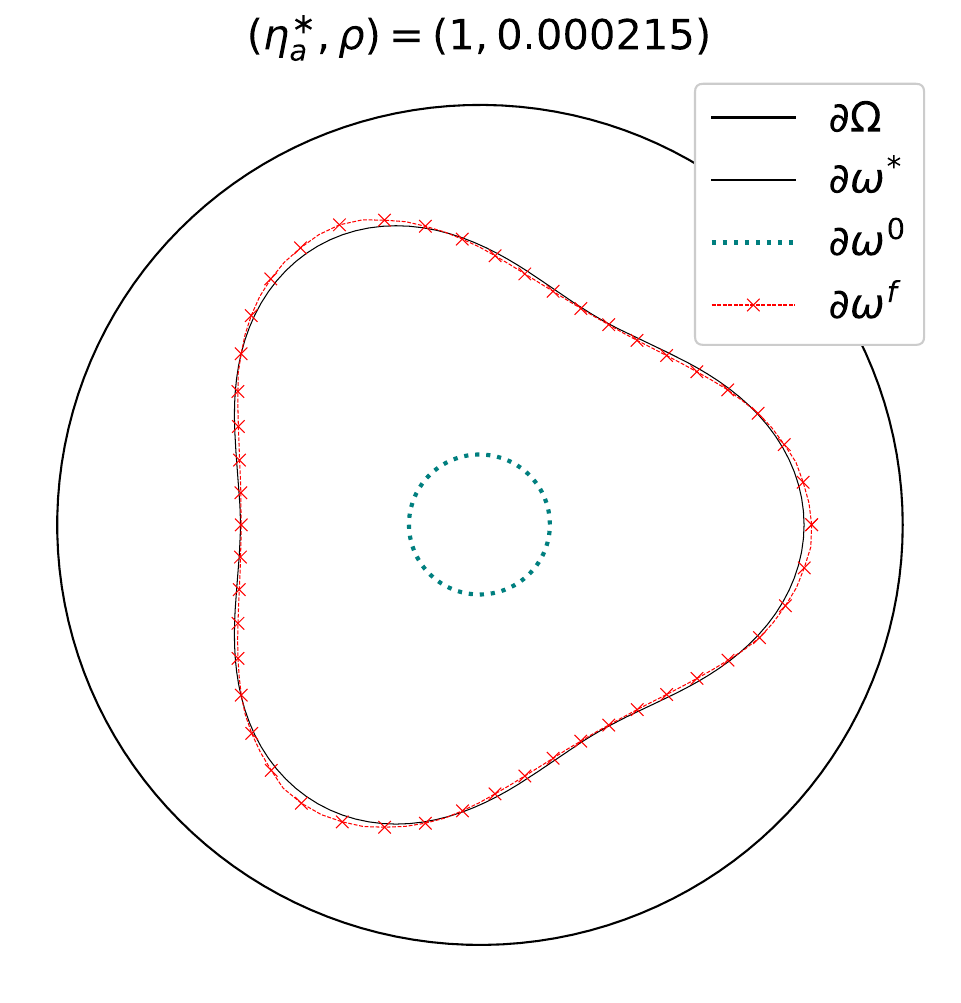}} \hfill
\resizebox{0.2\textwidth}{!}{\includegraphics{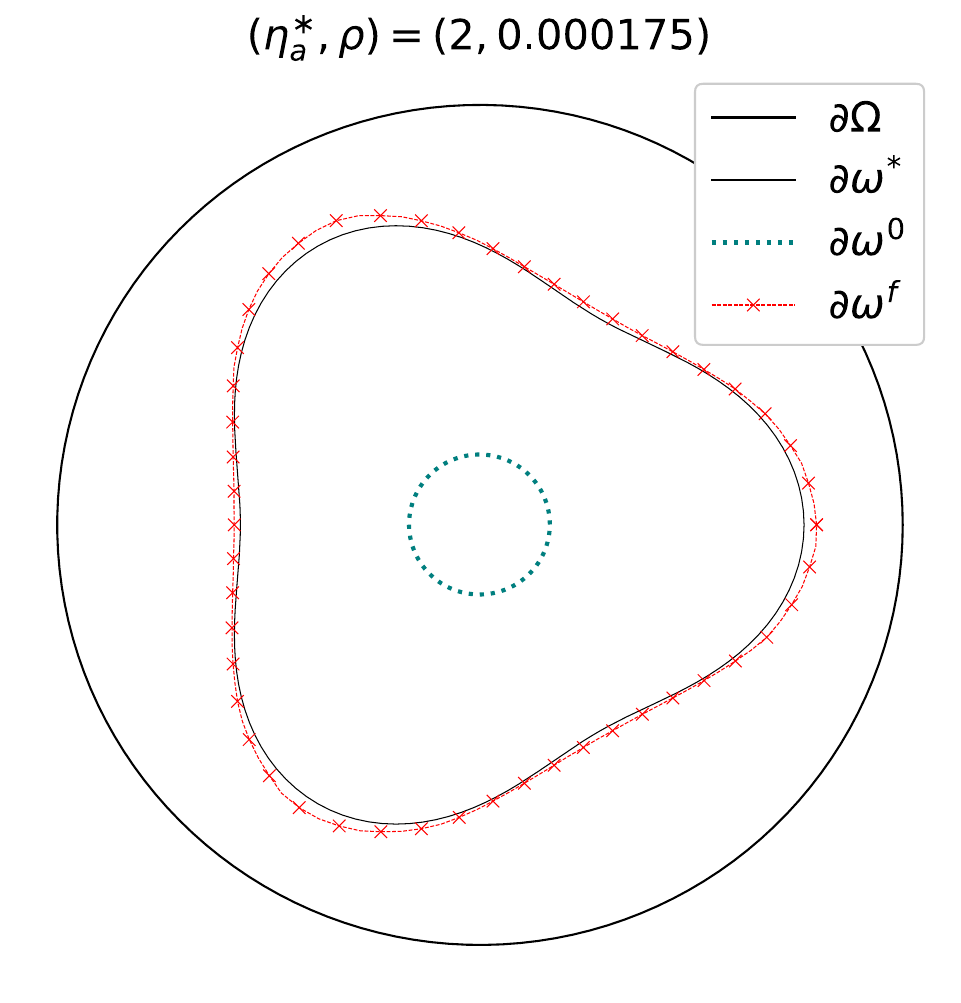}} \hfill
\resizebox{0.2\textwidth}{!}{\includegraphics{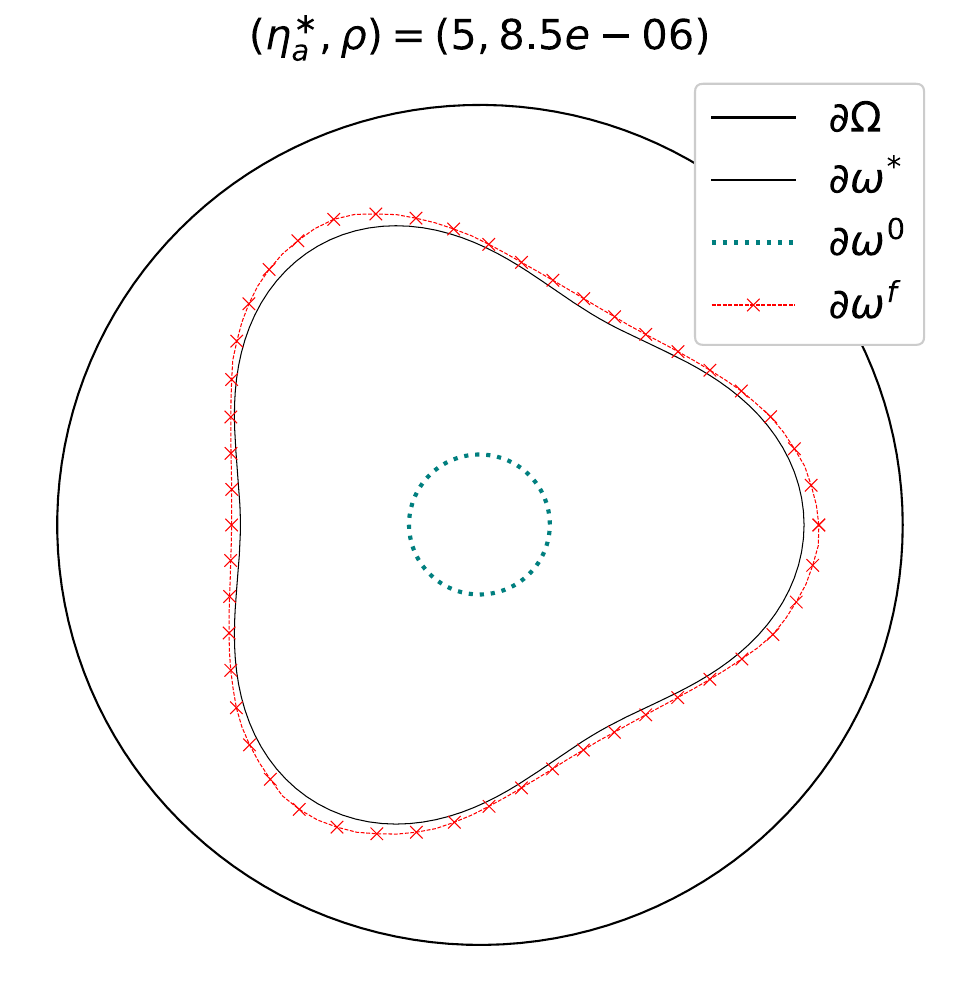}} \hfill\\
\resizebox{0.11\textwidth}{!}{\includegraphics{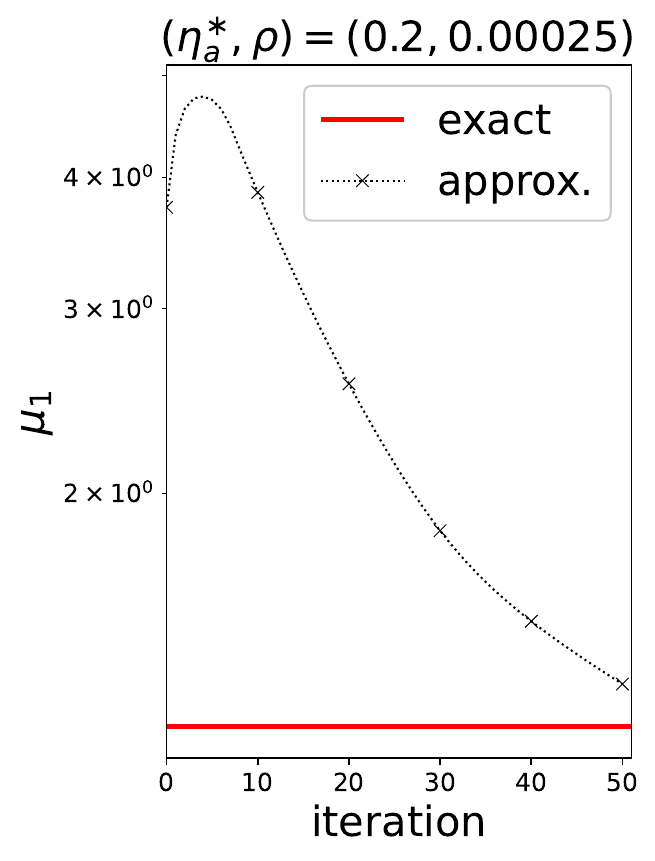}} \hfill
\resizebox{0.11\textwidth}{!}{\includegraphics{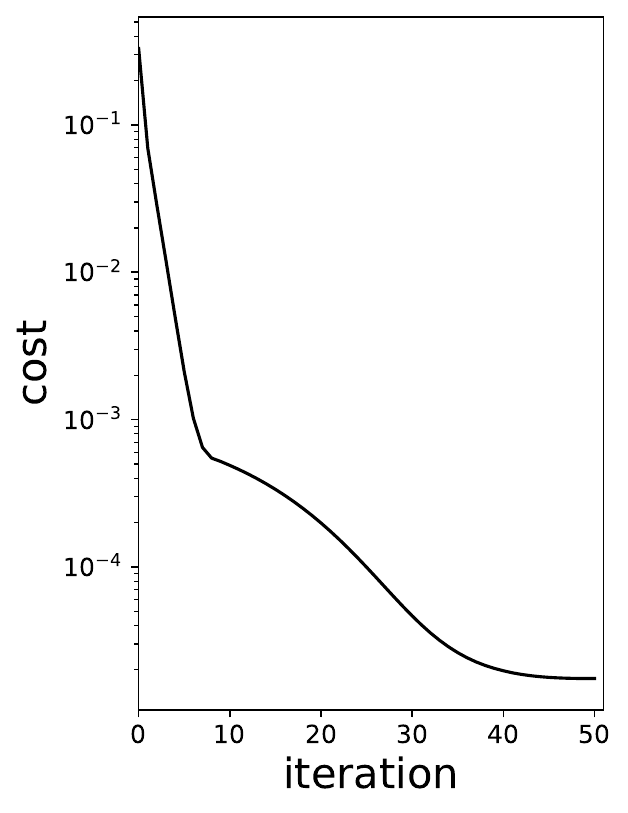}} \hfill
\resizebox{0.11\textwidth}{!}{\includegraphics{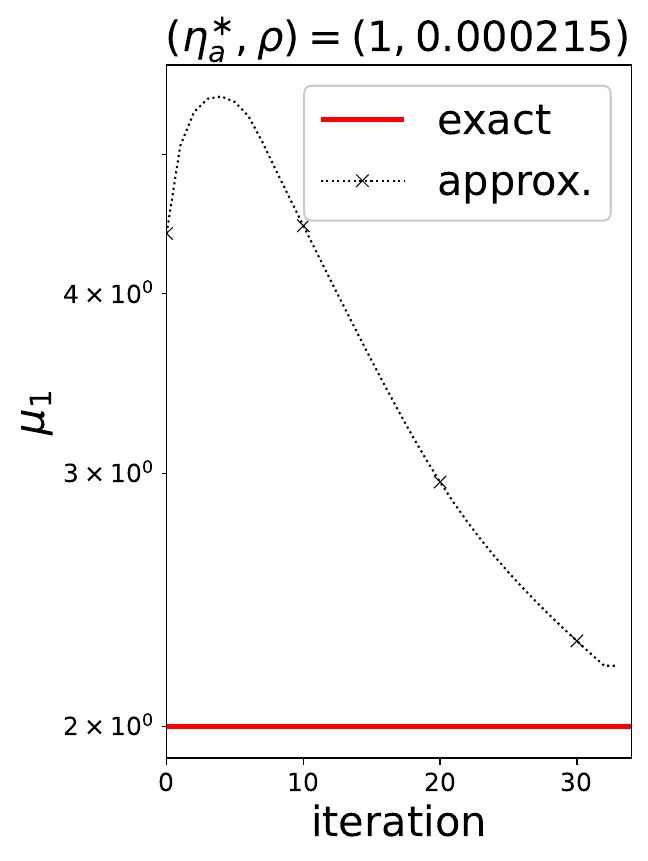}} \hfill
\resizebox{0.11\textwidth}{!}{\includegraphics{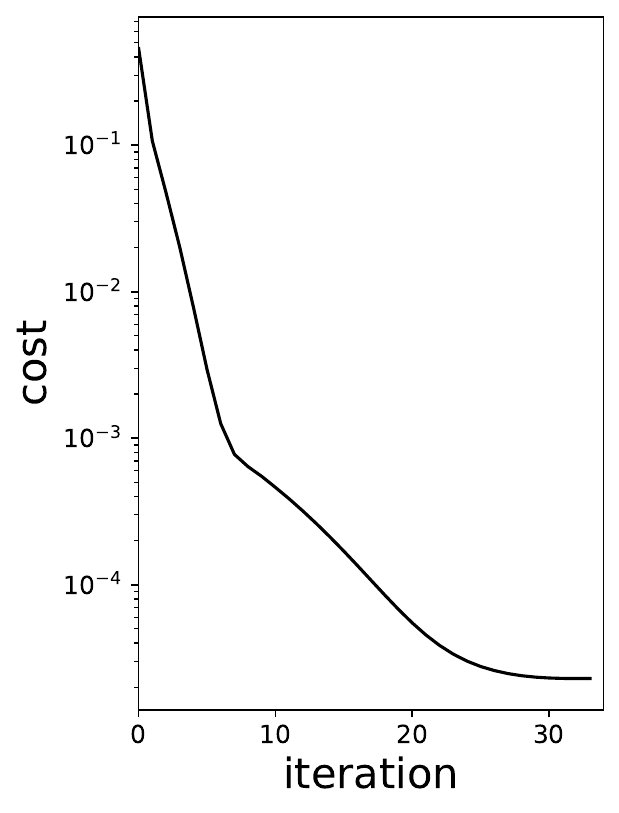}} \hfill
\resizebox{0.11\textwidth}{!}{\includegraphics{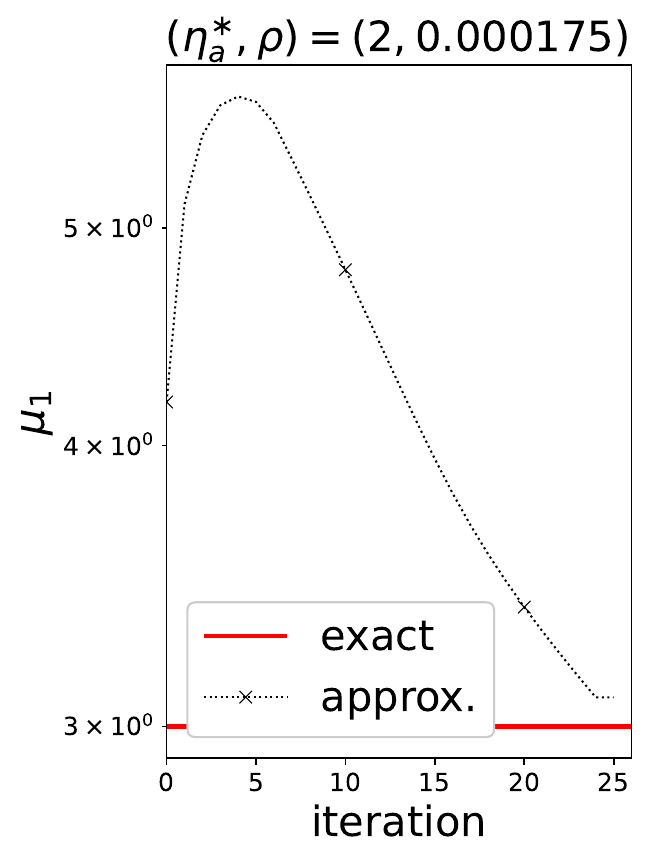}} \hfill
\resizebox{0.11\textwidth}{!}{\includegraphics{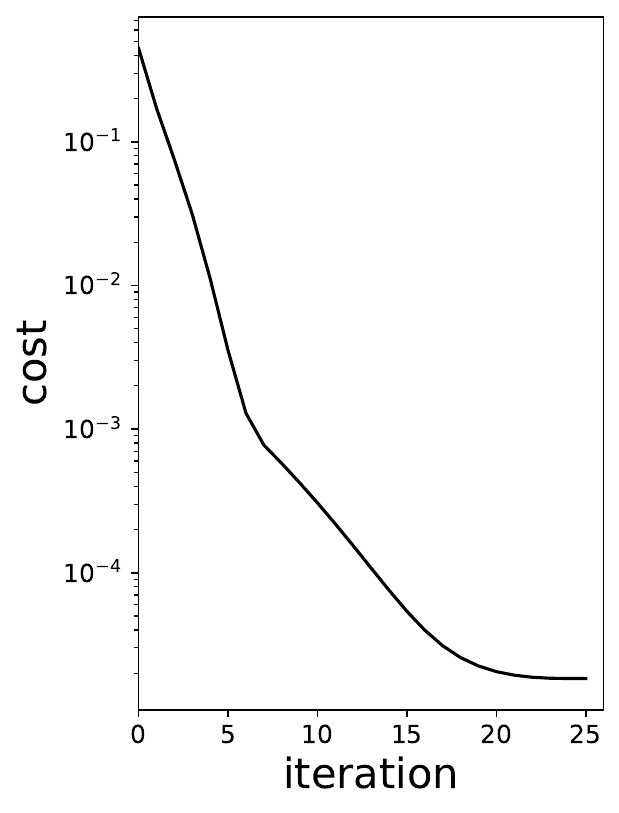}} \hfill
\resizebox{0.11\textwidth}{!}{\includegraphics{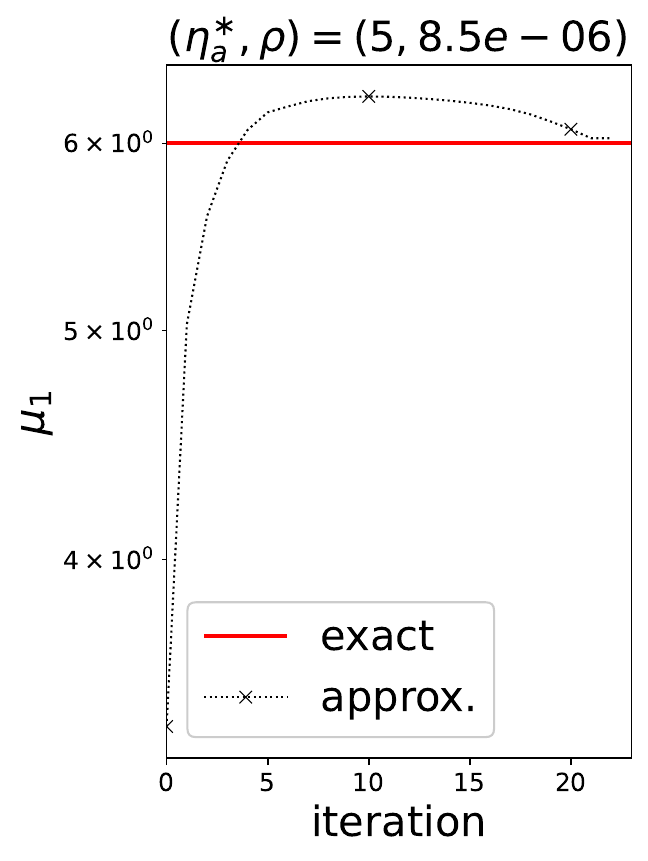}} \hfill
\resizebox{0.11\textwidth}{!}{\includegraphics{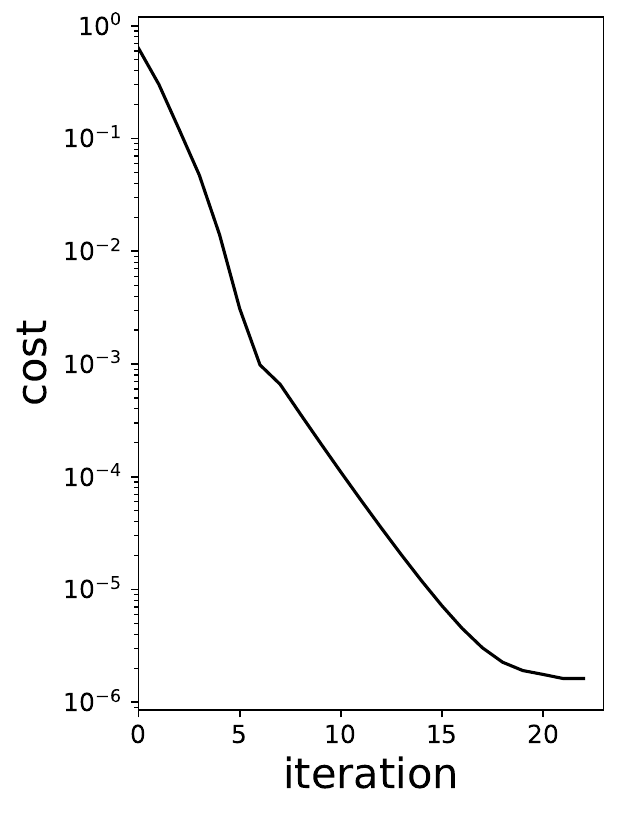}} \hfill
\caption{Results for a non-circular boundary interface with varying $\eta_{a}^{\ast}$ values}
\label{fig:constant_source_flower_3}
\end{figure}
%
%

For the reconstruction with noisy data at $\gamma = 0.05, 0.10$ and $\eta^{\ast} = 5$, the results are shown in Figure~\ref{fig:constant_source_flower_4}.
The Tikhonov regularization parameter $\rho$ is chosen based on the noise level, and a large step-size parameter $s = 5$ is used.
Despite the noise, the method successfully reconstructed the boundary interface and identified the absorption coefficient with good accuracy.
The histories of $\mu_{1}$ and the cost are also shown, with higher noise levels corresponding to larger cost values.
%
%
\begin{figure}[htp!]
\centering
\resizebox{0.2\textwidth}{!}{\includegraphics{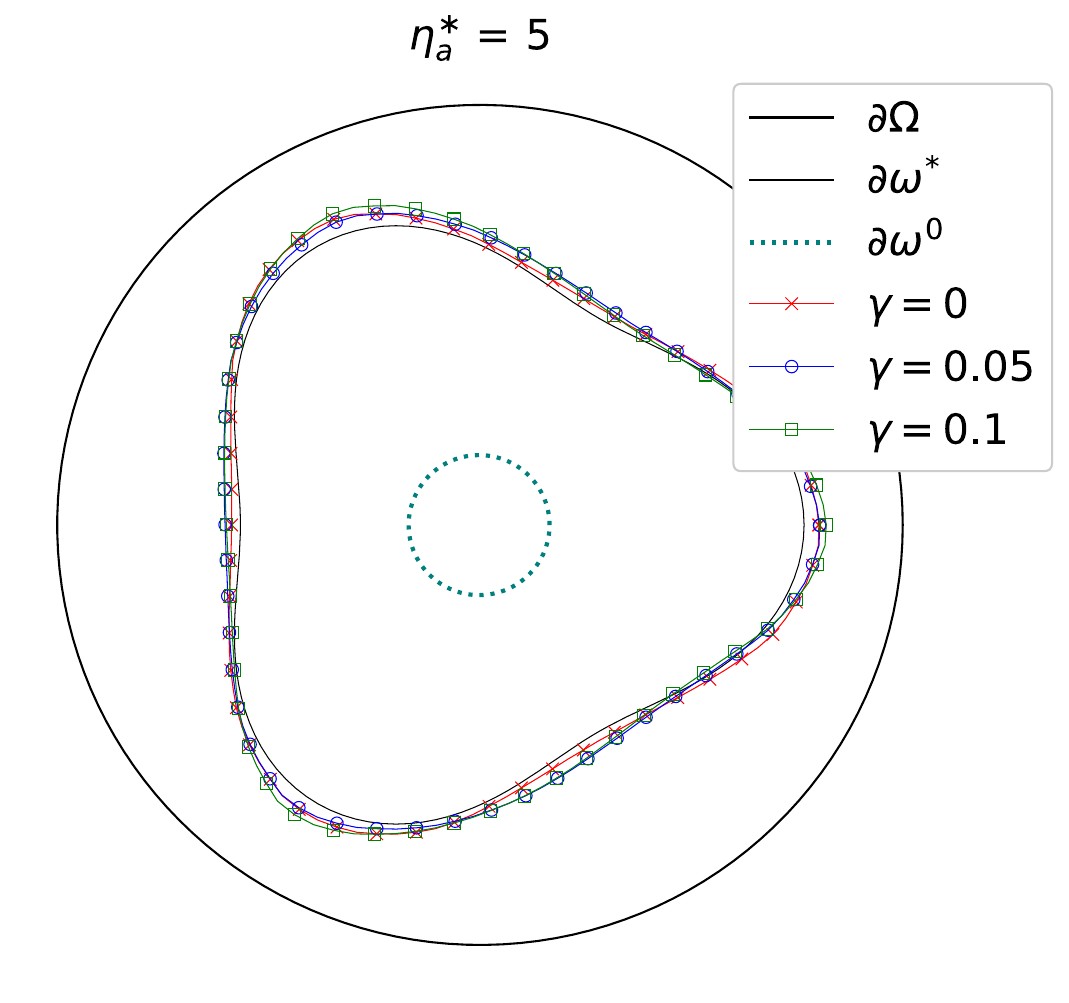}} \quad
\resizebox{0.19\textwidth}{!}{\includegraphics{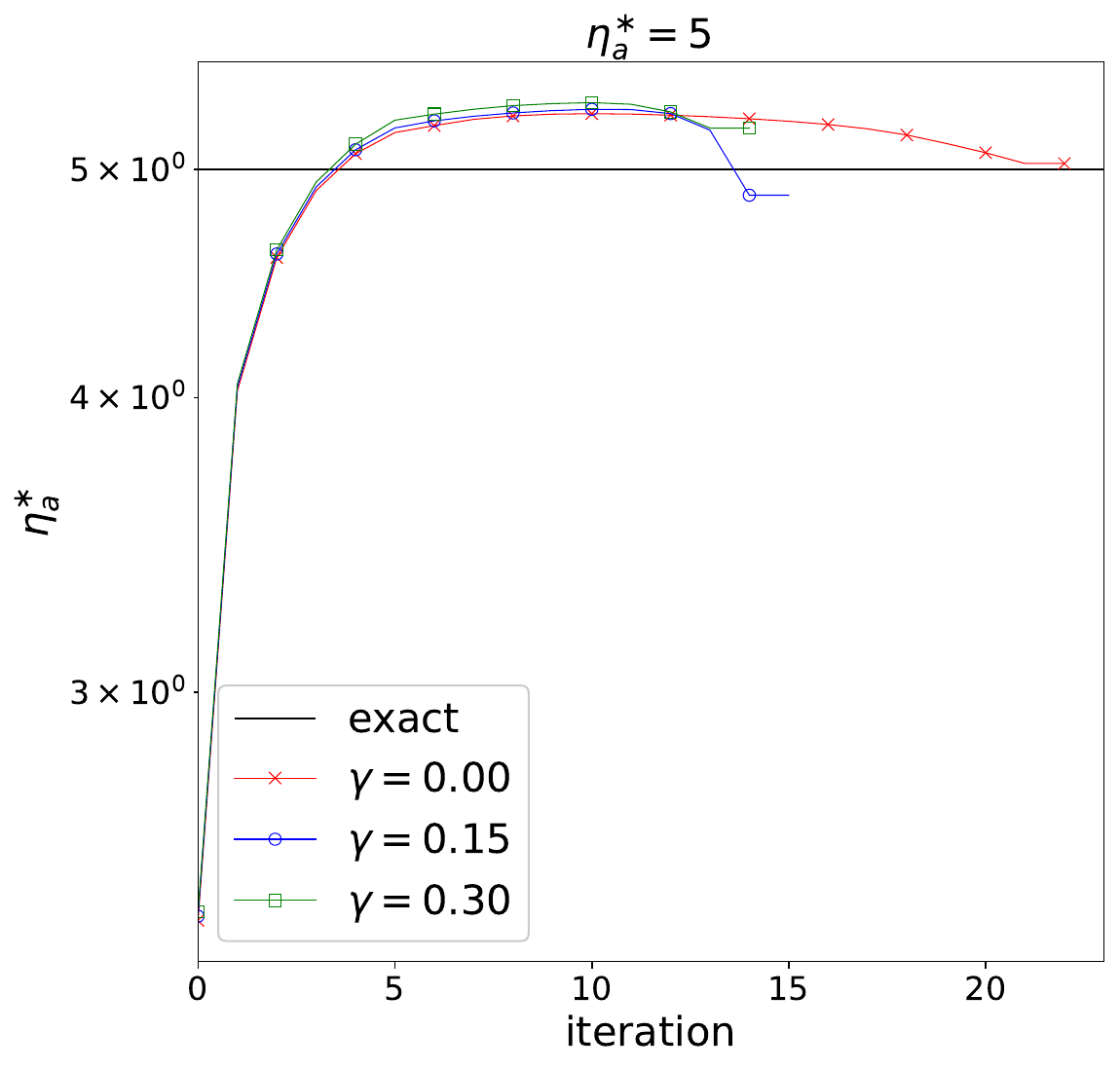}} \quad
\resizebox{0.18\textwidth}{!}{\includegraphics{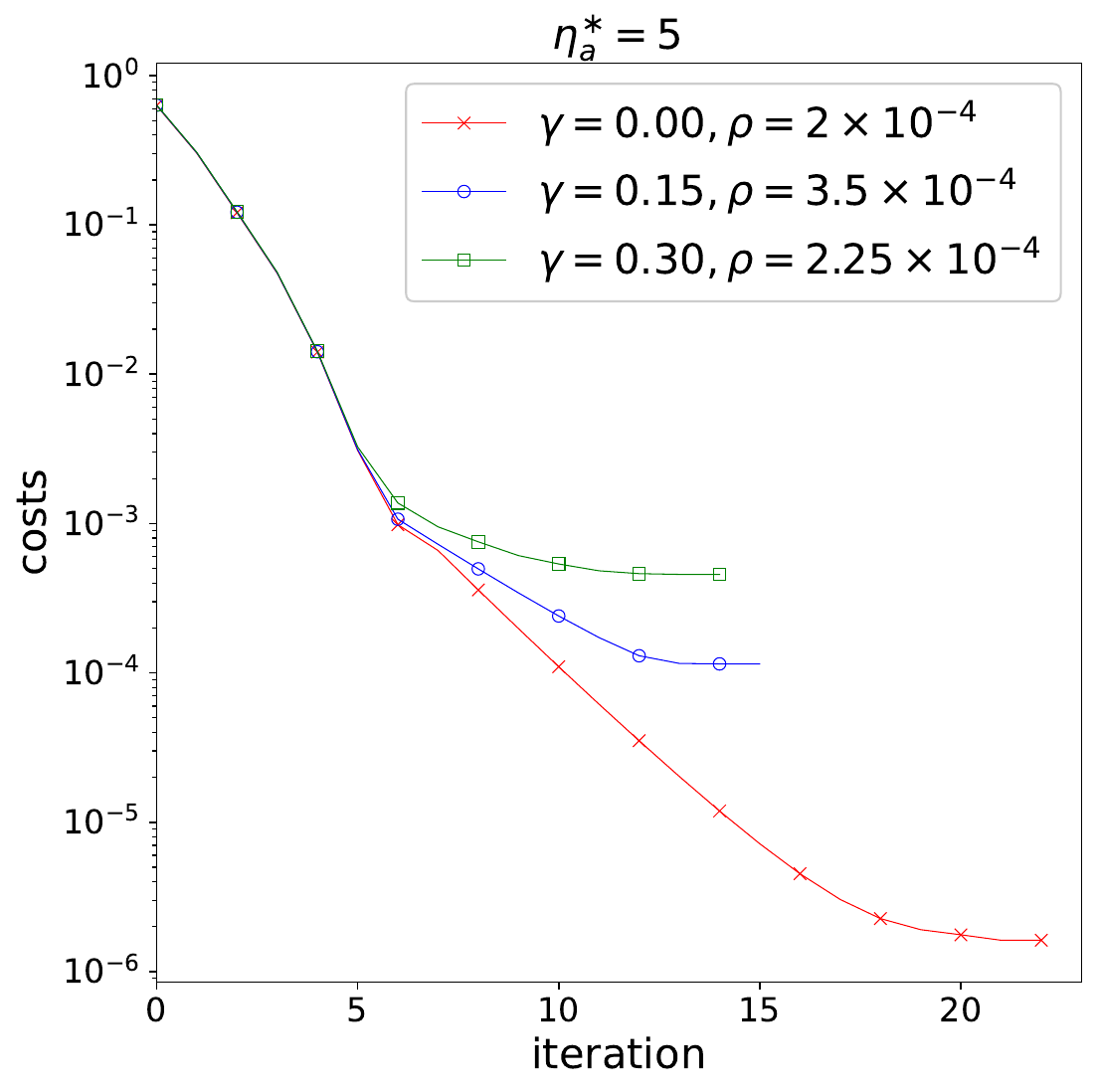}}
\caption{An example featuring a non-circular boundary interface with noisy data}
\label{fig:constant_source_flower_4}
\end{figure}
%
%

We repeated the experiment using a peanut-shaped boundary interface and fixed $\rho$ at 0.0001.
The results, shown in Figure~\ref{fig:constant_source_peanut_1}, demonstrate that despite the complex shape, the recovery of the absorption coefficient and boundary geometry remained accurate, even with high noise.
This highlights the robustness of the method for constant source.

To further assess the robustness of the proposed reconstruction method, we investigated the influence of different initial guesses under noisy data with noise level $\gamma = 0.10$.
{As shown in Figure~\ref{fig:constant_source_peanut_1_varying_initial_guess}, the initialization affects mainly the transient convergence behavior, including the number of iterations required and the intermediate evolution of the reconstructed interface. In particular, initial guesses located farther from the target inclusion generally require more iterations before stabilization. Nevertheless, despite these differences in convergence trajectories, all tested initializations converge to qualitatively similar reconstructed shapes and yield nearly identical values of the reconstructed absorption coefficient. This suggests that the proposed method possesses a certain degree of stability with respect to the initial guess, even in the presence of relatively high noise. We also observe that the method remains capable of recovering the principal geometric features of the inclusion for different initial geometries, although small discrepancies near regions of higher curvature persist due to the combined effects of noise and shape sensitivity.}
%
%
%
%
%
%
\begin{figure}[htp!]
\centering
\resizebox{0.2\textwidth}{!}{\includegraphics{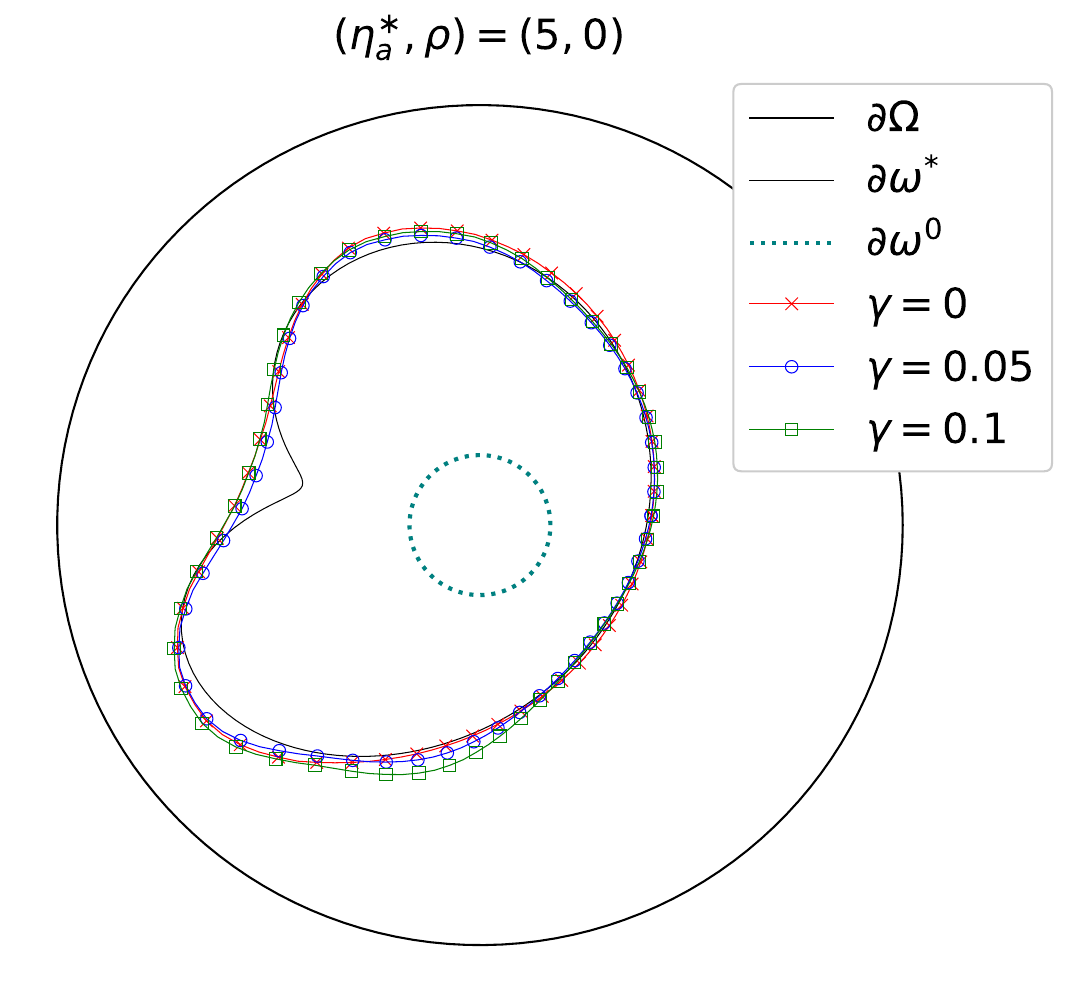}} \quad
\resizebox{0.19\textwidth}{!}{\includegraphics{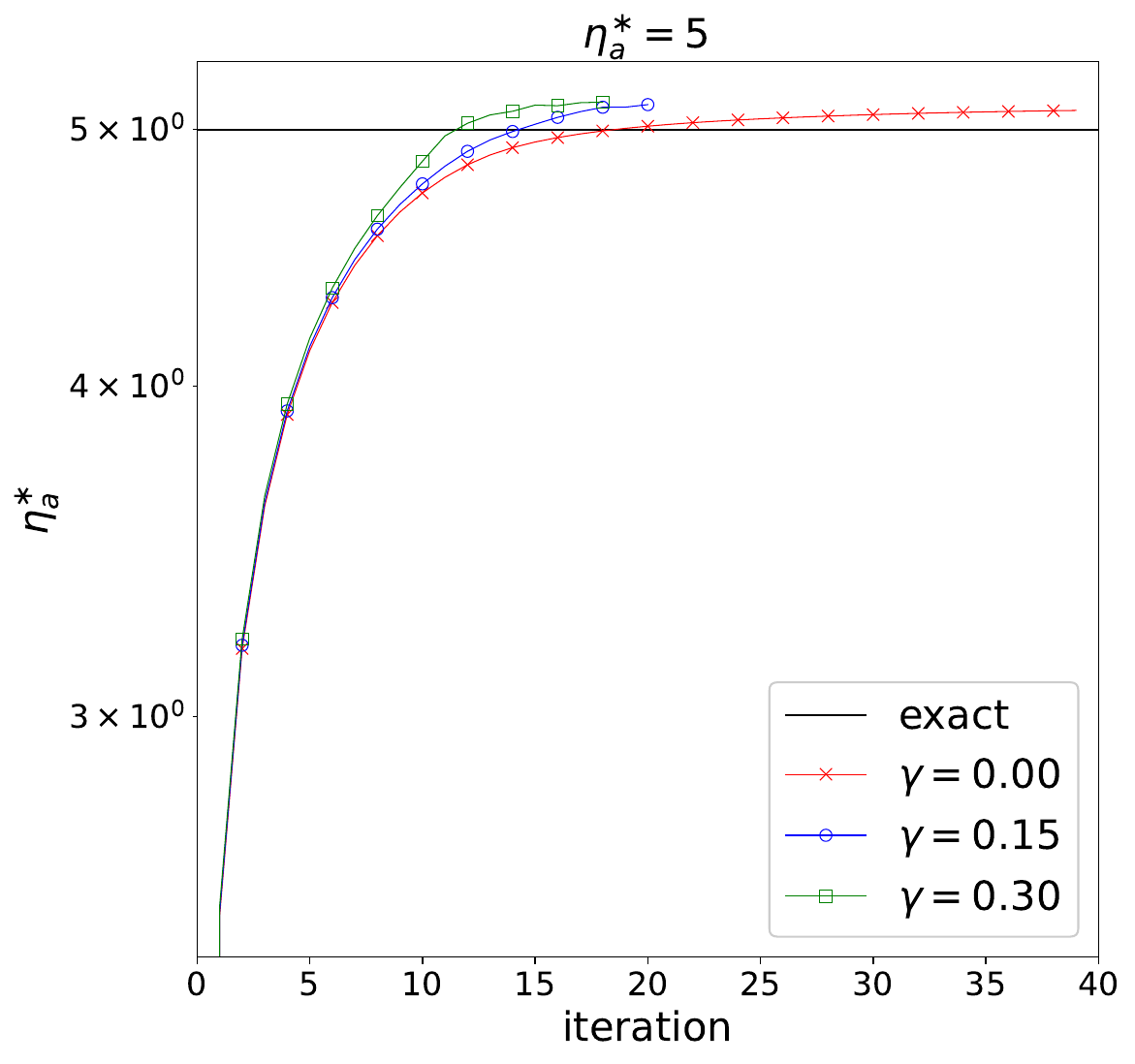}} \quad
\resizebox{0.18\textwidth}{!}{\includegraphics{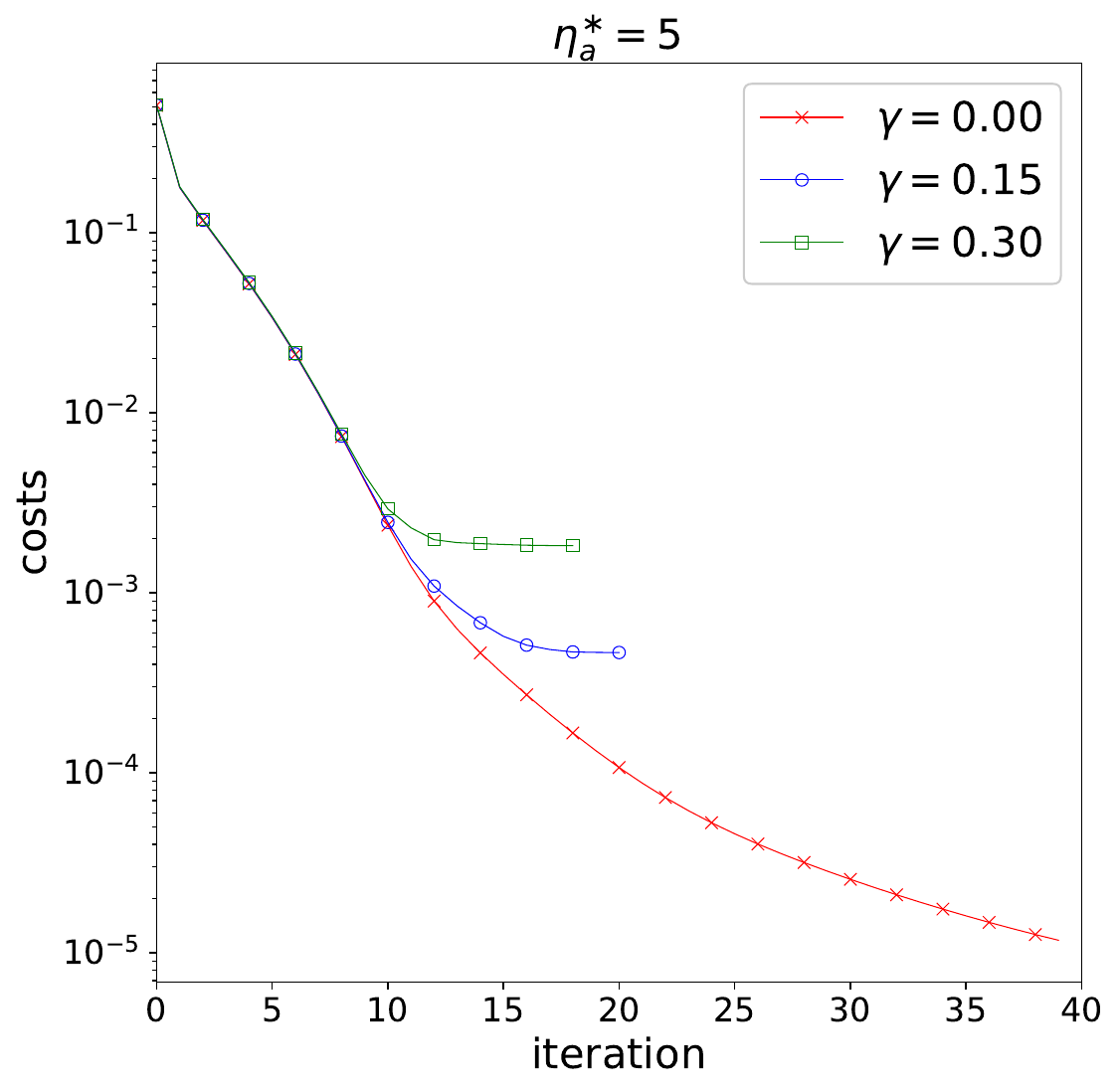}}
\caption{An example of a peanut-shaped boundary interface with noisy data}
\label{fig:constant_source_peanut_1}
\end{figure}
%
%
\begin{figure}[htp!]
\centering
\hfill
\resizebox{0.2\textwidth}{!}{\includegraphics{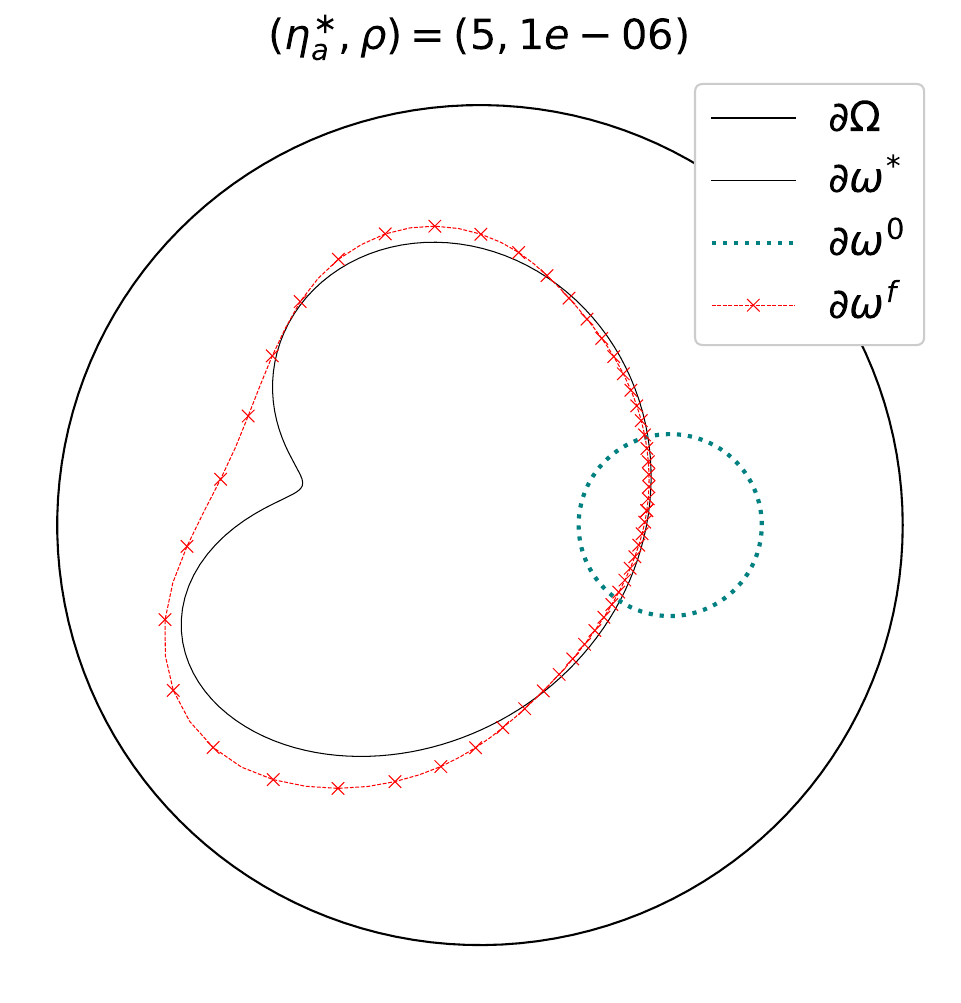}} \hfill
\resizebox{0.2\textwidth}{!}{\includegraphics{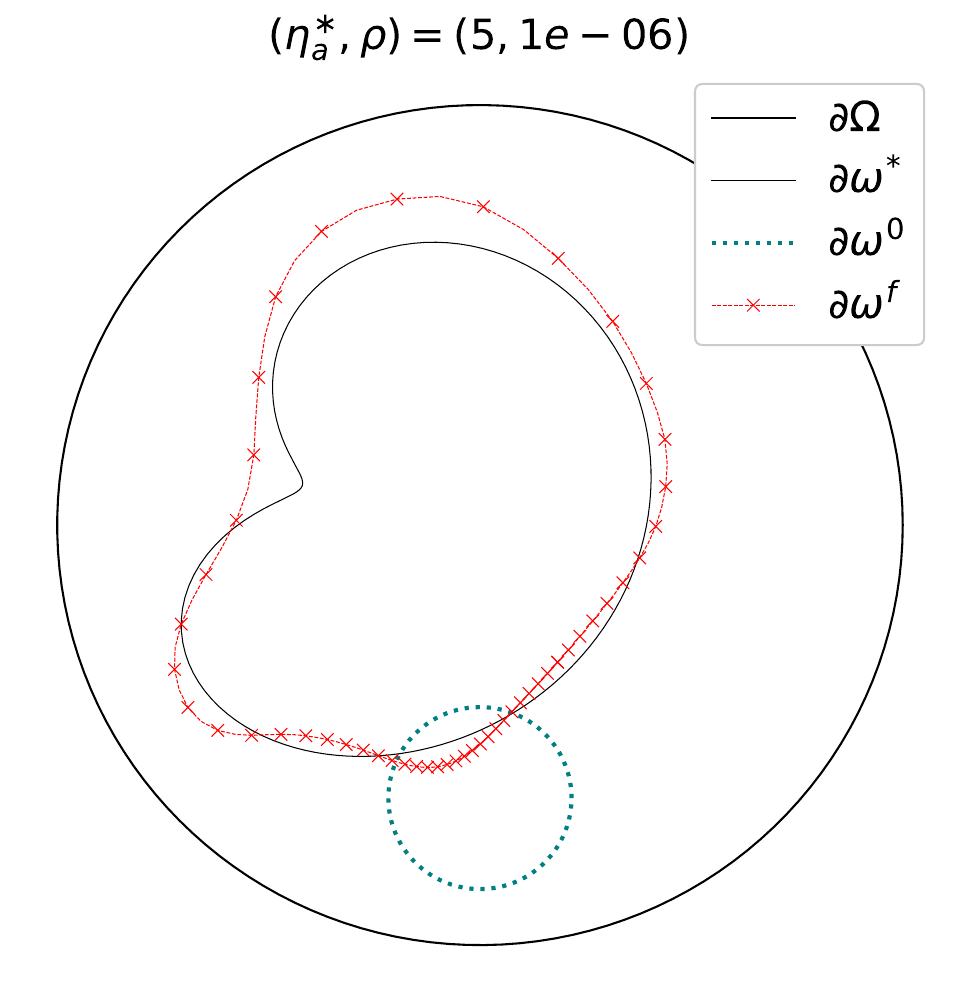}} \hfill
\resizebox{0.2\textwidth}{!}{\includegraphics{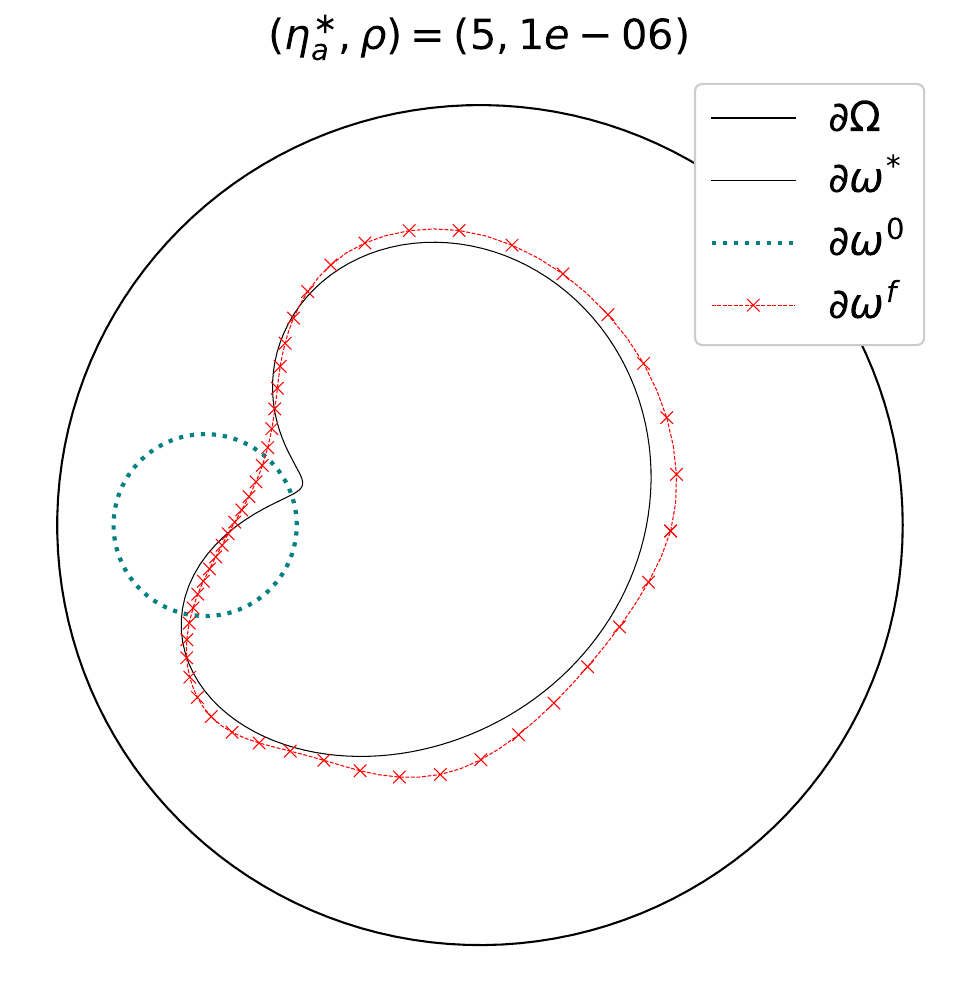}} \hfill
\resizebox{0.2\textwidth}{!}{\includegraphics{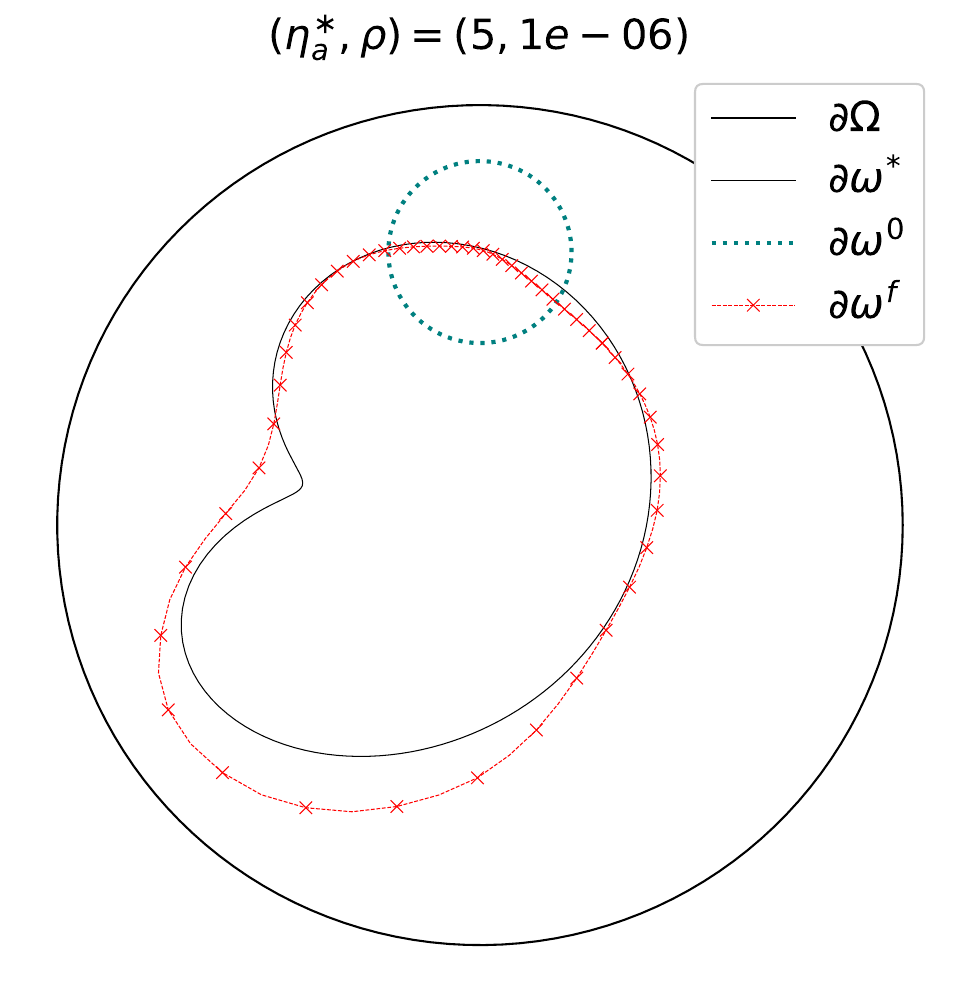}} \hfill\\
\resizebox{0.11\textwidth}{!}{\includegraphics{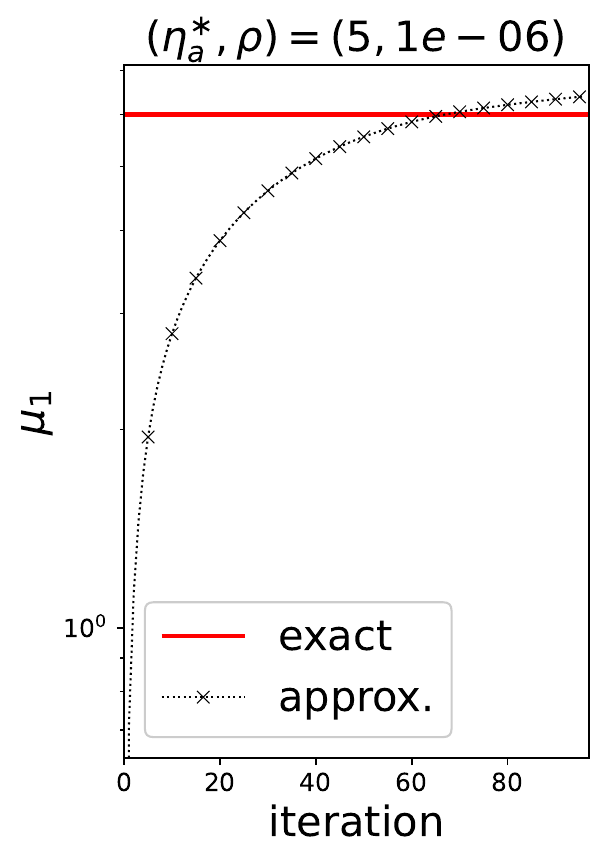}} \hfill
\resizebox{0.11\textwidth}{!}{\includegraphics{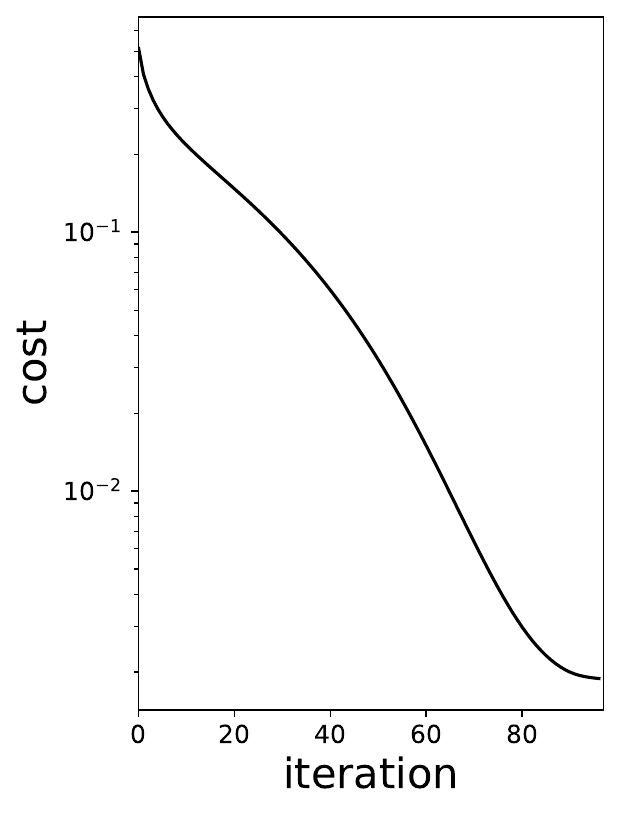}} \hfill
\resizebox{0.11\textwidth}{!}{\includegraphics{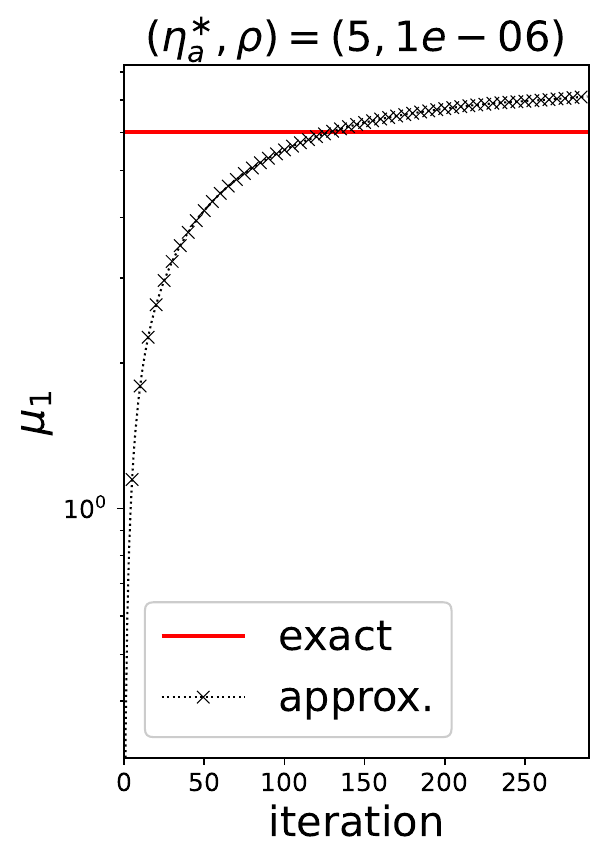}} \hfill
\resizebox{0.11\textwidth}{!}{\includegraphics{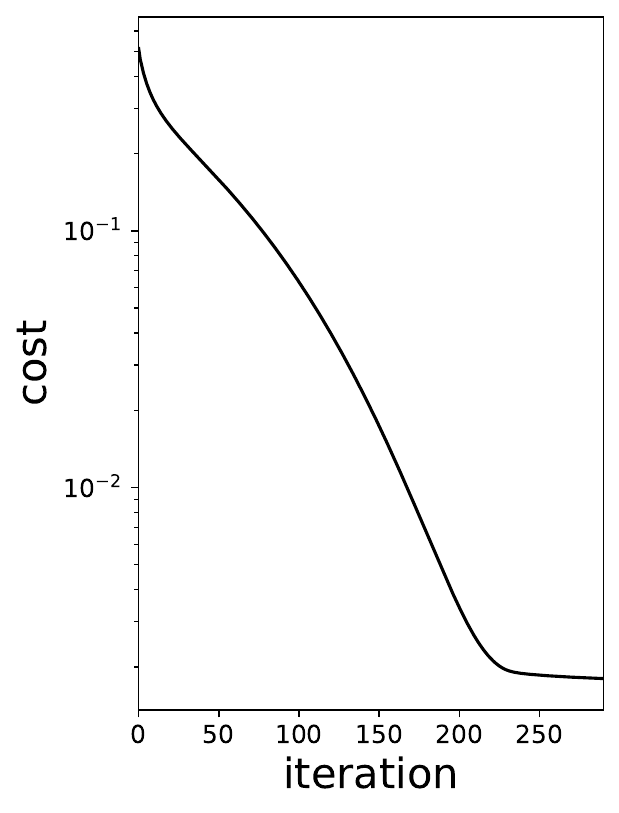}} \hfill
\resizebox{0.11\textwidth}{!}{\includegraphics{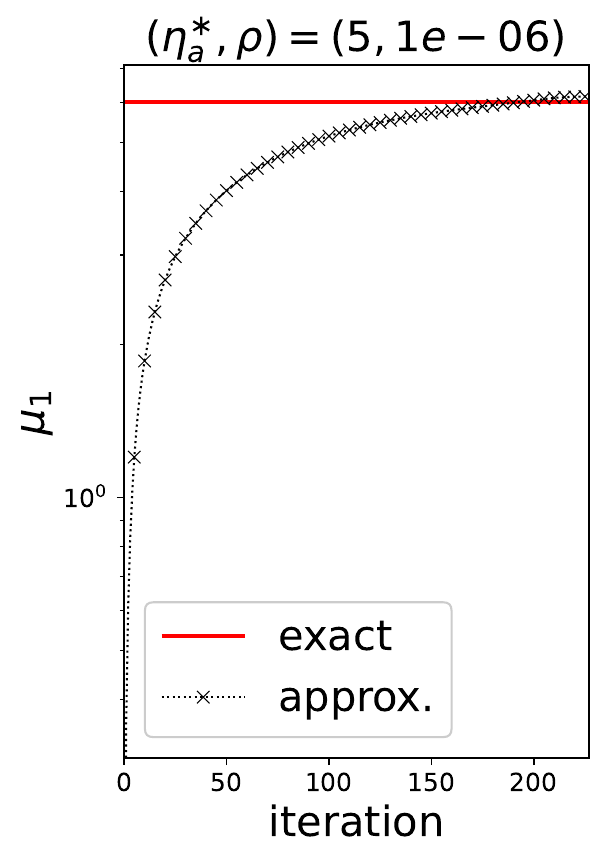}} \hfill
\resizebox{0.11\textwidth}{!}{\includegraphics{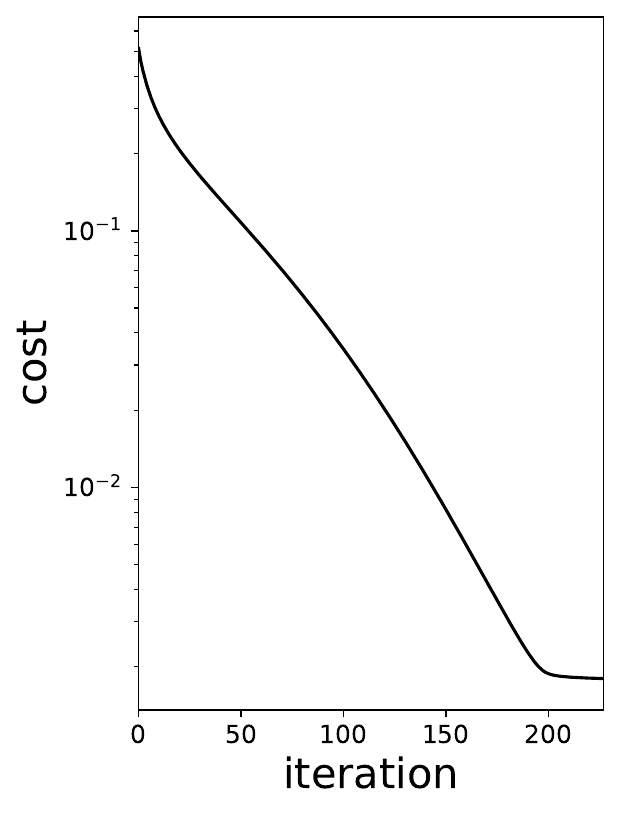}} \hfill
\resizebox{0.11\textwidth}{!}{\includegraphics{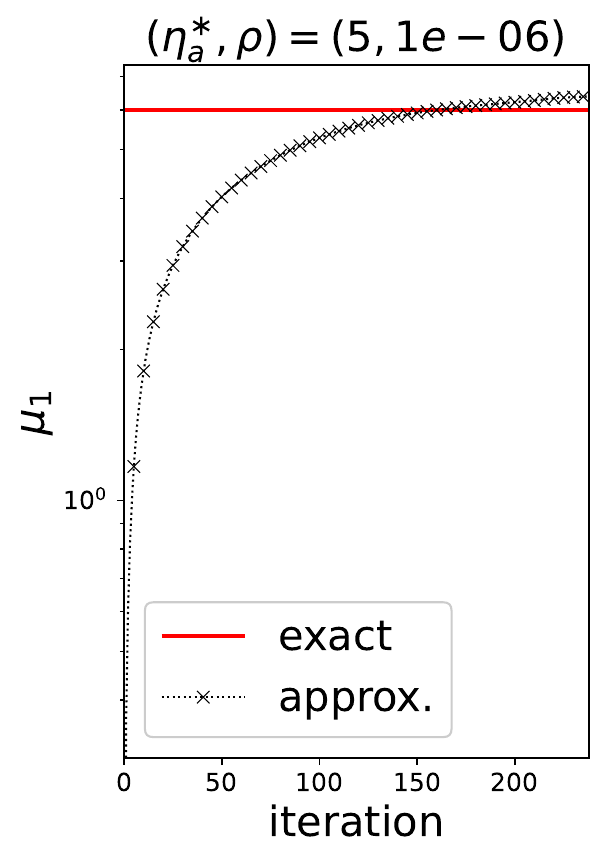}} \hfill
\resizebox{0.11\textwidth}{!}{\includegraphics{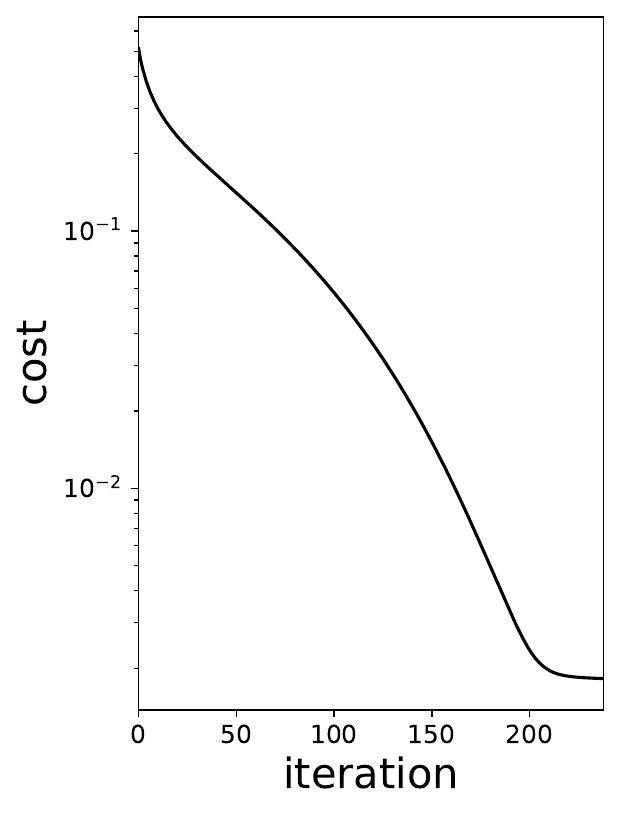}} \hfill
\caption{Effect of choice of initial guess}
\label{fig:constant_source_peanut_1_varying_initial_guess}
\end{figure}
%
%
%
\subsection{Numerical tests with point source and non-circular boundary interface}\label{subsec:point_source_non_circular}
We conduct numerical experiments using a point source to reconstruct a non-circular boundary interface, based on the setup from subsection~\ref{subsec:constant_source_non_circular}, with some modifications.
Using $f(x) = 100\delta(x)$ as the point source, we reconsider the reconstruction problem shown in Figure~\ref{fig:constant_source_flower_1}.
The results, in Figure~\ref{fig:point_source_flower}, compare scenarios with and without noise, with the regularization parameter $\rho$ set to $10^{-7}$.
The problem becomes more ill-posed, as small perturbations in the measurements lead to significant discrepancies in the reconstruction.
Apparently, reconstructing a non-circular boundary interface with a point source is particularly challenging.
%
%
\begin{figure}[htp!]
\centering
\resizebox{0.15\textwidth}{!}{\includegraphics{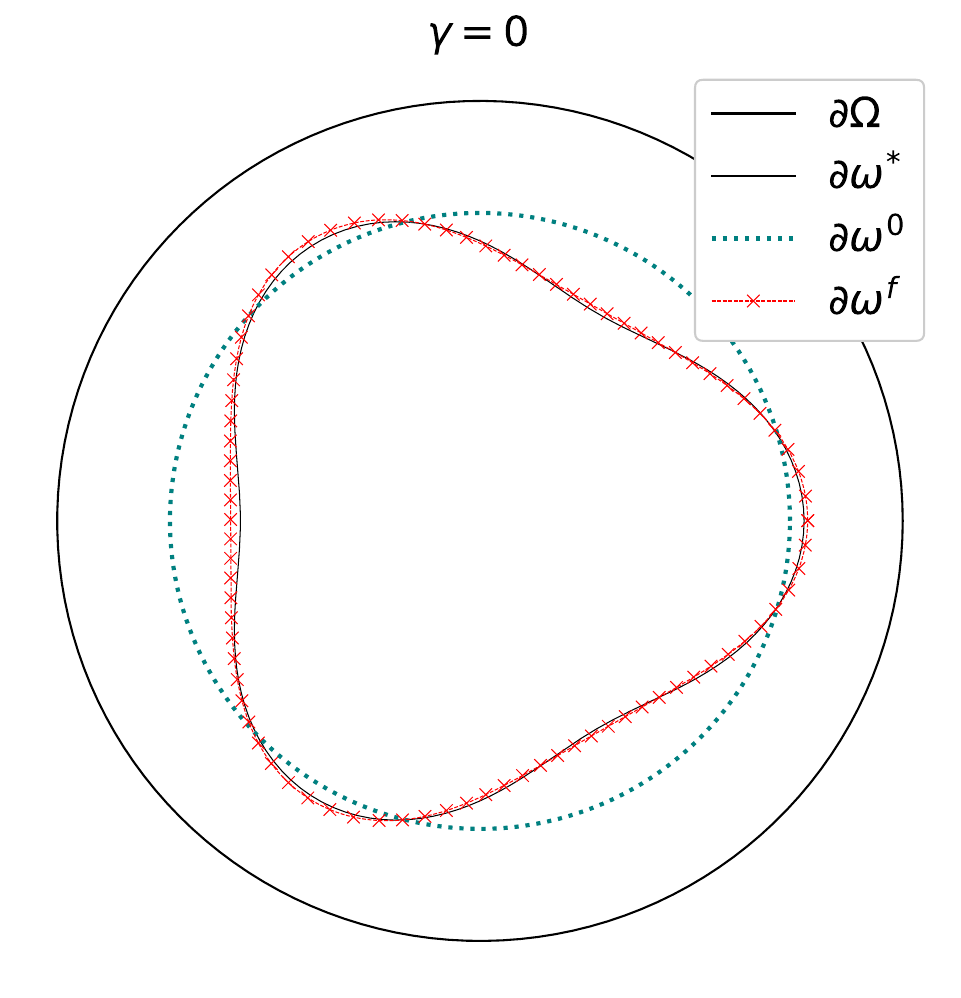}}  \quad
\resizebox{0.16\textwidth}{!}{\includegraphics{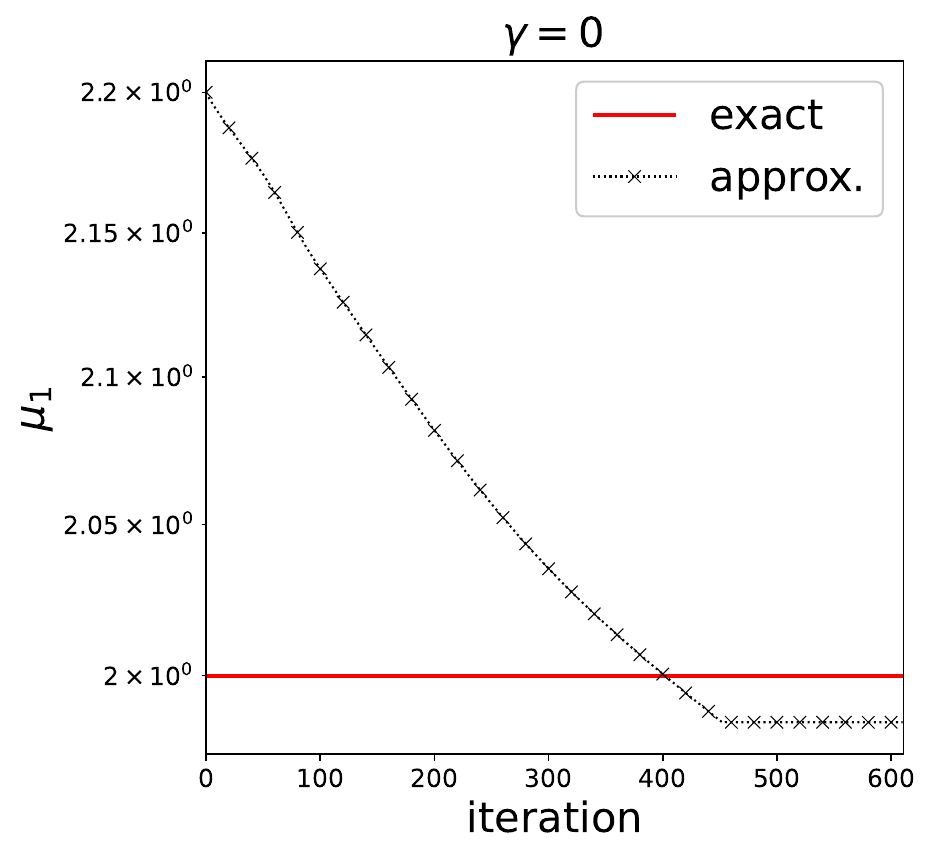}} \quad
\resizebox{0.15\textwidth}{!}{\includegraphics{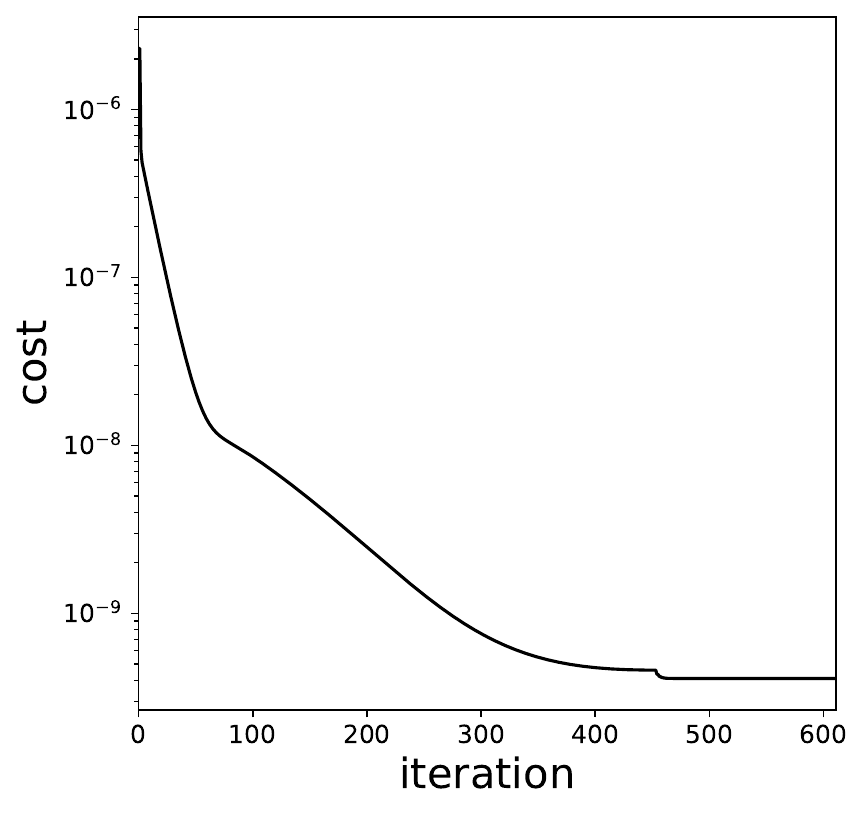}}\\
\resizebox{0.15\textwidth}{!}{\includegraphics{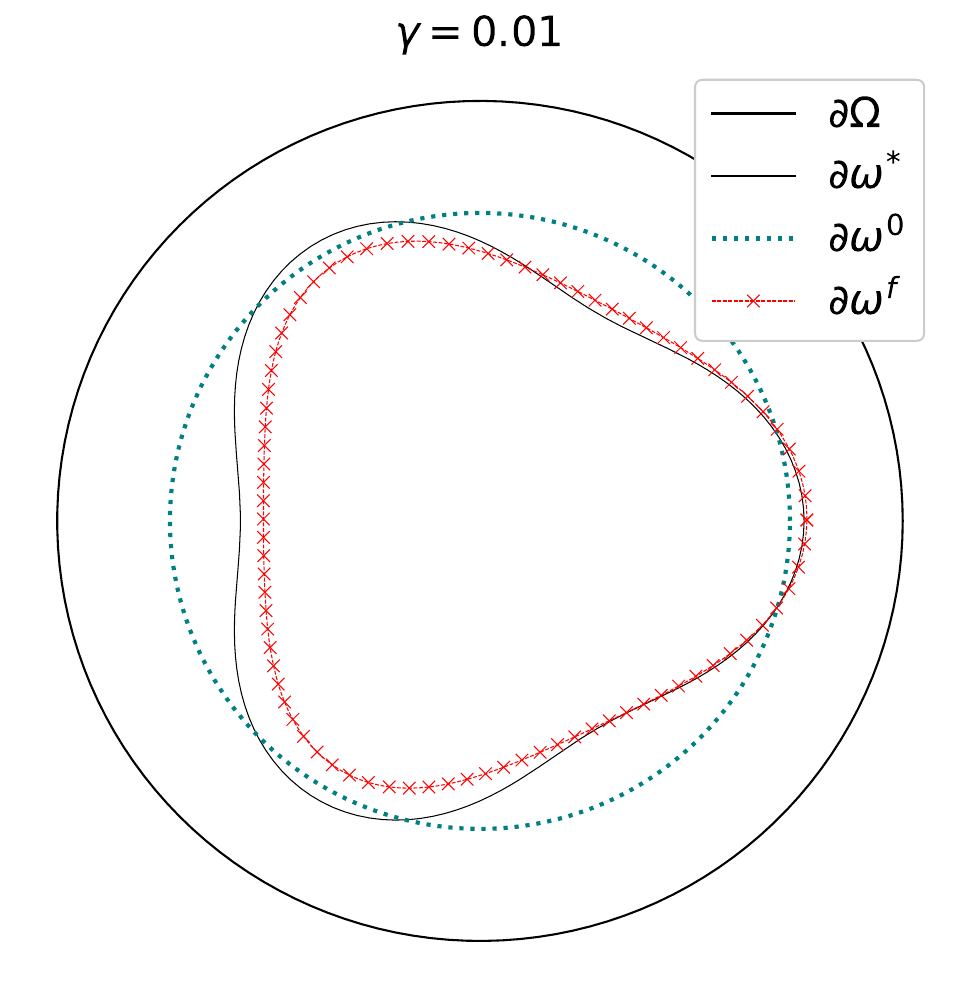}} \quad
\resizebox{0.16\textwidth}{!}{\includegraphics{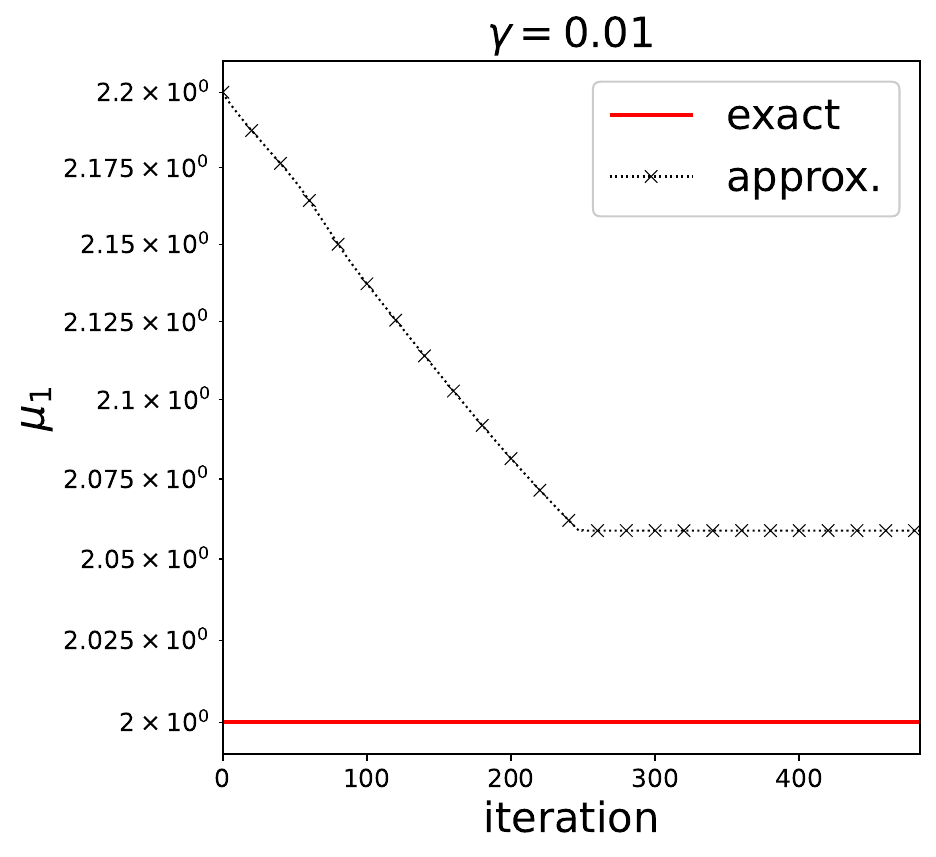}} \quad
\resizebox{0.15\textwidth}{!}{\includegraphics{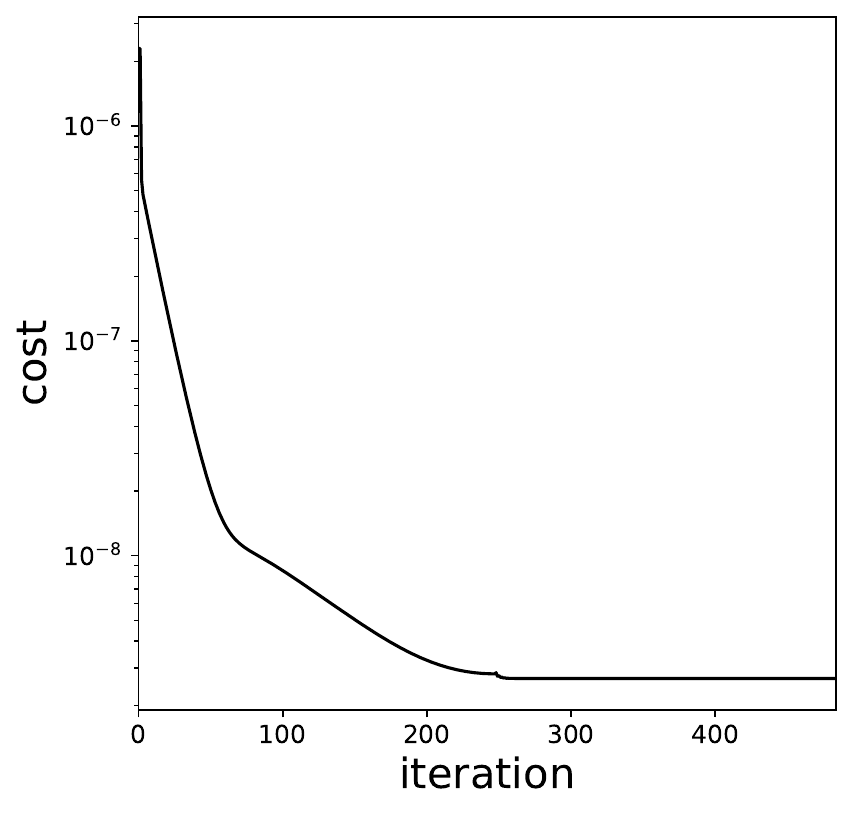}}
\caption{Reconstruction results for a non-circular interface with $f(x)=100\delta(x)$ using exact measurements (top row) and noisy measurements ($\gamma=0.01$, bottom row).}
\label{fig:point_source_flower}
\end{figure}
%
%

We also test our approach with a more complex boundary interface, parameterized as:
\[
	\domega^{\ast} = \left\{5(0.4 + 0.12\cos(3t))\begin{pmatrix}\cos{t}\\ \sin{t}\end{pmatrix}, \forall t \in [0, 2\pi) \right\}.
\]
We set $f(x) = 100\delta(x)$ and $\mu_{1}^{\ast} = 6$.
Reconstruction results using both exact measurements and noisy data with $\gamma = 0.01$ are shown in Figure~\ref{fig:point_source_fan}.
Reconstruction is highly accurate with exact measurements, but introducing noise ($\gamma = 0.01$) makes it slightly more challenging.
However, the reconstructed boundary interface and coefficient $\mu$ remain satisfactory.
For exact measurements, we set $\rho = 5 \times 10^{-6}$, and for noisy data, $\rho = 5 \times 10^{-5}$.
These reconstructions, like the previous examples, depend heavily on the initial guess.

%
%
\begin{figure}[htp!]
\centering
\resizebox{0.15\textwidth}{!}{\includegraphics{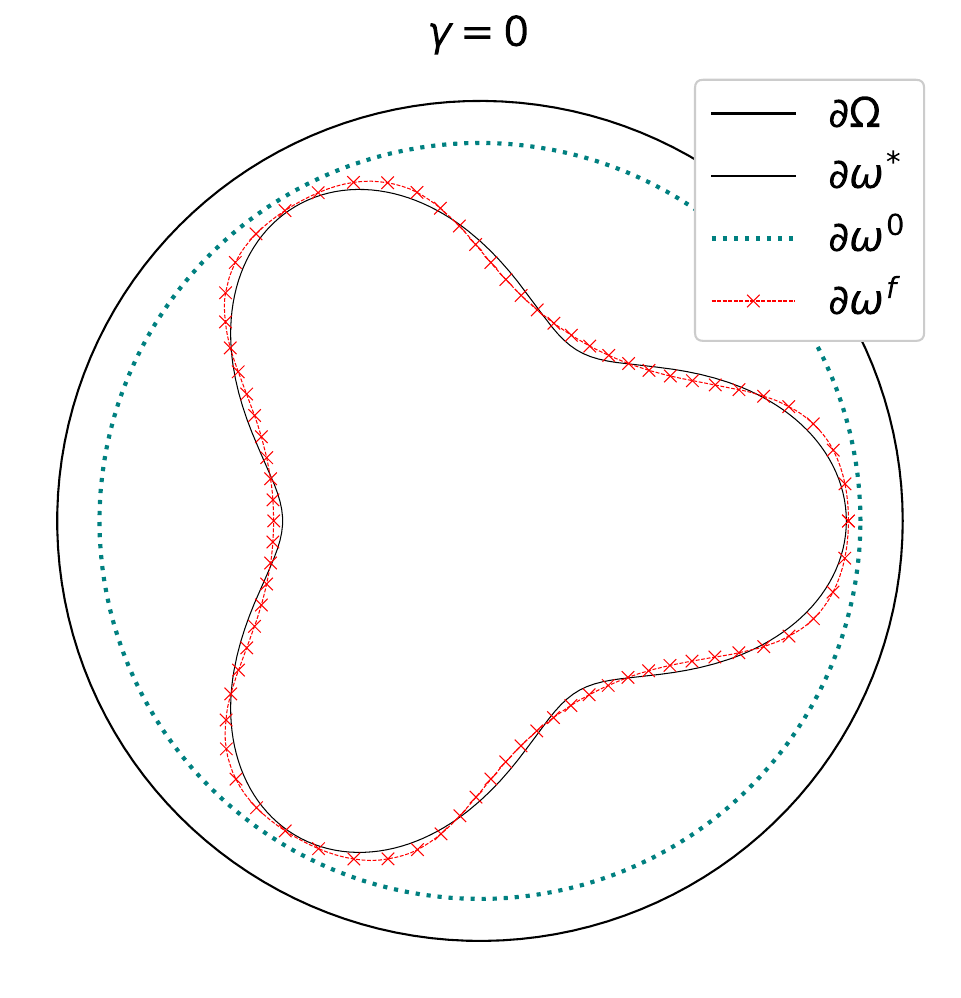}} \quad
\resizebox{0.15\textwidth}{!}{\includegraphics{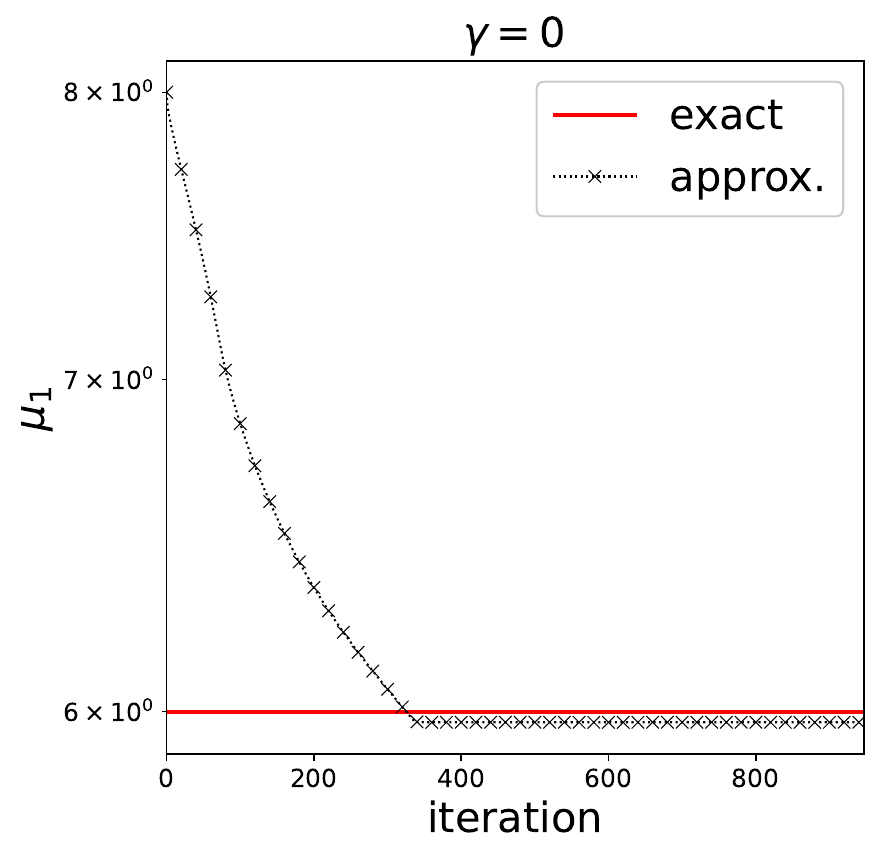}} \quad
\resizebox{0.15\textwidth}{!}{\includegraphics{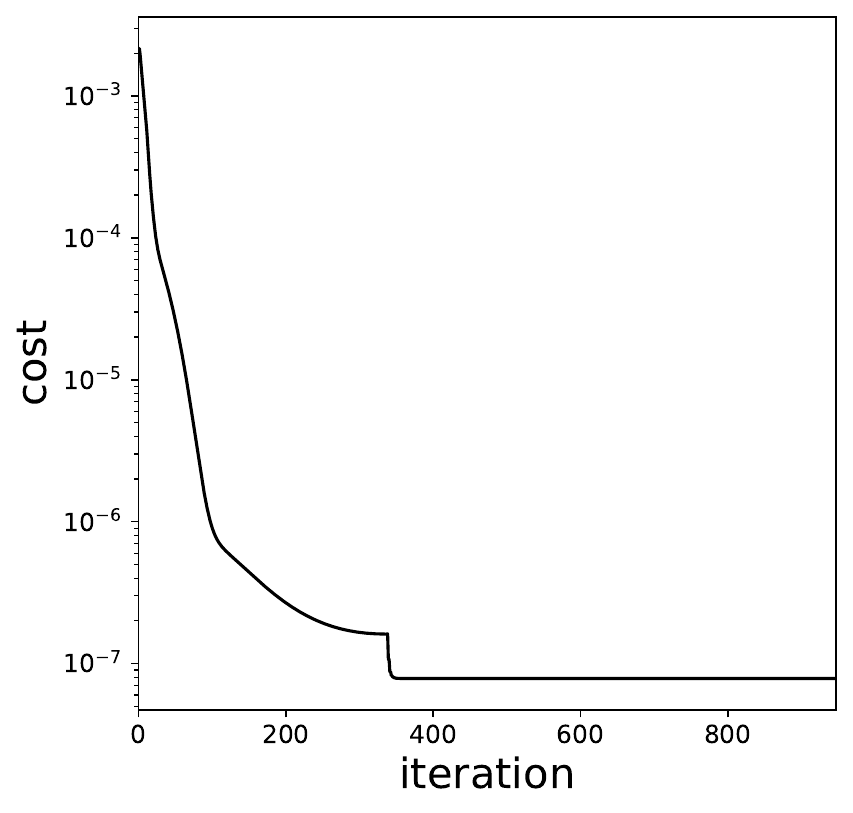}} \\
\resizebox{0.15\textwidth}{!}{\includegraphics{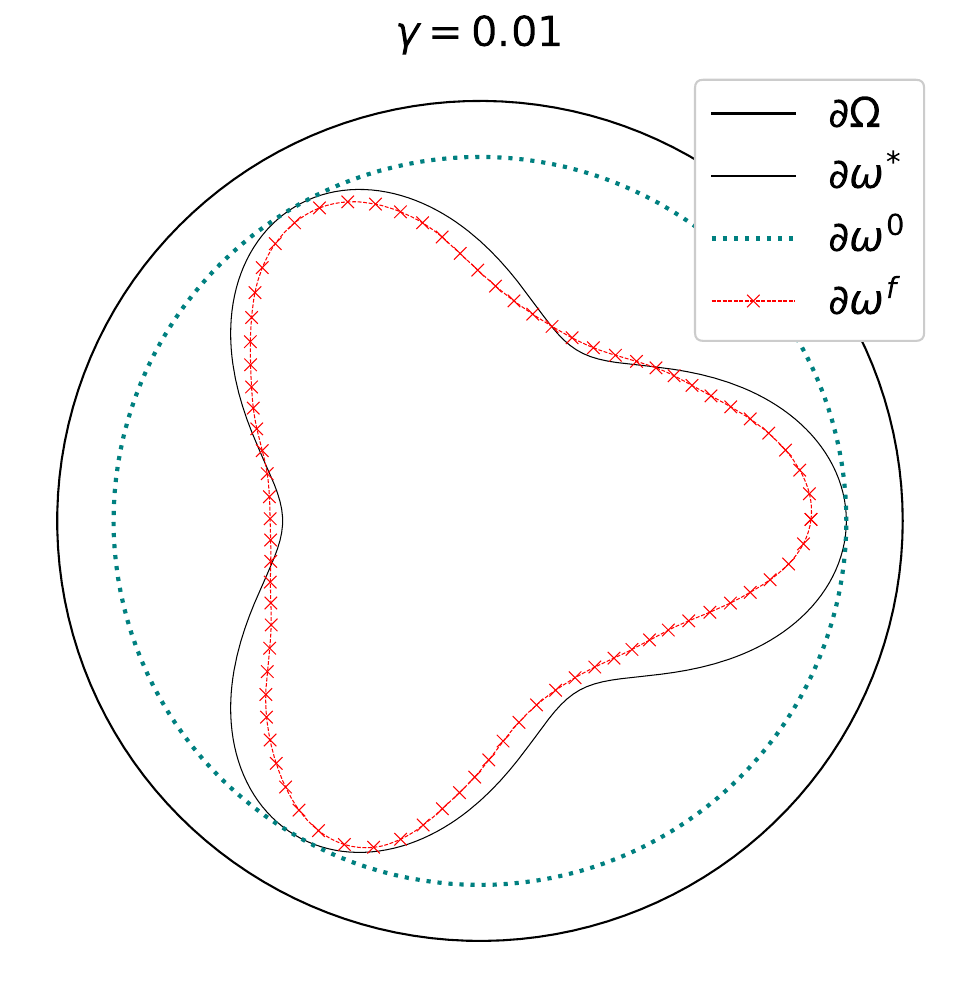}} \quad
\resizebox{0.15\textwidth}{!}{\includegraphics{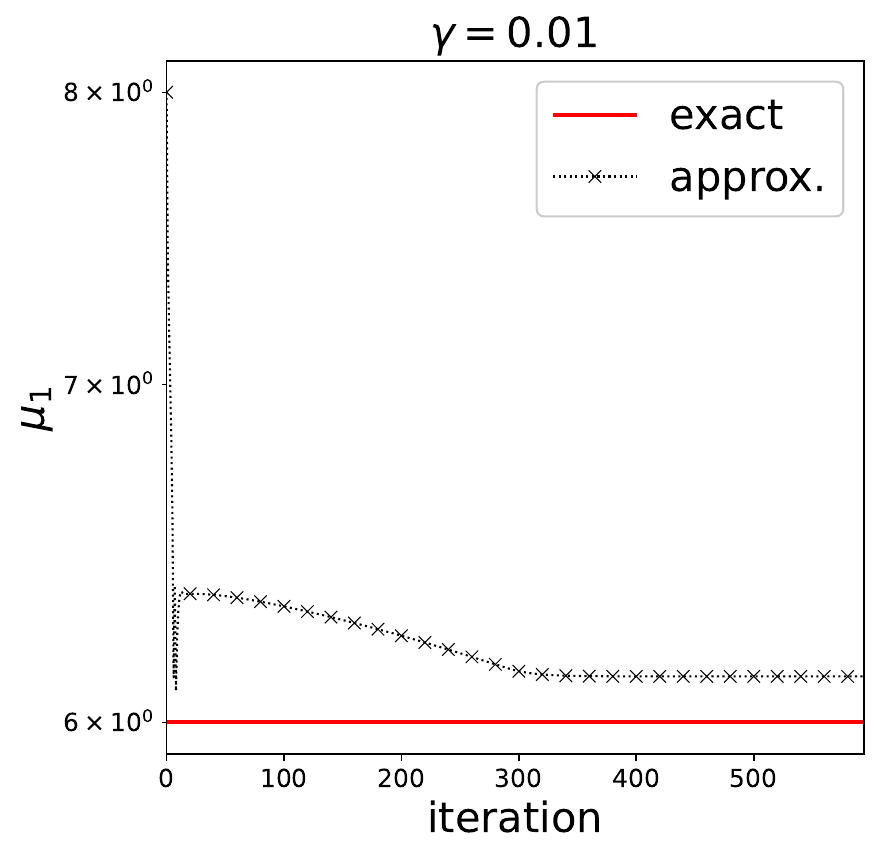}} \quad
\resizebox{0.15\textwidth}{!}{\includegraphics{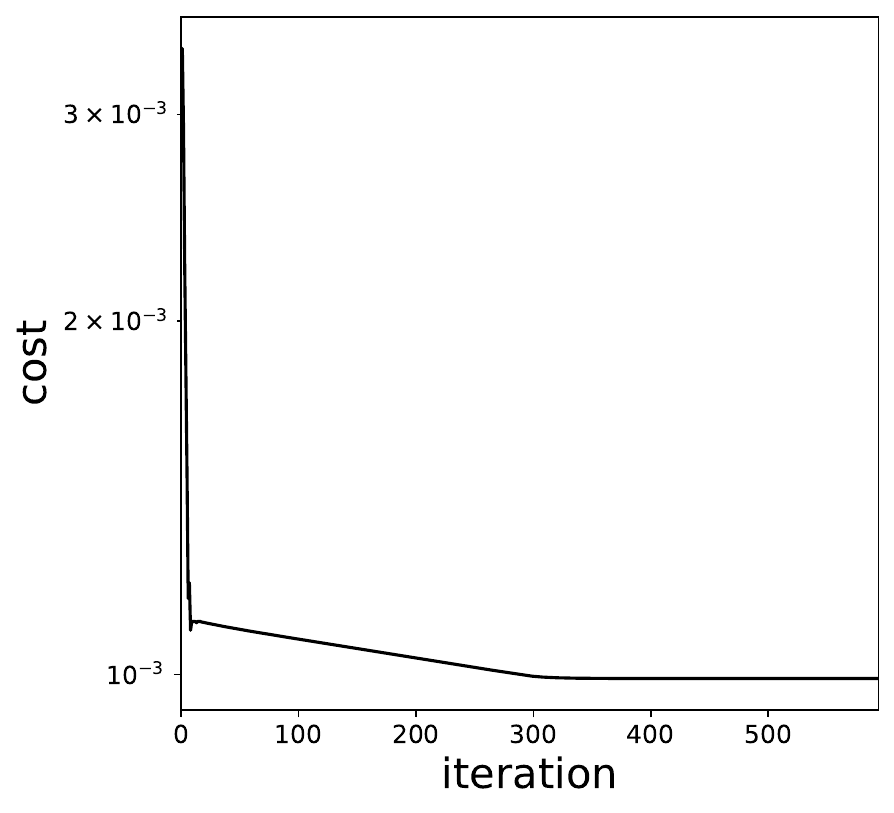}}
\caption{Results for a smooth non-circular boundary interface with $f(x) = 100\delta(x)$ and $ \mu_{1}^{\ast} = 6$ under exact (top row) and noisy measurements ($\gamma = 0.01$, bottom row)
}
\label{fig:point_source_fan}
\end{figure}
%
%

Finally, we consider a smaller boundary interface in the last test of this subsection.
This time, ${\omega}^{\ast}$ is parametrized as:
\[
    \domega^{\ast} = \left\{8(2 + 0.6\cos(4t))\begin{pmatrix}\cos{t}\\ \sin{t}\end{pmatrix}, \forall t \in [0, 2\pi) \right\}.
\]
The computational setup remains the same, but instead of fixing $\rho$, we apply \eqref{eq:balancing_principle} to evaluate its impact and accuracy.
The reconstruction results, shown in Figure~\ref{fig:point_source_boomer}, were obtained with $s = 10$ for exact measurements and $s = 1$ for noisy ones.
As expected, reconstructing smaller boundary interfaces, especially those farther from the exterior boundary, is challenging.
Reconstruction accuracy decreases with even small amounts of noise.
However, the method successfully identified concavities in the boundary interface, with the reconstructed geometry closely approximating the exact shape and providing a reasonably accurate absorption coefficient.
Figure~\ref{fig:point_source_boomer} also includes plots of the histories of values for $\mu_{1}$, cost values, and $\rho$.

From this point forward, all reconstructions use the balancing principle \eqref{eq:balancing_principle}.
The source is fixed as $f(x) = 100\delta(x)$, $\mu_{1}^{\ast} = 6$ and noisy measurements mean the noise level is set to $\gamma = 0.005$.
%
%
\begin{figure}[htp!]
\centering
\resizebox{0.15\textwidth}{!}{\includegraphics{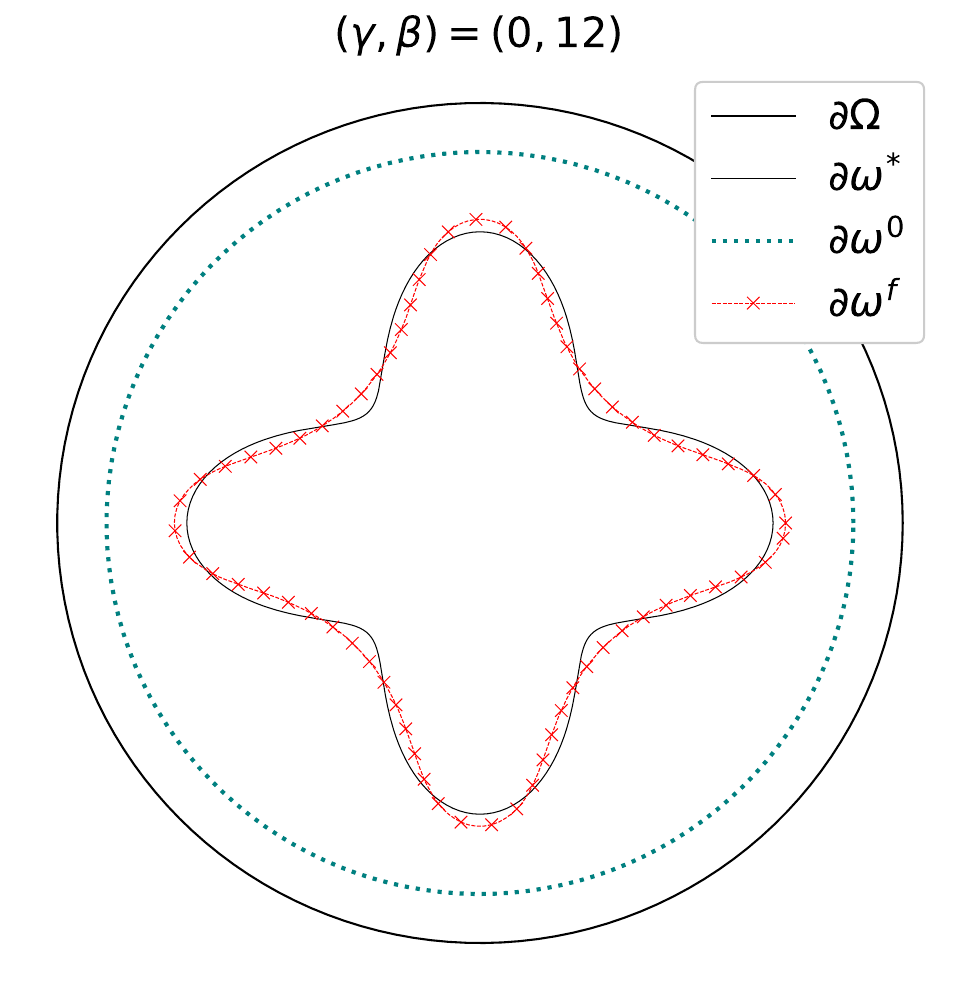}} \
\resizebox{0.15\textwidth}{!}{\includegraphics{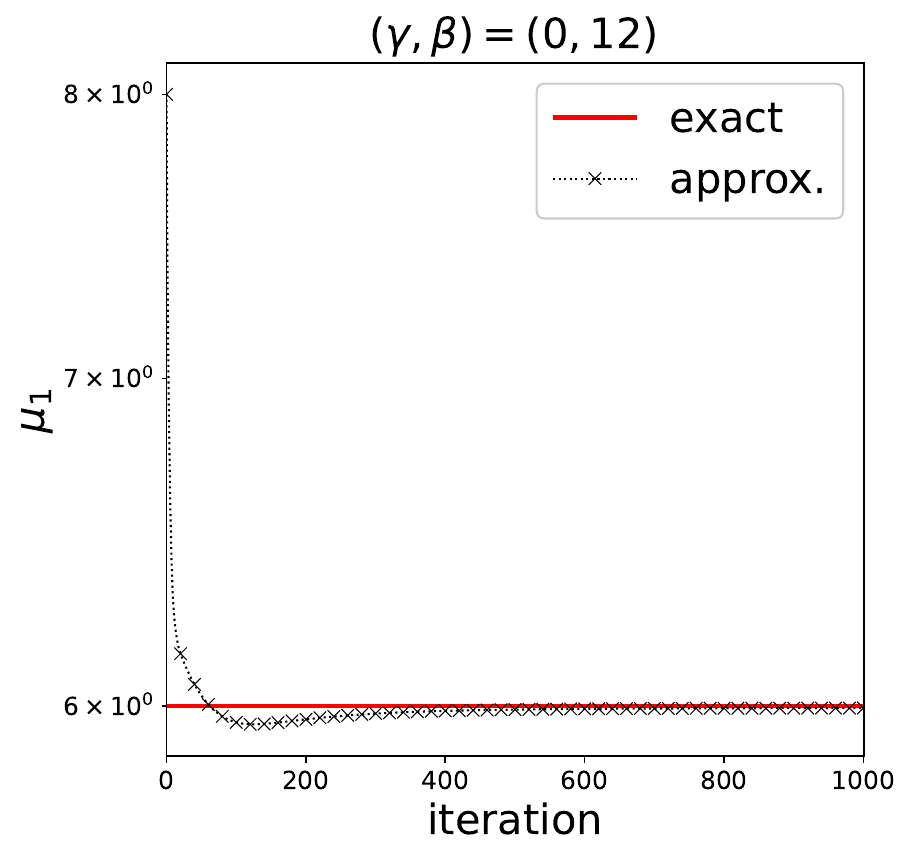}} \
\resizebox{0.15\textwidth}{!}{\includegraphics{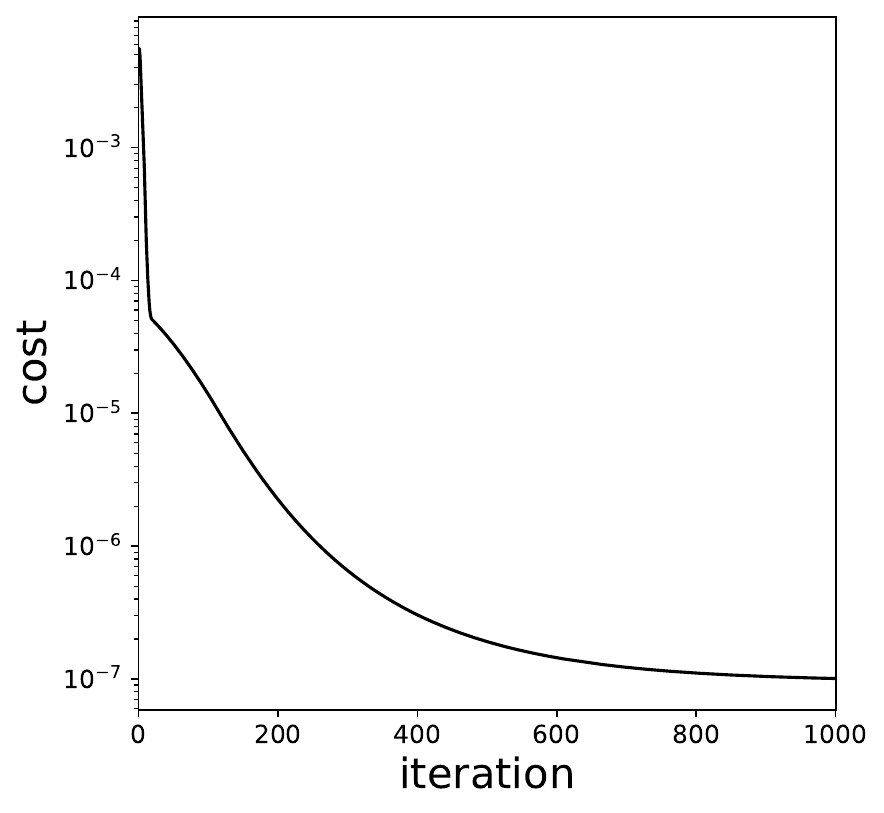}} \
\resizebox{0.15\textwidth}{!}{\includegraphics{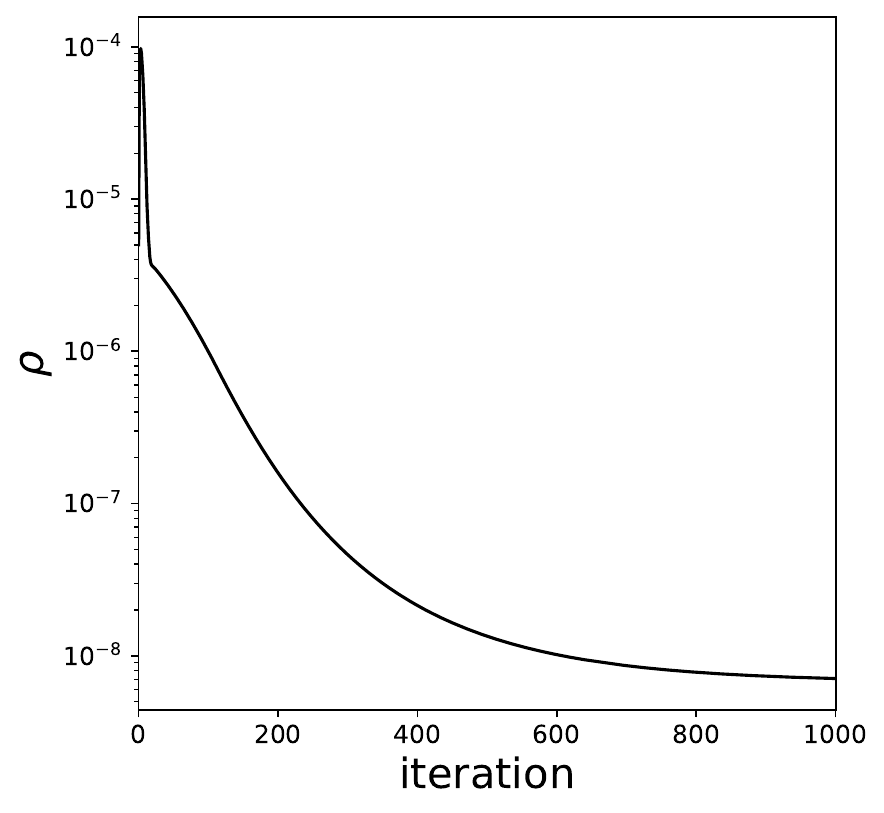}}
\\
\resizebox{0.15\textwidth}{!}{\includegraphics{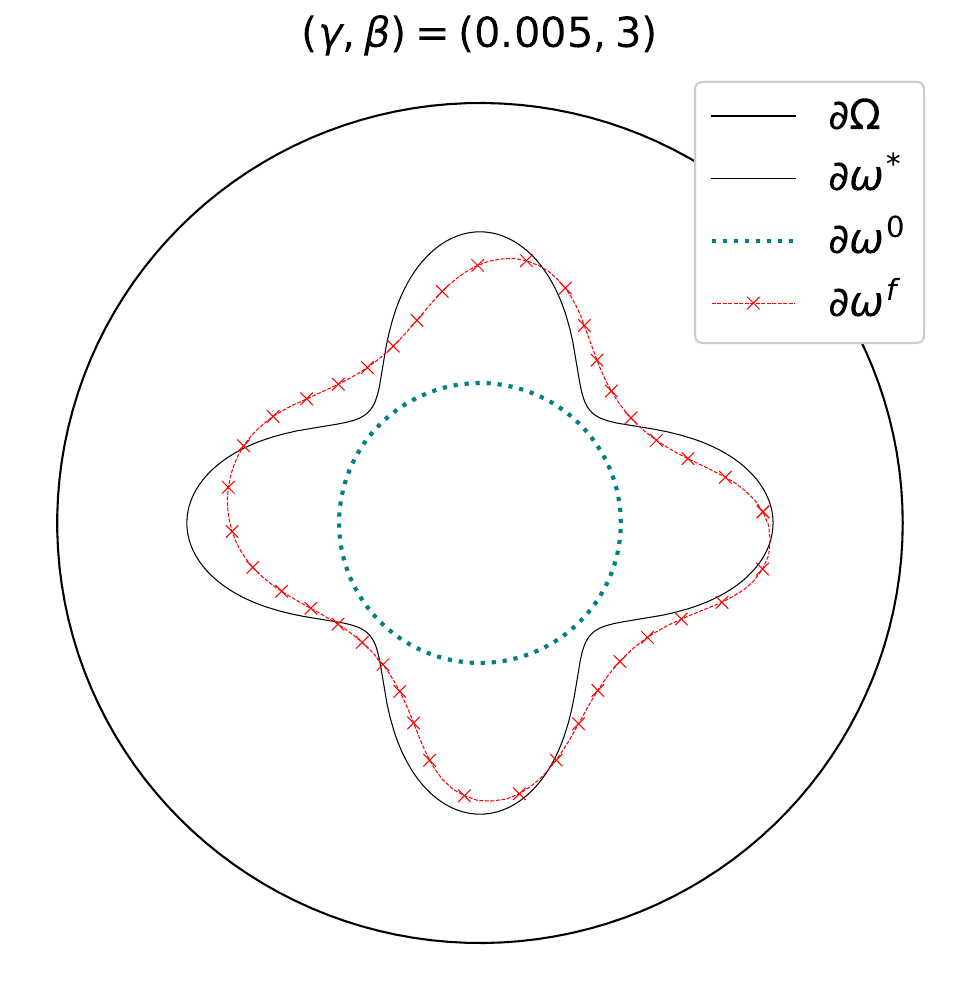}}  \
\resizebox{0.15\textwidth}{!}{\includegraphics{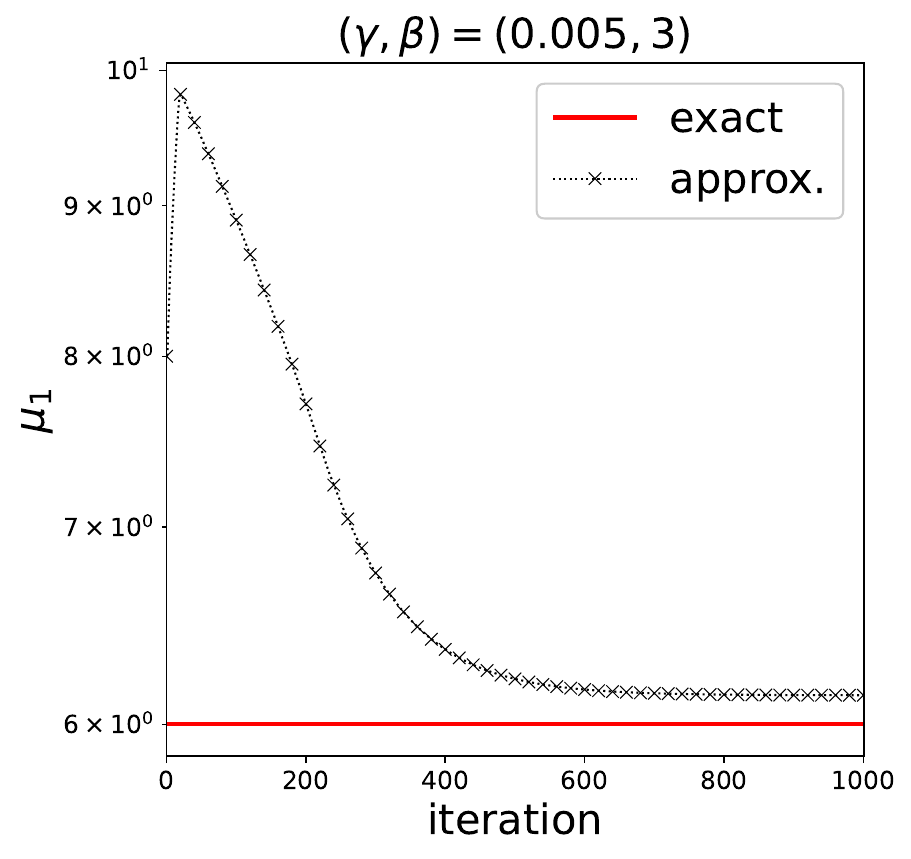}} \
\resizebox{0.155\textwidth}{!}{\includegraphics{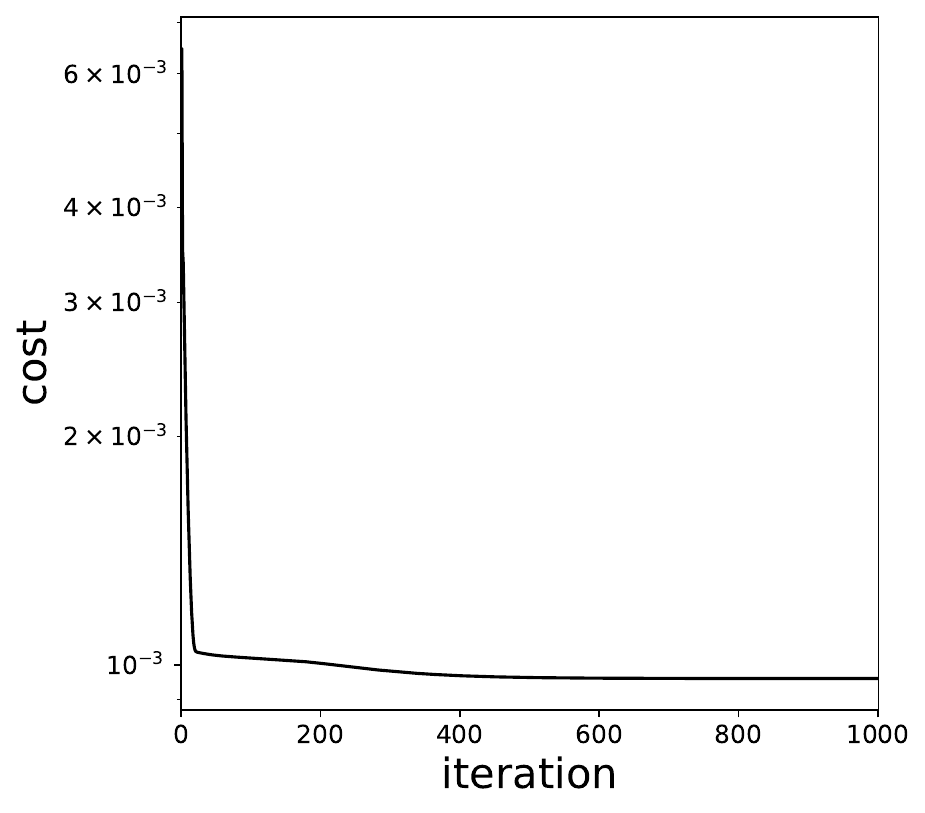}} \
\resizebox{0.15\textwidth}{!}{\includegraphics{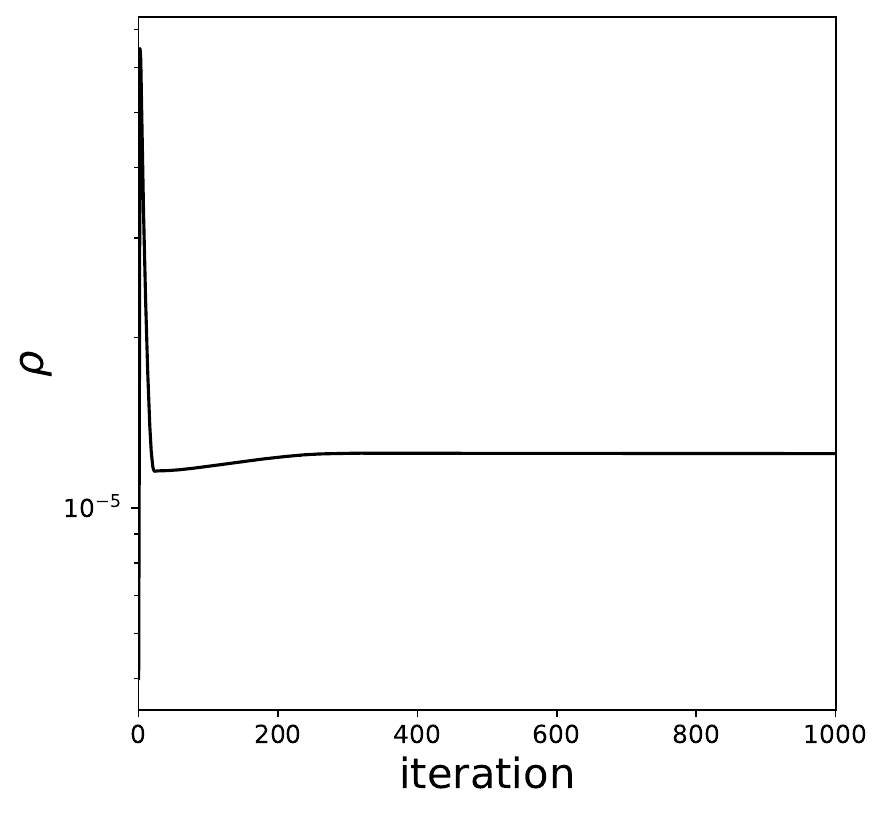}}
\caption{Results for a smooth non-circular boundary interface with $f(x) = 100\delta(x)$ and $\mu_{1}^{\ast} = 6$ under exact (top row) and noisy measurements ($\gamma = 0.005$, bottom row)
}
\label{fig:point_source_boomer}
\end{figure}
%
%
%
%
%
\subsection{Numerical tests with point source and boundary interface with sharp edges}\label{subsec:point_source_sharp_edges}
For the next experiment, we consider boundary interfaces with sharp edges, which are outside the regularity assumptions used in the analysis. These cases are included to illustrate the behavior of the numerical method under non-smooth geometries. The setup remains the same as in the previous subsection, with the only change being the modified boundary interface geometry that needs reconstruction. Specifically, we test a square boundary interface and an inverted {\textsf{T}-shaped} polygon.

The reconstruction results are shown in Figures~\ref{fig:point_source_non-smooth_interface_square} and~\ref{fig:point_source_non-smooth_interface_T} for exact and noisy measurements.
In Figure~\ref{fig:point_source_non-smooth_interface_square}, reconstructing the square's vertices is challenging.
Even so, the method successfully detects the edges with good accuracy.
Noise significantly affects the reconstruction, making it hard to accurately deduce the boundary geometry, but the method still identifies the interface and nearly reconstructs $\mu_{1}^{\ast}$ accurately.

{A similar observation can be made for the inverted {\textsf{T}-shaped} polygon shown in Figure~\ref{fig:point_source_non-smooth_interface_T}, which represents a more challenging, non-smooth geometry. While sharp vertices and edges are not fully resolved, the method is still able to capture the main geometric features of the interface, including its concavities. Moreover, the reconstructed value of $\mu_{1}^{\ast}$ remains accurate, even in the presence of noise.
In addition, the results suggest that the choice of the initial guess plays a more pronounced role in this setting. In particular, initial configurations that enclose the true interface tend to yield improved reconstructions, especially under noisy conditions, and allow for a more accurate identification of non-convex regions of the boundary.
These observations indicate that the proposed approach is capable of providing reliable reconstructions for non-smooth boundary interfaces, although certain fine geometric details may not be fully recovered.}
%
%
\begin{figure}[htp!]
\centering
\resizebox{0.15\textwidth}{!}{\includegraphics{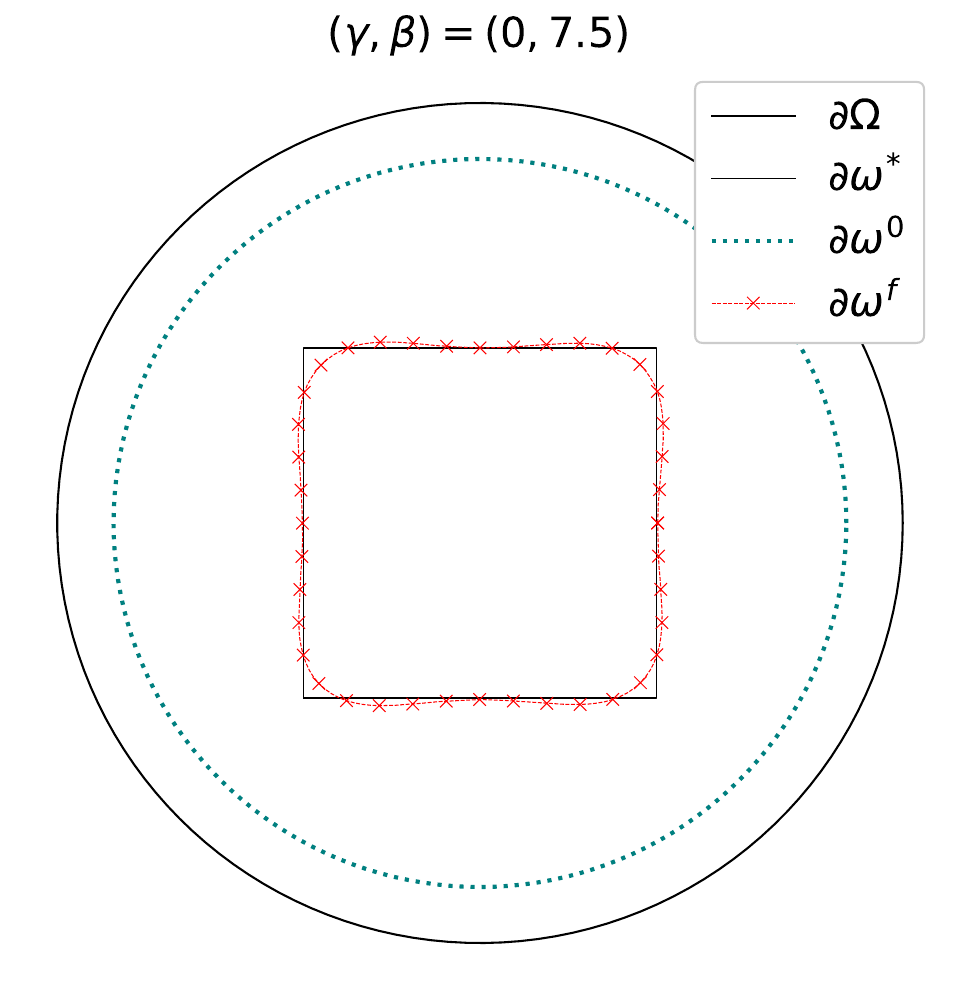}} \
\resizebox{0.15\textwidth}{!}{\includegraphics{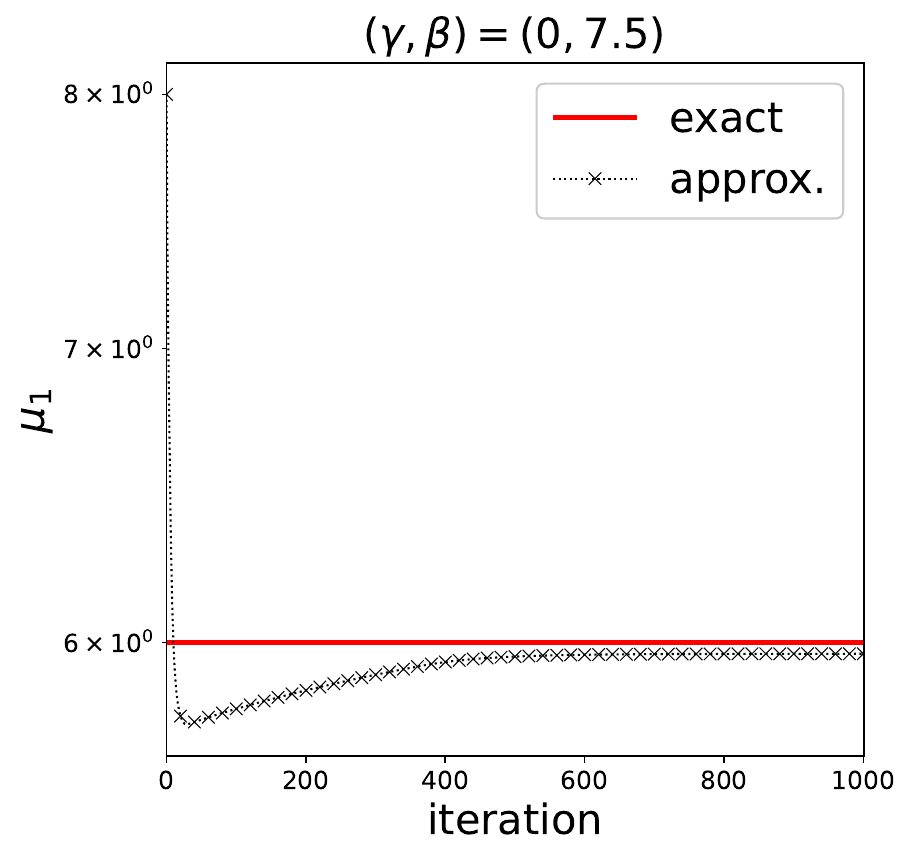}} \
\resizebox{0.15\textwidth}{!}{\includegraphics{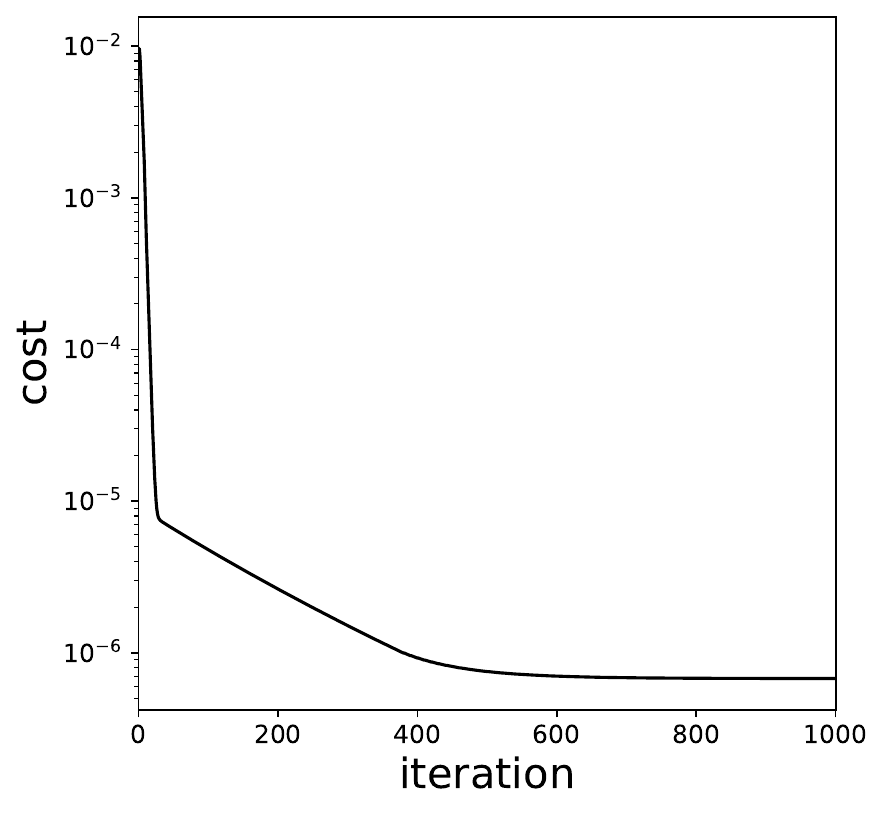}} \
\resizebox{0.15\textwidth}{!}{\includegraphics{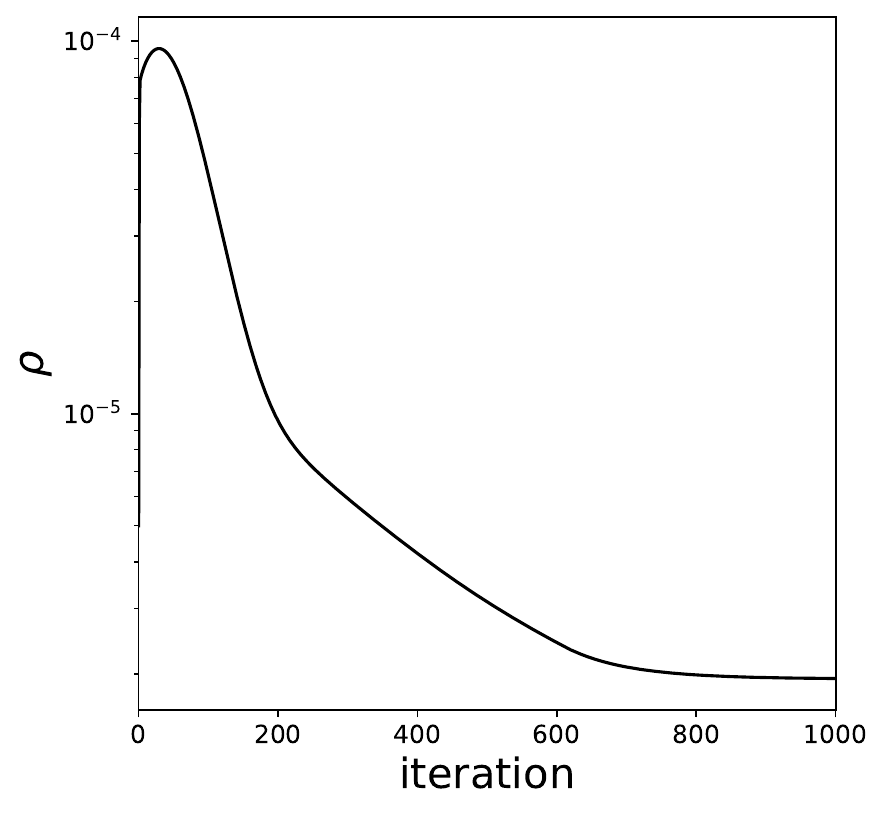}}
\\
\resizebox{0.15\textwidth}{!}{\includegraphics{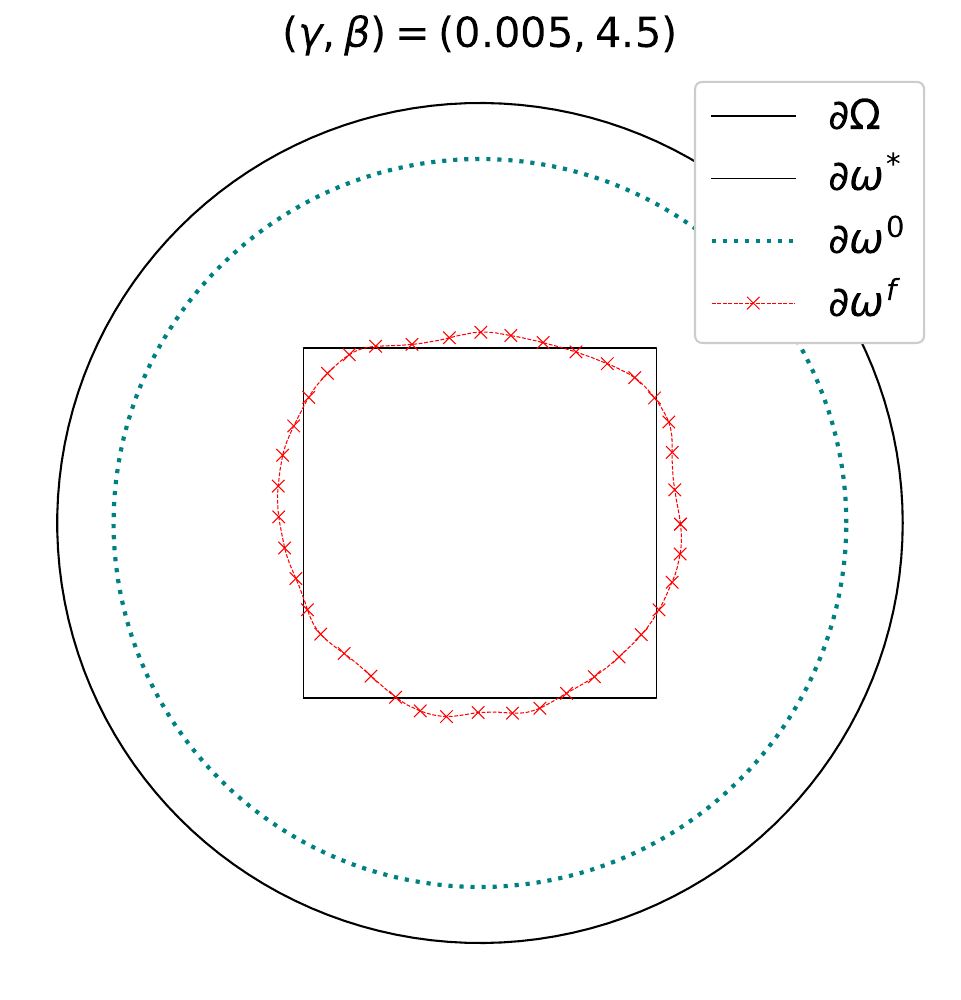}}  \
\resizebox{0.15\textwidth}{!}{\includegraphics{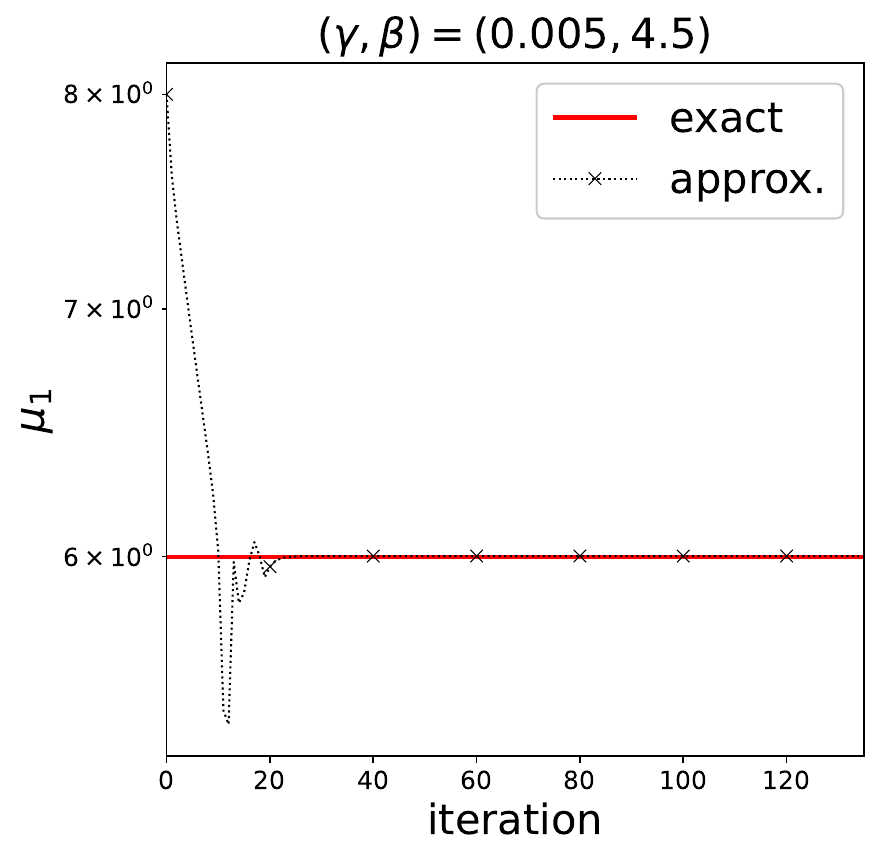}} \
\resizebox{0.15\textwidth}{!}{\includegraphics{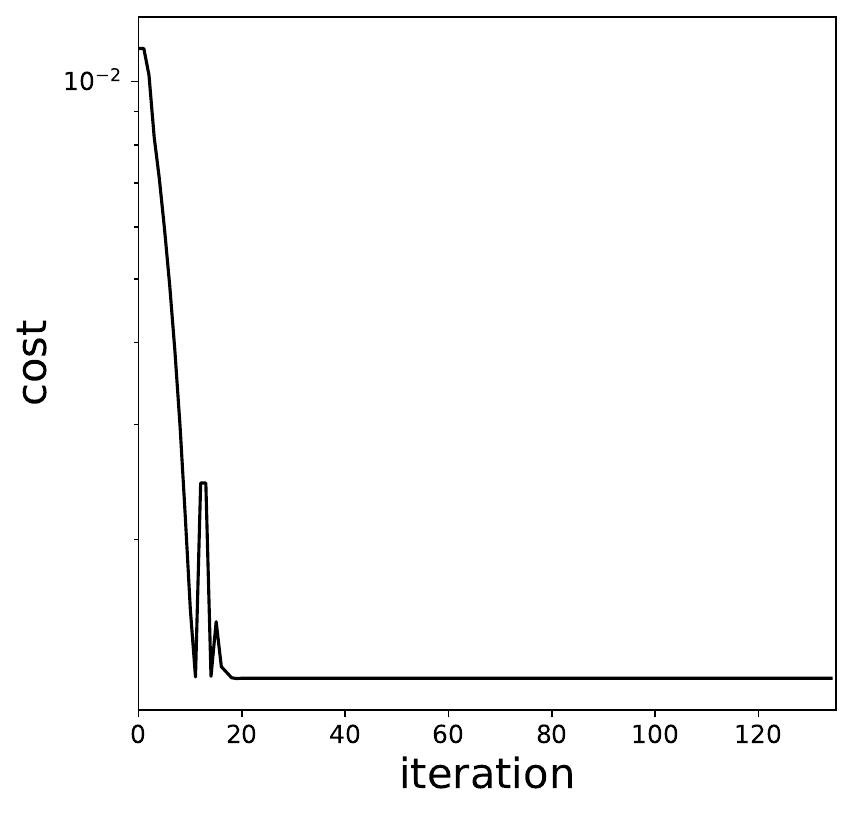}} \
\resizebox{0.15\textwidth}{!}{\includegraphics{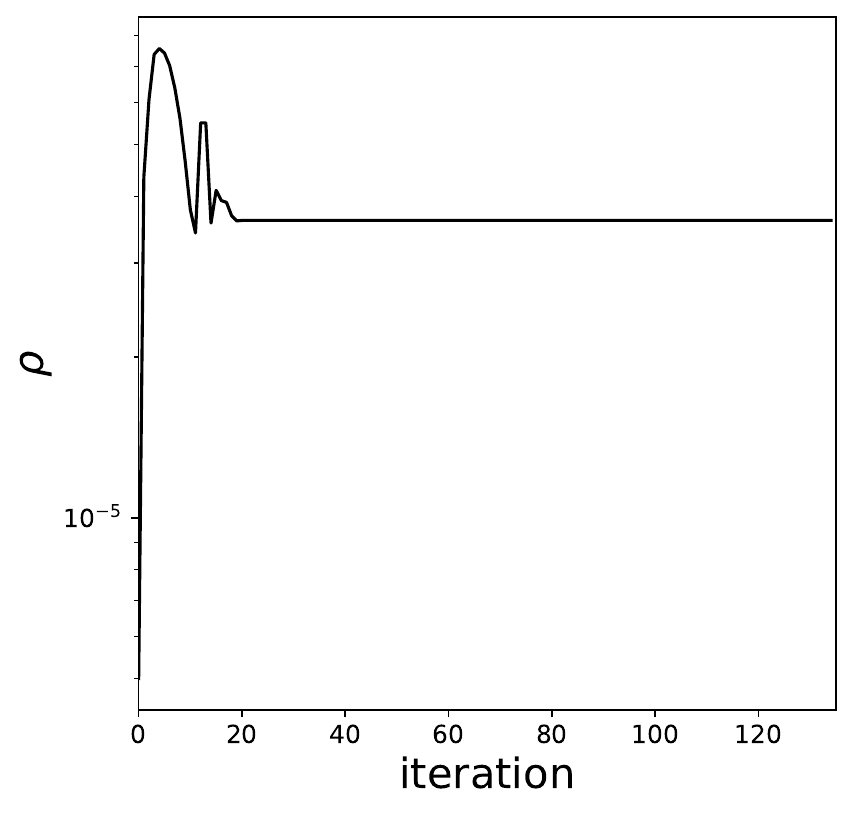}}
\caption{Results for a square boundary interface}
\label{fig:point_source_non-smooth_interface_square}
\end{figure}
%
%
%
%
\begin{figure}[htp!]
\centering
\resizebox{0.15\textwidth}{!}{\includegraphics{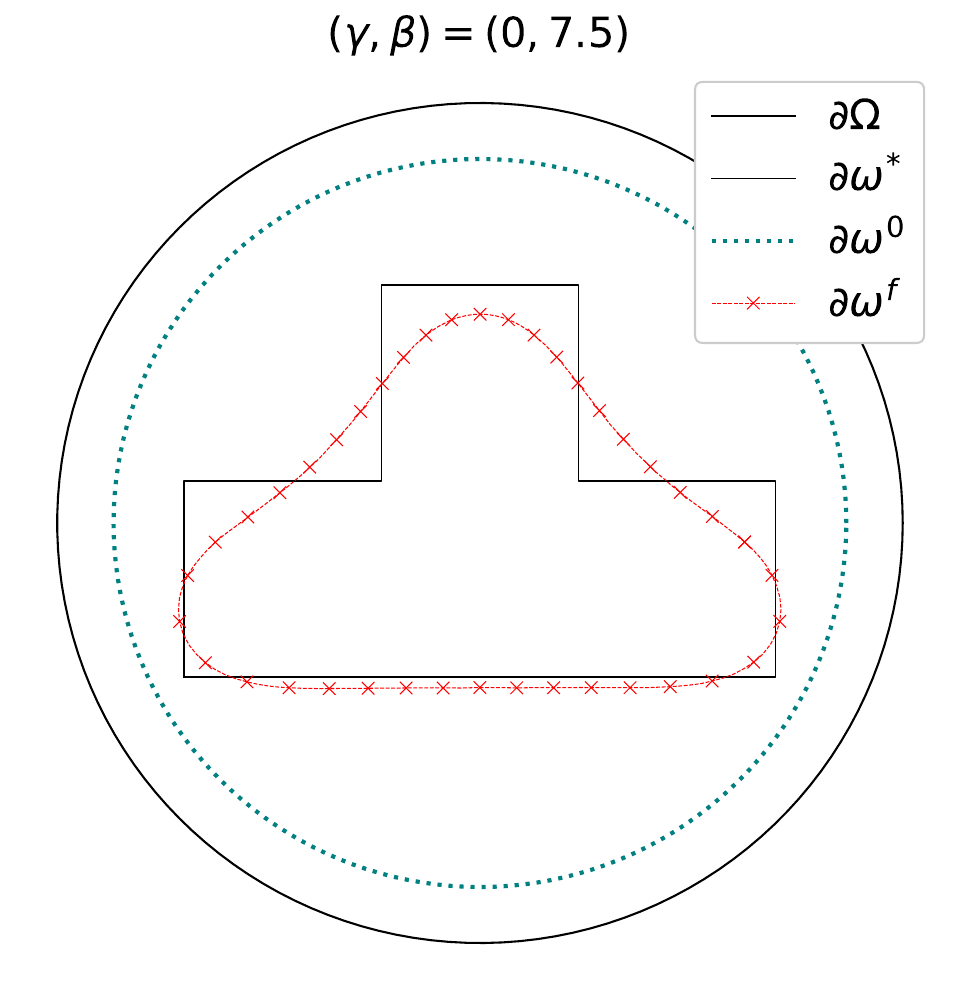}} \
\resizebox{0.15\textwidth}{!}{\includegraphics{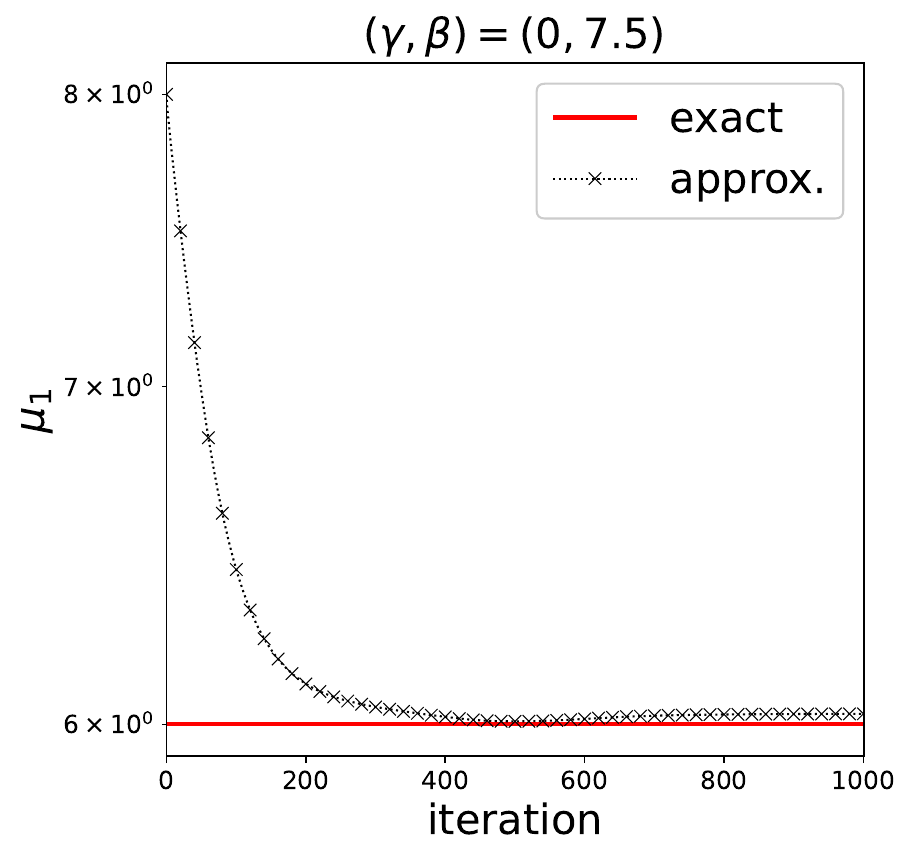}} \
\resizebox{0.15\textwidth}{!}{\includegraphics{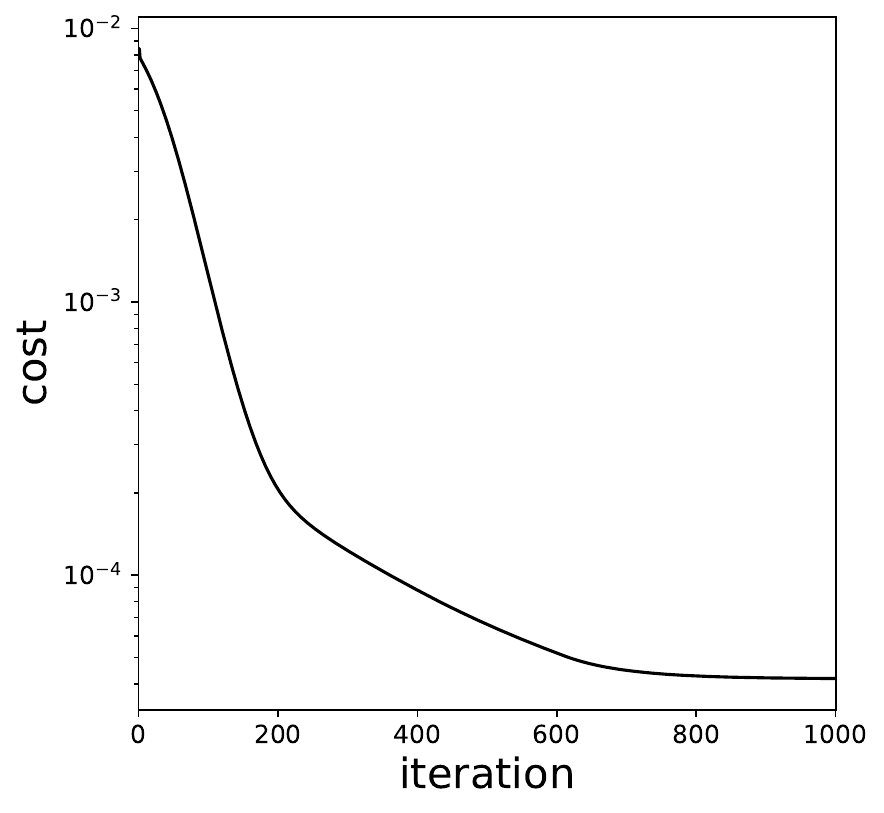}} \
\resizebox{0.15\textwidth}{!}{\includegraphics{rhoinvTb7.5n0ne0t1nu5a08R02.6p0.pdf}}
\\
\resizebox{0.15\textwidth}{!}{\includegraphics{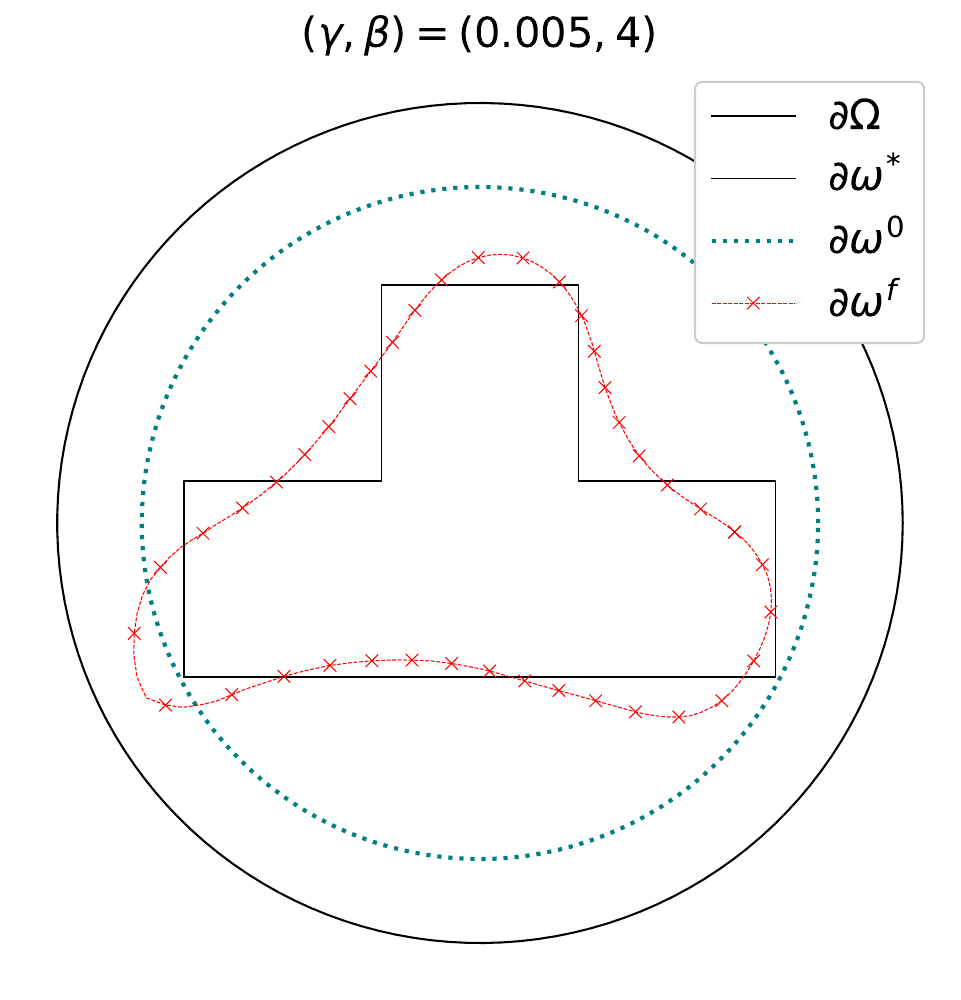}}  \
\resizebox{0.15\textwidth}{!}{\includegraphics{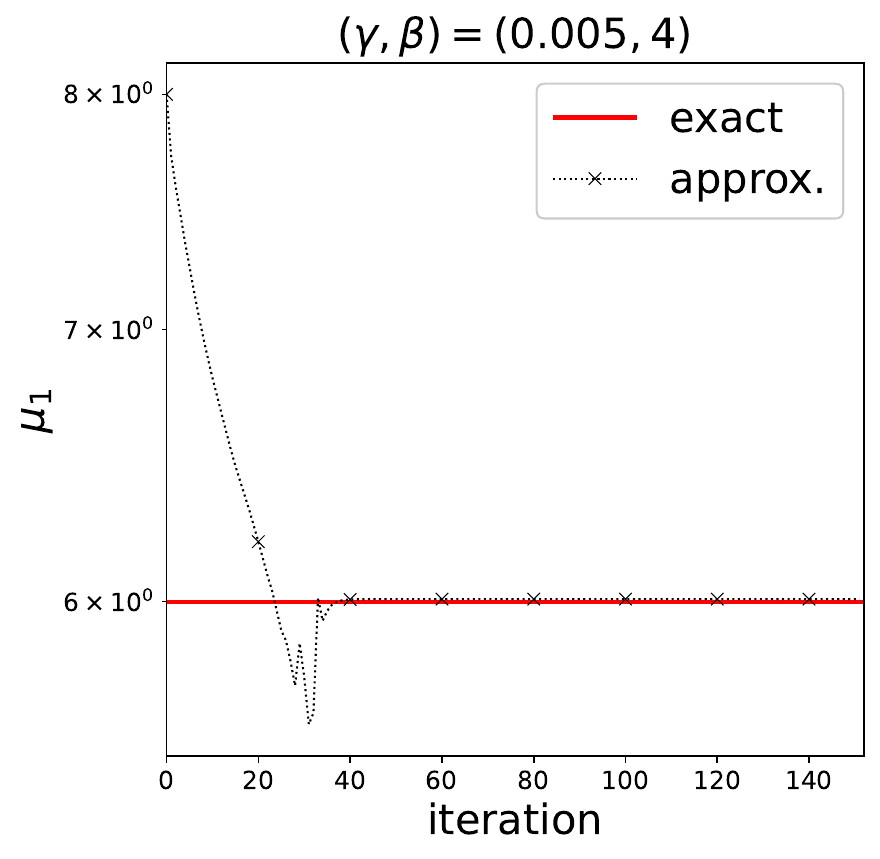}} \
\resizebox{0.15\textwidth}{!}{\includegraphics{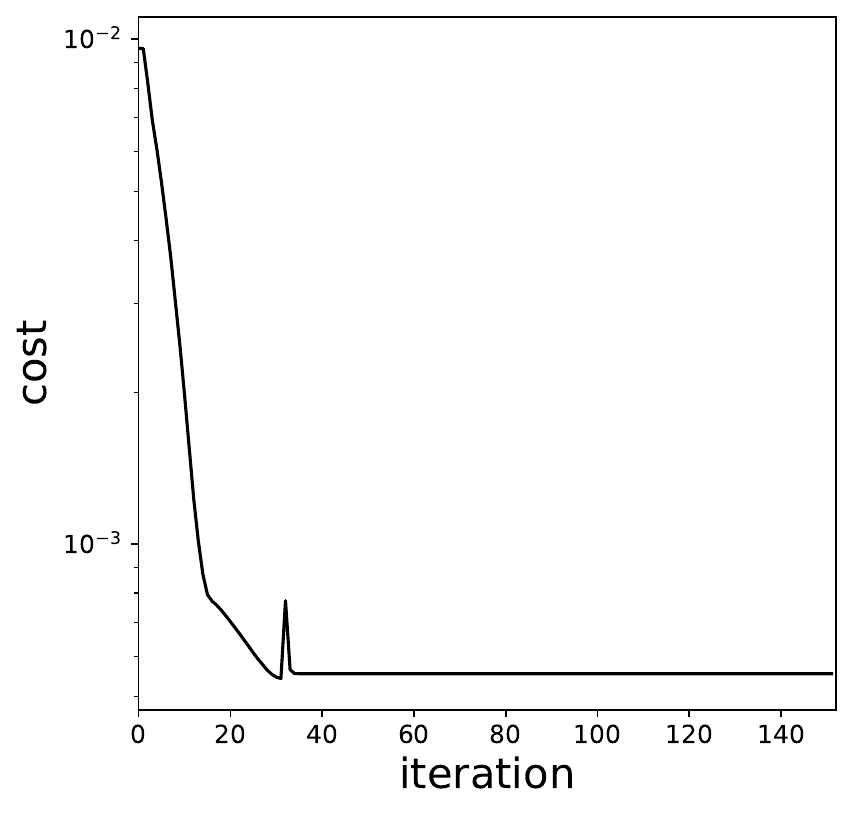}} \
\resizebox{0.15\textwidth}{!}{\includegraphics{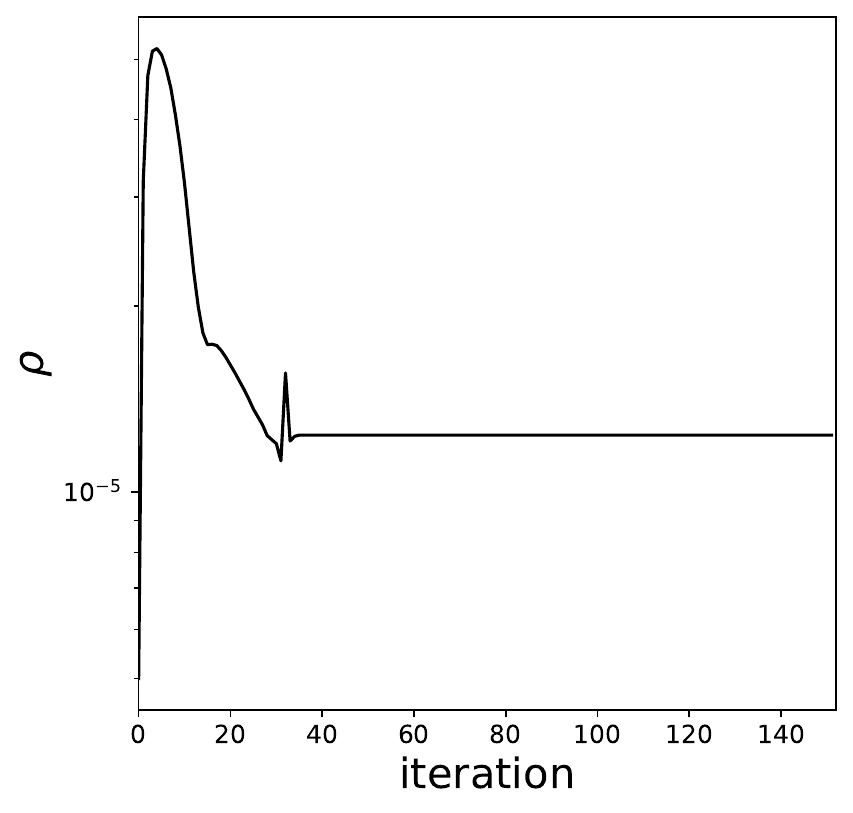}}
\\
\resizebox{0.15\textwidth}{!}{\includegraphics{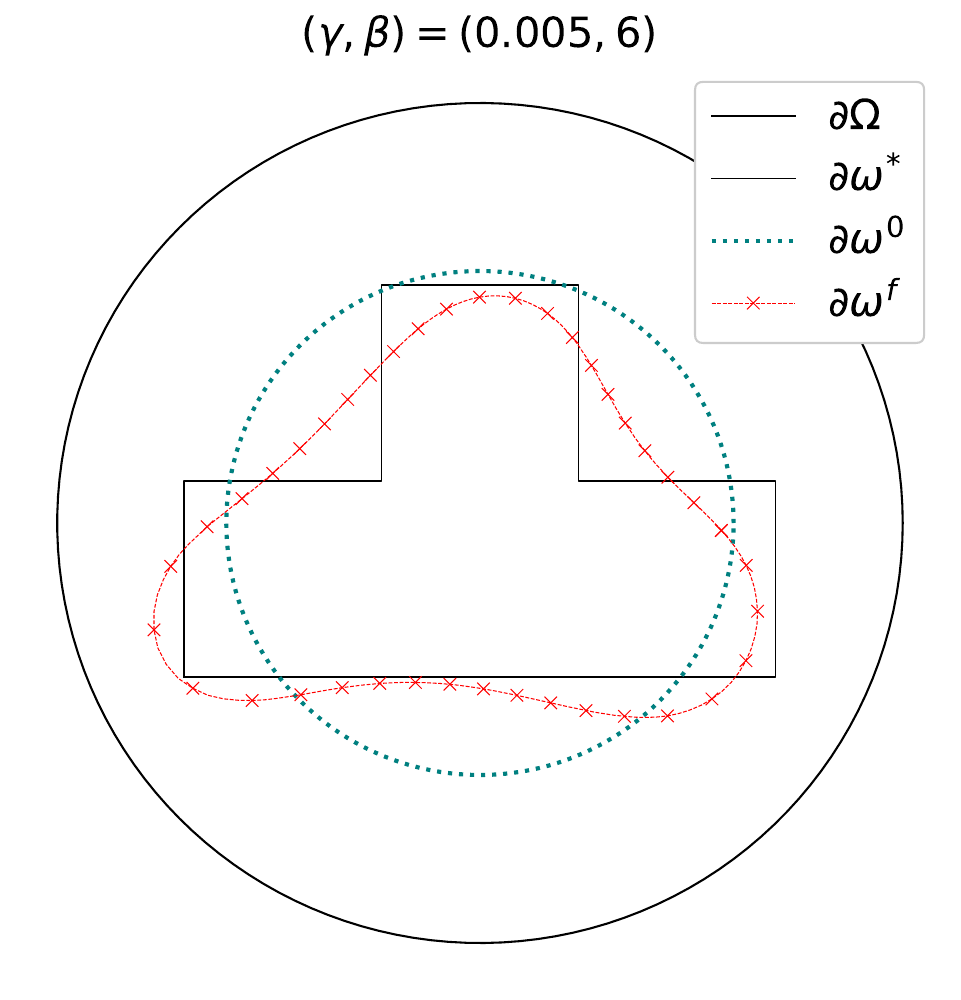}}  \
\resizebox{0.15\textwidth}{!}{\includegraphics{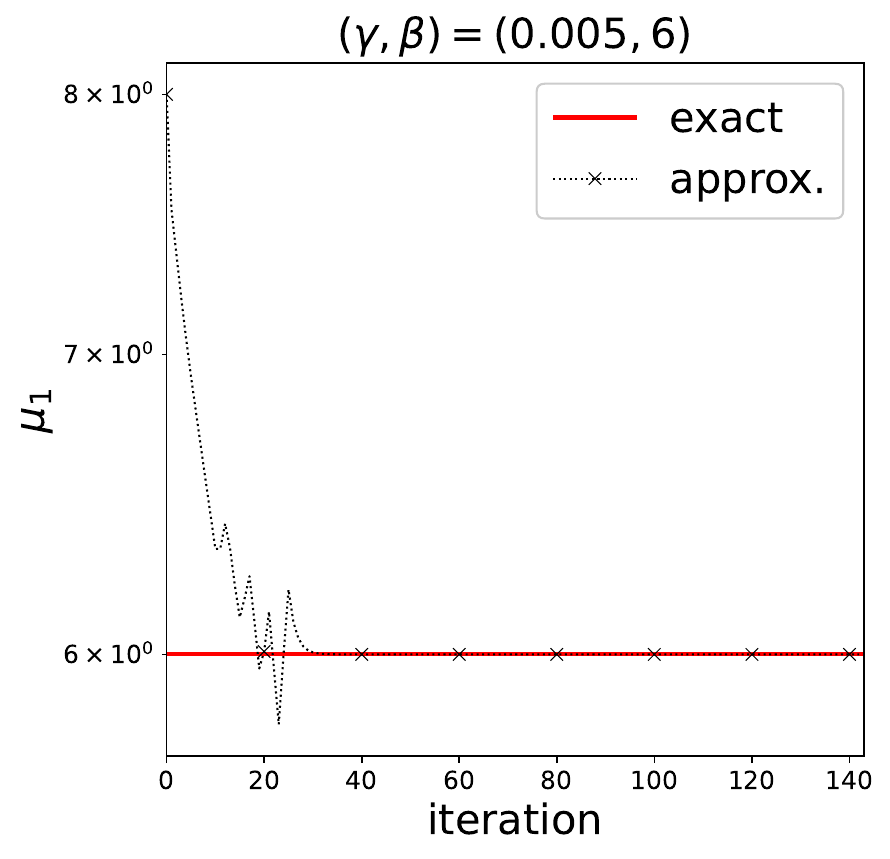}} \
\resizebox{0.15\textwidth}{!}{\includegraphics{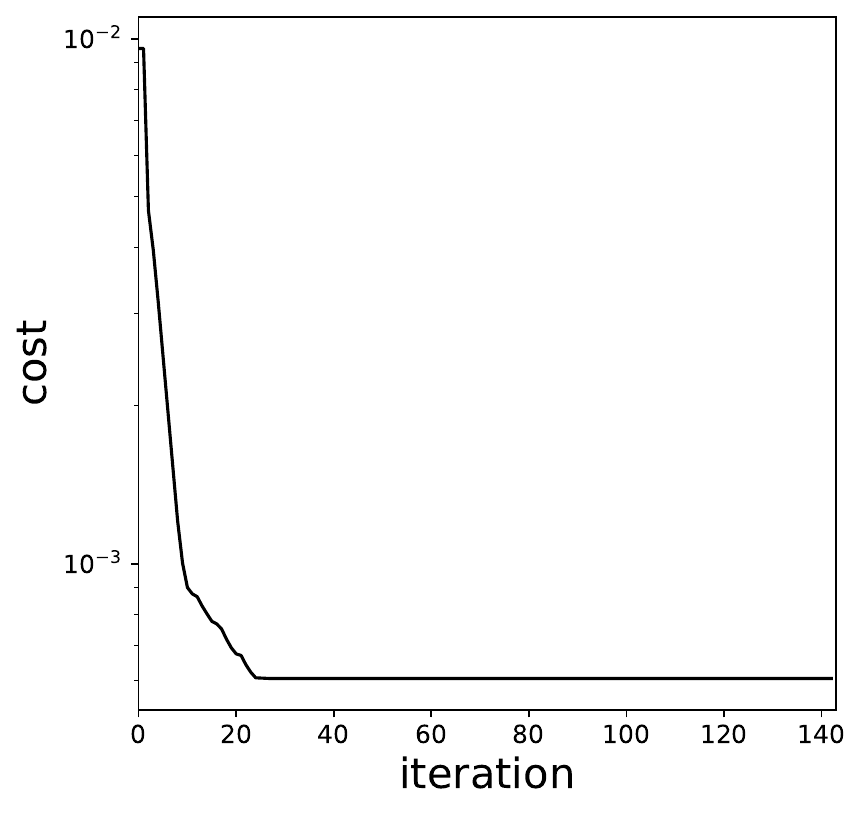}} \
\resizebox{0.15\textwidth}{!}{\includegraphics{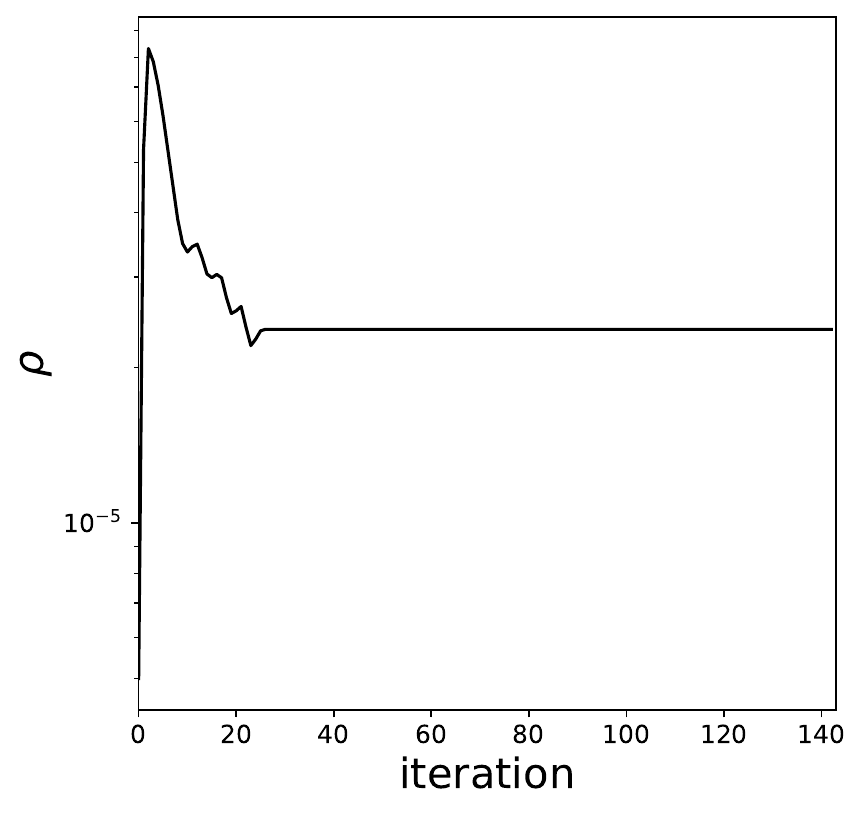}}
\caption{Results for an inverted {\textsf{T}-shaped} boundary interface  }
\label{fig:point_source_non-smooth_interface_T}
\end{figure}
%
%
\subsection{Numerical tests with sources close to the boundary}\label{subsec:point_sources_close_to_the_boundary}
To conclude our numerical examples, we consider cases with multiple (point) sources.
The sources are positioned near the boundary and we define $f$ as follows:
\begin{equation}\label{eq:multiple_sources_near_boundary}
	f(x) = \sum_{i=1}^{M} \exp\left(1 - \frac{(x - x_{i})^2 + (y - y_{i})^2}{\epsilon^2}\right),
\end{equation}
where $\epsilon > 0$, $M \in \mathbb{N}$, and $(x_{i}, y_{i}) = (\bar{R}_{f} \cos{\theta_{i}}, \bar{R}_{f} \sin{\theta_{i}})$ with $\theta_{i} \in [0,2\pi]$ and $\bar{R}_{f} \in (0, R)$. We set $\bar{R}_{f} = 2.99$, and in this subsection, noisy measurements mean $\gamma = 0.1$.

Figure~\ref{fig:multiple_sources_circular} presents results for a circular boundary interface with parameters $(\mu_{0}^{\ast}, \mu_{1}^{\ast}) = (1, 1.2)$ under exact and noisy measurements.
Thick black dots indicate the source positions, which remain consistent across all cases, and reconstructions are achieved without perimeter regularization.
Reconstruction accuracy decreases with fewer sources, as expected.
For instance, when $\theta_{i} = (2\pi/3)i$ for $i = 1, 2, 3$ (first column in Figure~\ref{fig:multiple_sources_circular}), the reconstructed shape deviates more from a circle compared to cases with more sources.
Nonetheless, the results remain reasonable, even with noise.
For these cases, we set $\epsilon = 0.5$ in \eqref{eq:multiple_sources_near_boundary}.
%
%
\begin{figure}[htp!]
\centering
\hfill
\resizebox{0.15\textwidth}{!}{\includegraphics{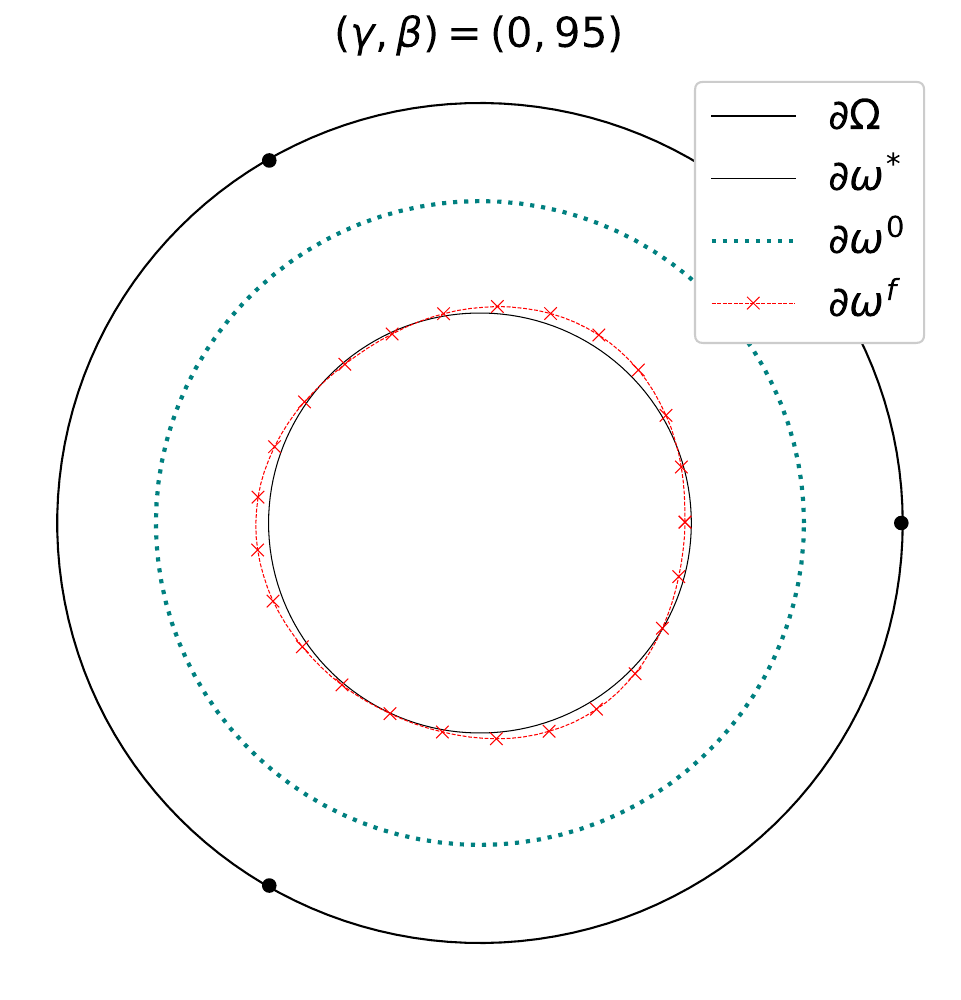}} \hfill
\resizebox{0.15\textwidth}{!}{\includegraphics{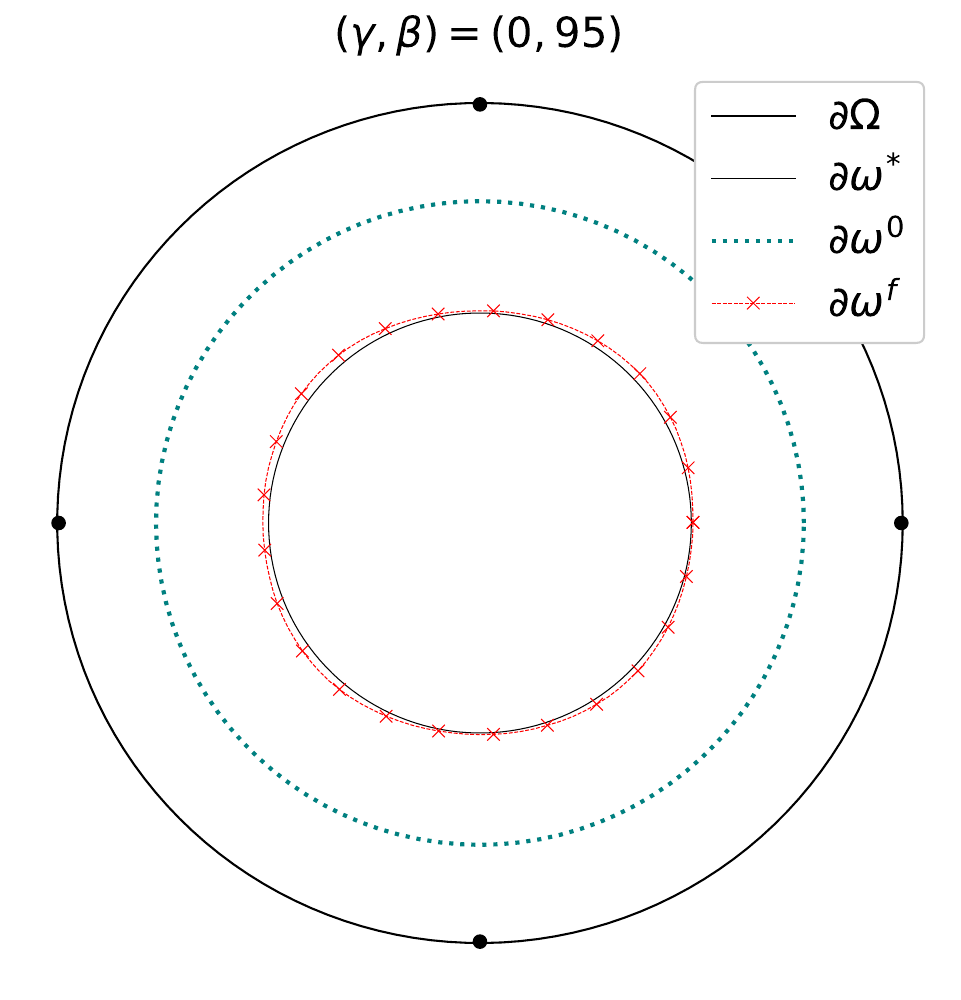}} \hfill
\resizebox{0.15\textwidth}{!}{\includegraphics{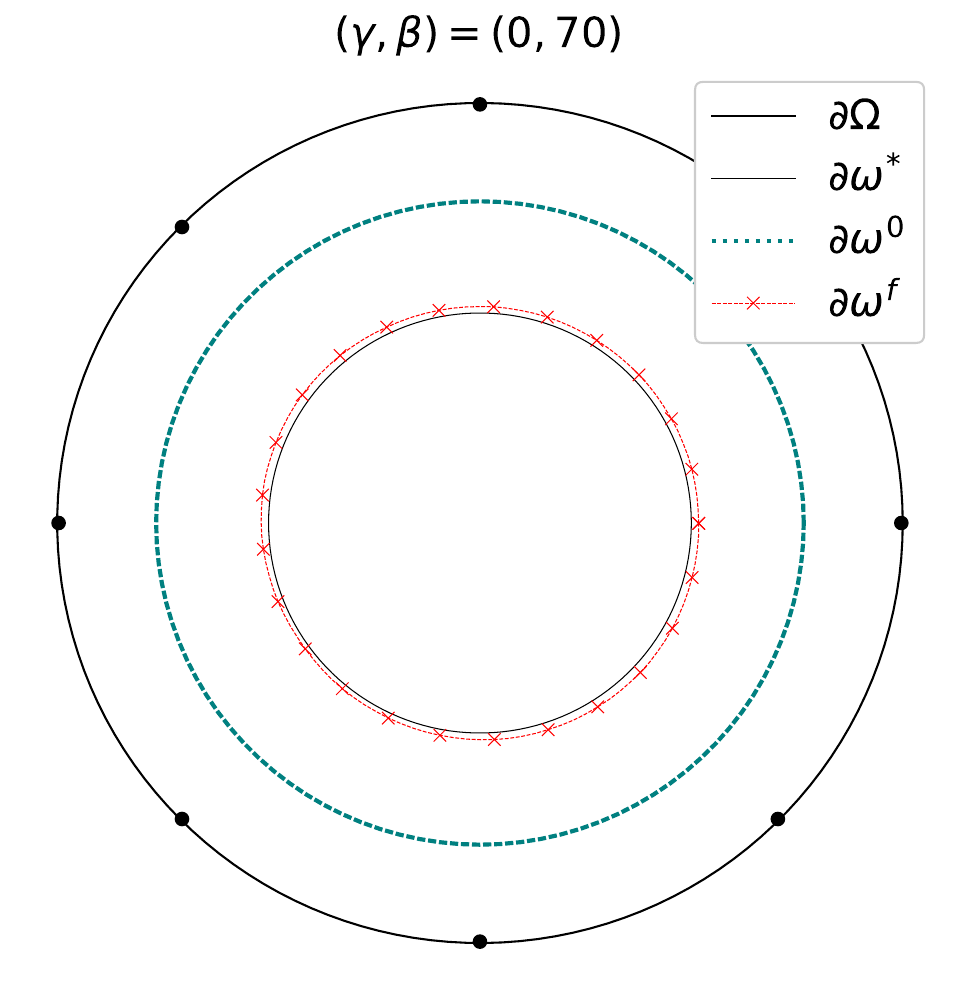}} \hfill
\resizebox{0.15\textwidth}{!}{\includegraphics{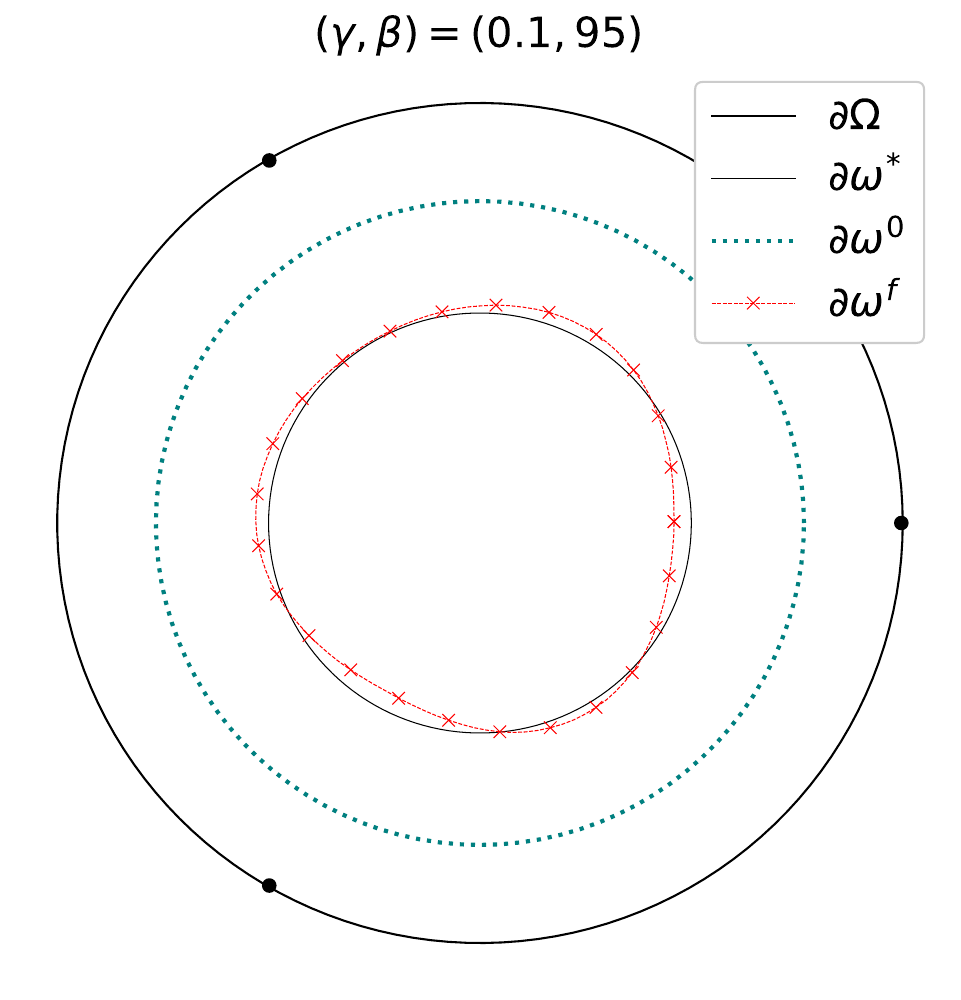}} \hfill
\resizebox{0.15\textwidth}{!}{\includegraphics{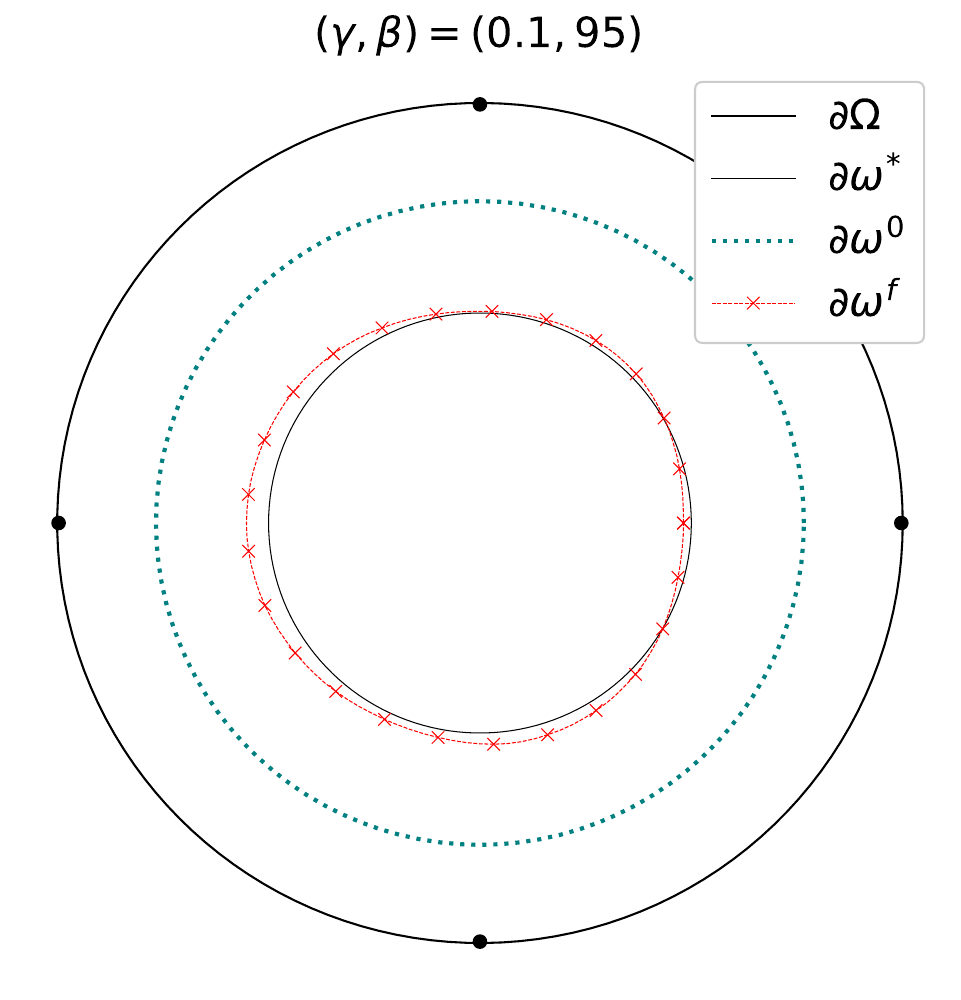}} \hfill
\resizebox{0.15\textwidth}{!}{\includegraphics{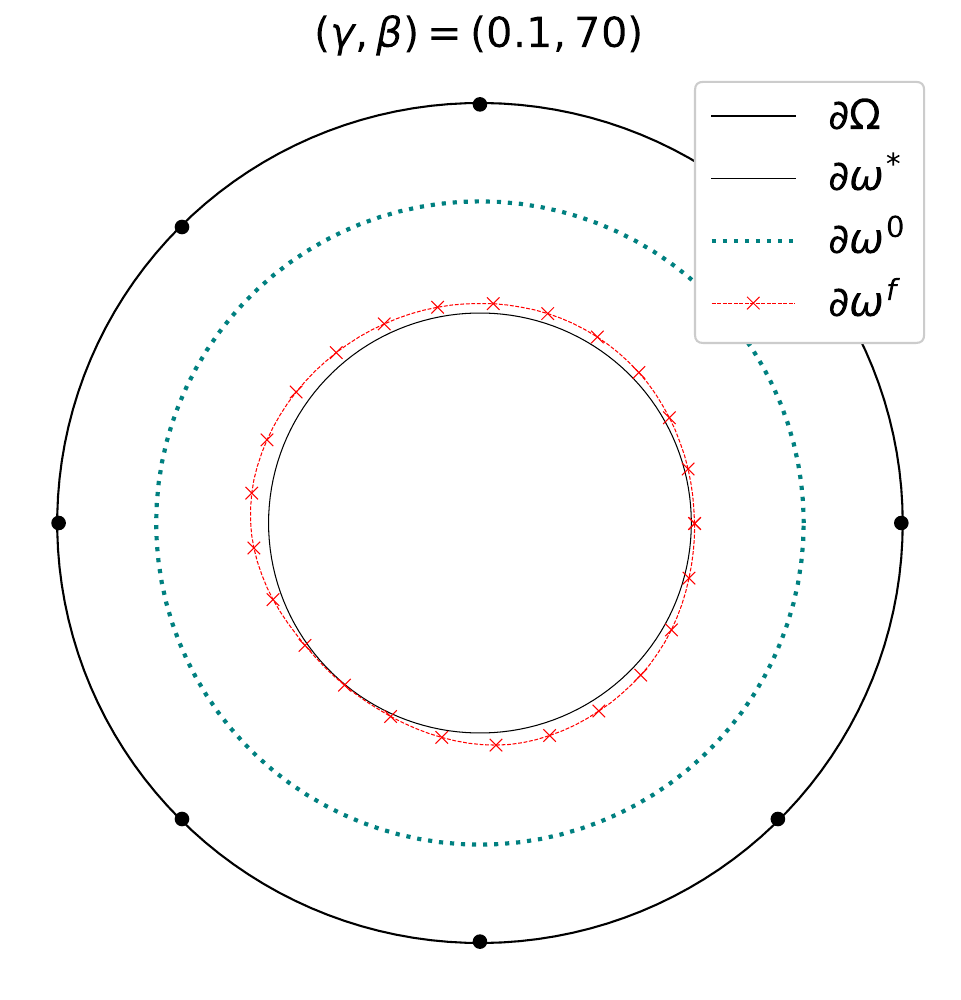}} \\
\resizebox{0.15\textwidth}{!}{\includegraphics{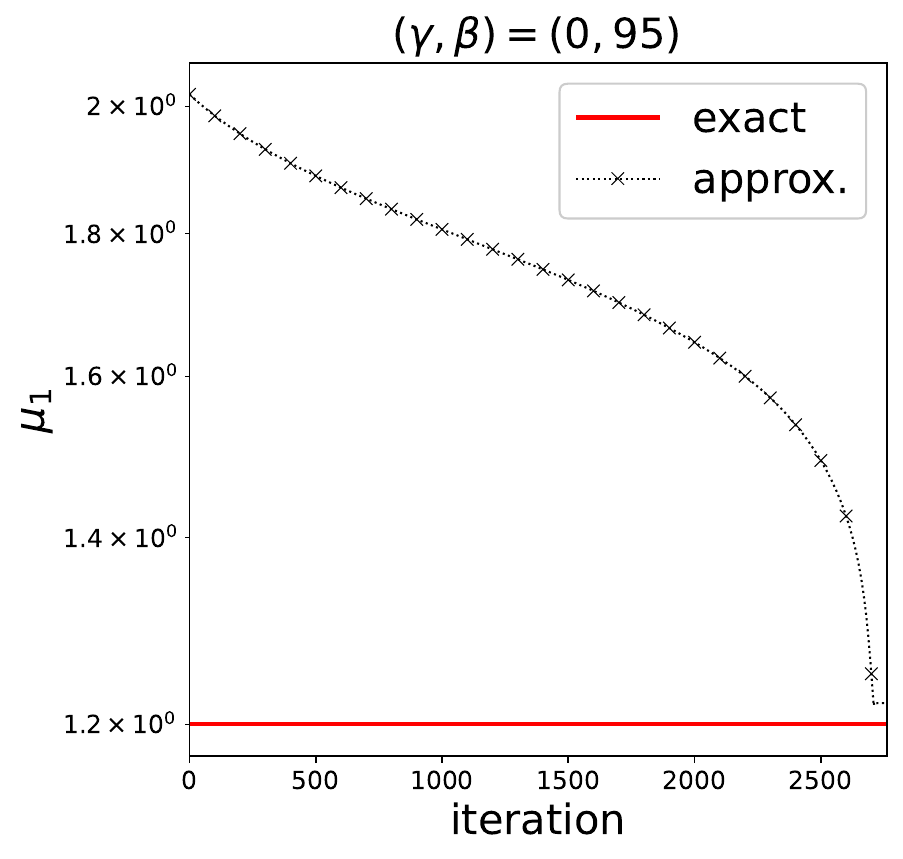}} \hfill
\resizebox{0.15\textwidth}{!}{\includegraphics{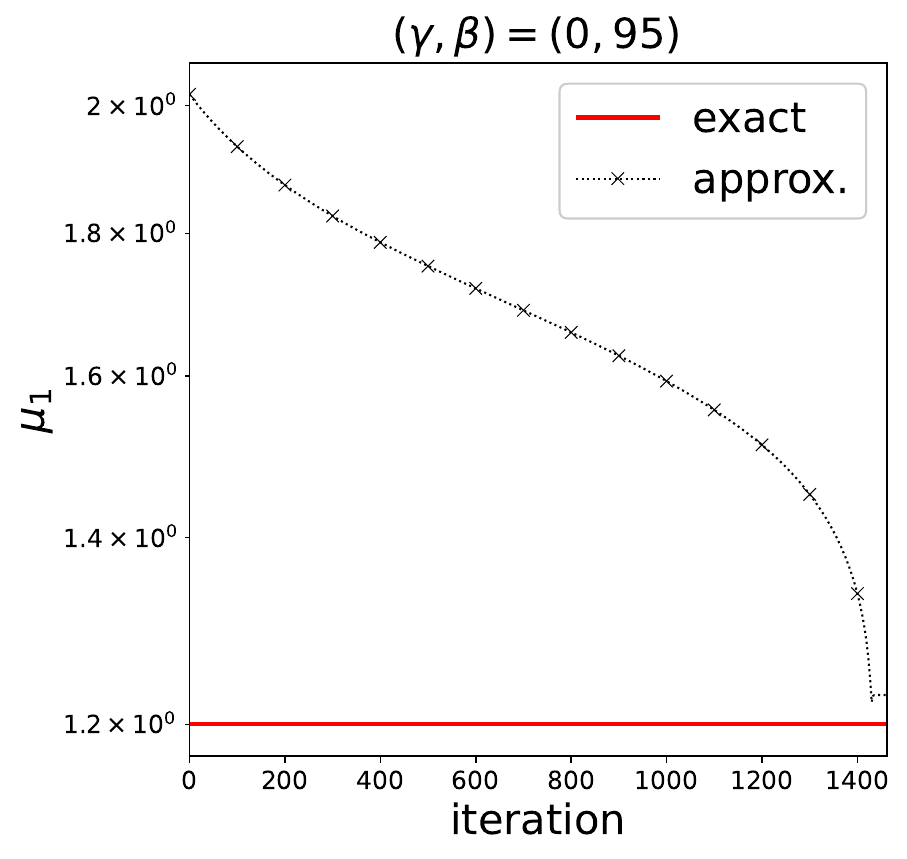}} \hfill
\resizebox{0.15\textwidth}{!}{\includegraphics{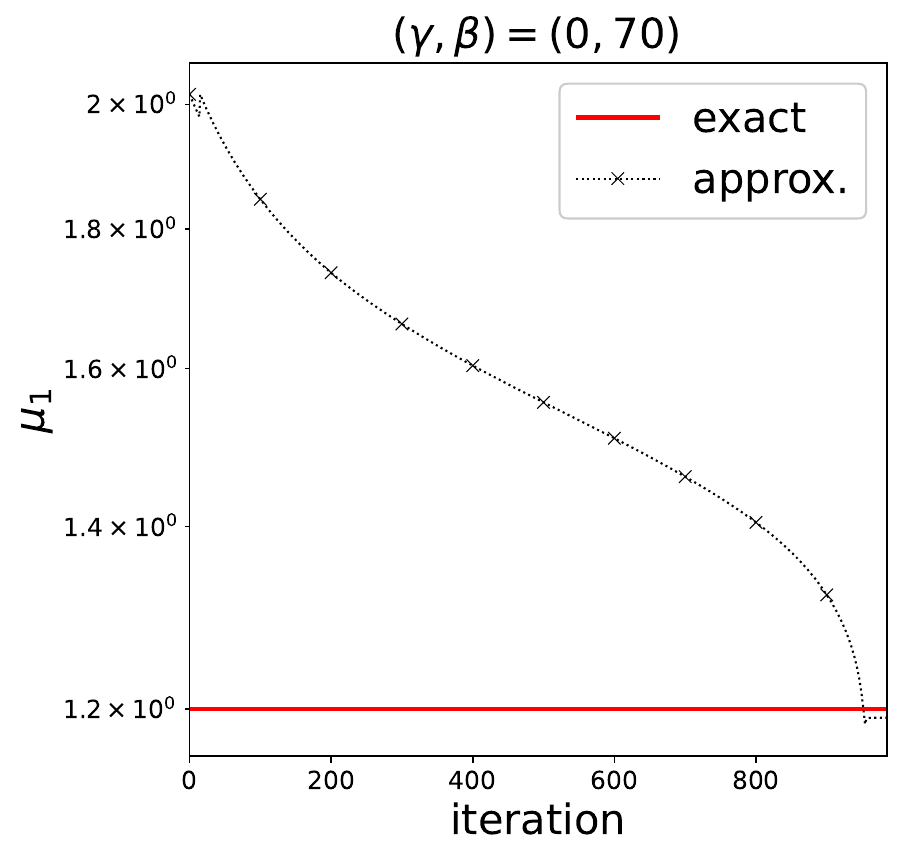}} \hfill
\resizebox{0.15\textwidth}{!}{\includegraphics{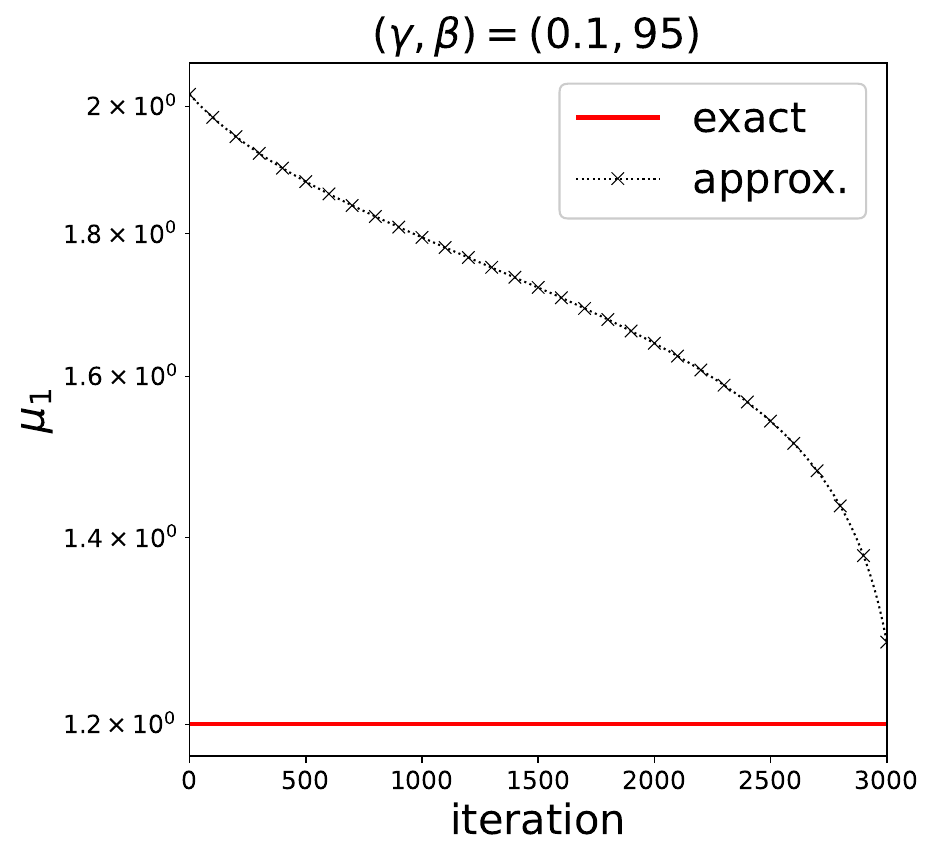}} \hfill
\resizebox{0.15\textwidth}{!}{\includegraphics{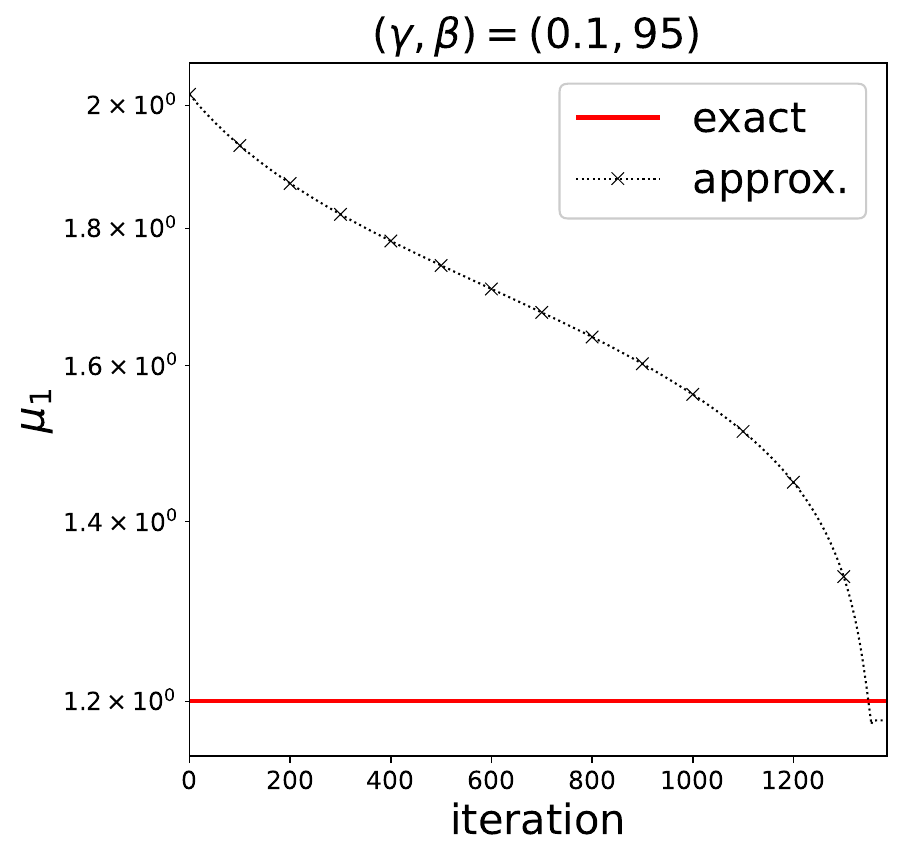}} \hfill
\resizebox{0.15\textwidth}{!}{\includegraphics{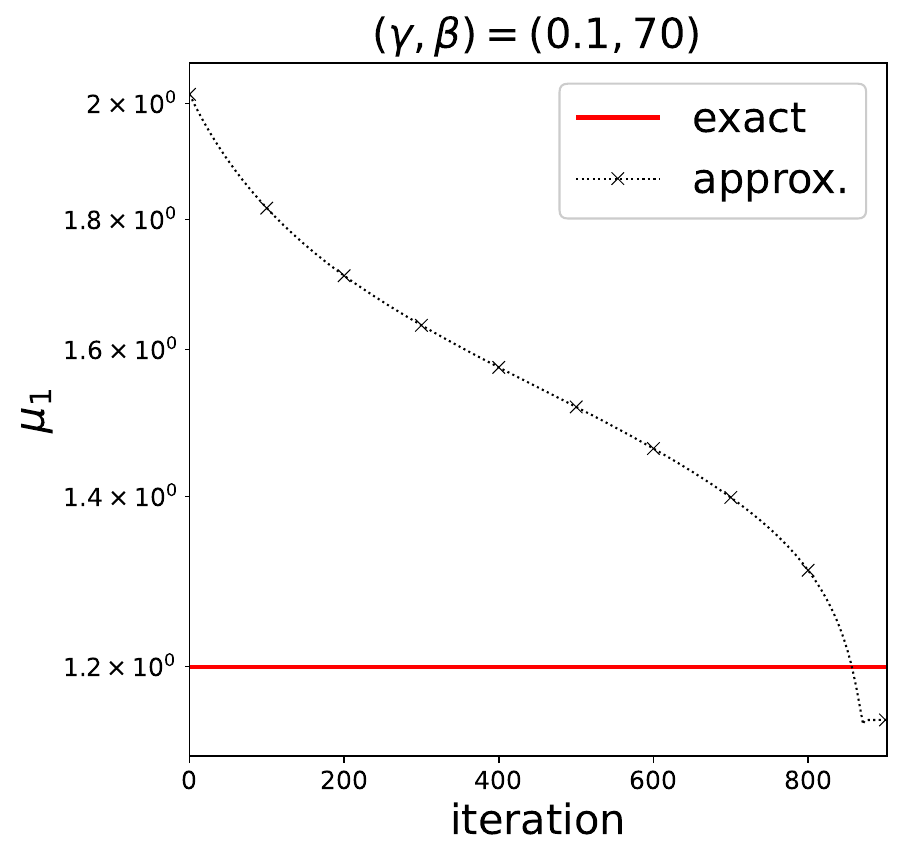}} \hfill
\caption{Results for a circular boundary interface with $f$ (given by \eqref{eq:multiple_sources_near_boundary} with $\epsilon = 0.5$) of various positions and instances near the boundary and parameters $(\mu_{0}^{\ast}, \mu_{1}^{\ast}) = (1, 1.2)$, under exact (first three columns from the left) and noisy measurements.
No perimeter regularization was applied in any of the cases.
}
\label{fig:multiple_sources_circular}
\end{figure}
%
%
%

Next, we {revisit} the problem with eight sources positioned at $\theta_{i} = \frac{\pi}{4}i$, where $i = 1, \ldots, 8$
and examine the effect of varying $\epsilon$ in \eqref{eq:multiple_sources_near_boundary}.
We consider $\epsilon = 0.5, 0.3, 0.2$ and apply perimeter penalization with a small weight.
For the remaining experiments, we set $(\mu_{0}^{\ast}, \mu_{1}^{\ast}) = (1, 6)$ and $\gamma = 0.1$, using a peanut-shaped boundary interface.

Figures~\ref{fig:multiple_sources_peanut_1} to \ref{fig:multiple_sources_peanut_3} show the results.
Smaller $\epsilon$ values lead to less accurate reconstructions, making it harder to capture concave boundary regions.
These results confirm that $\epsilon$ reflects the diffusion level of the sources.

{Finally, we analyze the reconstruction accuracy based on source positioning under noisy measurements.}
Using eight sources ($M = 8$) with $\epsilon = 0.5$ in \eqref{eq:multiple_sources_near_boundary}, we fix the penalization parameter $\rho_{1}$ to $6 \times 10^{-6}$. Reconstructions are evaluated for six configurations of source positions (see Figure~\ref{fig:positioning_sources} for two illustrations):
\[
	\text{Setup ${K}$}:\qquad \theta_{i} = \frac{{K}\pi}{3} + \frac{i\pi}{8}, \qquad i=1,\ldots,8, \quad {K} = 1,\ldots,6.
\]

The reconstruction results are summarized in Figure~\ref{fig:positioning_sources_results}.
Even with noisy data, the method reconstructs the boundary interface and the unknown coefficient $\mu$ effectively.
Also, observe that source positioning strongly influences accuracy, particularly for concave boundary regions.
Reconstructions are less accurate in areas farther from the sources, as expected.
Overall, the proposed method is robust and highly effective in reconstructing boundary interfaces with complex geometries under noisy measurements.
%
\begin{figure}[htp!]
\centering
\resizebox{0.15\textwidth}{!}{\includegraphics{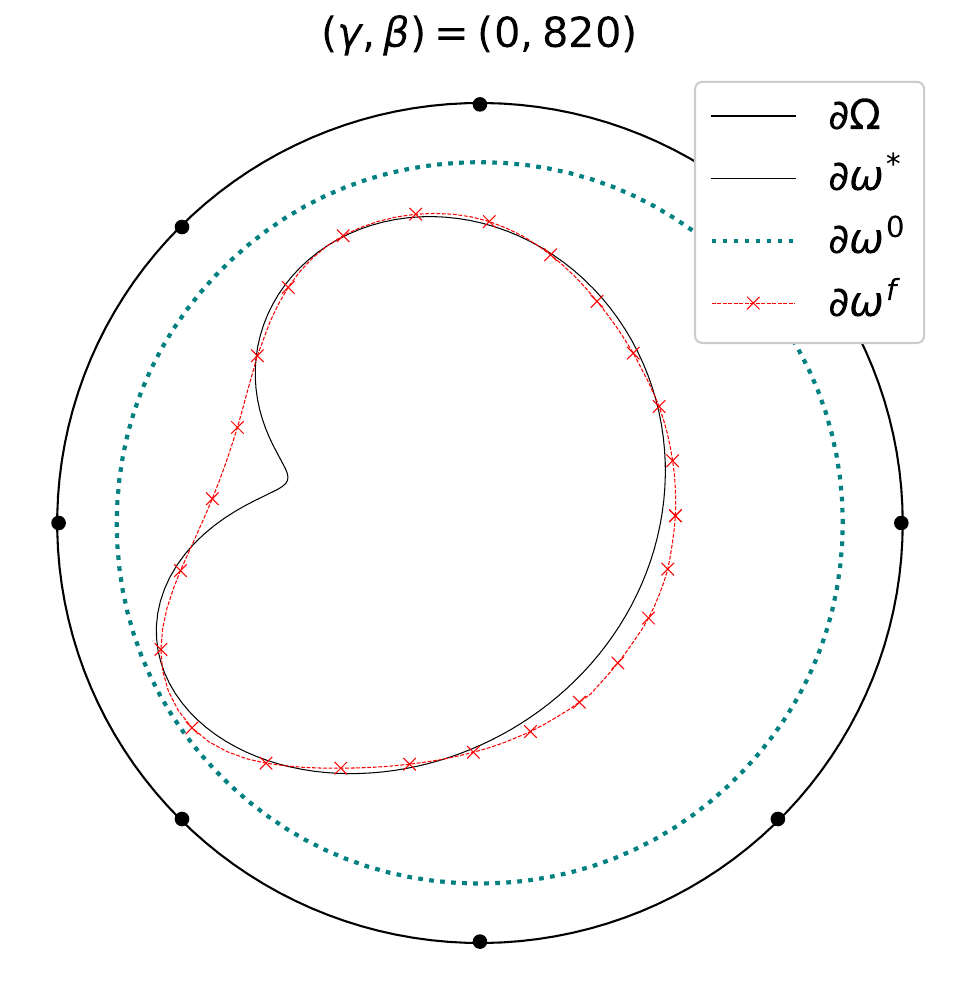}} \
\resizebox{0.15\textwidth}{!}{\includegraphics{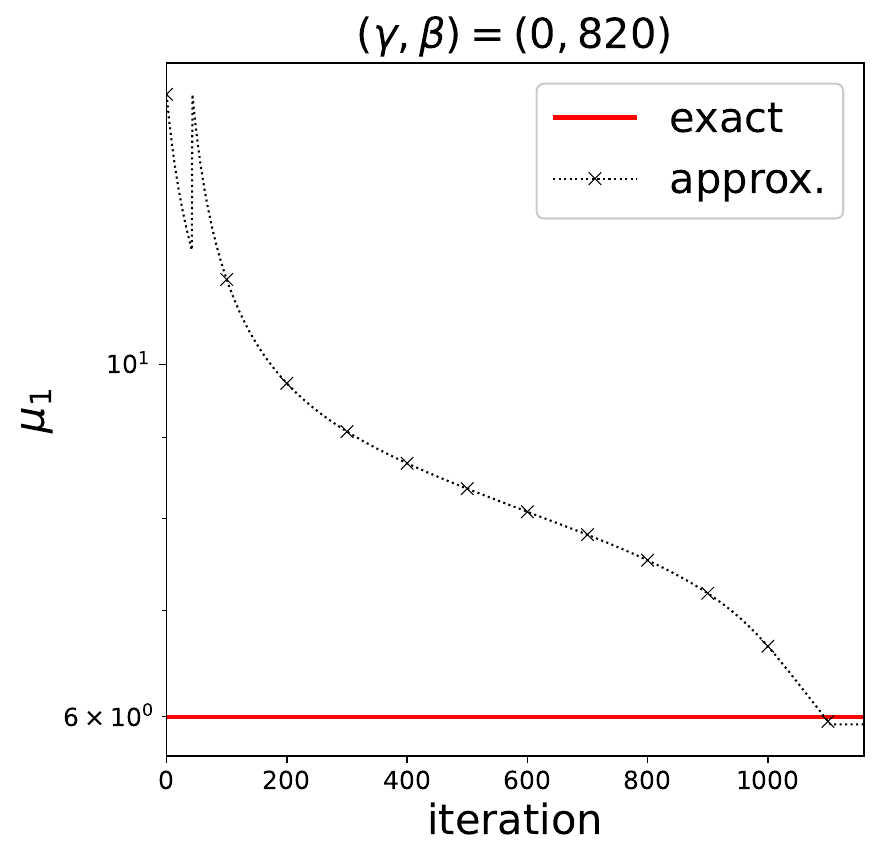}} \
\resizebox{0.15\textwidth}{!}{\includegraphics{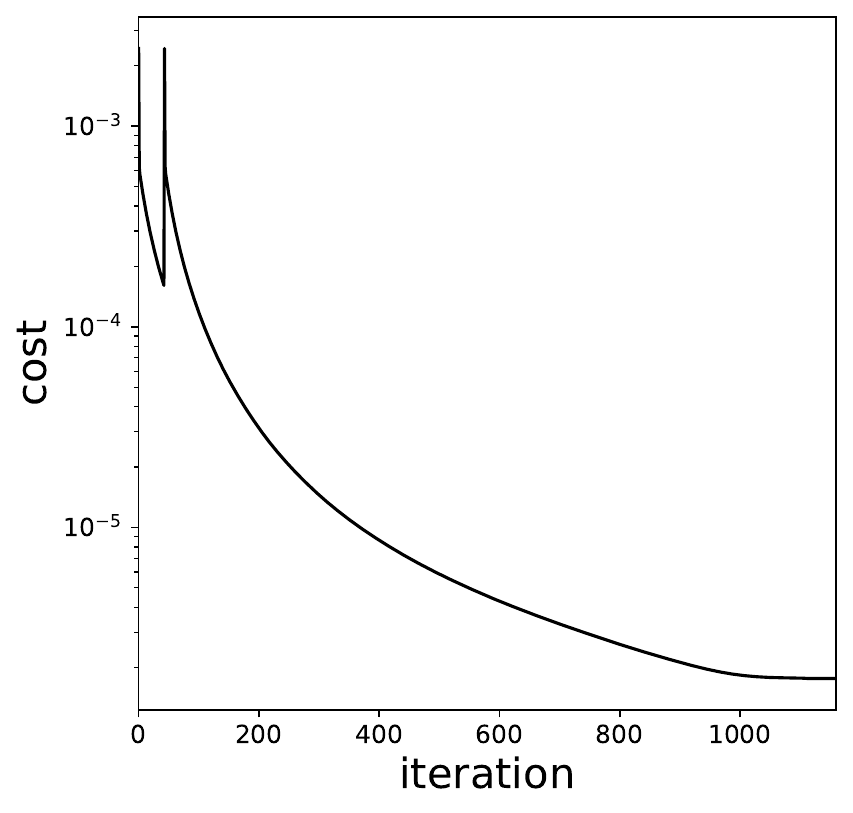}} \
\resizebox{0.15\textwidth}{!}{\includegraphics{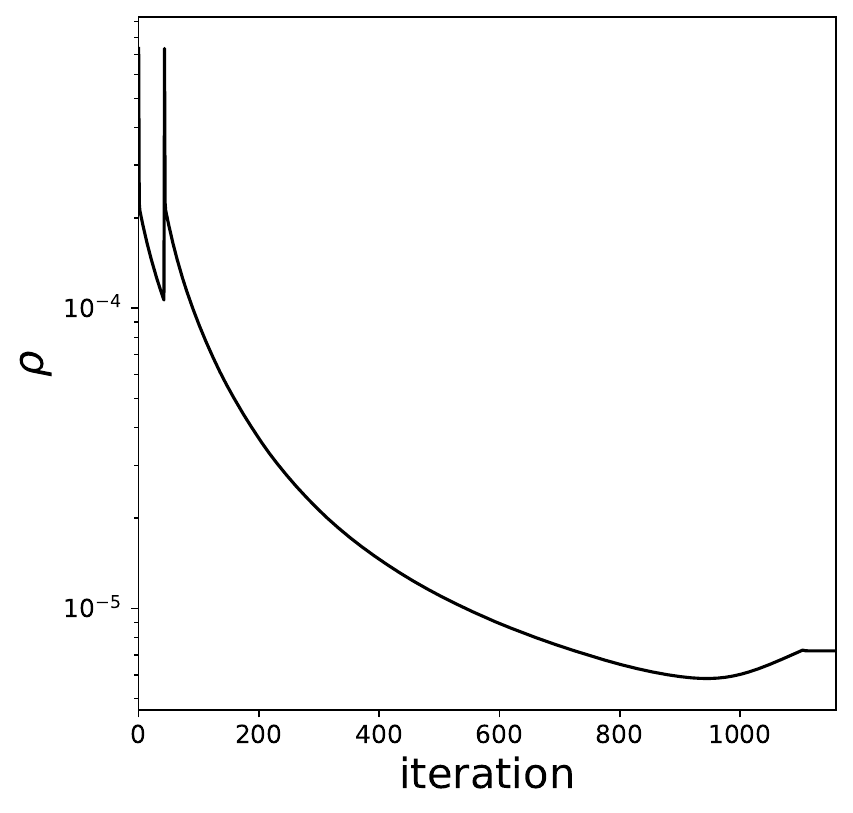}}
\\
\resizebox{0.15\textwidth}{!}{\includegraphics{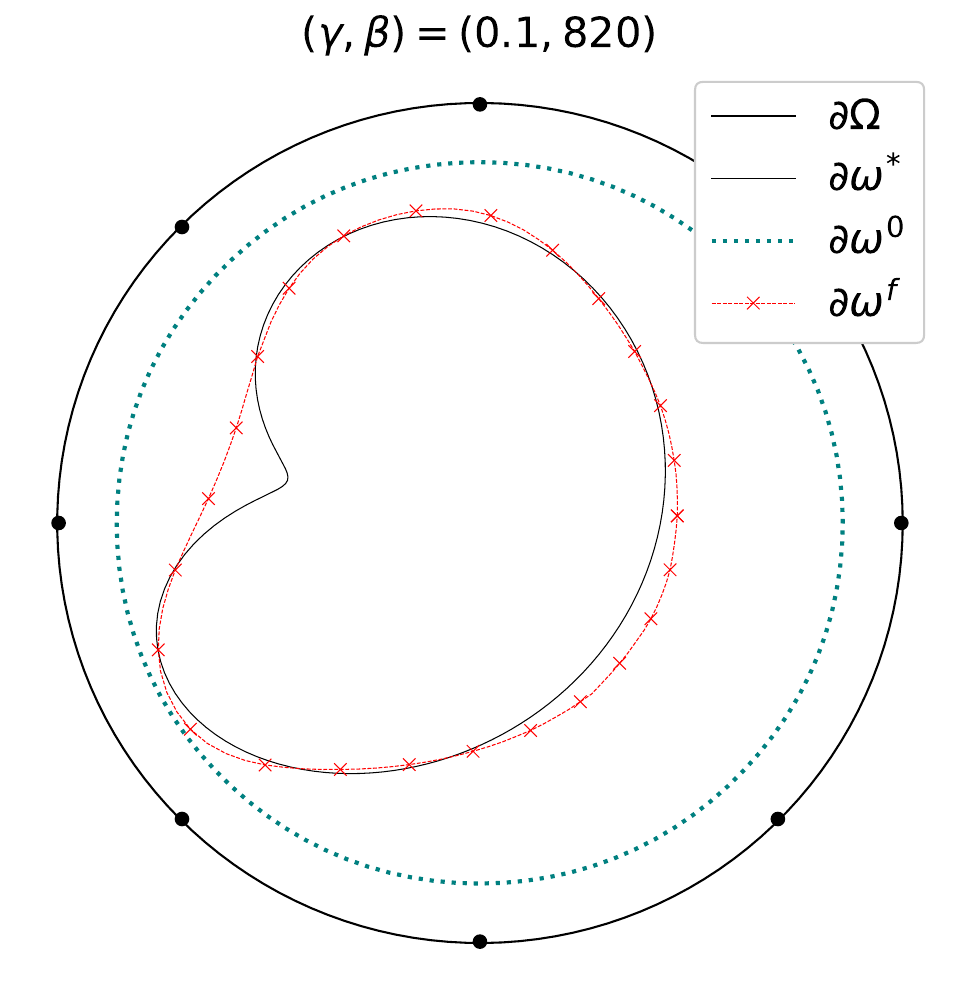}} \
\resizebox{0.15\textwidth}{!}{\includegraphics{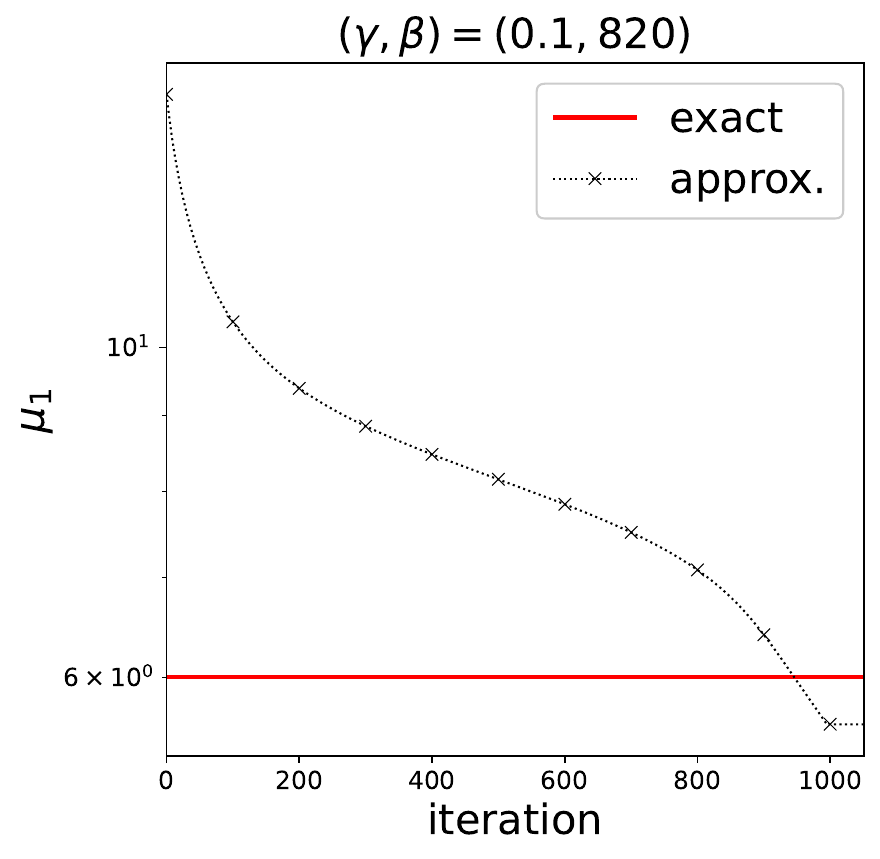}} \
\resizebox{0.15\textwidth}{!}{\includegraphics{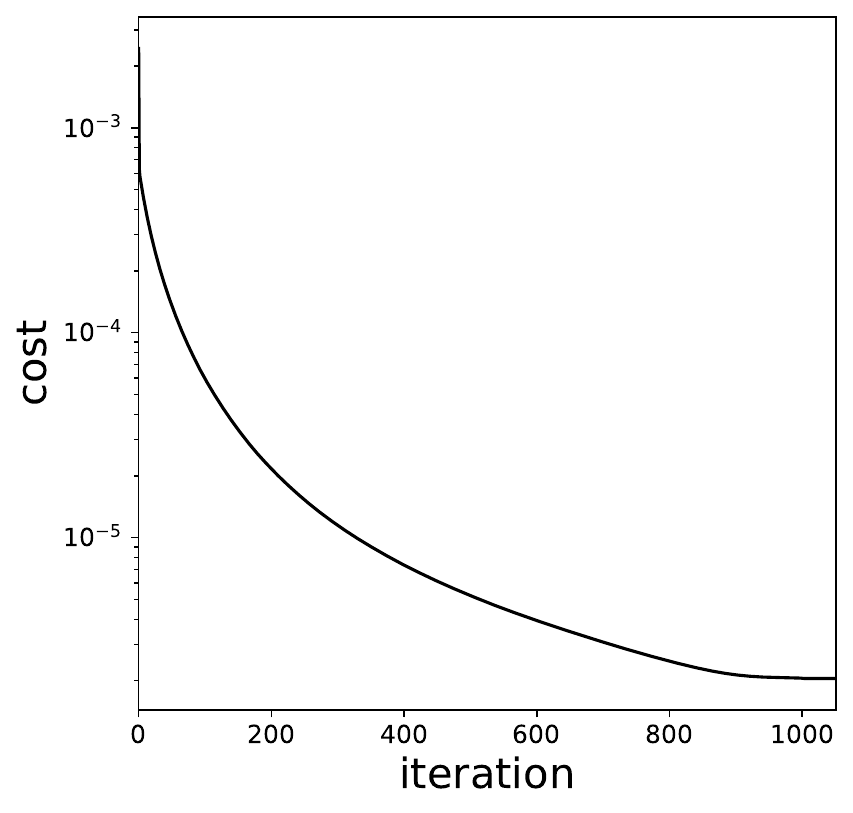}} \
\resizebox{0.15\textwidth}{!}{\includegraphics{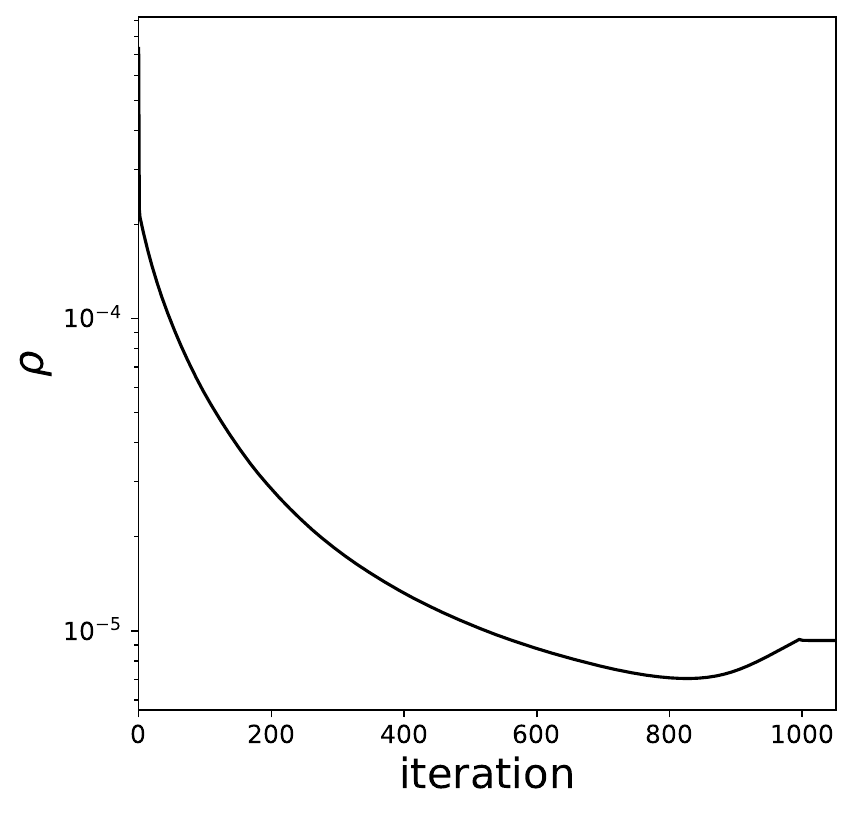}}
\caption{Results for a peanut-shape boundary interface with source $f$ (given by \eqref{eq:multiple_sources_near_boundary} with $\epsilon = 0.5$) near the boundary, under exact (top row) and noisy measurements.
Perimeter penalization was applied in all of the cases with $\rho_{1} = 0.00004$.
}
\label{fig:multiple_sources_peanut_1}
\end{figure}
%
%
%
%
\begin{figure}[htp!]
\centering
\resizebox{0.15\textwidth}{!}{\includegraphics{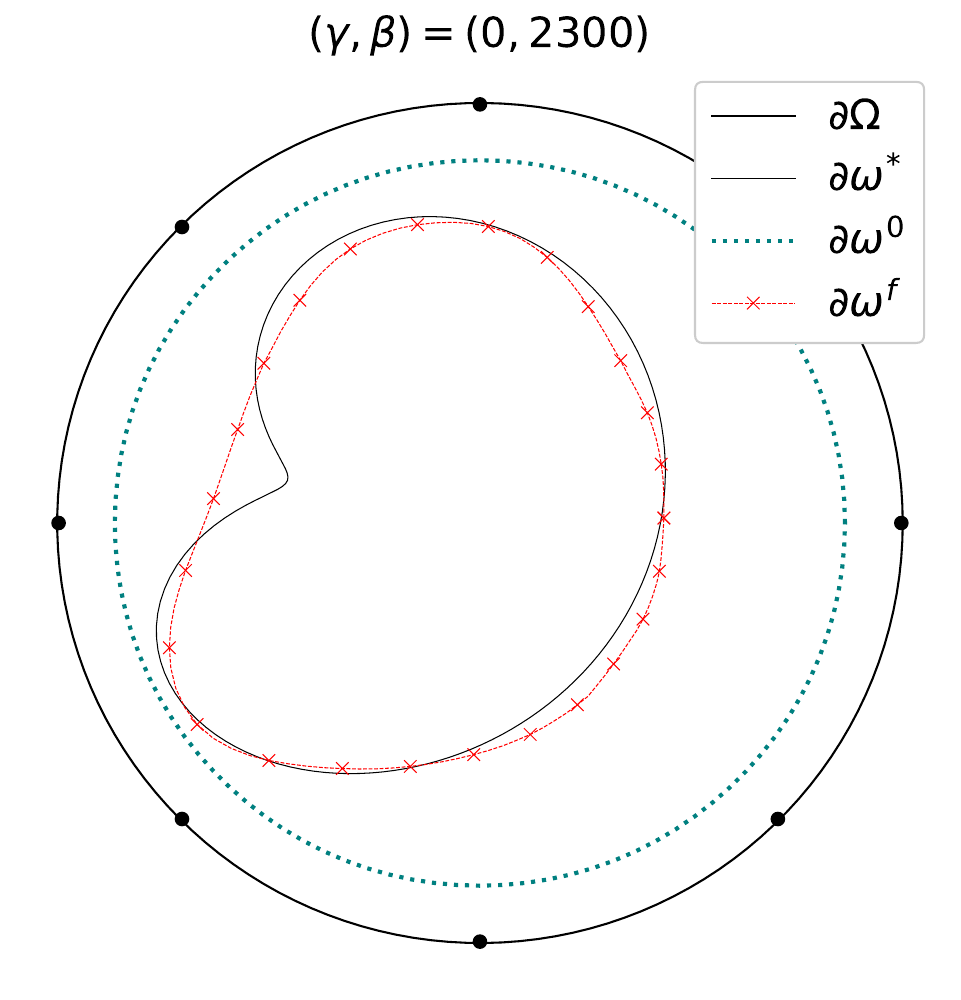}} \
\resizebox{0.15\textwidth}{!}{\includegraphics{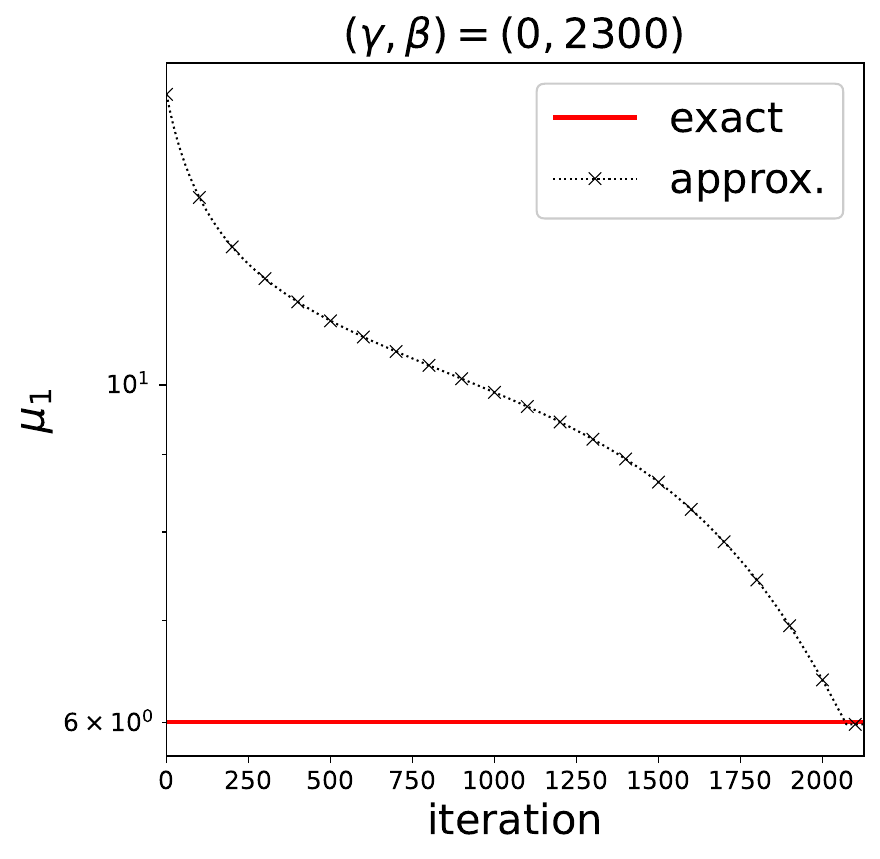}} \
\resizebox{0.15\textwidth}{!}{\includegraphics{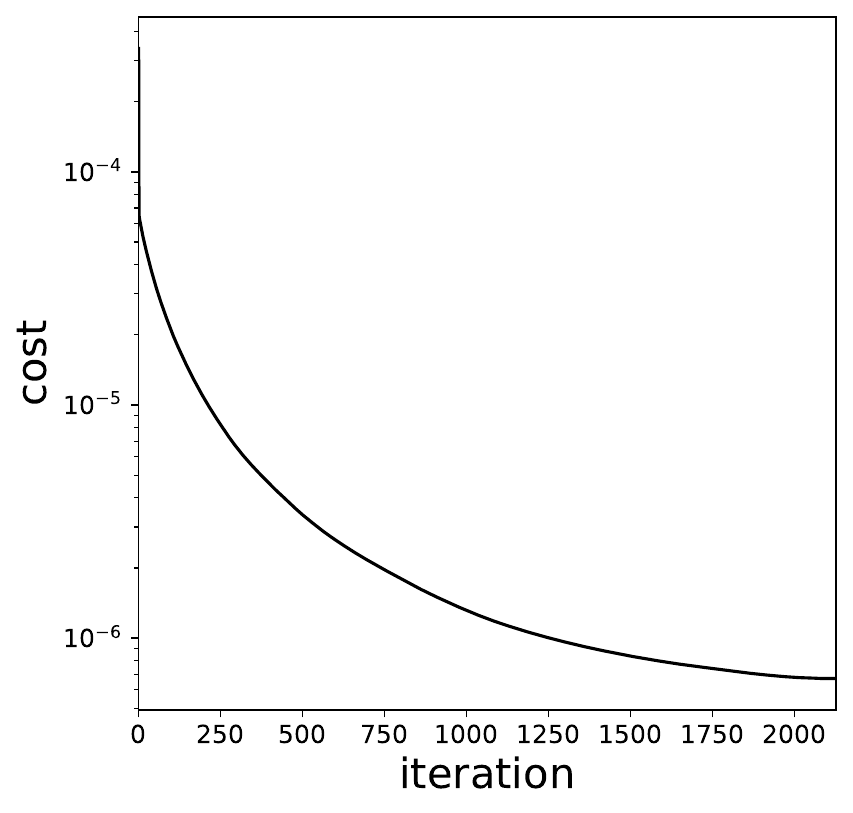}} \
\resizebox{0.15\textwidth}{!}{\includegraphics{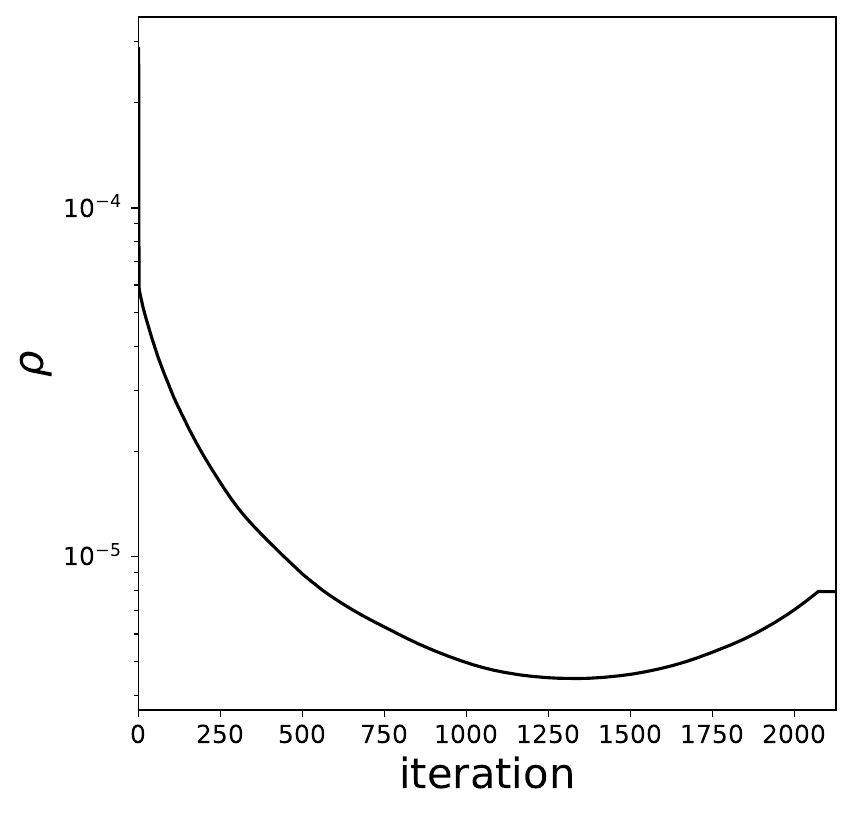}}
\\
\resizebox{0.15\textwidth}{!}{\includegraphics{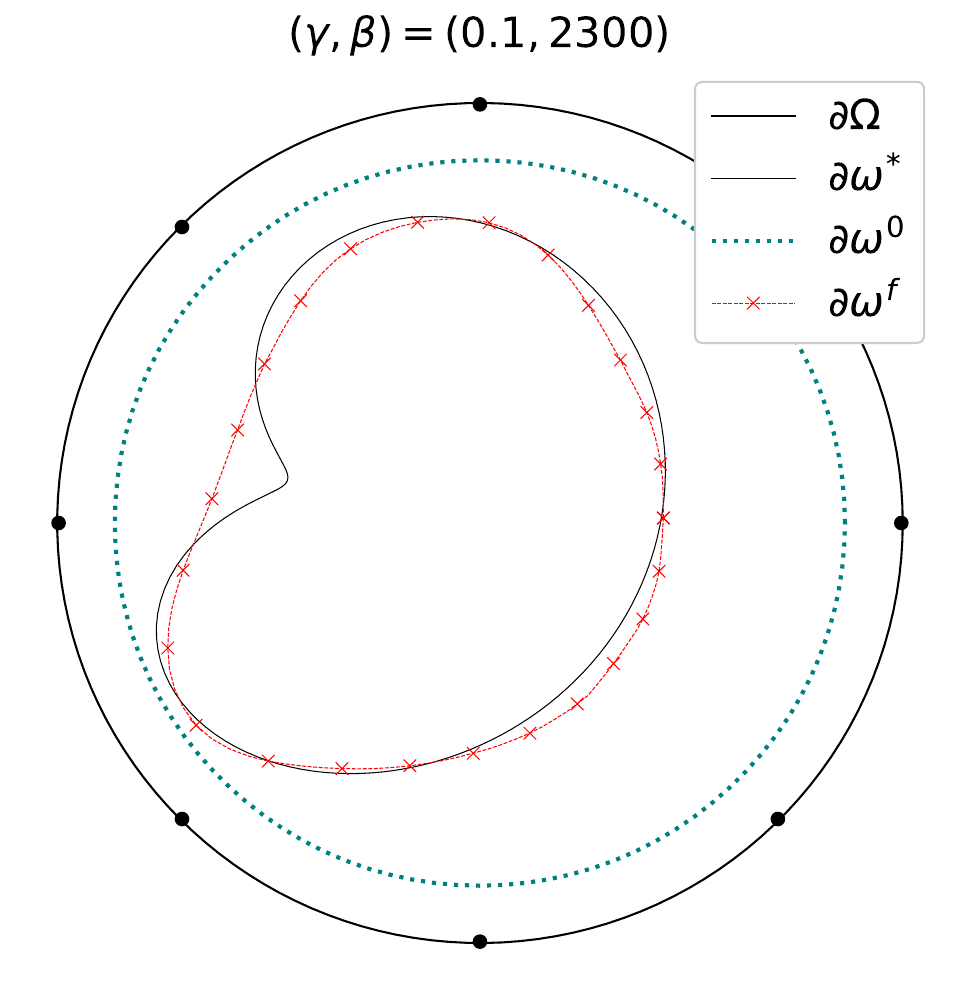}} \
\resizebox{0.15\textwidth}{!}{\includegraphics{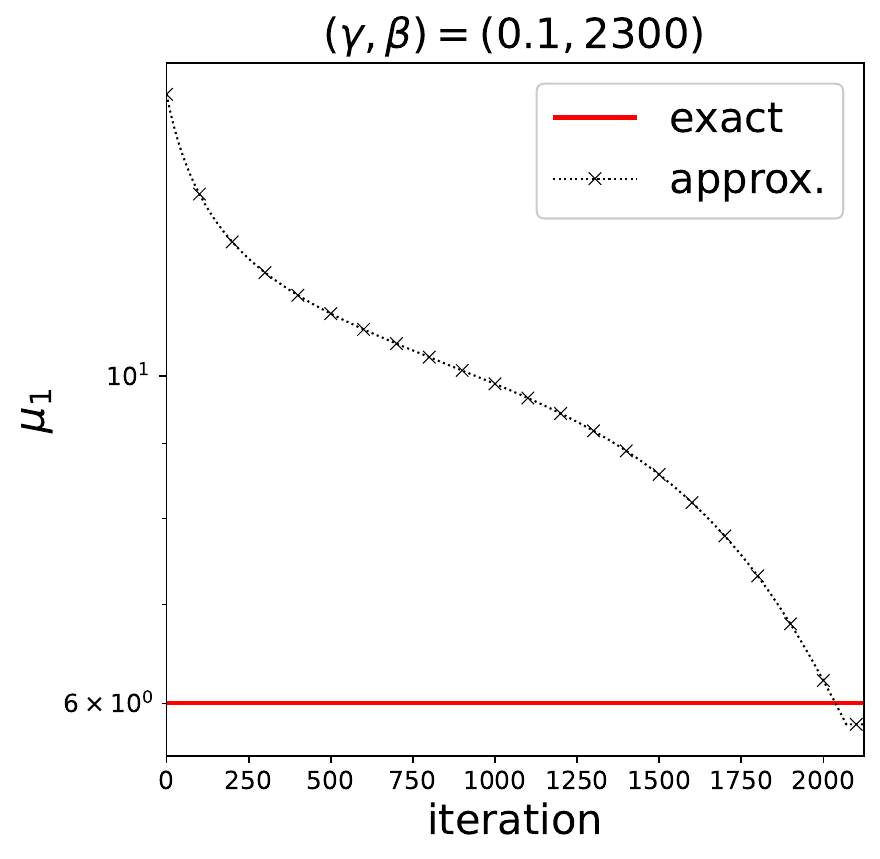}} \
\resizebox{0.15\textwidth}{!}{\includegraphics{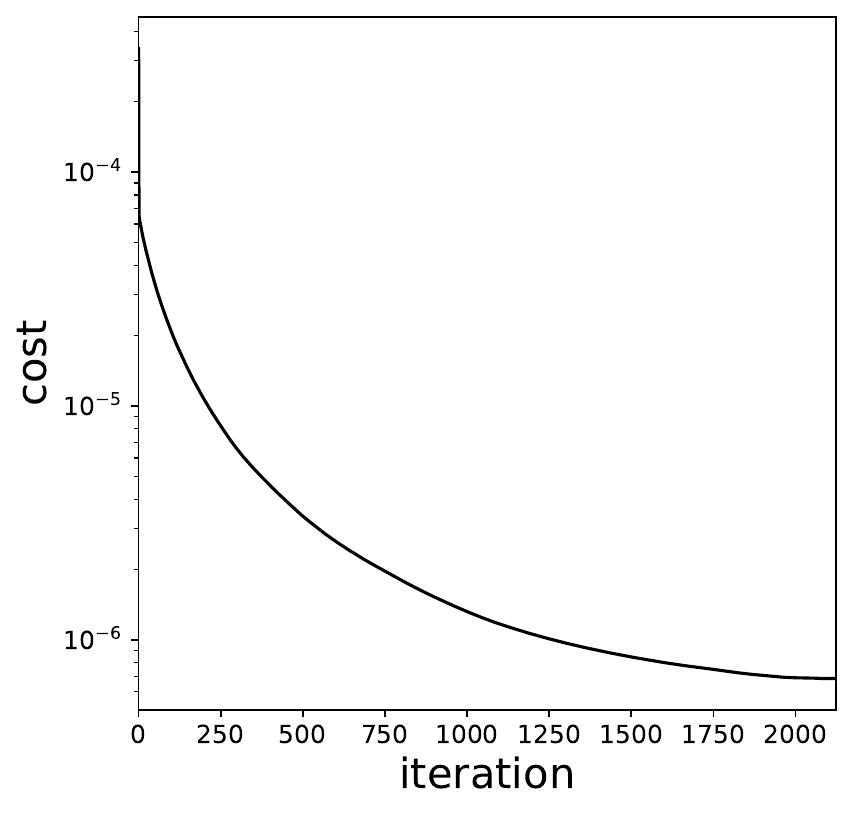}} \
\resizebox{0.15\textwidth}{!}{\includegraphics{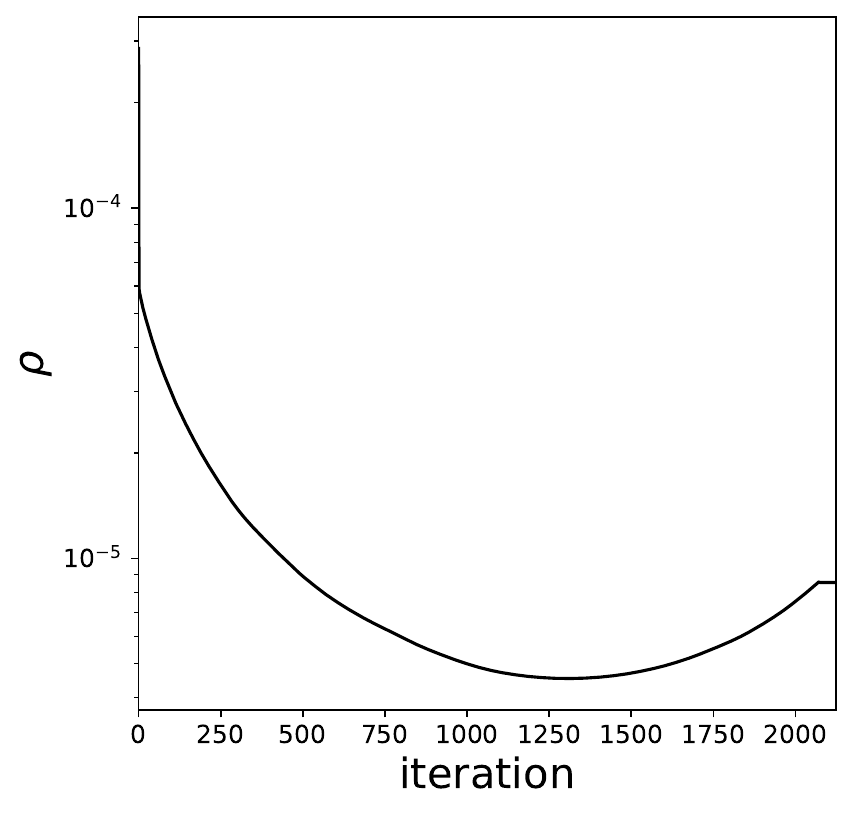}}
\caption{Results for a peanut-shape boundary interface with source $f$ (given by \eqref{eq:multiple_sources_near_boundary} with $\epsilon = 0.3$) near the boundary, under exact (top row) and noisy measurements.
Perimeter penalization was applied in all of the cases with $\rho_{1} = 0.00003$.
}
\label{fig:multiple_sources_peanut_2}
\end{figure}
%
%
%
%
\begin{figure}[htp!]
\centering
\resizebox{0.15\textwidth}{!}{\includegraphics{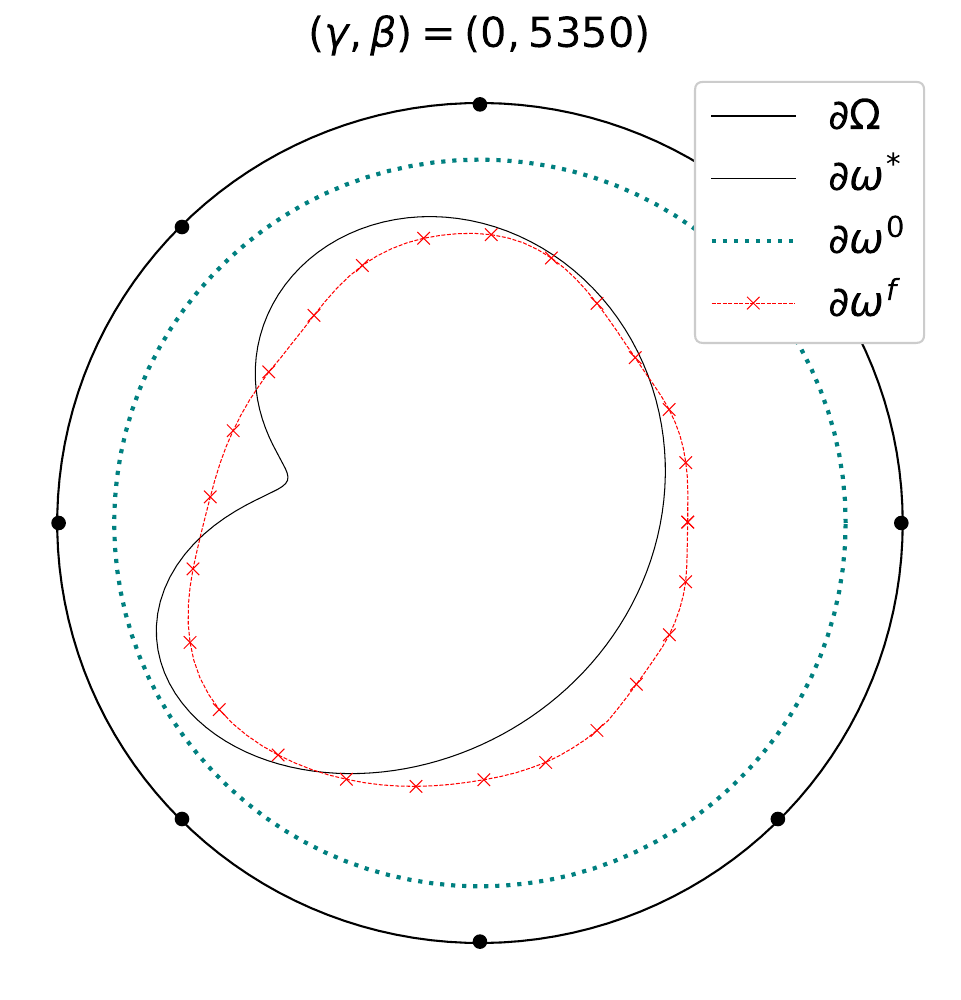}} \
\resizebox{0.15\textwidth}{!}{\includegraphics{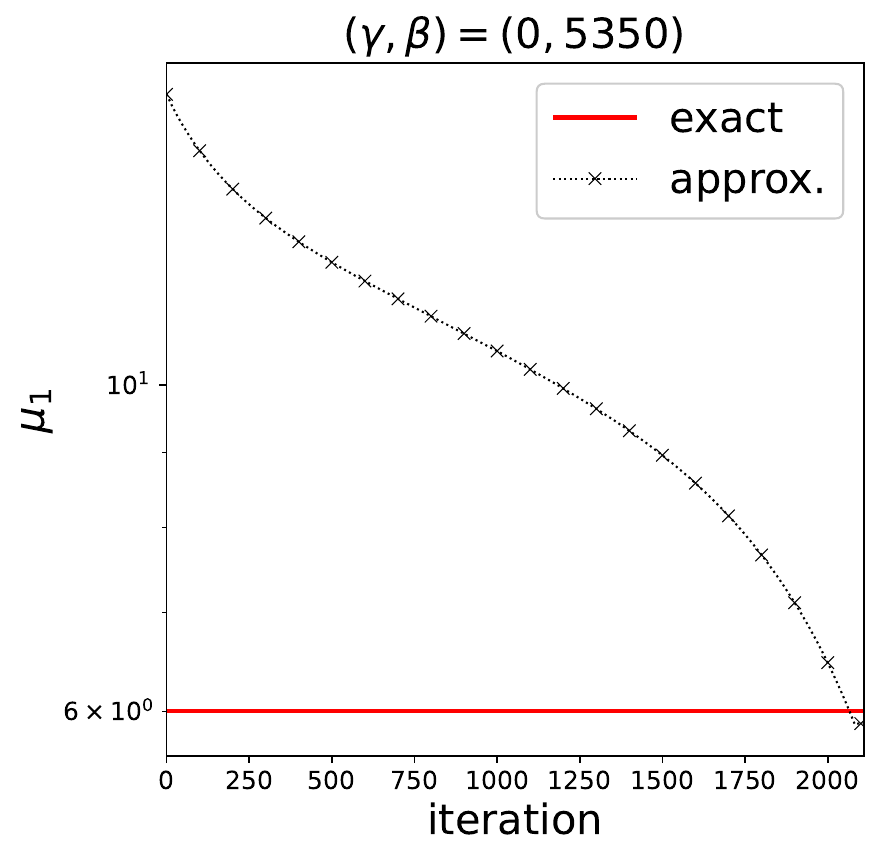}} \
\resizebox{0.15\textwidth}{!}{\includegraphics{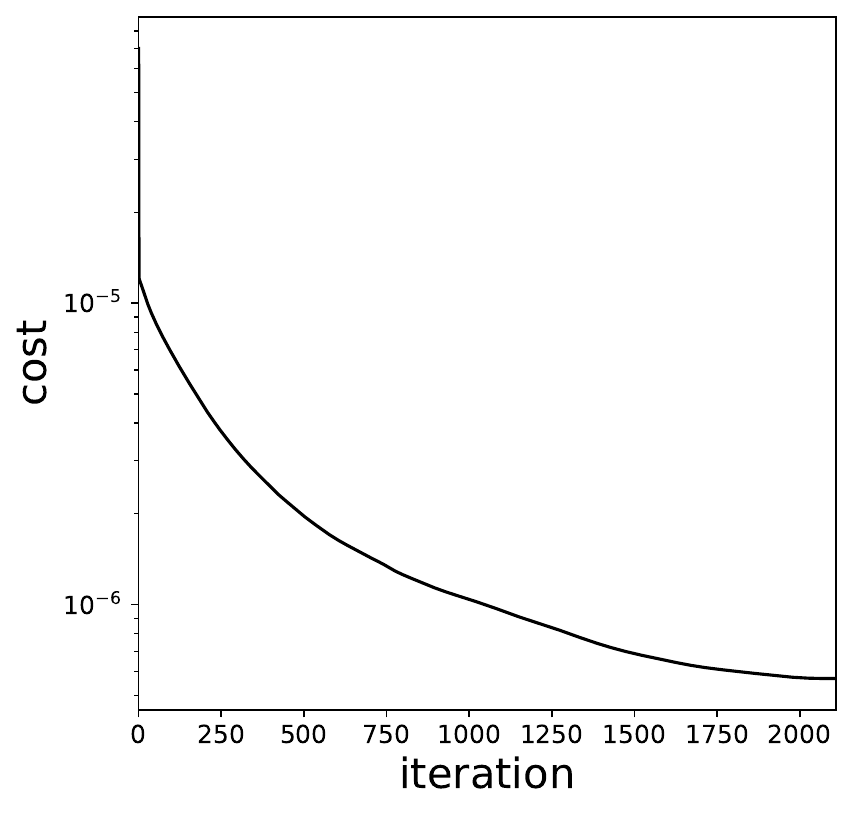}} \
\resizebox{0.15\textwidth}{!}{\includegraphics{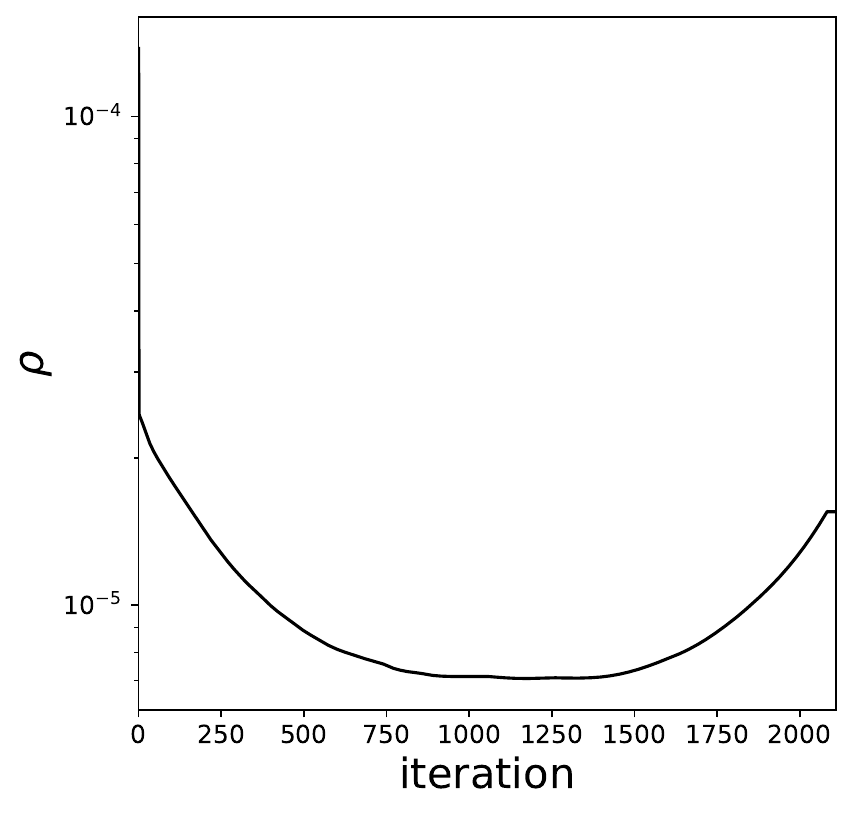}}
\\
\resizebox{0.15\textwidth}{!}{\includegraphics{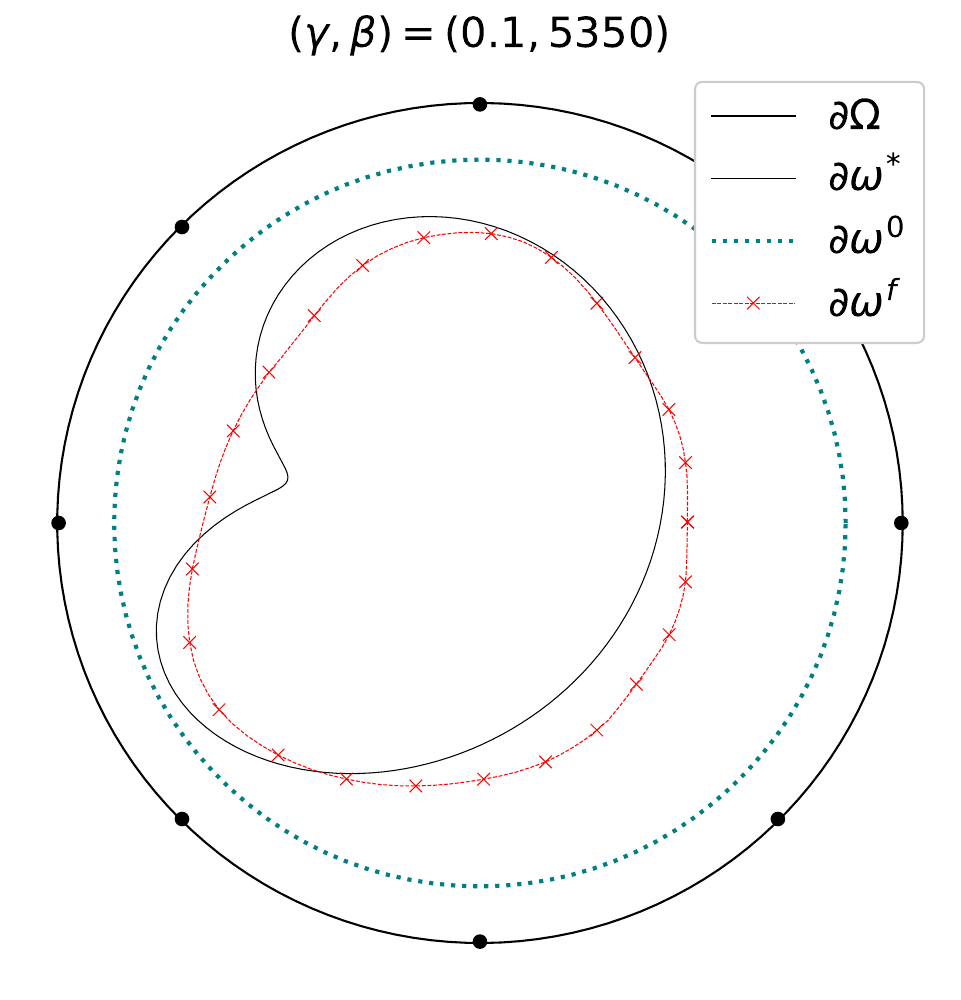}} \
\resizebox{0.15\textwidth}{!}{\includegraphics{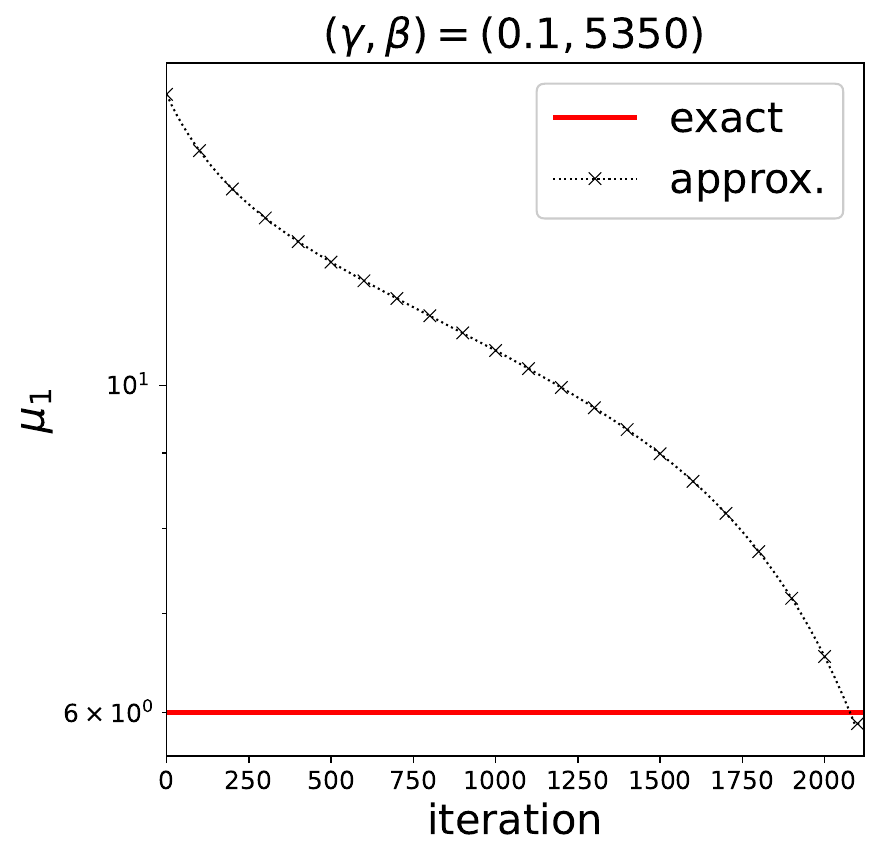}} \
\resizebox{0.15\textwidth}{!}{\includegraphics{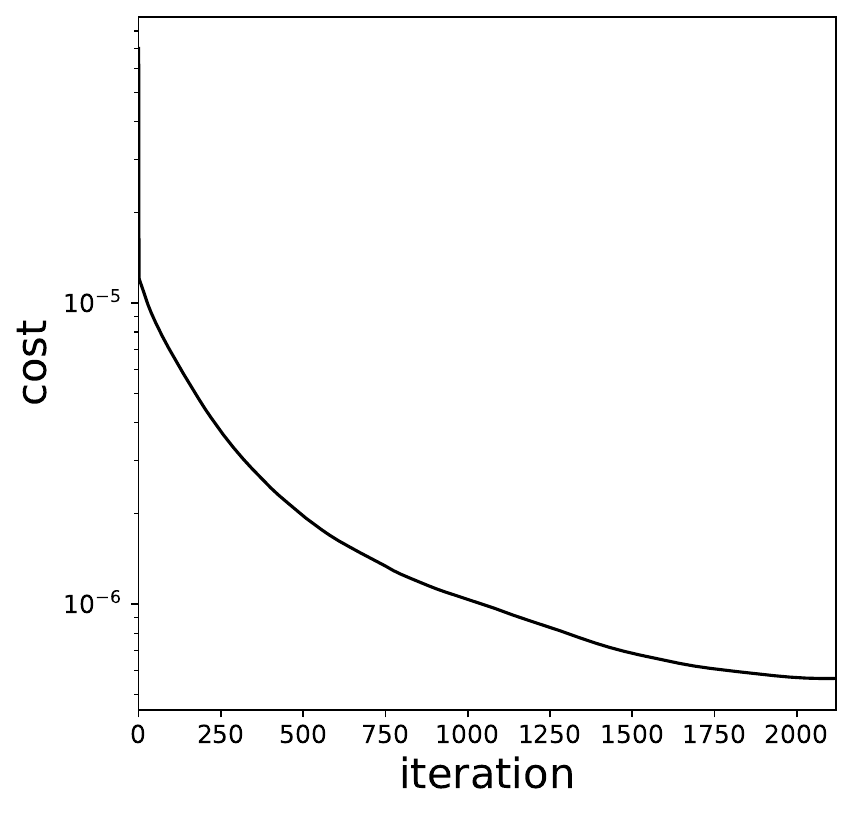}} \
\resizebox{0.15\textwidth}{!}{\includegraphics{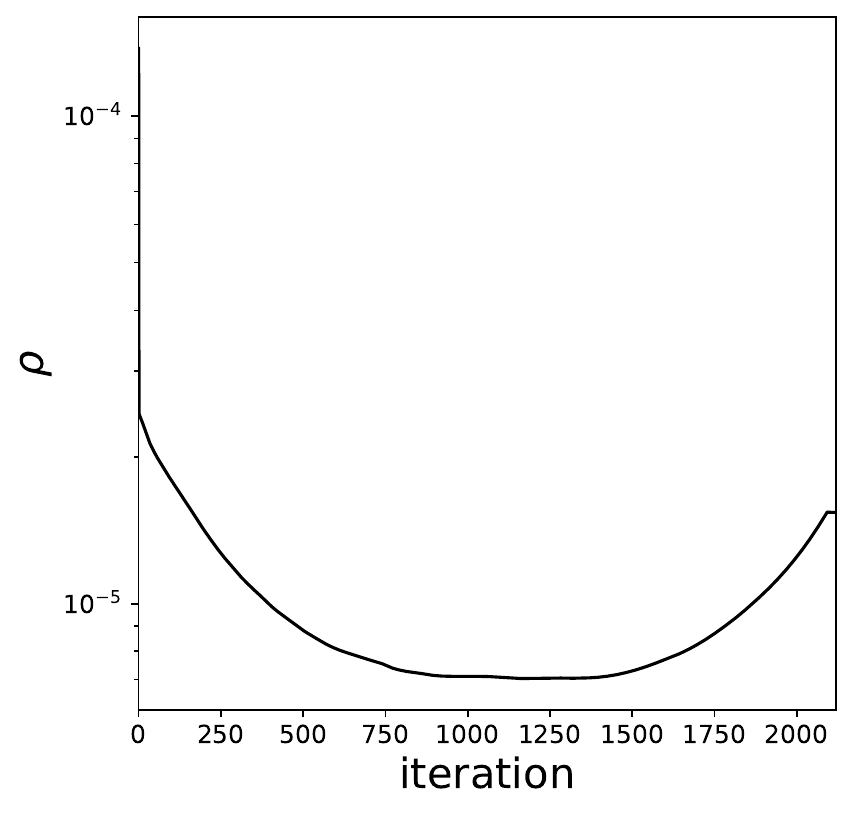}}
\caption{Results for a peanut-shape boundary interface with source $f$ (given by \eqref{eq:multiple_sources_near_boundary} with $\epsilon = 0.2$) near the boundary, under exact (top row) and noisy measurements.
The locations of the point sources are marked by thick black dots.
Perimeter penalization was applied in all of the cases with $\rho_{1} = 0.00003$.
}
\label{fig:multiple_sources_peanut_3}
\end{figure}
%
%
%
%
\begin{figure}[htp!]
\centering
\resizebox{0.25\textwidth}{!}{\includegraphics{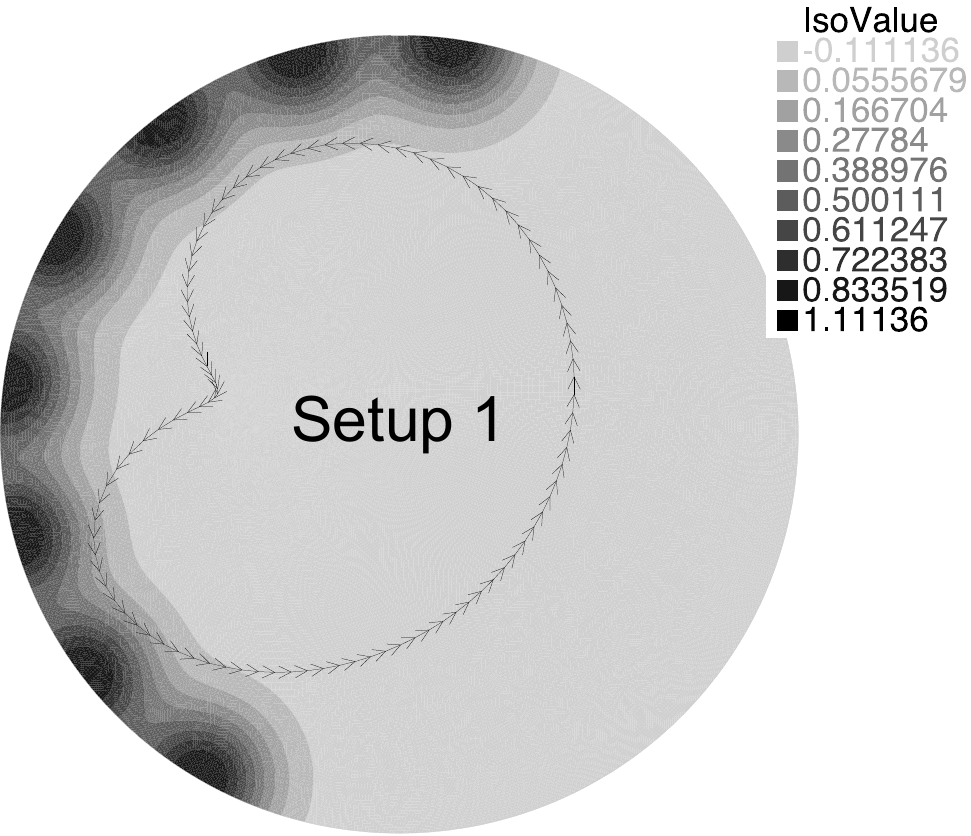}} \quad
\resizebox{0.25\textwidth}{!}{\includegraphics{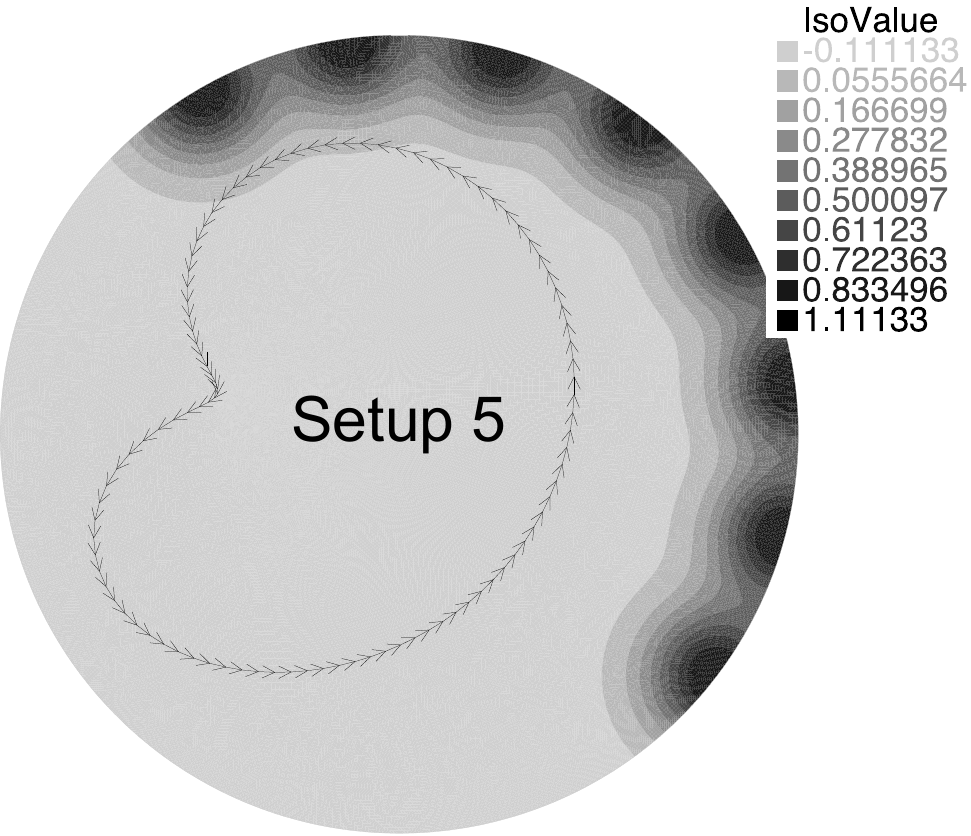}}
\caption{Positioning of sources}
\label{fig:positioning_sources}
\end{figure}
%
%

%
%
\begin{figure}[htp!]
\centering
\hfill
\resizebox{0.15\textwidth}{!}{\includegraphics{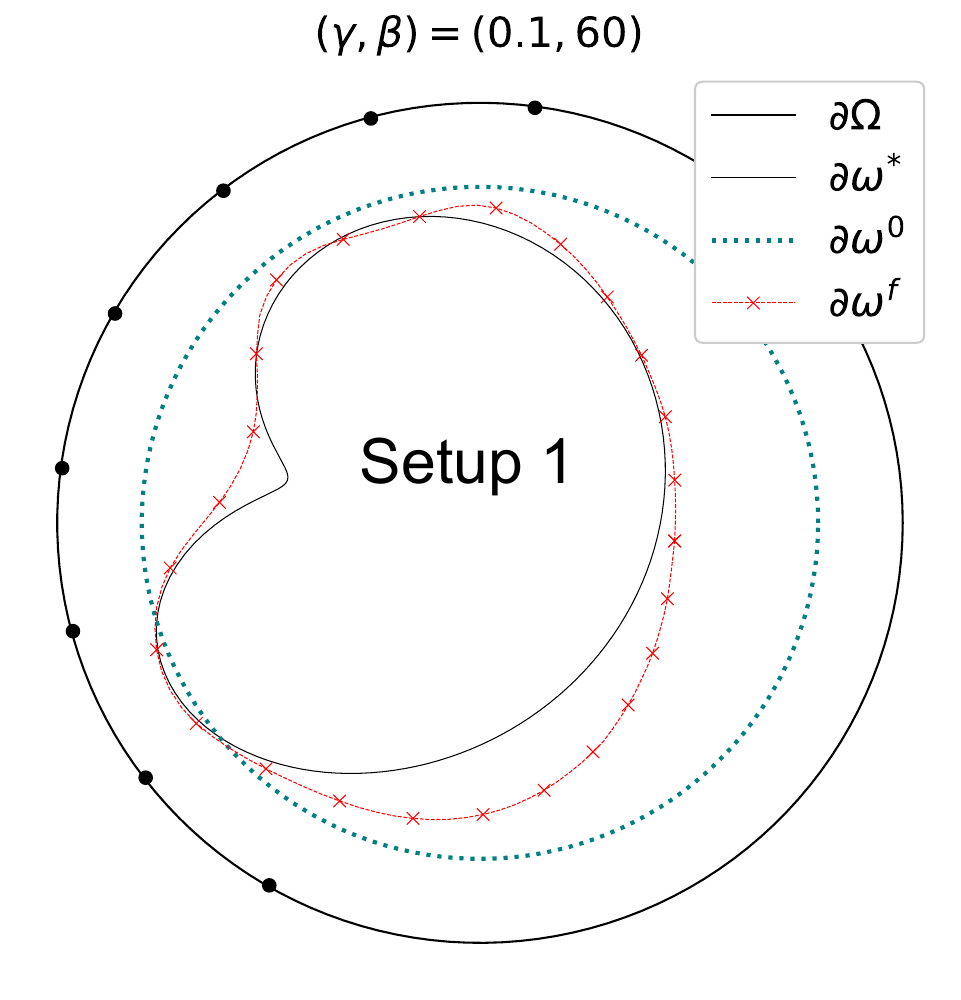}} \hfill
\resizebox{0.15\textwidth}{!}{\includegraphics{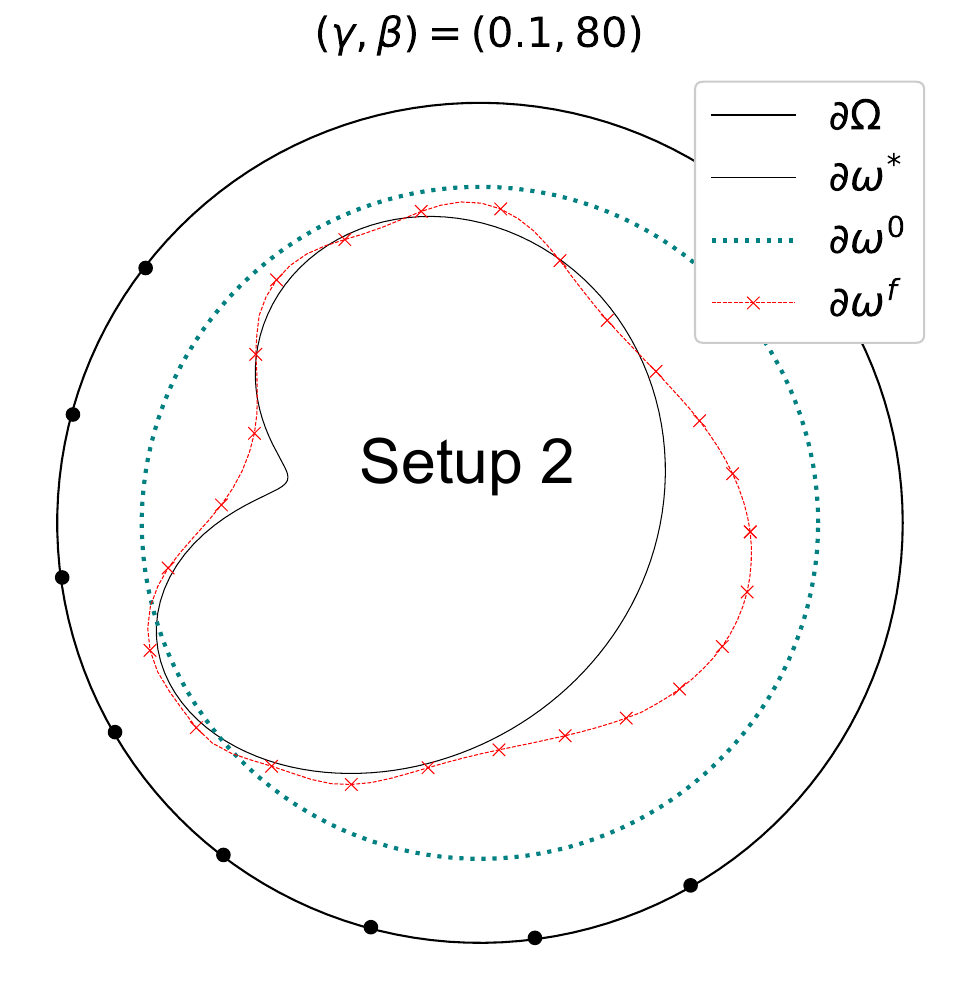}} \hfill
\resizebox{0.15\textwidth}{!}{\includegraphics{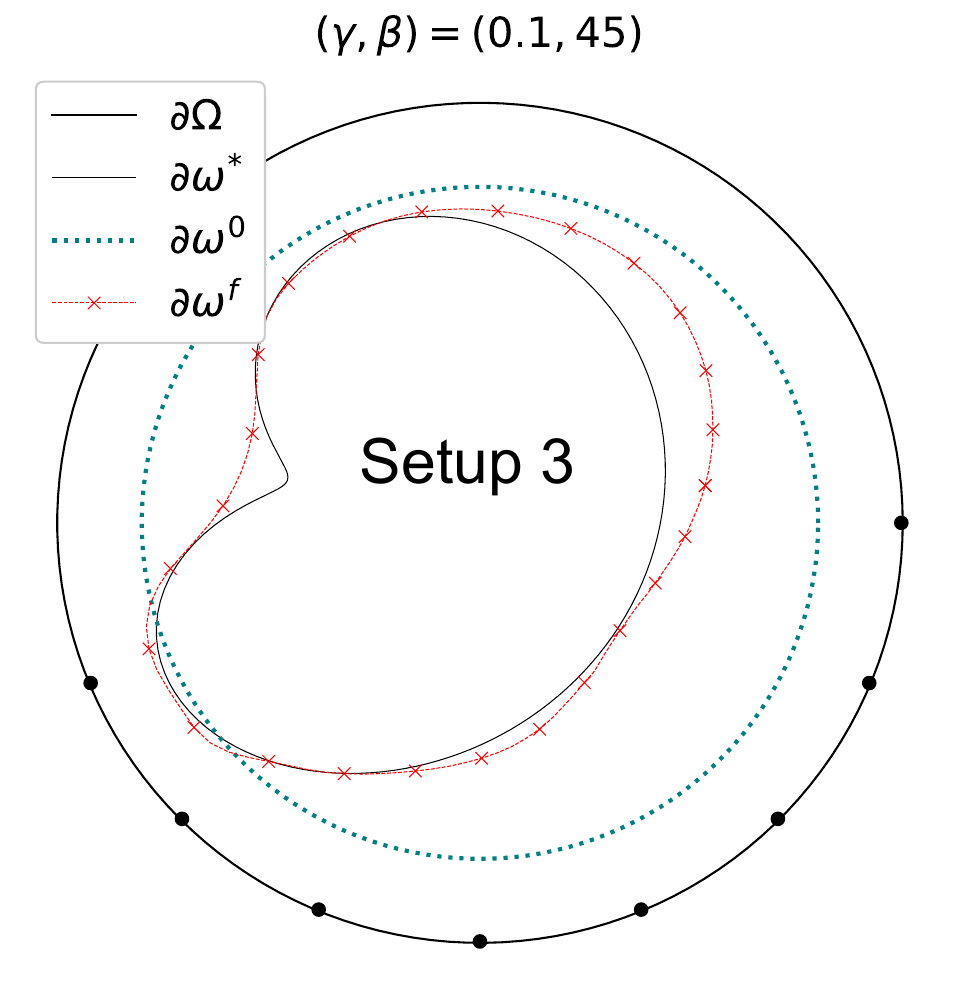}} \hfill
\resizebox{0.15\textwidth}{!}{\includegraphics{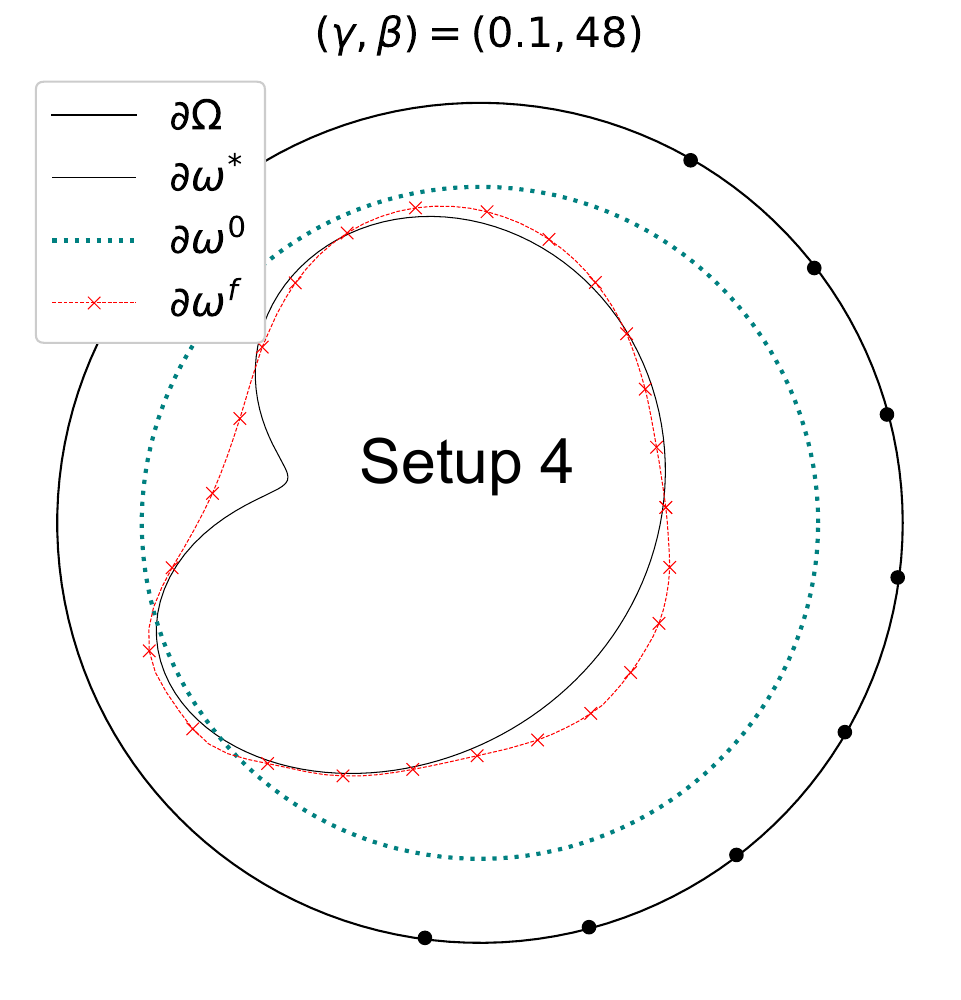}} \hfill
\resizebox{0.15\textwidth}{!}{\includegraphics{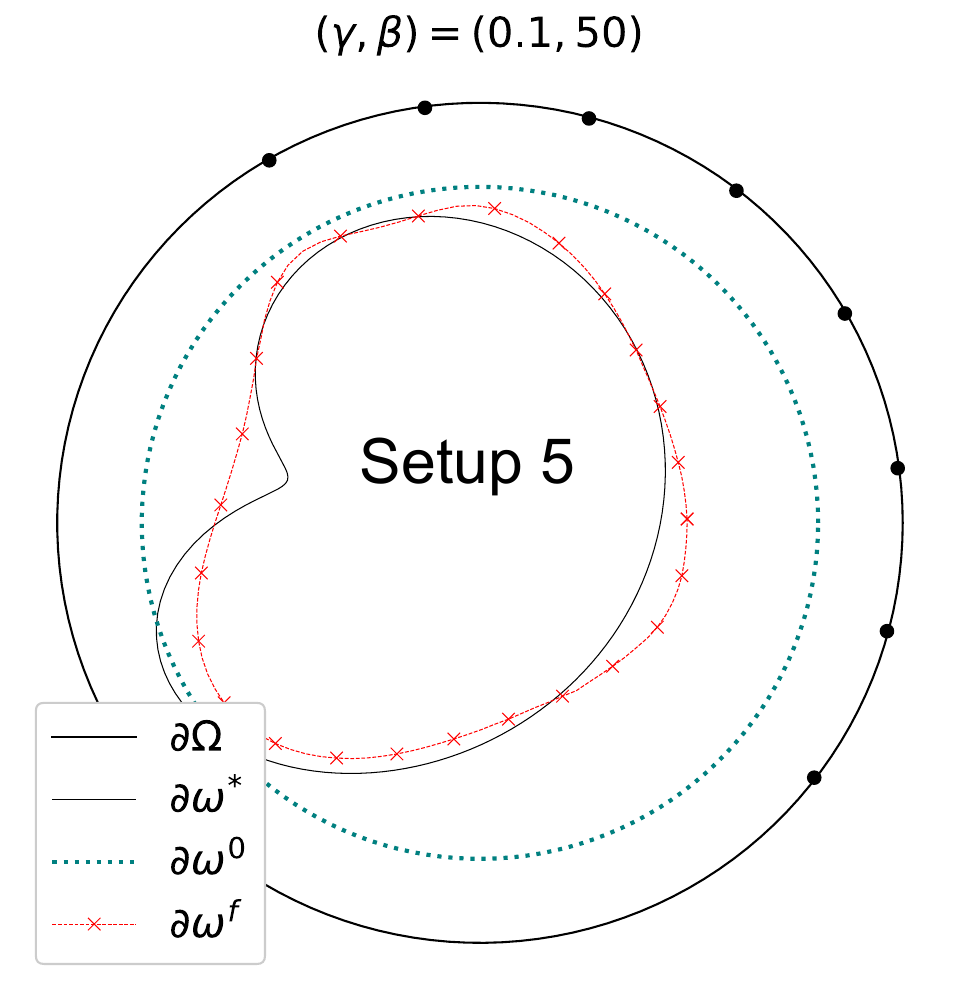}} \hfill
\resizebox{0.15\textwidth}{!}{\includegraphics{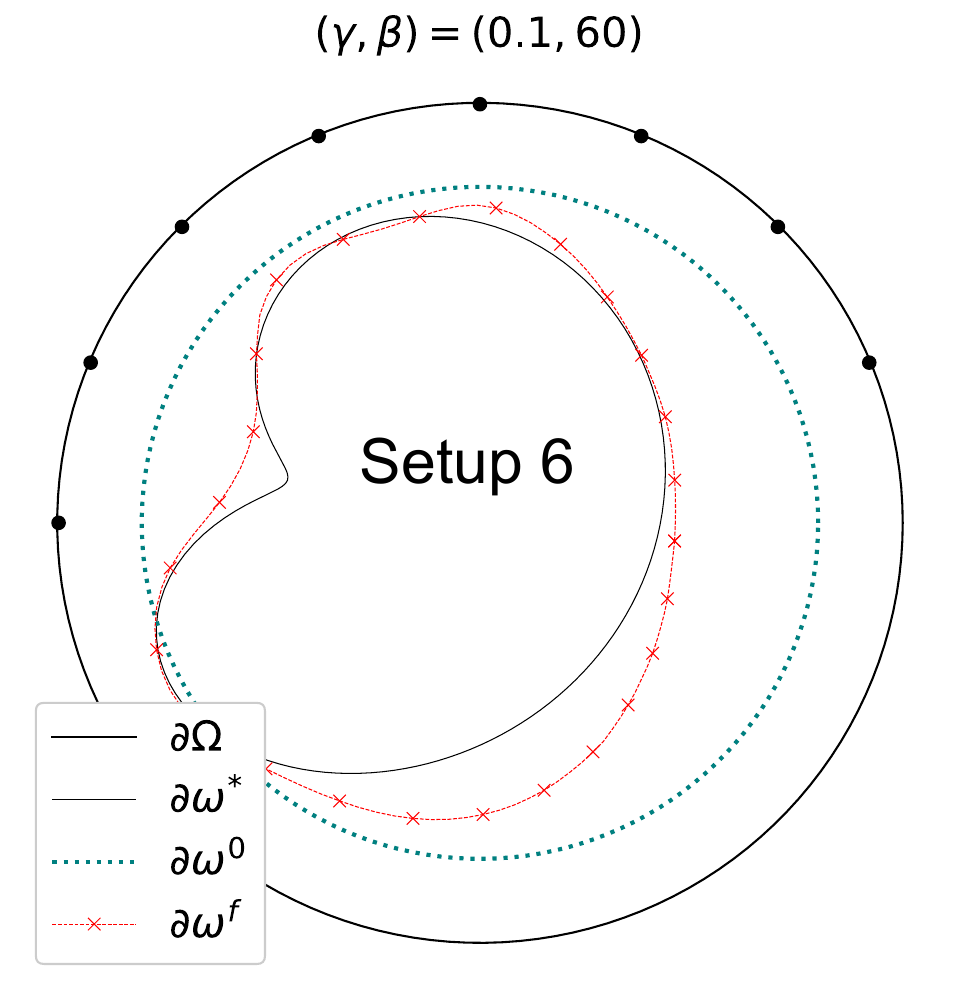}} \hfill\\
\resizebox{0.15\textwidth}{!}{\includegraphics{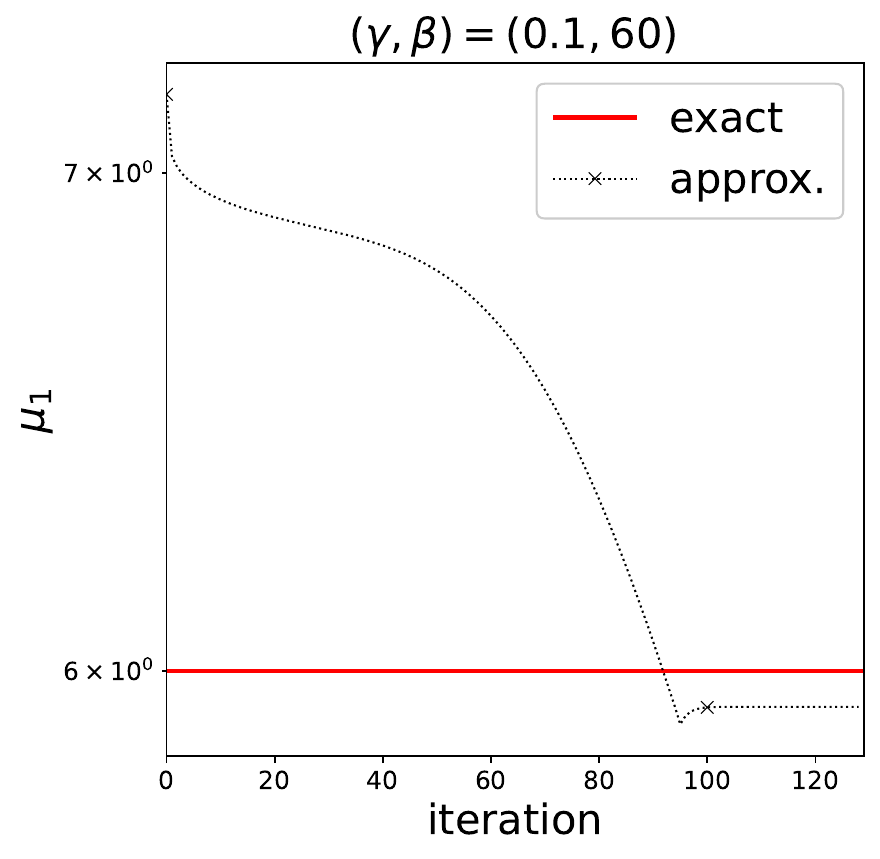}} \hfill
\resizebox{0.15\textwidth}{!}{\includegraphics{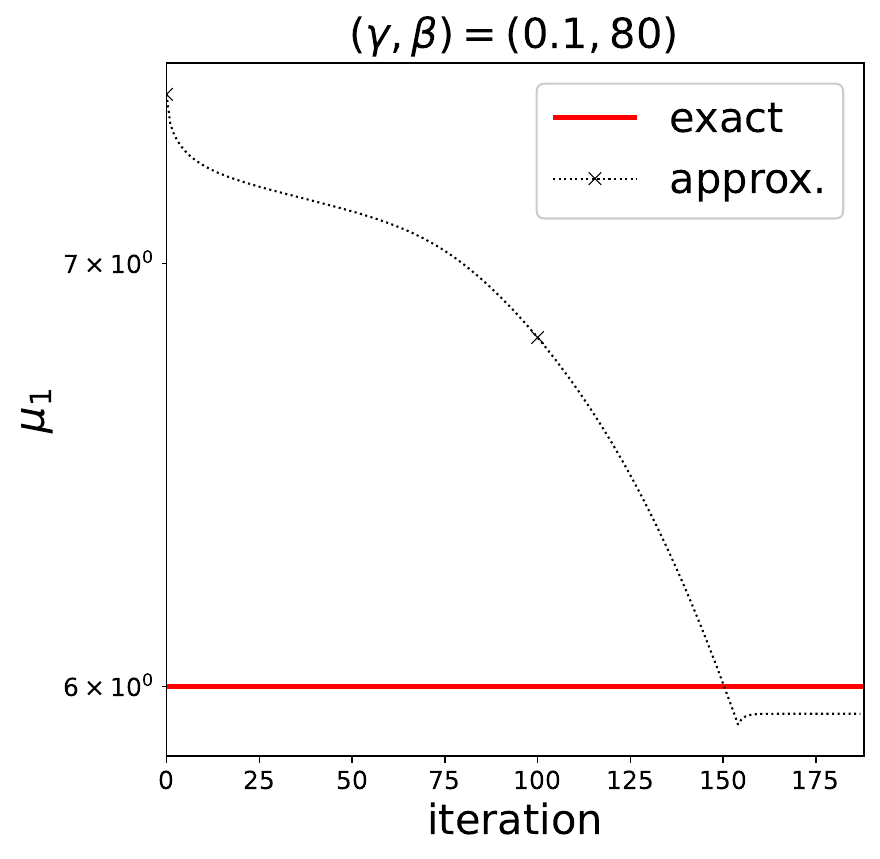}} \hfill
\resizebox{0.15\textwidth}{!}{\includegraphics{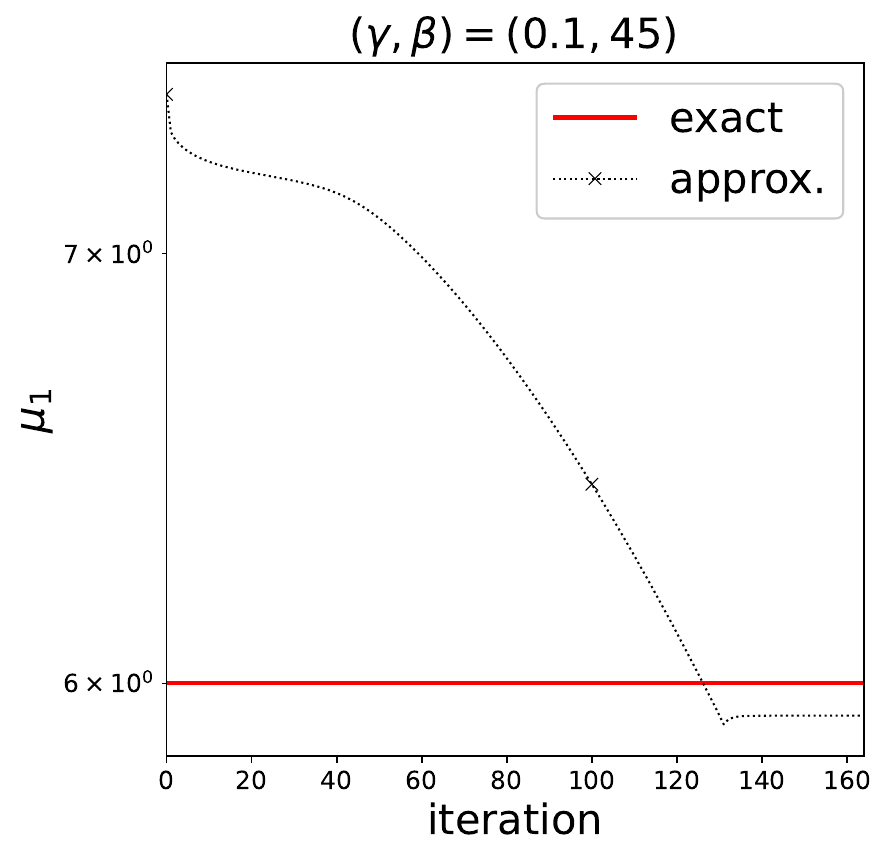}} \hfill
\resizebox{0.15\textwidth}{!}{\includegraphics{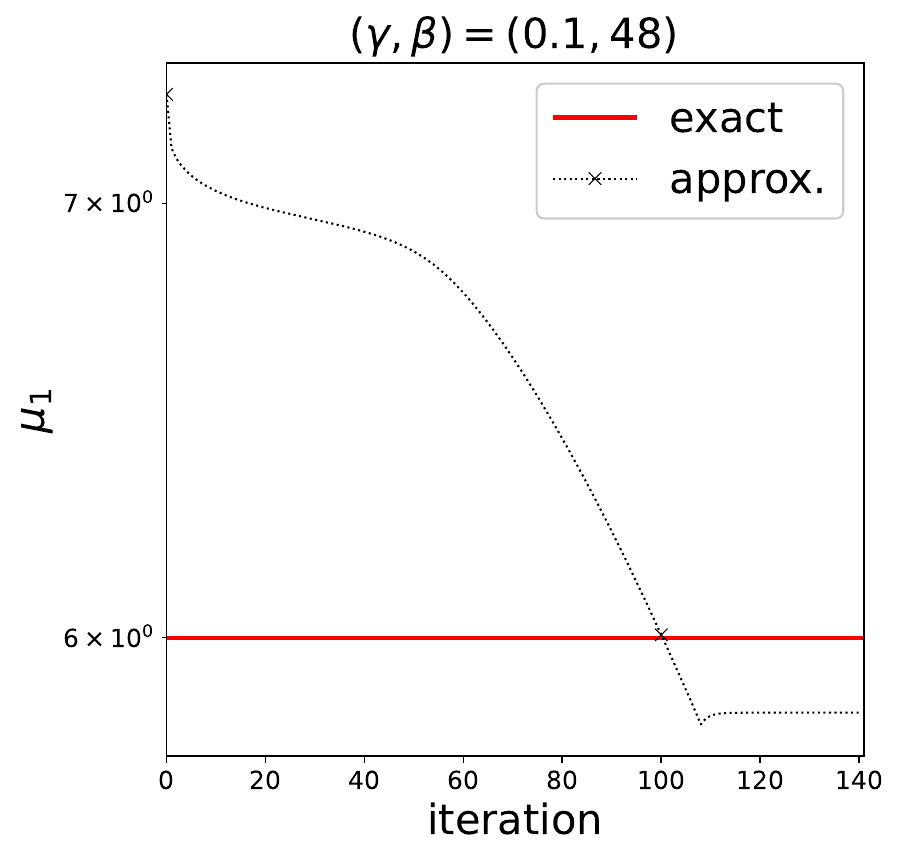}} \hfill
\resizebox{0.15\textwidth}{!}{\includegraphics{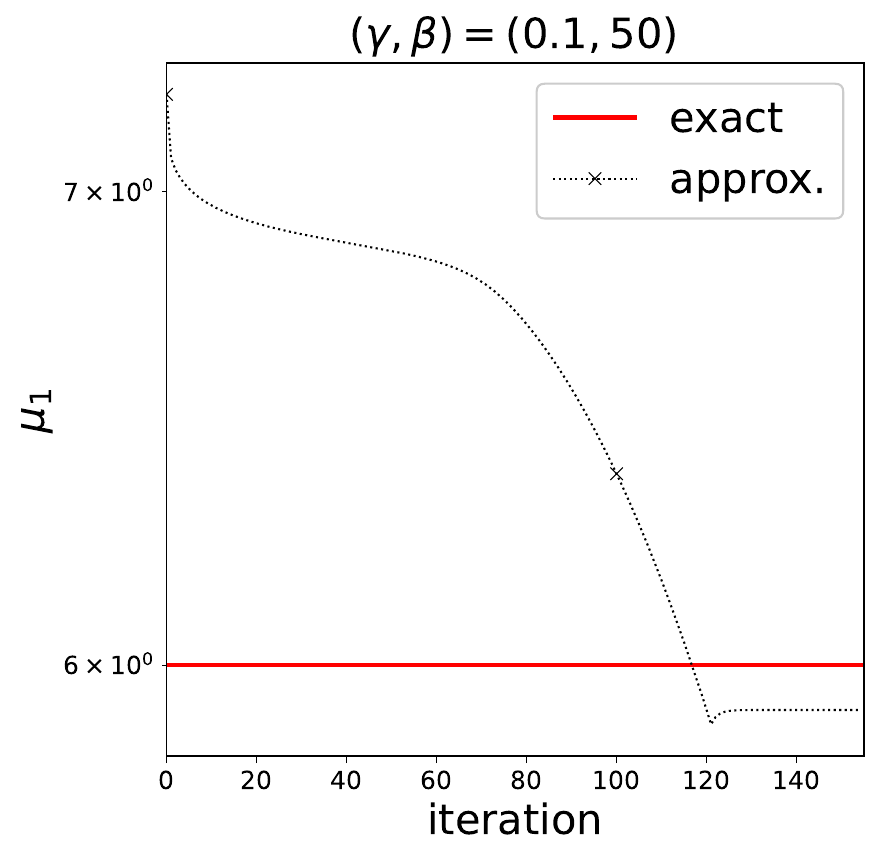}} \hfill
\resizebox{0.15\textwidth}{!}{\includegraphics{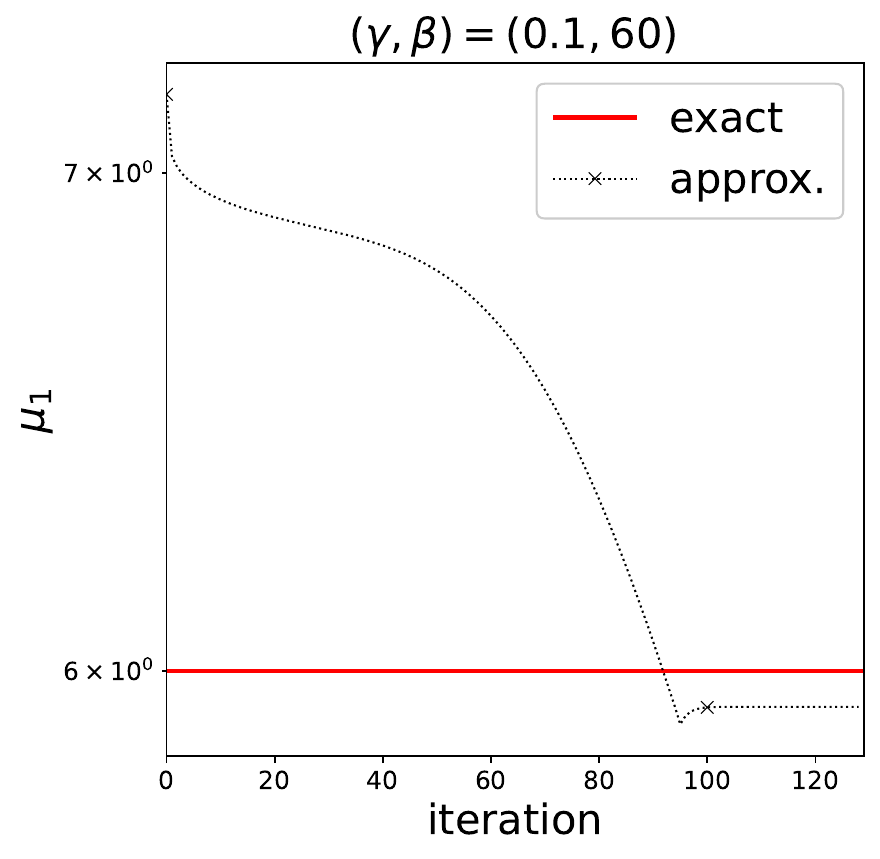}} \hfill
\\
\caption{Results for a peanut-shape boundary interface with $f$ (given by \eqref{eq:multiple_sources_near_boundary} with $\epsilon = 0.5$) of various positions and under noisy measurements.
Perimeter penalization was applied in all of the cases with $\rho_{1} = 0.000006$.
}
\label{fig:positioning_sources_results}
\end{figure}
%

{
\subsection{Discussion on the applicability of the numerical scheme}

The numerical experiments presented above primarily illustrate the performance of the proposed method
in the case of piecewise constant coefficients.
In this setting, the parameter $\mu$ is uniquely determined by the geometry of the evolving subdomains,
so that updating the interface automatically updates $\mu$.
Consequently, the state and adjoint problems remain well defined at each iteration.

For more general spatially varying coefficients, the numerical scheme can be interpreted in the following way.
The computational domain $\Omega$ is kept fixed throughout the iterations and serves as a hold-all domain,
while the interface separating the subdomains evolves.
The coefficient $\mu$ is defined on the whole of $\Omega$, and only its restriction to the subdomains
induced by the current interface is used in the state and adjoint equations.
At the discrete level, although the mesh and nodal locations may change after each interface update,
the coefficient $\mu$ remains available on the updated configuration by standard mesh-to-mesh interpolation.
Under this interpretation, the iterative scheme remains well defined numerically and is not restricted to piecewise constant coefficients.

\begin{remark}
Without the fixed-domain interpretation described above,
the iterative scheme is naturally restricted to piecewise constant coefficients,
since otherwise the coefficient $\mu$ would not be defined on the updated configuration.
\end{remark}

\begin{remark}
We have primarily used boundary-type shape gradients, which provide satisfactory reconstructions.
Accuracy could be further improved using distributed shape gradients \cite[Rem.~4.1]{Rabago2025}, which may offer better numerical accuracy in finite element implementations, particularly in challenging cases such as concave regions.
It is worth noting, however, that boundary-type shape gradients can attain comparable numerical accuracy when appropriate boundary corrections are incorporated \cite{GongZhu2021}.
Nonetheless, the current results already demonstrate the robustness of the boundary-based approach.
\end{remark}

\begin{remark}\label{rem:remeshing_in_3D}
Remeshing was not employed in our experiments, as interior subregions were sufficiently large or the interface remained near the accessible boundary.
However, in three-dimensional problems or when the initial interface is far from the true geometry, mesh quality may deteriorate, necessitating remeshing \cite{Rabago2025}.
Thus, the algorithm works without remeshing in the present cases, but remeshing should be considered in more challenging configurations, especially in 3D.
\end{remark}
}

%
%
\section{Conclusion}\label{sec:conclusion}
This study introduces a shape-optimization-based approach to tackle the complex, ill-posed problem of space-dependent parameter reconstruction in inverse {diffusion problems}.
By {reconstructing the constant} $\mu$ and the boundary interface with only one boundary measurement, {we demonstrated} the versatility and robustness {of this method}, particularly in scenarios involving non-smooth, non-convex boundaries.
Despite the difficulties in precisely capturing boundary vertices and edges, the method reliably reconstructs $\mu$ and accurately identifies concave features of the boundary interface, even under noisy conditions.
The influence of point source placement on reconstruction accuracy, especially in concave regions, highlights an expected spatial sensitivity.
Overall, the results confirm that the proposed approach is practical and effective for complex boundary interface reconstructions, emphasizing its applicability to various related reconstruction problems that focus on parameter identification with jump discontinuities.
{A key insight of the method is its non-trivial nature, as achieving an accurate reconstruction depends on a carefully chosen parameter $\beta$.}

In follow-up work, we will focus on stability analysis when the primary quantity of interest is the jump in the absorption coefficient.
Specifically, we will derive a local stability estimate for a parameterized, non-monotone family of domains and provide a quantitative stability result for the local optimal solution under perturbations of the absorption coefficient parameter.
This investigation will extend the current formulation by incorporating both the first-order and second-order {Eulerian derivative}s of the cost functional, offering a deeper understanding of the method’s robustness.
Additionally, exploring other {objective functional}s, including the well-known Kohn-Vogelius cost functional \cite{KohnVogelius1984,Meftahi2021}, will be the focus of future investigations.

{Harrach showed that two parameters can be uniquely determined using the time-independent diffusion equation \eqref{eq:main} if the diffusion coefficient is piecewise constant and the absorption coefficient is piecewise analytic \cite{Harrach2009}.
In this sense, the approach developed in this paper can be applied to the simultaneous reconstruction of two parameters.}

{In addition to DOT, the present method can potentially be applied to other optical imaging techniques, such as fluorescence DOT (FDOT) (see \cite[p.~165]{Jiang2011}, \cite{Milsteinetal2003,MycekPogue2003}, and \cite{Durduranetal2010}), ultrasound-modulated fluorescence techniques \cite{Liu2014}, and fluorescence molecular tomography (FMT) \cite{Nitziachristosetal2002}.}

\appendix
\counterwithin{theorem}{subsection}
\section{Appendices}\label{appx}
\renewcommand{\thesubsection}{\Alph{subsection}}
\renewcommand{\theequation}{\Alph{subsection}.\arabic{equation}}

\subsection{Proofs of some auxiliary results}\label{appx:proofs}
%
%
%
\subsubsection{Well-posedness of the state}
\begin{proof}[Proof of Lemma~\ref{lem:wellposedness_weak_formulation}]
	The proof follows from Lax-Milgram lemma.
	Indeed, it can be shown that the following inequalities hold:
	\begin{align}
		\abs{a(u,v)} &\leqslant \max\{{\alpha},\mu_{\max},\zeta^{-1}\} \norm{u}_{V} \norm{|v}_{V}, \quad (u, v \in V);\nonumber\\
		a(u,u) &\geqslant \min\{{\alpha}, \mu_{\min}\} \norm{u}_{V}^{2}, \quad (u \in V);\label{eq:coercivity}\\
		l(v) &\leqslant \cf \norm{v}_{V} , \quad \cf\coloneqq \norm{f}_{H^{-1}(\Omega)}, \quad (v \in V).\nonumber
	\end{align}
	The last two inequalities imply that
	\begin{equation}\label{eq:boundedness_of_u}
		\norm{u}_{V} \leqslant \cb \cf, \quad \cb\coloneqq  \frac{1}{\min\{{\alpha}, \mu_{\min}\}} > 0.
	\end{equation}
	The rest of the arguments are standard, so we omit it.
\end{proof}
\subsubsection{Continuity of the parameter-to-state map}
\begin{proof}[Proof of Proposition~\ref{prop:boundedness_of_F}]
	Let us write
	\begin{equation}\label{eq:bilinear_form}
	{a(\mu; u,v)} = \intO{( {\alpha} \nabla{u} \cdot \nabla{v} + {\mu} {u}{v})}
				+ \frac{1}{\zeta} \intG{{u}{v}}, \quad \mu\in {\mathcal{A}}, u,v \in {V}.
	\end{equation}
	Now, since $\tilde{u} = F(\tilde{\mu})$ and $\dbtilde{u} = F(\dbtilde{\mu})$,
	then, clearly, we have ${a(\tilde{\mu}; \tilde{u},v)} = l(v) = {a(\dbtilde{\mu}; \dbtilde{u},v)}$, for all $v \in V$.
	We let $w = w(x) = \tilde{u}(x) - \dbtilde{u}(x) \in V$.
	It can easily be verified that
	\[
        \left\{
        \begin{aligned}
          -\dive{\left( {\alpha} \nabla{w}\right)} + \tilde{\mu}{w} &= - (\tilde{\mu} - \dbtilde{\mu}) \dbtilde{u}  , \quad \text{in } \Omega,\\
          {\alpha} \dn{w} + \frac{1}{\zeta} w &= 0, \quad \text{on }  \bigdO.
        \end{aligned}
        \right.
	\]
	Using $a$ in \eqref{eq:bilinear_form}, we have the variational equation ${a(\tilde{\mu}; w,v)} = -( (\tilde{\mu} - \dbtilde{\mu}) \dbtilde{u} , v)_{\Omega}$, for all $v \in V$.
	We set $v = w$ and apply \eqref{eq:coercivity} and the Cauchy-Schwarz inequality to get
	\[
		\min\{{\alpha}, \mu_{\min}\} \norm{w}_{V}^{2}
		\leqslant \norm{\tilde{\mu} - \dbtilde{\mu}}_{L^{\infty}(\Omega)} \norm{\dbtilde{u}}_{L^{2}(\Omega)} \norm{w}_{L^{2}(\Omega)}
		\leqslant \norm{\tilde{\mu} - \dbtilde{\mu}}_{L^{\infty}(\Omega)} \norm{\dbtilde{u}}_{V} \norm{w}_{V}.
	\]
	Employing estimate \eqref{eq:boundedness_of_u}, we obtain
	\begin{equation}\label{eq:boundedness_of_w}
		\norm{w}_{V} \leqslant \cb^{2} \cf \norm{\tilde{\mu} - \dbtilde{\mu}}_{L^{\infty}(\Omega)}.
	\end{equation}
	Taking $c = \cb^{2} \cf \norm{\tilde{\mu} - \dbtilde{\mu}}_{L^{\infty}(\Omega)}$ concludes the proof.
\end{proof}
%
%
%
\subsubsection{Differentiability of the operator $F$}
\begin{proof}[Proof of Proposition~\ref{prop:state_derivative}]
	Let $u \in V$ be the unique solution to Problem~\ref{prob:weak_form_optimal_tomography}.
	Then, the existence of a unique weak solution to the variational equation \eqref{eq:sensitivity_equation_variational_form} can be verified easily using the Lax-Milgram lemma.
	We omit the proof since the argumentations are standard.
	
	We subtract \eqref{eq:sensitivity_equation_variational_form} from \eqref{eq:weak_form_del_u} to obtain ${a(\mu; {\wdot} - \udot,v)} = - (\nu {\wdot}, v)_{\Omega}$, for all $v \in V$.
	Taking $v= {\wdot} - \udot$, we obtain -- appealing to \eqref{eq:coercivity} and employing Cauchy-Schwarz inequality -- the estimate
	$
	\min\{{\alpha}, \mu_{\min}\} \norm{{\wdot} - \udot}_{V}^{2}
		= - (\nu {\wdot}, {\wdot} - \udot)_{\Omega}
		\leqslant \norm{\nu}_{L^{\infty}(\Omega)}\norm{{\wdot}}_{V} \norm{{\wdot} - \udot}_{V}.
	$
	{\small Consequently, we get $\norm{{\wdot} - \udot}_{V} \leqslant \cb\norm{\nu}_{L^{\infty}(\Omega)}\norm{{\wdot}}_{V}$.}
	In view of \eqref{eq:boundedness_of_w} and with respect to \eqref{eq:strong_form_of_delta_u}, we find that $\norm{{\wdot}}_{V} \leqslant \cb^{2} \cf \norm{\nu}_{L^{\infty}(\Omega)}$.
	Combining the last two estimates lead us to $\norm{{\wdot} - \udot}_{V} \leqslant \cb^{3} \cf \norm{\nu}_{L^{\infty}(\Omega)}^{2}$.
	This yields
	\begin{equation}\label{eq:Frechet_derivative_estimate}
		\dfrac{\norm{F(\mu + \nu)-F(\mu) - \udot}_{V}}{\norm{\nu}_{L^{\infty}(\Omega)}}
		\leqslant \dfrac{\norm{{\wdot} - \udot}_{V}}{\norm{\nu}_{L^{\infty}(\Omega)}}
		\leqslant  \cb^{3} \cf \norm{\nu}_{L^{\infty}(\Omega)}.
	\end{equation}
	In conclusion, $F$ is differentiable at $\mu$ and $DF(\mu)\nu = \udot$.
	
	Now, let us take $v = \udot \in {V}$ in \eqref{eq:sensitivity_equation_variational_form}.
	Then, again, in view of \eqref{eq:boundedness_of_w} and by the Cauchy-Schwarz inequality, we get
	\[
	\min\{{\alpha}, \mu_{\min}\} \norm{\udot}_{V}^{2}	
		\leqslant \abs{{a(\mu;\udot,\udot)}}
		= \abs{(\nu u, \udot)_{\Omega}}
		\leqslant \norm{\nu}_{L^{\infty}(\Omega)} \norm{u}_{V} \norm{\udot}_{V}
		\leqslant \cb \cf \norm{\nu}_{L^{\infty}(\Omega)} \norm{\udot}_{V}.
	\]
        This leads to the inequality
        \begin{equation}\label{eq:estimate_state_derivative}
            \norm{\udot}_{V} \leqslant c, \qquad c=\cb^{2} \cf \norm{\nu}_{{\infty}}.
        \end{equation}
        Thus, it follows that $DF(\mu)$ is uniformly bounded, thereby concluding the proof.
\end{proof}
%
%
%
\subsubsection{Second-order sensitivity analysis}
\begin{proof}[Proof of Proposition~\ref{eq:state_second_derivative}]
	The well-posedness of \eqref{eq:sensitivity_equation_variational_form} in Proposition~\ref{prop:state_derivative} implies the existence of unique solution $ \uddot \in V$ to \eqref{eq:second_order_sensitivity_equation_variational_form} by Lax-Milgram lemma.
	The proof is standard so we omit it.
	
	Now, we subtract \eqref{eq:second_order_sensitivity_equation_variational_form} from \eqref{eq:weak_form_del_del_u} to obtain ${a(\mu + \nu_{1};{\wddot},v) -  a(\mu;\uddot,v)}
		= -(\nu_{2} [ F(\mu+\nu_{1}) - F(\mu) - DF(\mu)\nu_{1} ]  , v)_{\Omega}$, for all $v \in V$,
	or equivalently, after rearrangement,
	\[
		{a(\mu; {\wddot} - \uddot, v)}
		= -(\nu_{2} [ F(\mu+\nu_{1}) - F(\mu) - DF(\mu)\nu_{1} ]  , v)_{\Omega} - (\nu_{1} {\wddot}, v)_{\Omega} , \quad \forall v \in V.
	\]

	By taking $ v = z \coloneqq  {\wddot} - \uddot $ and then employing the Cauchy-Schwarz inequality as well as the coercivity of the bilinear form $a$ (cf. \eqref{eq:coercivity}), we obtain
	\begin{equation}\label{eq:first_estimate}	
	\begin{aligned}
		\frac{1}{\cb} \norm{z}_{V}^{2}
		&\leqslant \norm{\nu_{2}}_{L^{\infty}(\Omega)} \norm{F(\mu+\nu_{1}) - F(\mu) - DF(\mu)\nu_{1}}_{V} \norm{z}_{V}\\
		&\quad + \norm{\nu_{1}}_{L^{\infty}(\Omega)} \norm{\wddot}_{V} \norm{z}_{V}.
	\end{aligned}
	\end{equation}
	Utilizing estimates \eqref{eq:Frechet_derivative_estimate} and \eqref{eq:estimate_state_derivative} -- employing similar argumentations while noting \eqref{eq:boundedness_of_w} -- we have
	\[
		\norm{F(\mu+\nu_{1}) - F(\mu) - DF(\mu)\nu_{1}}_{V}
			 \leqslant \cb^{3} \cf \norm{\nu_{1}}_{L^{\infty}(\Omega)}^{2}
	\]
	and
	\[
		\norm{\wddot}_{V} \leqslant \cb^{3} \cf \norm{\nu_{1}}_{L^{\infty}(\Omega)} \norm{\nu_{2}}_{L^{\infty}(\Omega)}.
	\]
	With these estimates, we deduce from \eqref{eq:first_estimate} the following bound
	\[
		\norm{z}_{V}^{2} \leqslant \cb^{3} \cf \norm{\nu_{1}}_{L^{\infty}(\Omega)}^{2} \norm{\nu_{2}}_{L^{\infty}(\Omega)},
	\]
	from which we obtain the estimate
	\begin{equation}\label{eq:second_order_Frechet_derivative_estimate}
	\begin{aligned}
		\dfrac{\norm{DF(\mu + \nu_{1})\nu_{2}-DF(\mu)\nu_{2} - \uddot}_{V}}{\norm{\nu_{1}}_{L^{\infty}(\Omega)}}
		&\leqslant \dfrac{\norm{{\wddot} - \uddot}_{V}}{\norm{\nu_{1}}_{L^{\infty}(\Omega)}}\\
		&\leqslant  \cb^{3} \cf \norm{\nu_{1}}_{L^{\infty}(\Omega)} \norm{\nu_{2}}_{L^{\infty}(\Omega)}.
	\end{aligned}
	\end{equation}
	This shows that $F$ is twice-differentiable at $\mu$ and $\uddot = D^{2}F(\mu)[\nu_{1}, \nu_{2}]$.

	Now, choosing $v = \uddot$ in \eqref{eq:second_order_sensitivity_equation_variational_form}, and then utilizing estimate \eqref{eq:estimate_state_derivative}, we obtain
	\begin{equation}\label{eq:estimate_state_second_derivative}
		\norm{\uddot}_{V} \leqslant \cb^{3} \cf \norm{\nu_{1}}_{L^{\infty}(\Omega)} \norm{\nu_{2}}_{L^{\infty}(\Omega)}.
	\end{equation}
	Choosing $c = {2} \cb^{3} \cf \norm{\nu_{1}}_{L^{\infty}(\Omega)} \norm{\nu_{2}}_{L^{\infty}(\Omega)}$, we conclude that $D^{2}F(\mu)$ is uniformly bounded.
	This proves the proposition.
\end{proof}
\subsubsection{{Strict convexity of the regularized functional}}\label{subsubsec:strict_convexity_proof}
\begin{proof}[Proof of Proposition~\ref{prop:strict_convexity_of_J}]
For brevity, we write $u=u(\mu)$.
	Observe that the first and the third term of $\Jedprime(\mu)\ddmu$ are non-negative, hence, we only need to examine the second term.
	We claim that it is positive.
	In view of \eqref{eq:boundedness_of_u} and \eqref{eq:estimate_state_second_derivative}, we get
	\begin{align*}
		\abs{\indO{u{}}{\uddot}}
		\leqslant \norm{u{}}_{L^{2}(\bigdO)}\norm{\uddot}_{L^{2}(\bigdO)}
		\lesssim \norm{u{}}_{{V}}\norm{\uddot}_{{V}}
		&\leqslant \cb^{3} \cf \norm{u{}}_{V}  \norm{\nu}_{L^{\infty}(\Omega)}^{2}\\
		&\leqslant \cb^{4} \cf^{2} \norm{\nu}_{L^{\infty}(\Omega)}^{2}.
	\end{align*}
	Let $\epszed = \cb^{4} \cf^{2} > 0$.
	Invoking our key assumption \eqref{key_assumption}, we get the following lower estimate
	{\small \[
	\Jedprime(\mu)\ddmu
	\geqslant \norm{\udot}_{L^{2}(\bigdO)}^{2} - \abs{\indO{u{}}{\uddot} } + \rho\norm{\dmu}_{L^{2}(\Omega)}^{2}\\
	\geqslant \norm{\udot}_{L^{2}(\bigdO)}^{2} + ( \rho - \epszed )\norm{\dmu}_{L^{\infty}(\Omega)}^{2}.
	\]}
	Choosing $\rho > \epszed > 0$, we get $\Jedprime(\mu)\ddmu > 0$ -- proving that $\Je$ is strictly convex.
\end{proof}
%
%
%

\subsection{Computation of the {Lagrangian} and {Eulerian derivative}s of the state}\label{appx:material_and_shape_derivative}
In this section, we use the shorthand $\int_{\Omega} = \int_{\Omega_{\pm}} = \int_{\Omega_{+}} + \int_{\Omega_{-}}$, assuming the context clarifies the division.
\begin{proof}[Proof of Theorem~\ref{thm:state_shape_derivative}]
	The proof consists of two primary steps: first, we characterize the material derivative of the state, followed by the derivation of the {Eulerian derivative} of the state; see \cite{AfraitesDambrineKateb2007} for a closely related derivation in the context of a transmission problem.
	
	\textit{First step:} Let $u_{t} = u(\omega_{t})$, where $t \in \textsf{I}$, and {$\Omega \in \mathcal{O}_{ad}^{k}$}, satisfying \eqref{prob:weak_form_optimal_tomography}.
	To prove the given proposition, we first show the existence of the material derivative $\dot{u}$ of $u$ which is defined as follows (see, e.g., \cite[Eq.~(3.38), p.~111]{SokolowskiZolesio1992}):
	\begin{equation}\label{eq:definition_of_the_material_derivative}
		\dot{u} = \dot{u}(\Omega)[\VV]  = \lim_{t \searrow 0} \frac{u(\Omega_{t}) \circ T_t - u(\Omega)}{t}
	\end{equation}
	provided the limit $\dot{u}$ exists in $H^{1}(\Omega)$ where $(u(\Omega_{t}) \circ T_t)(x) = u(\Omega_{t})(T_t(x))$, $x \in \Omega$.
	
	Let us consider $u_{t} \in V_{t} \coloneqq  H^{1}(\Omega_{t})$, the solution of the perturbed problem for a given variation {$\VV \in {{\sfTheta}^{k}}$} given by the solution of
	\begin{equation}\label{eq:perturbed_problem}
	a_{t}(u_{t},v_{t}) = l_{t}(v_{t}), \quad \forall v_{t} \in V_{t}.
	\end{equation}
	where
	\[
        \left\{
        \begin{aligned}
        	a_{t}(u_{t},v_{t}) &= \intOt{( {\alpha_{t}} \nabla{u_{t}} \cdot \nabla{v_{t}} + {\mu_{t}} {u_{t}}{v_{t}})}
        				+ \frac{1}{\zeta} \intGt{{u_{t}}{v_{t}}}, \quad \text{for } u_{t},v_{t} \in V_{t},\\
        	l_{t}(v_{t}) &= \intOt{f_{t}v_{t}}, \quad \text{for } v_{t} \in {V_{t}}.
        \end{aligned}
        \right.
        \]
        Here, $\alpha_{t}$, $\mu_{t}$, and $f_{t}$ are defined as $\alpha$, $\mu$, and $f$ but replacing $\Omega$ by the perturbed domain $\Omega_{t}$, and the gradient, $\nabla$, is taken with respect to the spatial variable $x \in \Omega$.

        By applying the change of variables (cf. \cite[subsec.~9.4.2--9.4.3, {pages}.~482--484]{DelfourZolesio2011}), one can write equation \eqref{eq:perturbed_problem} as follows:
	\begin{equation}\label{eq:transformed_problem}
		a^{t}(\ut,v) = l^{t}(v), \quad \forall v \in V,
	\end{equation}
	where
	\[
        \left\{
        \begin{aligned}
        	a^{t}(\ut,v) &= \intO{( {\alphat} \At \nabla{\ut} \cdot \nabla{v} + \dett {\mut} {\ut}{v})}
        				+ \frac{1}{\zeta} \intG{\bt {\ut}{v}}, \quad \text{for } \ut,v \in V,\\
        	l^{t}(v) &= \intO{\dett f^{t}v}, \quad \text{for } v \in {V}, \quad (\varphi^{t} = \varphi_{t} \circ T_{t} : \Omega \to \mathbb{R}).
        \end{aligned}
        \right.
        \]
        Here, observe that $\bt = \dett \abs{({D}T_t)^{-\top} \nn} = 1$ because $\VV$ vanishes on $\partial\Omega$.

        Now, for all $t \in [0, t_{0})$, with $t_{0}$ sufficiently small, one can show that $\wt = \ut - u \in V$ is a unique solution to the variational equation $a^{t}(\ut,v) - a(u,v) = l^{t}(v) - l(v)$, for all $v \in V$, which can equivalently be written as
        \begin{equation}\label{eq:difference_equation}
        		\tilde{a}(\wt, v) = \tilde{l}(v), \quad \forall v \in V,
        \end{equation}
	where
	\begin{equation}\label{eq:difference_equation_for_material_derivative}
        \left\{
        \begin{aligned}
        	\tilde{a}(\wt, v) &= \intO{( {\alphat} \nabla{\wt} \cdot \nabla{v} + {\mut} {\wt}{v})} + {\frac{1}{\zeta} \intG{{\wt}{v}}}, \quad \text{for } \wt, v \in V,\\
        	\tilde{l}(v) &= -\intO{(\alphat - \alpha)  \At \nabla{\ut} \cdot \nabla{v}}
				-\intO{ \alpha (\At - id)  \nabla{\ut} \cdot \nabla{v}}\\
			&\quad -\intO{ \dett (\mut - \mu) {\ut} {v}}
				-\intO{ (\dett - 1) \mu {\ut} {v}}
				\\
			&\quad +\intO{ \dett (f^{t} - f) {v}}
				+\intO{ (\dett - 1) f {v}}, \quad \text{for } \ut, v \in V.\\
        \end{aligned}
        \right.
        \end{equation}
	The well-posedness of \eqref{eq:difference_equation} essentially follows from the Lax-Milgram theorem, by applying standard arguments and noting that $\lim_{t \searrow 0} \At = \text{id}$ and $\lim_{t \searrow 0} \dett = 1$ uniformly on {${\Omega}$}, as well as the regularity assumptions on $\mu$ and $f$ given in Assumption~\ref{assumption:weak_assumptions} .
	As a consequence, one can deduce that $\norm{\wt}_{V} \lesssim \norm{u}_{V}$ ($t \in\textsf{I}=[0,t_{0})$).
	This means that the set $\{\wt \mid t \in \textsf{I}\}$ is bounded in $V$ for sufficiently small $t_{0}$.
	
	Let us define $\zt = \frac{1}{t} \wt$ for $t \in (0,t_{0})$ which also belongs to $V$.
	Then, we have
	\begin{equation}\label{eq:tilde_alpha}
	\begin{aligned}
		  	\tilde{a}(\zt, v)
			&=  -\intO{\left(\dfrac{\alphat - \alpha}{t} \right)  \At \nabla{\ut} \cdot \nabla{v}}
				-\intO{ \alpha \left(\dfrac{\At - id}{t} \right)  \nabla{\ut} \cdot \nabla{v}}\\
			&\quad -\intO{ \dett \left(\dfrac{\mut - \mu}{t} \right) {\ut} {v}}
				-\intO{ \left(\dfrac{\dett - 1}{t} \right) \mu {\ut} {v}}
				\\
			&\quad +\intO{ \dett \left(\dfrac{f^{t} - f}{t} \right) {v}}
				+\intO{ \left(\dfrac{\dett - 1}{t} \right) f {v}}\\
			&= \frac{1}{t} \tilde{l}(v) =: l_{t}(v), \quad (\forall v \in V).
	\end{aligned}
	\end{equation}
	By selecting $v = \zt$ as the test function in the equation above, we can infer the boundedness of the sequence $\{ \zt \}$ in $V$.
	Specifically, we consider a sequence $\{ t_n \}$ such that $\lim_{n \to \infty} t_n = 0$, and our goal is to demonstrate that $\lim_{n \to \infty} z^{t_n}$ exists.
        The properties of the transformation $T_{t}$ given in \eqref{eq:regular_maps}, together with the boundedness of $\wt$ in $V$ implies that $\nabla \zt$ is bounded in $L^{2}(\Omega)^{d}$, equivalently the sequence $\{z^{t_n}\}$ is bounded in $V$.
        Thus, there is a subsequence, which we still denote by $t_{n}$ with $t_{n} \searrow 0$ and an element $z \in V$ such that $z^{t_n} \rightharpoonup z$ weakly in $V$.
        Since $\nabla u^{t_n} \to \nabla u$ in $L^{2}(\Omega)^{d}$, $\lim_{t_n \searrow 0} {I}_{t_n} = 1$ and $\lim_{t_{n} \searrow 0} A_{t_n} = id$ uniformly on $\Omega$, and the derivatives of the maps $[t \mapsto \dett]$ and $[t \mapsto \At]$ given in \eqref{eq:derivatives_of_maps} we get
	\begin{equation}\label{eq:material_derivative_first_form}
	\begin{aligned}
		  	a_{0}(z,v) &\coloneqq  \intO{( {\alpha\nabla{z} \cdot \nabla{v} + {\mu} {z} {v})} }
						+ {\frac{1}{\zeta} \intG{ {z}{v}}}\\
			&=  -\intO{ \left( \nabla{\alpha} \cdot \VV (\nabla{u} \cdot \nabla{v}) + \alpha A \nabla{u} \cdot \nabla{v} \right) }\\
			&\quad - \intO{ \left( \nabla{\mu} \cdot \VV {u} {v} + \operatorname{div}{\VV} \mu {u} {v} \right)}\\
			&\quad + \intO{ \left( \nabla{f}\cdot\VV {v} + \operatorname{div}{\VV} f {v} \right) }\\
			&=: j_{1}(v) + j_{2}(v) + j_{3}(v) =: l_{0}(v), \qquad (\forall v \in V).
	\end{aligned}
	\end{equation}
	In above, the limit equation
	\begin{equation}\label{eq:divergence_expansion}
	\lim_{t \to 0} \frac{1}{t}\left( \dett \varphi^{t} - \varphi\right)
	= \operatorname{div}(\varphi\VV)
	= \varphi \operatorname{div}{\VV} + \nabla{\varphi} \cdot \VV
	\end{equation}
	was used which holds for any differentiable mapping $t \mapsto \dett \varphi^{t}$ from interval $\textsf{I}$ to $L^{2}(\Omega)$ with $\varphi \in V$ and $\VV \in {{\sfTheta}^{k}}$ (cf. \cite[Cor.~3.1]{IKP2006}).
	Since this equation has a unique solution, we deduce the weak convergence $z^{t_{n}} \rightharpoonup z$ in $V$ for any sequence $\{t_{n}\}$.
	Meanwhile, the strong convergence follows from the fact that
	$
		  	a_{0}(z,z)
			= \lim_{t_{n} \searrow 0}\tilde{a}({z^{t_{n}}} , {z^{t_{n}}})
			=  \lim_{t_{n} \searrow 0} l_{t_{n}}({z^{t_{n}}})
			= l_{0}(v).
	$ together with the weak convergence previously shown.
	This proves the characterization of the (unique) material derivative $z = \dot{u} \in V$ of $u \in V$ given in equation \eqref{eq:material_derivative}.

	\textit{Second step:}  Next, we shall derive the structure of the {Eulerian derivative} of the state.
	First, we recall that the function $u$ has a \textit{shape} derivative $\uprime$ at $0$ in the direction of the vector field $\VV \in {{\sfTheta}^{k}}$ if the limit
	\[
		\uprime = \lim_{t \searrow 0} \frac{u(\Omega_{t}) - u(\Omega)}{t},
	\]
	exist.
	This expression and the material derivative $\dot{u}$ are related by $\uprime = \dot{u} - (\nabla u \cdot \VV)$ provided that $\nabla u \cdot \VV$ exists in some appropriate function space \cite[Eq.~(3.38), p.~111]{SokolowskiZolesio1992}.
	We comment that the {Eulerian derivative} $\uprime$ of \eqref{eq:main} is not continuous across the interface $\domega$.
	As a consequence, $\uprime$ cannot be in $H^{1}(\Omega)$.
	Nonetheless, it belongs to $H^{1}(\Omega_{+}) \cup H^{1}(\Omega_{-})$.

	To proceed with the derivation of $\uprime$, we rewrite equation \eqref{eq:material_derivative_first_form} in another form.
	To do so, we observe that by applying the chain rule in conjunction with \eqref{eq:derivatives_of_maps}, the following expansions and identities hold (here $u$ is restricted to $\Omega_{\pm}$):
	\begin{equation}\label{eq:nabla_expansions}
	\begin{aligned}
		&- \left( \nabla{\alpha} \cdot \VV (\nabla{u} \cdot \nabla{v})
		+ \alpha A \nabla{u} \cdot \nabla{v} \right)\\
		& \qquad \qquad =
		\nabla{u} \cdot \left( \alpha (D\VV + D\VV^{\top}) - \operatorname{div}{(\alpha \VV)} id \right) \nabla{v} \\
		& \qquad \qquad = \operatorname{div}\left( \alpha (\VV \cdot \nabla{u}) \nabla{v}
				+ {\alpha (\VV \cdot \nabla{v}) \nabla{u} - \alpha (\nabla{u} \cdot \nabla{v}) \VV} \right)\\
		& \qquad \qquad - (\VV \cdot \nabla{u}) \operatorname{div}{(\alpha \nabla{v})}
			- (\VV \cdot \nabla{v}) \operatorname{div}{(\alpha \nabla{u})},
		\\	
		&\alpha \nabla(\nabla{u} \cdot \VV) \cdot \nabla{v} \\
		& \qquad \qquad = \alpha \nabla^{2}{u} \VV \cdot \nabla{v} + \alpha \nabla{\VV} \nabla{u} \cdot \nabla{v}\\
		& \qquad \qquad = \alpha (\nabla{v})^{\top} \nabla^{2}{u} \VV + \alpha D{\VV}\nabla{v} \cdot \nabla{u}\\
		& \qquad \qquad = \operatorname{div}{(\alpha(\VV \cdot \nabla{u}) \nabla{v})} - (\VV \cdot \nabla{u}) \operatorname{div}{(\alpha \nabla{v})},
		%
		%
	\end{aligned}
	\end{equation}
	where $\nabla{\VV} = (D\VV)^{\top} = {\partial{\theta}_{j}}/{\partial{x}_{i}}$ and $\nabla^{2}{u}$ denotes the Hessian of ${u}$ which is symmetric.
	On the other hand, taking $\nabla{v}_{\pm} \cdot \VV \in H^{1}(\Omega_{\pm})$, where $v_{\pm} \in V \cap H^{2}(\Omega_{\pm})$, $\VV \in {{\sfTheta}^{k}}$ (i.e., $\VV = \vect{0}$ in $\domega$), as a test function in Problem~\ref{prob:weak_form_optimal_tomography} yields the following equation
	\begin{equation}\label{eq:weak_form_substitution}
		\intO{( {\alpha} \nabla{u} \cdot \nabla{(\nabla{v} \cdot \VV)} + {\mu} {u}{(\nabla{v} \cdot \VV)})}
		= \intO{f(\nabla{v} \cdot \VV)}.
	\end{equation}
	Moreover, we have the following equivalent expressions
	\begin{align*}
		j_{2}(v) &= - \intO{ \operatorname{div}{(\mu{u}{v} \VV)} } + \intO{ \left( \mu (\VV \cdot \nabla{u}) {v} + \mu{u} (\VV \cdot \nabla{v}) \right)},\\
		j_{3}(v) &= \intO{ \operatorname{div}{(fv \VV)} } - \intO{f (\nabla{v} \cdot \VV)}.
	\end{align*}
	Utilizing the above identities in \eqref{eq:material_derivative_first_form} with $z$ replaced by $\dot{u}_{\pm} = {u^{\prime}_{\pm}} + \nabla{{u}_{\pm}} \cdot \VV \in H^{1}({\Omega}_{\pm})$ (observe that $(\nabla{u} \cdot \VV) \not\in H^{1}(\Omega)$ but ${u}_{\pm} \in H^{k+1}({\Omega}_{\pm})$, $k \geqslant 2$), we get
	\begin{equation}\label{eq:weak_form_full_expression}
	\begin{aligned}
		&a_{0}(\uprime, v) + \intO{( {\alpha\nabla{(\nabla{u} \cdot \VV)} \cdot \nabla{v} + {\mu} {(\nabla{u} \cdot \VV)} {v})} }\\
		&  = \intO{( {\alpha\nabla{(\nabla{u} \cdot \VV)} \cdot \nabla{v} + {\mu} {(\nabla{u} \cdot \VV)} {v})} } \\
		& \quad - \intO{ \left[ (\VV \cdot \nabla{v}) \operatorname{div}{(\alpha \nabla{u})} - \mu{u} (\VV \cdot \nabla{v}) + f (\nabla{v} \cdot \VV) \right]}\\
		& \quad + \intO{ \operatorname{div}{\left( \alpha (\VV \cdot \nabla{v}) \nabla{u} - \alpha (\nabla{u} \cdot \nabla{v}) \VV \right) }}\\
		& \quad - \intO{ \left[ \operatorname{div}{(\mu{u}{v} \VV)} - \operatorname{div}{(fv \VV)}  \right]}.
	\end{aligned}
	\end{equation}
	%
	
	For $\VV \in {{\sfTheta}^{k}}$, observe that by using  integration by parts, we have
	\[
		-\intO{  (\VV \cdot \nabla{v}) \operatorname{div}{(\alpha \nabla{u})}  } = \intO{ \alpha \nabla{u} \cdot \nabla{(\VV \cdot \nabla{v})}}.
	\]
	 Then, in view of \eqref{eq:weak_form_substitution}, equation \eqref{eq:weak_form_full_expression} can be simplified as follows
	\begin{equation}\label{eq:shape_derivative_of_the_state_initial_form}
	\begin{aligned}
		a_{0}(\uprime, v)
		 &= \intO{ \left[ \operatorname{div}{\left( \alpha (\VV \cdot \nabla{v}) \nabla{u} - \alpha (\nabla{u} \cdot \nabla{v}) \VV \right) }  \right] }\\
		 &\quad - \intO{ \left[ \operatorname{div}{(\mu{u}{v} \VV)}
			- \operatorname{div}{({f}{v} \VV)}  \right] },
	\end{aligned}
	\end{equation}
	which provides an inititial characterization of the {Eulerian derivative} of the state.
	
        To complete the proof, we express the right side of \eqref{eq:shape_derivative_of_the_state_initial_form} as a boundary integral over $\domega$.
        This is accomplished by applying the divergence theorem after partitioning the integral into two domains of integration: $\Omega \setminus \baromega$ and $\omega$. We then utilize the notation $\jump{\cdot}$, which denotes the difference between the traces of a function at the boundary interface $\domega$ as we approach from $\Omega \setminus \baromega$ and $\omega$, respectively.
	Specifically, by applying the divergence theorem in both $\Omega$ and $\Omega \setminus \baromega$, followed by integration by parts, we obtain
	\begin{equation}\label{eq:after_IBP}
	\begin{aligned}
		&\intO{( - \operatorname{div}{( \alpha\nabla{\uprime})} + {\mu} {\uprime} ) {v}  } + {\intG{ \left( \alpha \dn{\uprime} + \frac{1}{\zeta} {\uprime} \right) {v}}
						- \intdomega{\jump{ \alpha\dfrac{\partial{\uprime}}{\partial{\nn}} } {v} } }\\
		& = a_{0}(\uprime, v)\\
		& = \intO{( {\alpha\nabla{\uprime} \cdot \nabla{v} + {\mu} {\uprime} {v})} }
						+ {\frac{1}{\zeta} \intG{ {\uprime}{v}}}\\
		& = - \intdomega{ \left\{  (\VV \cdot \nabla{v})  \jump{\alpha \nabla{u}} \cdot \nn - \jump{\alpha \nabla{u}} \cdot \nabla{v} \Vn \right\} }\\
		&\quad   + \intdomega{ \left\{  \jump{\mu{u}} {v} \Vn  - \jump{f} {v} \Vn  \right\} },
	\end{aligned}
	\end{equation}
	where $\Vn = \VV \cdot \nn$.
	By comparing the left-most and right-most sides of the equation, while varying $v$ (which we assume to be sufficiently smooth--at least in $H^2(\Omega)$) over $\Omega= (\Omega\setminus{\baromega}) \cup {\baromega}$ and the over $\bigdO$, we deduce that the following equations hold at least in distributional sense:
	\[
        \left\{
        \begin{aligned}
          -\dive{\left( {\alpha} \nabla {\uprime} \right)} + \mu {\uprime} &= 0, \quad \text{in $\Omega \setminus \baromega$ and in $\omega$},\\
          {\alpha} \dn{{\uprime}} + \frac{1}{\zeta} {\uprime} &= 0, \quad \text{on } \bigdO.
        \end{aligned}
        \right.
	\]
	We next derive the equation for $\uprime$ on $\domega$.
	First, let us note that, by elliptic regularity result, ${u}_{\pm} \in H^{k+1}({\Omega}_{\pm})$ (for $d \in \{2,3\}$).
	Then, for some $k \in \mathbb{N}$, $k \geqslant 2$, we have $u \in {{C}}^{1,\alpha}(\overline{\Omega}_{\pm})$, $0 < \alpha \leqslant k - {d}/2$, $d \in \{2,3\}$, because of the Sobolev embedding $H^{k+1}({\Omega}_{\pm}) \hookrightarrow {{C}}^{1,\alpha}(\overline{\Omega}_{\pm})$ (see, e.g., \cite[Thm.~2.84, p.~98]{DemengelDemengel2012} or \cite[Thm.~4.12, p.~85]{AdamsFournier2003}).
	Now, because $\jump{u} = 0$ on $\domega$, we have $ \jump{ \nabla{u} } = \jump{ (\partial{u}/\partial{\nn})\nn }$ on $\domega$; that is, $\jump{\nabla_{\tau}{u}} = 0$ on $\domega$.
	Hence, $\jump{\dot{u}} = 0$ on $\domega$, and so $\jump{\uprime} = \jump{\dot{u}} - \jump{\nabla{u} \cdot \VV} = - \jump{\nabla{u} \cdot \VV}$ on $\domega$.
	By these equations, we deduce that
	\[
		\jump{\uprime} = - \Vn \jump{ \dfrac{\partial{u}}{\partial{\nn}} } \quad \text{on }\domega.
 	\]

	Next, we note that, from tangential Stokes' formula \cite{MuratSimon1976}, we have 
	\[
		\intdomega{ \VV \cdot \nabla_{\tau}{\varphi}} = - \intdomega{ \varphi \operatorname{div}_{\tau}{\VV}},
	\]
	when $\VV \cdot \nn = 0$ (i.e., $\VV$ is a tangential field).
	Here, the operators $\nabla_{\tau}$ and $\operatorname{div}_{\tau}$ are respectively the tangential gradient and tangential divergence operators (see, e.g., \cite{DelfourZolesio2011,HenrotPierre2018,SokolowskiZolesio1992}).
	We observe that $( (\jump{\alpha \nabla{u}} \cdot \nn)\VV - \Vn \jump{\alpha \nabla{u}} )\cdot \nn = 0$ on $\domega$.
		Hence, we can replace $\nabla{v}$ by $\nabla_{\tau}{v}$.
		In addition, we know that $\jump{\nabla_{\tau}{u}} = 0$ on $\domega$ which implies that
		\begin{align*}
		- \intdomega{ \Vn  \jump{\alpha} \nabla_{\tau}{u} \cdot \nabla{v}   }
		& = - \intdomega{ \Vn  \jump{\alpha \nabla{u}} \cdot \nabla{v}   } \\
		& = \intdomega{ \left\{  (\VV \cdot \nabla{v})  \jump{\alpha \nabla{u}} \cdot \nn
						- \left( \jump{\alpha \nabla{u}} \cdot \nabla{v} \right) \Vn   \right\} }\\
		& = \intdomega{ {v} \operatorname{div}_{\tau}( \Vn \jump{\alpha} \nabla_{\tau}{u} ) }.
		\end{align*}
	Using this identity, we arrive at the following equation
	\[
	 \intdomega{ \jump{\alpha\dfrac{\partial{\uprime}}{\partial{\nn}} } {v} }
		= \intdomega{ \left( \operatorname{div}_{\tau}{( \Vn \jump{\alpha} \nabla_{\tau}{u} )}
		 	-   \jump{\mu{u}}  \Vn
			+  \jump{f}  \Vn \right)  {v} },
	\]
	which holds for all $v \in V$.
	By varying $v$, we deduce that
	\[
		\jump{ \alpha\dfrac{\partial{\uprime}}{\partial{\nn}} }
		= K(u)[\VV]
		= \operatorname{div}_{\tau}{( \Vn \jump{\alpha} \nabla_{\tau}{u} )}
		 	-   \jump{\mu{u}}  \Vn
			+  \jump{f}  \Vn, \qquad \text{on } \domega.
	\]	
	This finally establishes the structure of the {Eulerian derivative} of the state given in equation \eqref{eq:state_shape_derivative}.
	For the more general structure of the {Eulerian derivative} without the aforementioned continuity conditions, see Theorem~\ref{thm:state_shape_derivative_weak_assumptions}.
\end{proof}

\subsection{{Existence of a shape solution}}\label{subsec:existence_of_optimal_shape_solution}
In this appendix, we address the question of the existence of an optimal solution to the optimization problem
\begin{equation}\label{eq:shape_optimization_problem}
	\min_{\omega \in \mathcal{O}_{\circ}^{1},\ \Omega \in \mathcal{O}_{ad}^{1}} J({\omega}).
\end{equation} 	
To establish this existence, we must impose a key assumption regarding the regularity of the boundary interface, which is fortunately a consequence of the definition of the set of admissible domains $ \mathcal{O}_{ad}^{1} $ provided in \eqref{eq:admissible_domains}.
For the purpose of our analysis, it suffices to assume that $ \Omega $ is Lipschitz continuous, which allows us to establish the desired existence result (refer to Proposition~\ref{prop:ensuring_epsilon_cone_property}).
{In this section, we assume that $f \in H^{-1}(\Omega)$ and that $\alpha$ exhibits jump discontinuities at the boundary interface.
Therefore, without further notice, we consider Problem~\ref{prob:weak_form_optimal_tomography} with $\alpha$ defined according to Assumption~\ref{assumption:weak_assumptions}.}

Because Problem~\ref{prob:weak_form_optimal_tomography} admits a unique weak solution by Lemma~\ref{lem:wellposedness_weak_formulation}, we can define the map $\omega \longmapsto {u}\coloneqq {u}(\omega)$, or equivalently, $\Omega \longmapsto {u}\coloneqq {u}(\Omega)$ (note that ${\Omega = (\Omega\setminus{\baromega}) \cup \baromega}$), and denote its graph by
\begin{align*}
	\mathcal{G}&=\{(\omega,{u}): \text{ $\omega \in {\mathcal{O}_{\circ}^{1}}$ and ${u}$ solves Problem~\ref{prob:weak_form_optimal_tomography}}\},
\end{align*}

The primary result we aim to establish is as follows:
\begin{theorem}\label{theorem:existence_of_optimal_shape_solution}
    The minimization problem \eqref{eq:shape_optimization_problem} admits at least one solution in $\mathcal{G}$.
\end{theorem}
To demonstrate the validity of this assertion, we first need to endow the set $\mathcal{G}$ with a topology that ensures its compactness and the lower semi-continuity of the functional ${J}$.
To achieve this, we introduce a topology on $\mathcal{G}$ induced by the Hausdorff convergence, denoted as $\Omega^{(n)} \stackrel{\text{H}}{\longrightarrow} \Omega$.
This framework enables us to prove the existence of the optimal solution to \eqref{eq:shape_optimization_problem} across arbitrary dimensions ($d \in \{2, 3\}$).
To prepare for our discussion, we will briefly review the definitions of Hausdorff distance, Hausdorff convergence, and the $\varepsilon$-cone property.
For further elaboration on these concepts, readers are referred to \cite[Ch.~3]{Pironneau1984}.
\begin{definition}
[{\cite[Def.~2.2.7, p.~30]{HenrotPierre2018}}]
\label{def:Hausdorff_distance}
	Let ${{\omega}}_{1}$ and ${{\omega}}_{2}$ be two (compact) subsets of $\mathbb{R}^{d}$, $d \geqslant 2$.
	The Hausdorff distance $\operatorname{dist}_{\text{H}}({{\omega}}_{1},{{\omega}}_{2})$ between ${{\omega}}_{1}$ and ${{\omega}}_{2}$ is defined as follows $\operatorname{dist}_{\text{H}}({{\omega}}_{1},{{\omega}}_{2}) = \max\{\rho({{\omega}}_{1},{{\omega}}_{2}),\rho({{\omega}}_{2},{{\omega}}_{1})\}$
	where $\rho({{\omega}}_{1},{{\omega}}_{2}) = \sup_{s\in{{\omega}}_{1}} \operatorname{dist}(x, {{\omega}}_{2})$ and $\operatorname{dist}(x,{{\omega}}_{2}) = \inf_{y \in {{\omega}}_{2}} \abs{x-y}$.
	Note that $\operatorname{dist}_{\text{H}}$ defines a topology on the closed bounded sets of $\mathbb{R}^{d}$.
\end{definition}
\begin{definition}[{\cite[Def.~2.2.8, p.~30]{HenrotPierre2018}}]
\label{def:Hausdorff_convergence}
	Let $\{{{\omega}}^{(n)} \}$ and ${{\omega}}$ be open sets included in $\Omega \subset \mathbb{R}^{d}$, $d \geqslant 2$.
	We say that the sequence ${{\omega}}^{(n)}$  converges in the sense of Hausdorff to ${{\omega}}$ if $\operatorname{dist}_{\text{H}}(\Omega\setminus{{\omega}}^{(n)}, \Omega\setminus{{\omega}}) \longrightarrow 0$ as $n \longrightarrow \infty$.
	We will denote this convergence by ${{\omega}}^{(n)} \stackrel{\text{H}}{\longrightarrow} {{\omega}}$ or simply by ${{\omega}}^{(n)} \longrightarrow {{\omega}}$ when there is no confusion.
\end{definition}
\begin{definition}[{\cite[Def.~2.4.1, p.~54]{HenrotPierre2018}}]
\label{definition:cone_property}
Let $\xi$ be a unitary vector in $\mathbb{R}^{d}$, $d \geqslant 2$, $\varepsilon > 0$ be a real number, and $y \in \mathbb{R}^{d}$.
A cone $C$ with vertex $y$, direction $\xi$, and dimension $\varepsilon$ is the set defined by
\[
	C(y,\xi,\varepsilon) = \{ x\in \mathbb{R}^{d} \mid \langle x-y, \xi\rangle_{\mathbb{R}^{d}} \geqslant \cos(\varepsilon) \|x-y\|_{\mathbb{R}^{d}} \ \ \text{and} \ \ 0 < \|x-y\|_{\mathbb{R}^{d}} < \varepsilon \},
\]
where $\langle \cdot,\cdot\rangle_{\mathbb{R}^{d}}$ is the Euclidean scalar product of $\mathbb{R}^{d}$ and $\|\cdot\|_{\mathbb{R}^{d}}$ is the associated {Euclidean norm}.

An open bounded set $\Omega \subset \mathbb{R}^{d}$ satisfies the $\varepsilon$-cone property, if for $x \in \bigdO$, there exists a unitary vector $\xi_{x} \in \mathbb{R}^{d}$ such that for all $y \in \overline{\Omega} \cap B_{\varepsilon}(x)$, we have $C(y,\xi,\varepsilon) \subset \Omega$, where $B_{\varepsilon}(x)$ denotes the open ball with center $x$ and radius $\varepsilon$.
\end{definition}
Given the definitions provided above, we hereby assert the ensuing proposition, pivotal in substantiating the proof of Theorem~\ref{theorem:existence_of_optimal_shape_solution}.
\begin{proposition}[{\cite[Thm.~2.4.7, p.~56]{HenrotPierre2018}}]\label{prop:ensuring_epsilon_cone_property}
	An open bounded set $\Omega \subset \mathbb{R}^{d}$ has the $\varepsilon$-cone property if and only if it has a Lipschitz boundary.
\end{proposition}
Proposition~\ref{prop:ensuring_epsilon_cone_property} guarantees that each admissible subdomain $\omega \in \mathcal{O}_{\circ}^{1}$ satisfies the $\varepsilon$-cone property, which is sufficient to establish Theorem~\ref{theorem:existence_of_optimal_shape_solution}.
We emphasize that given a sequence of open sets $\{\omega^{(n)}\}$ in $\mathcal{O}_{\circ}^{1}$, there exists an open set $\omega \in \mathcal{O}_{\circ}^{1}$ and a subsequence $\{\omega^{{(m)}}\}$ such that $\omega^{{(m)}} \to \omega$.
This convergence implies $\partial \omega^{{(m)}} \to \domega$.
These convergences also hold for characteristic functions and compact sets, as shown in \cite[Thm.~2.4.10, p.~59]{HenrotPierre2018}.
Moreover, the implied convergence ``$\omega^{(n)} \to \omega$ implies $\partial \omega^{(n)} \to \partial \omega$'' holds in the Hausdorff sense for domains with Lipschitz boundaries \cite[Ex.~3.2]{Holzleitner2001} or satisfying the cone property \cite{Chenais1975}.
For a detailed discussion on Hausdorff convergence, see \cite[Sec.~2.2.3, Def.~2.2.8, p.~30]{HenrotPierre2018}.
It is also worth noting that for a sequence of measurable sets $\{\omega^{(n)}\}$, the corresponding sequence of characteristic functions $\chi_{\omega^{(n)}}$ is weakly-${\ast}$ relatively compact in $L^{\infty}(\mathbb{R}^{d})$.
This means that we can find an element $\chi \in L^{\infty}(\mathbb{R}^{d})$ and a subsequence $\{\omega^{{(m)}}\}_{k \geqslant 0} \subset \{\omega^{(n)}\}_{n \geqslant 0}$ such that (cf.  \cite[Eq.~(2.3), p.~27]{HenrotPierre2018})
\begin{equation}\label{eq:weak_star_convergence_of_characteristic_functions}
	\text{for all $\psi \in L^{1}(\mathbb{R}^{d})$},
	\quad \lim_{{m \to \infty}} \int_{\mathbb{R}^{d}} \chi_{\omega^{{(m)}}} \psi \, d{x}
	=  \int_{\mathbb{R}^{d}} \chi_{\omega} \psi \, d{x}.
\end{equation}
In the above, the limit $\chi$ is generally not a characteristic function, as it takes values in $(0,1)$ \cite[Prop.~2.2.28, p.~45]{HenrotPierre2018}.
However, if the convergence occurs ``strongly'' in the sense of $L_{loc}^{p}$ for some $p \in [1, \infty)$, then $\chi$ becomes a characteristic function in the limit.
In this case, a subsequence can be extracted that converges almost everywhere, implying that $\chi$ takes on only the values $0$ and $1$, coinciding with the characteristic function of the set where it equals $1$ \cite[p.~27]{HenrotPierre2018}.
This remark is precisely stated in the following proposition.
\begin{proposition}[{\cite[Prop.~2.2.1, p.~27]{HenrotPierre2018}}]\label{prop:local_convergence_of_characteristic_functions_in_Lp}
If $\{\omega^{(n)}\}_{n\geqslant 0}$ and $\omega$ are measurable sets in $\mathbb{R}^{d}$ such that $\chi_{\omega^{(n)}}$
weakly-$\ast$ converges in $L^{\infty}(\mathbb{R}^{d}$) in the sense of \eqref{eq:weak_star_convergence_of_characteristic_functions} to $\chi_{\omega}$, then $\chi_{\omega^{(n)}} \longrightarrow \chi_{\omega}$  in $L_{loc}^{p}(\mathbb{R}^{d})$ for any $p<+\infty$ and almost everywhere.
\end{proposition}
Now, with the previous results at our disposal, we can easily prove the following proposition.

\begin{proposition}\label{proposition:convergence_to_solution}
	Let the following assumptions be satisfied:
	\begin{itemize}
		\item
		$\{\omega^{(n)}\} \subset {\mathcal{O}_{\circ}^{1}}$ is a sequence that converges to $\staromega \in {\mathcal{O}_{\circ}^{1}}$ in the Hausdorff sense and in the sense of characteristic functions;
		\item for each $n\in\mathbb{N}$, $\Omega^{(n)} \in {\mathcal{O}_{ad}^{1}}$, $\Omega^{(n)} \coloneqq  (\Omega \setminus{\baromega^{(n)}}) \cup {\baromega}^{(n)}$, and $\un \in H^{1}(\Omega^{(n)})$ solves Problem~\ref{prob:weak_form_optimal_tomography} with $\Omega = (\Omega \setminus{\baromega^{(n)}}) \cup {\baromega}^{(n)}$.
	\end{itemize}
	Then, the sequence $\un \in H^{1}(\Omega)$ converges (up to a subsequence) to a function $\staru$ in $H^{1}(\Omega)$-weak and in $L^{2}(\Omega)$-strong such that $\staru = {u}$ solves Problem~\ref{prob:weak_form_optimal_tomography} in $\Omega = (\Omega \setminus{\baromega}) \cup \baromega$ with $\omega = \staromega$.
	Moreover, $\chi_{\omega^{(n)}} \nabla \un$ converges strongly in $L^{2}(\omega)^{d}$ to $\chi_{\omega}\nabla{u}$.
	In addition, if the following compatibility conditions $\chi_{\Omega\setminus{\baromega}^{(n)}} {u} \longrightarrow {u}|_{\Omega\setminus{\baromega}}$ and $\chi_{\omega^{(n)}} {u} \longrightarrow {u}|_{\omega}$ strongly in $H^{1}({\Omega\setminus{\baromega}})$ and in $H^{1}(\omega)$, respectively, then the convergence $\un \longrightarrow {u}$ also holds strongly in $H^{1}(\Omega)$.
\end{proposition}
\begin{proof}
    Let the given assumptions be satisfied.
    To prove this proposition, we adapt the argument structure used in the proof of \cite[Prop.~2.2.3]{AfraitesRabago2024}, reproducing key analytical steps where appropriate.

    By definition of $\un$, we have
	\[
    \left\{
    \begin{aligned}
	\mathcal{E}^{(n)}
	&\coloneqq  \intO{ {\alpha_{0}} \chi_{\Omega\setminus\baromega^{(n)}} \nabla{\un} \cdot \nabla{v} }
		+ \intO{ {\alpha_{1}}  \chi_{\omega}^{(n)} \nabla{\un} \cdot \nabla{v} } \\
	&  \quad	+ \intO{ \chi_{\Omega\setminus\baromega^{(n)}} {\mu_{0}} {\un}{v} }
		+ \intO{ \chi_{\omega}^{(n)} {\mu_{1}} {\un}{v} }
				+ \frac{1}{\zeta} \intG{{\un}{v}}\\
	& = \intO{fv}, \ \text{for all } v \in {V}.
    \end{aligned}
    \right.
	\]
    Taking ${v} = \un \in H^{1}(\Omega)$ and using the equivalence between the norm {$\vertiii{v}_{\Omega}\coloneqq   ( \norm{\nabla{v}}_{L^{2}(\Omega)}^{2} + \norm{v}_{L^{2}(\bigdO)}^{2})^{1/2}$} and the usual $H^{1}(\Omega)$-Sobolev norm, we obtain the inequality $\norm{\un}_{H^{1}(\Omega)} \lesssim \norm{f}_{H^{-1}(\Omega)}$.
  Hence, $\{\un\}$ is bounded in $H^{1}(\Omega)$.
    By the Rellich-Kondrachov and Banach-Alaoglu theorems, we may extract a subsequence $\{{{\uk}}\} \subset \{\un\}$ such that we have weak convergence ${{\uk}} \rightharpoonup \staru$ in $H^{1}(\Omega)$ and strong convergence ${{\uk}} \rightarrow \staru$ in $L^{2}(\Omega)$, for some element $\staru \in H^{1}(\Omega)$.

	We next show that the limit point $\staru \in H^{1}(\Omega)$ actually solves Problem~\ref{prob:weak_form_optimal_tomography} in ${\Omega = (\Omega\setminus{\baromega}) \cup \baromega}$ (i.e., $\staru = {u}$ where ${u}$ solves Problem~\ref{prob:weak_form_optimal_tomography}) by passing through the limit and using the pointwise almost everywhere convergence of the characteristic functions $\chi_{\Omega\setminus{\baromega^{(n)}}}$ to $\chi_{\Omega\setminus{\barstaromega}}$ and $\chi_{\omega^{(n)}}$ to $\chi_{\staromega}$.
	From Proposition \ref{prop:local_convergence_of_characteristic_functions_in_Lp}, we know that $\chi_{\omega^{(n)}}$ almost everywhere converges to $\chi_{\staromega}$ in $L^{1}(\Omega)$.
	As a consequence, we get (cf. \cite[p.~130]{HenrotPierre2018})
\begin{equation}\label{equa:L2D_convergence}
\left.
\begin{aligned}
\chi_{\Omega\setminus\baromega^{(n)}} \nabla\psi
&\longrightarrow
\chi_{\Omega\setminus\barstaromega}\nabla\psi\\[0.5em]
\chi_{\omega^{(n)}} \nabla\psi
&\longrightarrow
\chi_{\staromega}\nabla\psi
\end{aligned}
\quad \right\}
\quad
\text{strongly in }L^2(\Omega).
\end{equation}
	We show that $\staru = {u}$ actually solves Problem~\ref{prob:weak_form_optimal_tomography} by proving that $\mathcal{E}^{(n)}  \longrightarrow \mathcal{E}^{(\infty)}$ as $n \longrightarrow \infty$ where
    \[
    \left\{
    \begin{aligned}
	\mathcal{E}^{(\infty)}
	&\coloneqq  \intO{ {\alpha_{0}} \chi_{\Omega\setminus\baromega} \nabla{u} \cdot \nabla{v} }
	 	+ \intO{ {\alpha_{1}} \chi_{\omega} \nabla{u} \cdot \nabla{v} } \\
	&  \quad + \intO{ \chi_{\Omega\setminus\baromega} {\mu_{0}} {u}{v} }
		+ \intO{ \chi_{\omega} {\mu_{1}} {u}{v} }
				+ \frac{1}{\zeta} \intG{{u}{v}} \\
	& = \intO{fv}, \quad \text{for all } v \in {V}.
    \end{aligned}
    \right.
    \]
	\sloppy Using \eqref{equa:L2D_convergence}, the weak convergence $\un \rightharpoonup {\staru}$ in $H^{1}(\Omega)$, and the weak-$^{\ast}$ convergences  $\chi_{\Omega\setminus{\baromega^{(n)}}} {\rightharpoonup} \chi_{\Omega\setminus{\barstaromega}}$ and $\chi_{\omega^{(n)}} \stackrel{\ast}{\rightharpoonup} \chi_{\staromega}$ in $L^{\infty}(\Omega)$, we see that
    \[
    \left\{
    \begin{aligned}
	&\intO{ {\alpha_{0}} \chi_{\Omega\setminus\barstaromega} \nabla{\staru} \cdot \nabla{v} }
	 	+ \intO{ {\alpha_{1}} \chi_{\staromega} \nabla{\staru} \cdot \nabla{v} } \\
	&  + \intO{ \chi_{\Omega\setminus\barstaromega} {\mu_{0}} {\staru}{v} }
		+ \intO{ \chi_{\staromega} {\mu_{1}} {\staru}{v} }
				+ \frac{1}{\zeta} \intG{{\staru}{v}} \\
	& = \intO{fv}, \quad \text{for all } v \in {V}.
    \end{aligned}
    \right.
    \]
	By the uniqueness of the limits (see Lemma~\ref{lem:wellposedness_weak_formulation}), we deduce that $\mathcal{E}^{(n)} \longrightarrow \mathcal{E}^{(\infty)}$.
	Thus, we conclude that $\staru = u((\Omega \setminus {\barstaromega) \cup \barstaromega})$ -- recovering Problem~\ref{prob:weak_form_optimal_tomography}.

	The proofs of the final two statements Proposition~\ref{proposition:convergence_to_solution} follow a similar approach as the proof of the last part of Proposition 2.2.3 in \cite{AfraitesRabago2024} (see also \cite[Proof of Cor.~3.7.4., p.~130]{HenrotPierre2018}), and are therefore omitted.
	This completes the proof of the proposition.
\end{proof}
To close out this appendix, we provide the proof of Theorem~\ref{theorem:existence_of_optimal_shape_solution}.
\begin{proof}[Proof of Theorem~\ref{theorem:existence_of_optimal_shape_solution}]
	Observe that the infimum of ${J}(\omega)$ is finite.
	Hence, we can find a minimizing sequence $\{\omega^{(n)}\} \subset {\mathcal{O}_{\circ}^{1}}$ which is bounded such that $\lim_{n\to\infty} {J}(\omega^{(n)}) \\= \inf_{\omega \in {\mathcal{O}_{\circ}^{1}}} {J}(\omega)$.
	By \cite[Thm.~2.4.10, p.~59]{HenrotPierre2018}, there exists ${{\staromega}} \in {\mathcal{O}_{\circ}^{1}}$, and a subsequence $\{\omega^{{(m)}}\} \subset \{\omega^{(n)}\}$ such that $\omega^{{(m)}}$ converges to ${{\staromega}}$ in the sense of Hausdorff (Definition~\ref{def:Hausdorff_convergence}) and also in the sense of characteristic functions.
	This implies that the first assumption in Proposition~\ref{proposition:convergence_to_solution} is satisfied.
	With the second premise of Proposition~\ref{proposition:convergence_to_solution}, we know that $\un \in H^{1}(\Omega)$ (of functions $\un \in H^{1}(\Omega)$ which solves Problem~\ref{prob:weak_form_optimal_tomography} on each of its respective domain $\Omega = (\Omega\setminus{\baromega^{(n)}}) \cup {\baromega}^{(n)}$) -- taking a further subsequence if necessary -- converges to (the unique limit) ${{\staru}} \in H^{1}(\Omega)$ where ${{\staru}} = {u}((\Omega \setminus {\barstaromega) \cup \barstaromega})$ solves Problem~\ref{prob:weak_form_optimal_tomography} in ${{\Omega}=(\Omega \setminus {\barstaromega) \cup \barstaromega}}$.
	Now, to conclude, it is left to show that the shape functional ${J}(\omega)$ is lower-semicontinuous; that is, we have ${J}({{\staromega}}) \leqslant \lim_{{m \to \infty}} {J}(\omega^{{(m)}}) = \inf_{{{\widehat{\omega}}} \in {\mathcal{O}_{\circ}^{1}}} {J}({{\widehat{\omega}}}) \leqslant {J}({{\omega}})$.
	From Proposition~\ref{proposition:convergence_to_solution}, we know that the maps $(\Omega\setminus{\baromega}) \mapsto {u}(\Omega\setminus{\baromega} )$ and $\omega \mapsto {u}(\omega)$ are continuous.
	Therefore, the map $\omega \mapsto {J}(\omega)$ is also continuous, in particular, it is lower-semicontinuous.	
	This proves Theorem~\ref{theorem:existence_of_optimal_shape_solution}.
\end{proof}
\section*{Acknowledgments}
The authors sincerely thank the anonymous referees for their careful reading of the manuscript and for their insightful and constructive comments, which have significantly improved the clarity and overall quality of this paper. They are also grateful for the referees' suggestions of the related references \cite{Arridgeetal2008,BalRen2006,BelhachmiDhifMeftahi2023}, which have strengthened the discussion and context of the work. The authors further warmly thank Lekbir Afraites (Sultan Moulay Slimane University, Morocco) for his valuable comments, fruitful discussions, and for bringing reference \cite{AfraitesDambrineKateb2007} to their attention, which contributed to improving the first draft of this paper.

JFTR is supported by the JSPS Postdoctoral Fellowships for Research in Japan and partially by the JSPS Grant-in-Aid for Early-Career Scientists under Japan Grant Number JP23K13012.
HN is partially supported by JSPS Grants-in-Aid for Scientific Research under Grant Numbers JP26H02185, JP25K00920, JP24H00188, and JP21H04431.
JFTR and HN are also partially supported by the JST CREST Grant Number JPMJCR2014.










\bigskip

\medskip

\end{document}